\documentclass{amsart}
\usepackage{eucal}
\usepackage{amsmath}
\usepackage{amsthm,amssymb,stmaryrd,color,enumerate,accents,mathtools,times}

\usepackage[usenames,dvipsnames]{xcolor}
\usepackage[shortlabels]{enumitem}
\usepackage{url}
\usepackage{inconsolata} 
\definecolor{sangria}{rgb}{0.57, 0.0, 0.04}
\definecolor{royalblue}{rgb}{0.0, 0.22, 0.66}
\usepackage[colorlinks=true, urlcolor=sangria ,hyperindex, linkcolor=royalblue, pagebackref=false, citecolor=sangria]{hyperref}
\usepackage{physics} 
\usepackage{bm}
\usepackage[bb=ams, cal=cm, scr=rsfso]{mathalfa}
\usepackage{tikz-cd}  
\usepackage{marginnote}
\usepackage[textcolor=blue,
    linecolor=blue!10!white,
    bordercolor=blue!10!white,
    backgroundcolor=white,]{todonotes}

\usepackage{fullpage}
\usepackage[all]{xy}
\theoremstyle{plain}
\newtheorem{thm}[subsubsection]{Theorem}
\newtheorem{prop}[subsubsection]{Proposition}
\newtheorem{lemma}[subsubsection]{Lemma}
\newtheorem{cor}[subsubsection]{Corollary}
\newtheorem{conj}[subsubsection]{Conjecture}

\newtheorem{thmapp}{Theorem}[section]

\newtheorem{thmx}[subsubsection]{Theorem}

\newtheorem{corx}[subsubsection]{Corollary}

\newtheorem*{thm*}{Theorem}
\newtheorem*{prop*}{Proposition}
\newtheorem*{lemma*}{Lemma}
\newtheorem*{cor*}{Corollary}
\newtheorem*{conj*}{Conjecture}

\theoremstyle{definition}
\newtheorem{defn}[subsubsection]{Definition}
\newtheorem{rmk}[subsubsection]{Remark}
\newtheorem{notation}[subsubsection]{Notation}

\newtheorem{example}[subsubsection]{Example}

\newtheorem{rmkx}[subsubsection]{Remark}

\newtheorem*{defn*}{Definition}
\newtheorem*{rmk*}{Remark}
\newtheorem*{notation*}{Notation}
\newtheorem*{observation*}{Observation}
\newtheorem*{exercise*}{Exercise}
\newtheorem*{convention*}{Convention}
\newtheorem*{example*}{Example}
\newtheorem*{fact*}{Fact}
\newtheorem*{acknowledgement*}{Acknowledgement}

\newtheorem{rmkapp}[thmapp]{Remark}
\newtheorem{defapp}[thmapp]{Definition}

\numberwithin{equation}{subsection}
\numberwithin{figure}{subsection} 
\numberwithin{table}{subsection} 
\makeatletter
\let\c@equation\c@subsubsection
\makeatother
\makeatletter
\let\c@table\c@subsubsection
\makeatother
\makeatletter
\let\c@figure\c@subsubsection
\makeatother

\newcommand{\PR}[1]{\left(#1\right)}
\newcommand{\CB}[1]{\left\{#1\right\}}
\newcommand{\BR}[1]{\left[#1\right]}
\newcommand{\DP}[1]{(\!(#1)\!)}
\newcommand{\DB}[1]{\llbracket#1\rrbracket}
\newcommand{\RG}[1]{\langle#1\rangle}
\newcommand{\pma}[1]{{\begin{pmatrix}#1\end{pmatrix}}}
\renewcommand{\mod}{\bmod}
\newcommand{\ra}{\rightarrow}
\newcommand{\xra}[1]{\xrightarrow{#1}}

\newcommand{\mono}{\hookrightarrow}
\newcommand{\epi}{\twoheadrightarrow}

\newcommand{\risom}{\buildrel\sim\over\rightarrow} 
\newcommand{\ov}{\overline}
\newcommand{\ud}{\underline}
\renewcommand{\hat}{\widehat}
\newcommand{\til}{\widetilde}
\newcommand{\floor}[1]{\lfloor #1 \rfloor}
\newcommand{\ceil}[1]{\lceil #1 \rceil}

\newcommand{\GL}{\mathrm{GL}}

\newcommand{\GSp}{\mathrm{GSp}}
\newcommand{\Sp}{\mathrm{Sp}}

\newcommand{\GSO}{\mathrm{GSO}}
\DeclareMathOperator{\simc}{\mathrm{sim}}
\newcommand{\Iw}{{\mathrm{Iw}}}

\newcommand{\Fl}{\mathrm{Fl}} 
\DeclareMathOperator{\Rep}{Rep}
\newcommand{\rmG}{\mathrm{G}}
\newcommand{\uG}{\ud{G}}
\newcommand{\uB}{\underline{B}}
\newcommand{\uT}{\ud{T}}

\newcommand{\uC}{\ud{C}} 

\newcommand{\LuG}{\prescript{L}{}{\underline{G}}}

\newcommand{\uPhi}{\underline{\Phi}}
\newcommand{\uDel}{\ud{\Delta}}
\newcommand{\uLam}{\underline{\Lambda}}
\newcommand{\tilW}{\til{W}}
\newcommand{\uW}{\underline{W}}
\newcommand{\utilW}{\underline{\til{W}}}
\newcommand{\uOm}{\underline{\Omega}}
\newcommand{\tilw}{\til{w}} 
\newcommand{\tilz}{\til{z}} 
\newcommand{\ctensor}{\hat{\otimes}}
\newcommand{\ctimes}{\hat{\otimes}}

\newcommand{\btimes}{\boxtimes}

\DeclareMathOperator{\spec}{\mathrm{Spec}}

\DeclareMathOperator{\spf}{\mathrm{Spf}}

\DeclareMathOperator{\Frob}{\mathrm{Frob}}

\DeclareMathOperator{\Hom}{\mathrm{Hom}}
\DeclareMathOperator{\coker}{\mathrm{coker}}

\DeclareMathOperator{\End}{\mathrm{End}}

\DeclareMathOperator{\Ind}{\mathrm{Ind}}
\DeclareMathOperator{\ind}{\mathrm{ind}}

\DeclareMathOperator{\Lie}{\mathrm{Lie}}
\DeclareMathOperator{\Aut}{\mathrm{Aut}}
\DeclareMathOperator{\std}{{\mathrm{std}}}

\DeclareMathOperator{\Gal}{\mathrm{Gal}}
\DeclareMathOperator{\Art}{\mathrm{Art}}
\DeclareMathOperator{\Ext}{\mathrm{Ext}}
\DeclareMathOperator{\HT}{{\mathrm{HT}}}

\newcommand{\rhobar}{\overline{\rho}}
\newcommand{\rbar}{\overline{r}}
\newcommand{\mo}{{-1}}

\renewcommand{\ss}{{\mathrm{ss}}}

\newcommand{\Gr}{\mathrm{Gr}}
\newcommand{\ur}{\mathrm{ur}}

\newcommand{\et}{\normalfont{\text{{\'et}}}}

\newcommand{\diag}{\mathrm{Diag}}
\newcommand{\JH}{\mathrm{JH}}

\newcommand{\rec}{\mathrm{rec}}
\newcommand{\ad}{\mathrm{ad}}

\newcommand{\id}{\mathrm{id}}
\newcommand{\univ}{\mathrm{univ}}

\newcommand{\pret}{\prescript{t}{}}
\newcommand{\Sym}{{\mathrm{Sym}}}
\newcommand{\red}{{\mathrm{red}}}
\newcommand{\reg}{{\mathrm{reg}}}

\newcommand{\nba}{\nabla}
\newcommand{\bss}{{\backslash}}
\newcommand{\Conv}{{\mathrm{Conv}}}
\newcommand{\Ad}{\mathrm{Ad}}
\newcommand{\Adm}{{\mathrm{Adm}}}
\newcommand{\loccit}{\emph{loc.~cit.}}
\newcommand{\tres}{{\text{tr\`es ramifi\'ee}}}

\newcommand{\dash}{\text{-}}
\newcommand{\loc}{\mathrm{loc}}
\newcommand{\alg}{\mathrm{alg}}
\newcommand{\gen}{\mathrm{gen}}
\newcommand{\BM}{\mathrm{BM}}
\newcommand{\der}{\mathrm{der}}
\newcommand{\obv}{\mathrm{obv}}

\newcommand{\Mod}{\mathrm{Mod}}
\newcommand{\CNL}{\hat{\cC}} 

\newcommand{\recgt}{{\mathrm{rec_{GT}}}}
\newcommand{\ddv}[1]{\frac{d{#1}}{dv}}
\newcommand{\rmor}{{\mathrm{or}}}    
\newcommand{\ix}[1]{^{(#1)}}
\newcommand{\nm}{{\mathrm{nm}}}

\newcommand{\nv}{{\mathrm{nv}}}

\newcommand{\pgma}{(\varphi,\Gamma)}
\newcommand{\mm}{{--}}
\newcommand{\pcp}{{\wedge_p}} 

\newcommand{\AP}{\mathrm{AP}}

\newcommand{\jj}{{j\in \cJ}}

\newcommand{\Irr}{\mathrm{Irr}}

\newcommand{\nbaz}{{\nabla_0}}

\newcommand{\catphiMfour}{\Phi\text{-}\mathrm{Mod}_K^{\et,4}}
\newcommand{\catphiMone}{\Phi\text{-}\mathrm{Mod}_K^{\et,1}}
\newcommand{\catphiMsym}{\Phi\text{-}\mathrm{Mod}_K^{\et,\Sym}}
\newcommand{\catphiMn}{\Phi\text{-}\mathrm{Mod}_K^{\et,n}}

\newcommand{\tilC}{\til{C}}
\newcommand{\psip}{{\psi_p}}
\newcommand{\osig}{\overline{\sigma}}

\newcommand{\idx}[1]{^{(#1)}}

\newcommand{\hcS}{\hat{\cS}}

\newcommand{\recGTp}{\rec_{\mathrm{GT},p}}
\newcommand{\Isom}{\mathrm{Isom}}

\newcommand{\ax}{{\mathrm{aux}}}
\newcommand{\fixed}{{\mathrm{fixed}}}

\newcommand{\mc}{\mathcal}
\newcommand{\mf}{\mathfrak}
\newcommand{\mbf}{\mathbf}

\newcommand{\Q}{\mathbf{Q}}

\newcommand{\Qp}{\mathbf{Q}_p}
\newcommand{\Qpbar}{\overline{\mathbf{Q}}_p}

\newcommand{\Z}{\mathbf{Z}}
\newcommand{\Zp}{\mathbf{Z}_p}

\newcommand{\Zpbar}{\overline{\mathbf{Z}}_p}
\newcommand{\R}{\mathbf{R}}
\newcommand{\C}{\mathbf{C}}
\newcommand{\N}{\mathbf{N}}
\newcommand{\F}{\mathbf{F}}

\newcommand{\Fp}{\mathbf{F}_p}
\newcommand{\Fpbar}{\overline{\mathbf{F}}_p}

\newcommand{\G}{\mathbf{G}}

\newcommand{\T}{\mathbb{T}}
\newcommand{\A}{\mathbb{A}}

\newcommand{\fa}{{\mf{a}}}
\newcommand{\fb}{{\mf{b}}}
\newcommand{\fc}{{\mf{c}}}

\newcommand{\fm}{{\mf{m}}}

\newcommand{\fp}{{\mf{p}}}

\newcommand{\fM}{{\mf{M}}}
\newcommand{\fN}{{\mf{N}}}

\newcommand{\fP}{{\mf{P}}}

\newcommand{\fS}{{\mf{S}}}

\newcommand{\cA}{{\mc{A}}}

\newcommand{\cC}{{\mc{C}}}
\newcommand{\cD}{{\mc{D}}}
\newcommand{\cE}{{\mc{E}}}
\newcommand{\cF}{{\mc{F}}}
\newcommand{\cG}{{\mc{G}}}

\newcommand{\cI}{{\mc{I}}}
\newcommand{\cJ}{{\mc{J}}}

\newcommand{\cM}{{\mc{M}}}
\newcommand{\cN}{{\mc{N}}}
\newcommand{\cO}{{\mc{O}}}
\newcommand{\cP}{{\mc{P}}}

\newcommand{\cR}{{\mc{R}}}
\newcommand{\cS}{{\mc{S}}}
\newcommand{\cT}{{\mc{T}}}
\newcommand{\cU}{{\mc{U}}}

\newcommand{\cX}{{\mc{X}}}
\newcommand{\cY}{{\mc{Y}}}
\newcommand{\cZ}{{\mc{Z}}}

\newcommand{\bfa}{{\mbf{a}}}

\newcommand{\al}{\alpha}
\newcommand{\be}{\beta}
\newcommand{\ga}{\gamma}
\newcommand{\Del}{\Delta}
\newcommand{\del}{\delta}
\newcommand{\eps}{\epsilon}
\newcommand{\veps}{\varepsilon}

\newcommand{\Lam}{\Lambda}
\newcommand{\lam}{\lambda}
\newcommand{\sig}{\sigma}

\title{Emerton--Gee stacks, Serre weights, and  Breuil--M\'ezard conjectures for $\GSp_4$}
\author{Heejong Lee}
\address{Korea Institute for Advanced Study, 85 Hoegi-ro, Dongdaemun-gu, Seoul 02455, Republic of Korea}
\email{heejonglee@kias.re.kr}

\begin{document}

\begin{abstract}
   We construct a moduli stack of rank 4 symplectic projective \'etale $\pgma$-modules and prove its geometric properties for any prime $p>2$ and finite extension $K/\Qp$. 
   When $K/\Qp$ is unramified, we adapt the theory of local models recently developed by Le--Le Hung--Levin--Morra to study the geometry of potentially crystalline substacks in this stack. In particular, we prove the unibranch property at torus fixed points of local models and deduce that  tamely potentially crystalline deformation rings are domains under genericity conditions. As applications, we prove, under appropriate genericity conditions, a $\GSp_4$-analogue of the Breuil--M\'ezard conjecture for tamely potentially crystalline deformation rings, the weight part of Serre's conjecture formulated by Gee--Herzig--Savitt for global Galois representations valued in $\GSp_4$ satisfying Taylor--Wiles conditions, and a modularity lifting result for tamely potentially crystalline representations. 
\end{abstract}

\maketitle

\setcounter{tocdepth}{2}
\tableofcontents

\section{Introduction}

\subsection{Background}
Serre's modularity conjecture predicts that every continuous, irreducible, odd Galois representation $\rbar: \Gal(\ov{\Q}/\Q) \ra \GL_2(\Fpbar)$ arises from modular forms \cite{Serre87}. Its refined version predicts the minimal weight of modular forms that give rise to $\rbar$ in terms of the restriction of $\rbar$ at $p$. The weight part of Serre's conjecture, saying that the original conjecture implies the refined conjecture, was crucial for the proof of the conjecture due to Kisin and Khare--Winterberger \cite{KW1,KW2,Kisin2adic}.

The weight part of Serre's conjecture has been generalized, most notably, to 2-dimensional Galois representations over a totally real field \cite{BDJ}, to Galois representations over a totally real field valued in $\GL_n$ \cite{Herzig-Duke-2009SerreweightMR2541127} and in more general groups $G$ \cite{GHS-JEMS-2018MR3871496}. (Both \cite{Herzig-Duke-2009SerreweightMR2541127}  and \cite{GHS-JEMS-2018MR3871496} assume that $p$ is unramified in the totally real field and the Galois representations are tamely ramified at places above $p$. See \cite[Conjecture 1.6.2]{LLHLM-2020-localmodelpreprint} for a generalization to wildly ramified representations, and \cite[Definition 1.1.1]{LLHLM-extreme} for a generalization to a totally real field in which $p$ ramifies.) In these generalizations, the notion of weight is replaced by \emph{Serre weights}, namely, isomorphism classes of irreducible $\Fpbar$-representations of $G(\F_q)$. 

Although stated in the global context, the weight part of Serre's conjecture is closely related to the mod $p$ Langlands program. Let $K/\Qp$ be a finite extension. The mod $p$ Langlands program predicts a certain correspondence between admissible smooth $\Fpbar$-representations of a $p$-adic reductive group $G(K)$ and continuous representations of $\Gal(\ov{K}/K)$ valued in its dual group $G^\vee(\Fpbar)$. 
In this article, we only consider $G=\GL_n$ or $\GSp_4$, in which case $G^\vee$ is isomorphic to $G$.  
Such a correspondence has been established only in the case $G=\GL_2$ and $K=\Qp$ \cite{Col,Pas} or $G=\GL_1$. Despite many recent advances (for example, see \cite{6author,BHHMS2}) in the program, its intricate nature remains poorly understood, and a precise conjecture for $G=\GL_n$ was only made recently \cite{EGH22} in a categorical language. However, using global methods, one can construct a candidate for a mod $p$ Langlands correspondence, namely, an admissible smooth $\Fpbar$-representation $\Pi$ of $\GL_n(K)$ associated with an $n$-dimensional continuous $\Fpbar$-representation $\rhobar$ of $\Gal(\ov{K}/K)$. One major open problem is to prove that $\Pi$ only depends on $\rhobar$ and not on non-canonical choices made in the global argument. The weight part of Serre's conjecture describes the Serre weights appearing in the $\GL_n(\cO_K)$-socle of $\Pi$ explicitly in terms of $\rhobar|_{I_K}$. Thus, the weight part of Serre's conjecture has been a guiding principle in the mod $p$ Langlands program.

The theory of Galois deformations has played an important role in many advances in proving the weight part of Serre's conjecture (\cite{LLHLM1,LLHLM2,LLHLM-2020-localmodelpreprint}), in the mod $p$ and $p$-adic Langlands program (\cite{Col,Pas,6author,BHHMS1}), and in proving modularity lifting results (\cite{TW,W,CHT,BLGGT-MR3152941,BCGP-ab_surf-2018abelian}). Recently, Emerton--Gee constructed an object that algebraically interpolates deformation spaces of $n$-dimensional $p$-adic Galois representations (called the Emerton--Gee stack \cite{EGstack}). The Emerton--Gee stack allows us to employ more geometric methods to study $p$-adic Galois representations. More importantly, it is an essential ingredient in formulating the categorical mod $p$ and $p$-adic Langlands conjecture \cite{EGH22}. When $K/\Qp$ is unramified, \cite{LLHLM-2020-localmodelpreprint} developed a theory of local models to study the geometry of potentially crystalline substacks in Emerton--Gee stacks. In particular, it is applied to analyze singularities of tamely potentially crystalline deformation rings. This led to several applications, including the weight part of Serre's  conjecture and the Breuil--M\'ezard conjecture. The objective of this article is to generalize Emerton--Gee stacks, local models, and their applications to Serre weight conjectures and the Breuil--M\'ezard conjecture to the group $\GSp_4$.

\subsection{Emerton--Gee stacks for $\GSp_4$}
Let $p$ be a prime and let $K/\Qp$ be a finite extension with the ring of integers $\cO_K$, a uniformizer $\varpi$, and the residue field $k$.  We also fix a sufficiently large finite extension $E/\Qp$ with the ring of integers $\cO$, a uniformizer $\varpi$, and the residue field $\F$.  

In \cite{EGstack}, Emerton--Gee constructed $\cX_{n,K}$ a moduli stack of rank $n$ projective \'etale $\pgma$-modules. For a finite $\cO$-algebra $A$, rank $n$ projective \'etale $\pgma$ modules with $A$-coefficients are equivalent to continuous representations of $G_K:=\Gal(\ov{K}/K)$ on rank $n$ projective $A$-modules. However, for general $A$, \'etale $\pgma$-modules behave better than $G_K$-representations in algebraic families. In particular, a family of reducible \'etale $\pgma$-modules can converge to an irreducible one, unlike $G_K$-representations. Thus, $\cX_{n,K}$ is considered the correct notion of a moduli stack of ``$p$-adic Langlands parameters''.

To generalize Emerton--Gee stacks to $\GSp_4$, we define a \emph{symplectic} projective \'etale $\pgma$-module to be a triple $(M,N,\al)$ where $M$ is a rank 4 projective \'etale $\pgma$ module, $N$ is a rank 1 projective \'etale $\pgma$-module, and $\al: M\simeq M^\vee \otimes N$ is an essentially self-dual isomorphism satisfying a certain sign condition  (Definition \ref{def:sym-et-pgma}). These reflect the underlying 4-dimensional vector space of $\GSp_4$, the similitude character of $\GSp_4$, and the non-degenerate skew-symmetric bilinear form on the 4-dimensional vector space, respectively. Under Fontaine's equivalence, symplectic projective \'etale $\pgma$-modules with $A$-coefficients for finite local $\cO$-algebra $A$ are equivalent to continuous representations $\rho: G_K \ra \GSp_4(A)$. We let $\cX_{\Sym,K}$ be a category fibered in groupoids over $\spf \cO$ whose groupoid of $A$-points, for a $p$-adically complete $\cO$-algebra $A$, is the groupoid of symplectic projective \'etale $\pgma$-modules with $A$-coefficients.  The following Theorem generalizes main properties of $\cX_{n,K}$ to $\cX_{\Sym,K}$.

\begin{thmx}[Theorem \ref{thm:sym-pgma-stack-reptbl}, Proposition \ref{prop:crys-ss-stacks}, Theorem \ref{thm:irred-comp-pgma-stack}]\label{thmx:EGstaack}
Suppose that $p>2$.
\begin{enumerate}
    \item The category fibered in groupoids $\cX_{\Sym,K}$ is a Noetherian formal algebraic stack.
    \item For each Hodge type $\lam$ and inertial type $\tau$, there exists a closed substack $\cX_{\Sym,K}^{\lam,\tau}$ (resp.~$\cX_{\Sym,K}^{\ss,\lam,\tau}$) which is a $p$-adic formal algebraic stack and  flat over $\cO$. It is uniquely characterized as the $\cO$-flat closed substack such that for any finite flat $\cO$-algebra $A$, $\cX_{\Sym,K}^{\lam,\tau}(A)$ (resp.~$\cX_{\Sym,K}^{\ss,\lam,\tau}(A)$) is equivalent to the groupoid of representations $\rho:G_K \ra \GSp_4(A)$ such that $\rho\otimes_A E$ is potentially crystalline (resp.~semistable)  with Hodge type $\lam$ and inertial type $\tau$.
    \item The underlying reduced substack $\cX_{\Sym,K,\red}\subset \cX_{\Sym,K}$ is an algebraic stack over $\F$ equidimensional of dimension $4[K:\Qp]$. Moreover, its irreducible components are naturally labeled by Serre weights.
\end{enumerate}
\end{thmx}

The proof of item (1) and (2) use the corresponding results for the Emerton--Gee stack for $\GL_4\times \GL_1$ and can be easily generalized to $\GSp_{2n}$. To prove item (3), we first construct irreducible components corresponding to Serre weights and prove that their union equals the stack $\cX_{\Sym,K,\red}$, following the strategy for $\GL_n$. The construction of irreducible components requires a generalization of a result of Emerton--Gee constructing a family of extensions using abelian Galois cohomology. One difference between $\GSp_4$ and $\GL_n$ is the presence of a non-abelian unipotent radical of a maximal parabolic subgroup $Q$ called the Klingen parabolic subgroup. It is unclear how the construction of Emerton--Gee can be generalized to the non-abelian case. 
Instead, we construct a family of extensions valued in $Q$ inside a family of extensions valued in the minimal parabolic of $\GL_4$ containing $Q$ by only using abelian Galois cohomology. To prove that their union is equal to the stack $\cX_{\Sym,K,\red}$, we need to prove the existence of crystalline lifts for any continuous representation $G_K \ra \GSp_4(\F)$. In this case, it seems inevitable to use non-abelian Galois cohomology. We prove this by using the main result in \cite{Lin}, which develops an obstruction theory for lifting mod $p$ representations of $G_K$ valued in reductive groups.

Recently, Zhongyipan Lin constructed a generalization of the Emerton--Gee stack to an arbitrary reductive group proved its Noetherian formal algebraicity using a Tannakian formalism \cite{LinEG}. 

\subsection{Local models for potentially crystalline stacks}
Assume that $K/\Qp$ is unramified of degree $f$.

In \cite{LLHLM-2020-localmodelpreprint}, the authors constructed local models of potentially crystalline substacks of $\cX_{n,K}$ using Breuil--Kisin modules. One can associate a Breuil--Kisin module with descent data to a lattice in a potentially crystalline representation. Using this, they identify the potentially crystalline stacks (with bounded $p$-adic Hodge type) with a certain closed substack of the moduli stack of Breuil--Kisin modules with descent data constructed in \cite{CL-ENS-2018-Kisinmodule-MR3764041}. It is proved in \loccit~that the moduli stack of Breuil--Kisin modules with descent data is smoothly equivalent to  a  Pappas--Zhu local model studied in \cite{PZ13-Inv-local_model-MR3103258}. The innovation in \cite{LLHLM-2020-localmodelpreprint} is a construction of certain subvarieties in Pappas--Zhu local models whose open neighborhoods are, after $p$-adic completion, smoothly equivalent to  open neighborhoods of potentially crystalline stacks.

The Pappas--Zhu local models exist for any connected reductive group which splits over a tamely ramified base change. Let $M(\le\lam)$ be the Pappas--Zhu local model over $\spec \cO$  associated to the group $\Res_{K/\Qp}\GSp_4$, Iwahori subgroup of $\GSp_4(K)$, and  a regular Hodge type $\lambda$. Following the idea of \cite{LLHLM-2020-localmodelpreprint}, we construct a closed subvariety $M(\lam,\nba_{\bfa_\tau})\subset M(\le\lam)$ for a tame inertial type $\tau$ (see Definition \ref{def:local-model} and \S\ref{sub:products-local-models}).   
It is a projective variety equipped with a natural $T^f$-action. Its $T^f$-fixed points are given by certain elements $\tilz$ in the extended affine Weyl group of $(\GSp_4)^f$. Our main result on the local model is the following.

\begin{thmx}[Theorem \ref{thm:unibranch-product}]\label{thmx:unibranch}
Suppose that $\lam$ is regular and $\tau$ is sufficiently generic (depending on $\lam$). Then $M(\lam,\nba_{\bfa_\tau})$ is unibranch at any $T^f$-fixed point $\tilz \in M(\lam,\nba_{\bfa_\tau})(\F)$.
\end{thmx}

Let $\cX^{\le \lam,\tau}_{\Sym,K,\reg}$ be the scheme-theoretic union of $\cX^{\lam',\tau}_{\Sym,K}$ for $\lam'\le \lam$ regular. For sufficiently generic $\tau$, the $p$-adic completion of the variety $M(\lam,\nba_{\bfa_\tau})$ is locally smoothly equivalent to $\cX^{\le \lam,\tau}_{\Sym,K,\reg}$ (Theorem \ref{thm:pcrys-local-model}). As a result, Theorem \ref{thmx:unibranch} implies the following result on Galois deformation rings. This is the main ingredient for applications below.

\begin{corx}\label{corx:domain}
Suppose that $\lam$ is regular and $\tau$ is sufficiently generic (depending on $\lam$). If  $\rhobar: G_K \ra \GSp_4(\F)$ is tame, then the potentially crystalline deformation ring $R_{\rhobar}^{\lam,\tau}$ with Hodge type $\lam$ and tame inertial type $\tau$ is a domain if it is nonzero.
\end{corx}

\begin{rmkx}
As in \cite{LLHLM-2020-localmodelpreprint}, our main results on local models, as well as their applications, hold under suitable genericity conditions, which will be made explicit in the text. In particular, genericity conditions imply that the Hodge type $\lam$ is small relative to $p$. We refer readers to \S1.3 in \loccit~for discussions on genericity conditions.
\end{rmkx}

\subsection{The inertial local Langlands correspondence} For a mildly generic tame inertial type $\tau$, we construct an irreducible representation $\sig(\tau)$ of $\GSp_4(k)$ in characteristic zero such that if a smooth representation $\pi$ of $\GSp_4(K)$ contains $\sig(\tau)$ (viewed as a $\GSp_4(\cO_K)$-representation via inflation) as a $\GSp_4(\cO_K)$-subrepresentation, then $\tau$ is isomorphic to the restriction to inertia of the $L$-parameter of $\pi$ (Theorem \ref{thm:inertial-LLC}).  
This result is essential to our global arguments. 

Very recently, Suzuki and Xu proved the explicit local Langlands correspondence for $\GSp_4$ and $\Sp_4$ \cite{SuzukiXuGSp4}. We expect that one can use their result to extend the inertial local Langlands for more general $\tau$. 

\begin{rmkx}[Transfer from $\GSp_4$ to $\GL_4$]
    We define natural maps sending Deligne--Lusztig representations and Serre weights of $\GSp_4(k)$ to those of $\GL_4(k)$ (Proposition \ref{prop:transfer-DL-SW}), which might be of independent interest. They are compatible with mod $p$ reduction. In the case of Deligne--Lusztig representations, it is also compatible with the inertial local Langlands in a suitable sense. Moreover, this shows that a Serre weight $\sig$ of $\GSp_4(k)$ is in the set $W^?(\rhobar|_{I_{K}})$ of Serre weights associated with tame $\rhobar: G_K \ra \GSp_4(\F)$ only if its transfer to $\GL_4(k)$ is in the set of Serre weights of $\GL_4(k)$ associated with $\rhobar$ viewed as $\GL_4(\F)$-valued representation, and the converse holds under a very mild assumption (Corollary \ref{cor:transfer-W?}). Thus, it seems reasonable to expect that this map respects the property of being modular.
\end{rmkx}

\subsection{Applications}
From now on, we assume that $K/\Qp$ is unramified of degree $f$. 

The importance of Corollary \ref{corx:domain} is related to the patching argument. Roughly speaking, the patching argument constructs a module $M_\infty(\lam-\eta,\tau)$ over (a certain modification of) $R_{\rhobar}^{\lam,\tau}$. A standard commutative algebra argument shows that the support of $M_\infty(\lam-\eta,\tau)$ is a union of irreducible components in $\spec R_{\rhobar}^{\lam,\tau}$. It is a folklore conjecture that the support is indeed equal to $\spec R_{\rhobar}^{\lam,\tau}$, which almost immediately implies a modularity lifting result in a relevant setup. Moreover, it has applications to the Serre weight conjectures and the Breuil--M\'ezard conjecture. 
Unfortunately, this is hard to prove in general. However, if $M_\infty(\lam-\eta,\tau)\neq 0$ and $ R_{\rhobar}^{\lam,\tau}$ is a domain, then the support of $M_\infty(\lam-\eta,\tau)$ has to be equal to $\spec R_{\rhobar}^{\lam,\tau}$. Thus, Corollary \ref{corx:domain} together with Taylor's ``Ihara avoidance'' argument implies the following modularity lifting result.

\begin{thmx}[Theorem \ref{thm:modularity-lifting}]\label{thmx:modularity-lifting}
Let $r,r':G_F \ra \GSp_4(\cO)$ be continuous representations  that are isomorphic modulo $\varpi$, unramified  at all but finitely many places, and potentially crystalline with regular Hodge type $\lam$ and sufficiently generic (depending on $\lam$) tame inertial type $\tau$ at places above $p$. We further assume that $r \mod \varpi$ is tame at places above $p$ and satisfies Taylor--Wiles conditions. If $r'$ is automorphic, then $r$ is also automorphic.
\end{thmx}

\subsubsection{The geometric Breuil--M\'ezard conjecture}

The original Breuil--M\'ezard conjecture \cite{BM02} measures the complexity of the special fiber of potentially crystalline (or semistable) deformation rings in terms of the mod $p$ representation theory of $\GL_n(k)$. Its geometric formulations are stated in \cite{EG-geom_BM-MR3134019} (for deformation rings) and \cite{EGstack} (for Emerton--Gee stacks)  and proven in \cite{LLHLM-2020-localmodelpreprint} in the tamely potentially crystalline case under genericity assumptions. Dotto formulated the Breuil--M\'ezard conjecture for central division algebras and proved that it follows from the conjecture for $\GL_n$ \cite{DottoBM}.

We formulate the geometric Breuil--M\'ezard conjecture for $\GSp_4$ in the tamely potentially crystalline case and prove it under genericity assumptions. As a preliminary, we establish the inertial local Langlands correspondence for $\GSp_4$ for a mildly generic tame inertial type (Theorem \ref{thm:inertial-LLC}). For a tame inertial type $\tau$, let $\sig(\tau)$ be the corresponding tame type (i.e.~an irreducible representation of $\GSp_4(k)$ over $E$).

\begin{thmx}[Corollary \ref{cor:geom-BM-poly-gen} and \ref{cor:versal-BM-poly-gen}]\label{thmx:BM}
Let $\Lam$ be a finite set of regular Hodge types $\lam$. 
\begin{enumerate}
    \item For each Serre weight $\sigma$, there exists a $4f$-dimensional cycle $\cZ_{\sigma}$ in $\cX_{\Sym,K,\red}$ such that
    \begin{align*}
    \cZ_{\lam,\tau} = \sum_{\sigma} [\ov{\sigma(\tau) \otimes V(\lam-\eta)} : \sigma] \cZ_{\sigma}
    \end{align*}
    holds for any regular Hodge type $\lam\in \Lam$ and any sufficiently generic (depending on $\Lam$) regular tame inertial type $\tau$.
    \item Let $\rhobar: G_K \ra \GSp_4(\F)$ be sufficiently generic (depending on $\Lam$). For each Serre weight $\sigma$, there exists a $(4f+11)$-dimensional cycle $Z_{\sigma}$ in $\spec R_{\rhobar}^\square/\varpi$ such that 
    \begin{align}\label{eqn:versal-BM}
    Z(R_{\rhobar}^{\lam,\tau}/\varpi) = \sum_{\sigma} [\ov{\sigma(\tau) \otimes V(\lam-\eta)} : \sigma] Z_{\sigma}
\end{align}
    holds for any regular Hodge type $\lam\in \Lam$ and any regular tame inertial type $\tau$.
\end{enumerate}
\end{thmx}

By using patching functors and Corollary \ref{corx:domain}, we first construct cycles $Z_{\sig}$ satisfying \eqref{eqn:versal-BM} for regular $\lam$, sufficiently generic $\tau$, and tame $\rhobar$ (Theorem \ref{thm:versal-BM}). Then we algebraically interpolate $Z_{\sig}$ for various $\rhobar$ to construct $\cZ_{\sig}$ in item (1) following the axiomatic argument in \cite[\S8.3]{LLHLM-2020-localmodelpreprint}. Item (2) (where we do \emph{not} assume that $\rhobar$ is tame) essentially follows from item (1) by pulling back the cycles $\cZ_{\sig}$ to a versal ring for $\cX_{\Sym,K,\red}$ at $\rhobar$.

\subsubsection{The weight part of Serre's conjecture}
Let $F$ be a totally real field of even degree over $\Q$ in which $p$ is unramified. Let $\cG$ be an inner form of $\GSp_4$ over $F$ which is compact modulo center at infinity and splits at all finite places. Fix an isomorphism $\iota: \Qpbar \risom \C$. Given a Hecke character $\chi: \A_F^\times/F^\times \ra \C^\times$, a   level $U\subset \cG(\A^\infty_F)$, and an $\cO[\cG(\cO_F\otimes_\Z \Z_p)]$-module $W$, we define a space of algebraic automorphic forms $S_\chi(U,W)$ to be the $\cO$-module of continuous functions $f : \cG(F) \bss \cG(\A^\infty_F) \ra W$ such that  $f(zg) =(\iota^\mo\circ\chi(z))f(g)$ and $f(gu)=u_p^\mo \cdot f(g)$ for all $z\in Z(\A_F^\infty), g\in \cG(\A^\infty_F)$, and $u \in U$.

Let $\rbar: G_F \ra \GSp_4(\F)$ be a continuous and absolutely irreducible representation which is the mod $p$ reduction of the Galois representation attached to a regular algebraic cuspidal automorphic representation of $\cG(\A_F)$ (or equivalently, of $\GSp_4(\A_F)$).
Then $\rbar$ determines a maximal ideal $\fm_{\rbar}$ of an appropriate Hecke algebra, and  $S_\chi(U,\sigma)_{\fm_{\rbar}}\neq 0$ for some $\chi$, $U$, and Serre weight $\sigma$. In this case, we say $\sigma$ is a \emph{modular} Serre weight for $\rbar$ and write $W(\rbar)$ for the set of modular Serre weights for $\rbar$. 
When $\rbar$ is tame at places above $p$, \cite{GHS-JEMS-2018MR3871496} defines a set $W^?(\otimes_{v|p}\rbar|_{I_{F_v}})$ (see Definition \ref{def:W?}) using combinatorial recipes, where ${I_{F_v}}$ denotes the inertia subgroup at $v$, and conjectures that $W(\rbar)=W^?(\otimes_{v|p}\rbar|_{I_{F_v}})$. We verify this conjecture under technical genericity assumptions.

\begin{thmx}[Theorem \ref{thm:SWC}]\label{thmx:SWC}
Let $\rbar: G_F \ra \GSp_4(\F)$ be as above. Moreover, we assume that $\rbar|_{I_{F_v}}$ is tame and sufficiently generic for $v|p$, and $\rbar$ satisfies \emph{Taylor--Wiles conditions} (Definition \ref{def:odd-TW-conditions}). Then $W(\rbar) =W^?(\otimes_{v|p}\rbar|_{I_{F_v}})$. 
\end{thmx}

Previously, a similar conjecture was made by Herzig--Tilouine \cite[Conjecture 1]{HerzigTilouine} when $F=\Q$ using \'etale cohomology groups of Siegel modular varieties instead of algebraic automorphic forms on $\cG$. We expect that our method can be used to prove the conjecture of Herzig--Tilouine (or its generalizations to totally real fields) if the conjecture on vanishing of mod $p$ \'etale cohomology groups of Hilbert--Siegel modular varieties localized at a non-Eisenstein maximal ideal outside the middle degree is known (cf.~\cite[Conjecture 4.3]{CarICM}).

Analogous conjectures for a rank $n$ unitary group $U(n)$ over a totally real field were proven  under technical assumptions (when $U(n)$ splits at places above $p$ by \cite{GLS15} for $n=2$, by \cite{LLHLM1,LLHLM2} for $n=3$, and by \cite{LLHLM-2020-localmodelpreprint} for general $n$; when $U(n)$ ramifies at places above $p$ by \cite{KM-2020-unitary} for $n=2$). For the group $\GSp_4$, \cite{Gee-Geraghty-companion-MR2876931} proved modularity of \emph{obvious weights} for $\rbar$ ordinary at places above $p$ assuming modularity of a single obvious weight. See also \cite[Theorem 1.3]{Yam} for a result obtained by a different approach. Our result is independent of \cite{Gee-Geraghty-companion-MR2876931,Yam}.

\begin{rmkx}[The work of Arthur]
    We use \cite[Theorem 3.5]{Mok-Comp-2014-MR3200667} and \cite[Theorem 7.4.1]{GeeTaibi19} for the construction of a Galois representation attached to a regular algebraic cuspidal automorphic representation of an inner form of $\GSp_4$. These results were proven using the main results in \cite{ArthurBook}, which are conditional on the twisted weighted fundamental lemma announced in \cite{CL10fundlemma}, but whose proofs have not appeared. See also \cite[\S1.6]{BCGP2}.
\end{rmkx}

\subsubsection*{Organization of the paper}
In \S\ref{sec:2}, we first recall basic notions related to the representation theory of $\GSp_4(k)$ from \cite{LLHLM-2020-localmodelpreprint}. Essentially the only new result here is the inertial local Langlands correspondence for $\GSp_4$ in \S\ref{sub:ILLC}. 

In \S\ref{sec:local-models}, we introduce local models and prove their properties. Our presentation closely follows \cite{LLHLM-2020-localmodelpreprint}. Where possible, we simplify our argument by deducing results from analogous results for $\GL_4$. After introducing the global affine Grassmannian and universal local models in \S\ref{sub:global-affine-Grass}--\ref{sub:uni-local-models}, we introduce mixed characteristic local models and prove the unibranch property in \S\ref{sub:mixed-char-local-models}--\ref{sub:products-local-models}.  In \S\ref{sub:special-fiber-local-models}--\ref{sub:torus-fixed-point}, we discuss irreducible components in the special fiber of mixed characteristic local models. In \S\ref{sub:torus-fixed-point}, we classify torus fixed points in each irreducible component under genericity conditions. 

We introduce moduli stacks of symplectic \'etale $\pgma$-modules in \S\ref{sub:pgma-stack}, of Breuil--Kisin modules in \S\ref{sub:BK-stack}, and of \'etale $\varphi$-modules in \S\ref{sub:et-phi-stack}. Then we bring them together to prove our results on local model diagrams for potentially crystalline stacks (Theorem \ref{thm:pcrys-local-model}) and their mod $p$ reduction (Theorem \ref{thm:mod-p-local-model}). Both results rely on the existence of certain local lifts whose proofs are postponed to \S\ref{sub:patching-argument}.

In \S\ref{sec:global}, we discuss our setup for global arguments and construct patching functors. In the course, we prove a weight elimination result in \S\ref{subsec:alg-aut-forms}. As applications, we prove the existence of certain local lifts in \S\ref{sub:patching-argument}.

Finally, we formulate the geometric Breuil--M\'ezard conjecture and prove its restricted versions in \S\ref{sub:BM}. We prove our result on the weight part of Serre's conjecture in \S\ref{sub:SWC} and a modularity lifting result in \S\ref{sub:modularity-lifting}.

\subsection{Acknowledgement}
The author would like to thank his advisor Florian Herzig for suggesting him to work on the weight part of Serre's conjecture for $\GSp_4$, his constant encouragement, and numerous comments. The debt that this article owes to the work of Daniel Le, Bao V.~Le Hung, Brandon Levin, and Stefano Morra will be obvious to the reader, and it is a pleasure to acknowledge this. The author would like to thank Daniel Le, Bao V.~Le Hung, and Stefano Morra for discussions and for answering many questions about their work. He would like to thank Pablo Boixeda Alvarez, Wee Teck Gan, Tasho Kaletha, and
Zhongyipan Lin for helpful correspondence. He would also like to thank Shaun Stevens for informing him about Jamie Lust's thesis. 

Part of the work was carried out during visits to the Ulsan National Institute of Science and Technology and the Korea Institute of Advanced Science. The author would like to heartily thank these institutions and Chol Park and Myungjun Yu for their support. The author was partially supported by the Ontario Trillium Scholarship.

\subsection{Notation}
Let $p>2$ be a prime. We let $K/\Qp$ denote a finite field extension with the ring of integers  $\cO_K$ and  the residue field $k$. Throughout this paper, except \S\ref{sub:pgma-stack}, we assume $K$ to be unramified of degree $f$ over $\Q_p$.  
We take $E$ to be a finite extension of $\Q_p$ with ramification degree $e$, the ring of integers $\cO$, and the residue field $\F$. We assume that $E$ is sufficiently large unless otherwise stated. We let $[\cdot]:\F^\times \ra W(\F)^\times$ denote the multiplicative lift.

Let $F$ be a field. We write $G_F:=\Gal(\ov{F}/F)$ where $\ov{F}$ is a separable closure of $F$. If $F$ is a non-archimedean local field, we write $I_F \subset G_F$ for the inertia subgroup and $W_F$ for the Weil group. 

Fix a separable closure $\ov{K}$ of $K$ and suppose that $K/\Qp$ is unramified of degree $f$. We choose $\pi \in \ov{K}$ such that $\pi^{p^f-1}=-p$. We denote by $\omega_K : G_K \ra \cO_K^\times$ the character defined by $g(\pi)= \omega_K(g)\pi$ for $g\in G_K$. For an embedding $\sig: K \mono E$, we write $\omega_{K,\sig} = \sig \circ \omega_K$.

Let $\eps$ denote the $p$-adic cyclotomic character. If $V$ is a de Rham representation of $G_K$ over $E$, then for each $\sigma \in \Hom_{\Qp}(K,E)$, we let $\HT_\sig(V)$ denote the multiset of Hodge--Tate weights labelled by $\sig$ normalized so that $\HT_\sig(\eps)=\CB{1}$.

Let $F$ be a number field. For a place $v$ of $F$, we let $F_v$ be the completion of $F$ at $v$ with the ring of integers $\cO_{F_v}$, a uniformizer $\varpi_v$, and the residue field $k_v$ of size $q_v$. We write $\Frob_{F_v}$ for a geometric Frobenius element in $G_{F_v}$. We normalize the Artin map $\Art_{F_v} : F_v^\times \risom W^{\mathrm{ab}}_{F_v}$ so that uniformizers are mapped to geometric Frobenius elements.

Let $G$ be a split connected reductive group over $\Z$. In the body of the paper, we take $G=\GSp_4$ most of the time. We write $B\subset G$ for a choice of Borel subgroup, $T\subset B$ for a maximal torus, and $U\subset B$ for the unipotent radical of $B$.  Let $\Phi^+ \subset \Phi$ be the subset of positive roots for $(G,B,T)$. We denote by $\Del$ the set of simple roots. We write $X^*(T)$ for the group of characters of $T$ and $X_*(T)$ for the group of cocharacters of $T$.  We write $\Lam_R \subset X^*(T)$ and $\Lam_R^\vee\subset X_*(T)$ for the root lattice and coroot lattice. We let $W$ denote the Weyl group, $W_a$ denote the affine Weyl group, and $\tilW$ denote the extended affine Weyl group for $G$.

We write $G^\vee$ for the split reductive group over $\Z$ defined by the root datum $(X_*(T),X^*(T),\Phi^\vee,\Phi)$. We write $T^\vee \subset G^\vee$ for the induced maximal split torus. We have isomorphisms $X^*(T^\vee)\simeq X_*(T)$ and $X_*(T^\vee)\simeq X^*(T)$. We let $W^\vee$, $W_a^\vee$, and $\tilW^\vee$ denote the Weyl group, affine Weyl group, and extended affine Weyl group for $G^\vee$.

We denote by $\cO_p$ a finite \'etale $\Z_p$-algebra. In the body of this paper, we take $\cO_p$ to be either $\cO_{K}$ or $\cO_F\otimes_{\Z}\Z_p$ for some totally real field $F$ in which $p$ is unramified. Let $F_p = \cO_p[1/p]$. Then $F_p$ is isomorphic to a product $\prod_{v\in S_p}F_v$ for a finite set $S_p$ and finite unramified extensions $F_v/\Qp$. There is also an isomorphism $\cO_p\simeq \prod_{v\in S_p}\cO_{F_v}$ where $\cO_{F_v}$ is the ring of integers of $F_v$.

We define $G_0 := \Res_{\cO_p / \Z_p}G_{/\cO_p}$ and $\uG := (G_0)_{/\cO}$.  Let $B$ be a choice of Borel subgroup and $T\subset B$ be a maximal split torus. We define $B_0, \uB$ and $T_0, \uT$ similarly to $G_0, \uG$. Let $\cJ = \Hom_{\Z_p}(\cO_p, \cO)$. Then $(\uG,\uB,\uT)$ is naturally identified with ($G^\cJ_{/\cO},B^\cJ_{/\cO},T^\cJ_{/\cO})$.
 The root datum of $(\uG,\uB,\uT)$ is given by
 \begin{align*}
    (X^*(\uT),X_*(\uT),\uPhi,\uPhi^\vee) \simeq (X^*(T)^\cJ,X_*(T)^\cJ,\Phi^\cJ,\Phi^{\vee,\cJ}).
 \end{align*}
We have $\uLam_R \simeq \Lam_R^\cJ$, $\uW\simeq W^\cJ$, $\uW_a\simeq W_a^\cJ$, $\utilW\simeq \tilW^\cJ$, and similarly for $\uLam_R^\vee$, $\uW^\vee$, $\uW_a^\vee$, $\utilW^\vee$. 

Let $\varphi$ be the absolute Frobenius on $\cO_p/p$ and its lift to $\cO_p$. If $S$ is a set and $s=(s_\sig)_{\sig \in \cJ} \in S^\cJ$, then we define $\pi(s)$ by $\pi(s)_\sig = s_{\sig \circ \varphi^\mo}$. When $\cO_p = \cO_K$, we fix an embedding $\sig_0: K \mono E$ and define $\sig_j := \sig_0\circ \varphi^{-j}$ for $j\in \Z/f \Z$. This identifies $\cJ$ with $\Z/f\Z$.

We write $\diag(a_1,\dots ,a_n)$ for the diagonal matrix in $\GL_n$ with entries $a_1,\dots, a_n$. If $\mu$ is a cocharacter and $a$ is a scalar, we write $a^\mu$ to denote $\mu(a)$. We write $\Ind$ for the unnormalized parabolic induction and $\ind$ for the compact induction. Any union of schemes or algebraic stacks is defined to be a scheme-theoretic union unless stated otherwise.

\section{Types and weights}\label{sec:2}
In this section, we recall basic notions regarding the representation theory and the extended Weyl group of $\GSp_4$ and prove the inertial local Langlands correspondence for $\GSp_4$ (Theorem \ref{thm:inertial-LLC}). We follow \cite[\S2]{LLHLM-2020-localmodelpreprint} closely throughout this section except \S\ref{sub:ILLC}.

\subsection{Preliminaries}\label{sub:prelim}

\subsubsection{The group $\GSp_4$}
Let  $G=\GSp_4$ be the split reductive group over $\Z$ defined by
\begin{align*}
    \GSp_4(R) = \CB{ g\in \GL_4(R) \mid \pret g J g = \simc(g)J \text{ for some $\simc(g)\in R^\times$}}
\end{align*}
for any commutative ring $R$, where
\begin{align*}
    J= \pma{ & & & 1 \\ & &1& \\&-1& & \\-1 & & & }.
\end{align*}
Then $\simc: g\mapsto \simc(g)$ defines a character of $\GSp_4$, called the \emph{similitude character}. Let $T\subset \GSp_4$ be the maximal diagonal torus and $W = N(T)/T$ be the Weyl group of $G$. We identify $W$ with the subgroup of $N(T)$ generated by two simple reflections 
\begin{align*}
       s_1:=\pma{& 1 & & \\ 1& & & \\ & & &1 \\ & & 1&}, \ s_2:=\pma{1 & & & \\ & & 1& \\ &-1  & & \\ & & &1}.
\end{align*}
We write $w_0=s_1s_2s_1s_2$ for the longest element in $W$. We use the same notation to denote $\cJ$-fold product of $w_0$ in $\uW$.

Let $B$ be the upper triangular Borel subgroup of $\GSp_4$. For any subset $A\subset \{s_1,s_2\}$, we have a standard parabolic subgroup $P_A$ generated by $A$ and $B$. When $A= \{s_1\}$, we call $S:=P_{A}$ the \emph{Siegel parabolic subgroup}, and when $A=\{s_2\}$, we call $Q:=P_A$ the \emph{Klingen parabolic subgroup}. If $P$ is a standard parabolic subgroup with unipotent radical $N$, we write $\ov{P}$ and $\ov{N}$ to denote the opposite parabolic and its unipotent radical.

We identify the character group $ X^*(T)$ with $\Z^3$ by defining the character corresponding to $(a,b;c)\in \Z^3$ by 
\begin{align*}
    (a,b;c):\diag(x,y,zy^\mo,zx^\mo) \mapsto x^ay^bz^c. 
\end{align*}
Similarly, we identify the cocharacter group $X_*(T)$ with $\Z^3$ by
\begin{align*}
    (a,b;c): x \mapsto \diag(x^a, x^b, x^{c-b}, x^{c-a}).
\end{align*}
Sets of simple roots and coroots are given by
\begin{align*}
    \Delta = \CB{\alpha_1=(1,-1;0), \alpha_2 = (0,2;-1))}, \ \ \     \Delta^\vee = \CB{\alpha_1^\vee=(1,-1;0), \alpha_2^\vee = (0,1;0)}.
\end{align*}
Sets of positive roots and coroots are given by
\begin{align*}
    \Phi^+ = \CB{\al_1,\al_2,\al_3=\al_1 + \al_2, \al_4=2\al_1+\al_2}, \ \ \ 
   \Phi^{\vee,+} = \CB{\al_1^\vee, \al_2^\vee, \al_3^\vee = \al_1^\vee + 2\al_2^\vee, \al_4^\vee= \al_1^\vee + \al_2^\vee}.
\end{align*}

The dual group of $\GSp_4$ is isomorphic to $\GSp_4$ by an exceptional isomorphism. We often write $\GSp_4^\vee$ (and $T^\vee$ for its diagonal maximal torus) to emphasize that we are working on the dual side. 
We fix the duality isomorphism by
\begin{align*}
    (X^*(T), X_*(T), \Phi, \Phi^\vee) & \risom (X_*(T^\vee), X^*(T^\vee), \Phi^\vee, \Phi) \\
    \phi:  X^*(T) & \xrightarrow{} X_*(T^\vee) \\
    (a,b;c) & \mapsto (a+b+c,a+c;a+b+2c).
\end{align*}
Then $\phi$ maps $\alpha_1$ to $\alpha_2^\vee$ and $\alpha_2$ to $\alpha_1^\vee$. We also write a map $\phi: W \risom W^\vee$ sending $s_1 \mapsto s_2$, $s_2\mapsto s_1$, so that
\begin{align*}
    \phi(w(\lam)) = \phi(w)(\phi(\lam))
\end{align*}
for any $w\in W$ and $\lam\in X^*(T)$. We fix an element $\eta = (2,1;0)\in X^*(T)$. We often use the same notation to denote the $\cJ$-fold product of $\eta$ in $X^*(\uT)$ and its image in $X_*(\uT^\vee)$ under $\phi$.

\subsubsection{The group $\GL_n$}
If $G=\GL_n$, we add subscript $n$ to the objects introduced above when we need to distinguish them from the case $G=\GSp_4$. For example, we write $T=T_n, B=B_n, U=U_n, W=W_n$. We identify $X^*(T_n) \simeq X_*(T_n) \simeq \Z^n$ in the standard way. We fix an element $\eta' = (n-1,n-2,\dots, 0)\in X^*(T_n)$ and use the same notation for its $\cJ$-fold product in $X^*(\uT_n)$. 

\subsubsection{Affine Weyl group} 
Let $G$ be a split reductive group over $\Z$. Recall that the affine Weyl group $W_a\simeq \Lam_R \rtimes W$ and  the extended Weyl group $\tilW \simeq X^*(T) \rtimes W$ for $G$. 
Similarly, we write $W^\vee_a \simeq \Lam^\vee_R \rtimes W^\vee$ and $\tilW^\vee \simeq X_*(T^\vee) \rtimes W^\vee$ for the affine Weyl group and extended Weyl group for $G^\vee$.

Let $\cA$ be the set of alcoves of $X^*(T)\otimes_\Z \R$. We denote by $A_0$ the dominant base alcove with respect to our choice of the Borel subgroup $B$. The group $W_a$ acts simply transitively on $\cA$, and the choice of $A_0$ defines a Bruhat order on $W_a$ which we denote by $\le$. We also recall the upper arrow ordering on the set $\cA$ (see, \cite[\S II.6.5]{JantzenBook}). This induces an ordering on $W_a$ which we denote by $\uparrow$.

Let $\Omega\le \tilW$ be the subgroup stabilizing $A_0$. Then we have $\tilW \simeq W_a \rtimes \Omega$. We extend the Bruhat order and upper arrow order to $\tilW$ by: for $\tilw_1, \tilw_2 \in W_a$ and $\del_1,\del_2\in \Omega$, when $\del_1=\del_2$, $\tilw_1 \del_1 \le \tilw_2 \del_2$ (resp.~ $\tilw_1 \del_1 \uparrow \tilw_2 \del_2$) if and only if  $\tilw_1 \le \tilw_2$ (resp.~$\tilw_1\uparrow \tilw_2$), and when $\del_1\neq \del_2$, $\tilw_1 \del_1$, $\tilw_2 \del_2$ are incomparable.

We define a Bruhat order on $W_a^\vee$ induced by the antidominant base alcove and extend to $\tilW^\vee$ as in the previous paragraph.
We define a bijection
\begin{align*}
    (-)^*: \tilW &\ra \tilW^\vee \\
    \tilw=t_\nu w &\mapsto \tilw^*:= \phi(w)^\mo t_{\phi(\nu)}.
\end{align*}
This is an anti-automorphism of a group and preserves the Bruhat order on both sides.
We also write $(-)^*$ for its inverse, i.e.~if $\tilw\in \tilW$ and $\tilz = \tilw^*$, then $\tilz^{*}=\tilw$.

\begin{defn}
We define the \emph{admissible set} associated to $\lam_0\in X^*(T)$ as
\begin{align*}
    \Adm(\lam_0):= \CB{\til{w}\in \til{W} \mid \til{w}\le t_{w\lam_0} \text{ for some $w\in W$}}.
\end{align*}
If $\lam\in X^*(\ud{T})$, we define $\Adm(\lam)= \prod_{j\in \cJ} \Adm(\lam_j)$
\end{defn}
We similarly define $\Adm^\vee(\lam)$ for $\lam\in X_*(\uT^\vee)$. The map $(-)^*$ induces a bijection between $\Adm(\phi^\mo(\lam))$ and $\Adm^\vee(\lam)$.

The \emph{$p$-dot action} of $t_\nu w \in \tilW$ on $\lambda \in X^*(T)\otimes \R$ is defined as
\begin{align*}
     t_\nu w\cdot \lambda := w(\lambda + \eta) - \eta  +p\nu 
\end{align*}
A \emph{$p$-alcove} is a connected component of the complement
\begin{align*}
    \PR{X^*(T)\otimes \R } \backslash \cup_{\al\in \Phi,m\in \Z} \CB{x \mid \RG{x+\eta,\alpha^\vee} = pm}.
\end{align*}
We define $p$-alcoves in $X^*(\uT)\otimes \R$ similarly.

We say that an alcove $A$ (resp.~a $p$-alcove $C$) is \emph{restricted (resp.~$p$-restricted)} if for all $x\in A$ (resp.~$x\in C$) and $\al\in \Phi^+$, $0< \RG{x,\al^\vee}<1$ (resp.~$0< \RG{x+\eta,\al^\vee}<p$). 
In the case $G=\GSp_4$, there are four $p$-restricted alcoves given by
\begin{align*}
    C_0 &:= \CB{(a,b;c) \in X^*(T)\otimes \R \mid 0 < b<a, \  0<a+b < p } \\
    C_1 &:= \CB{(a,b;c) \in X^*(T)\otimes \R \mid b< a<p,\  a+b < 2p } \\
    C_2 &:= \CB{(a,b;c) \in X^*(T)\otimes \R \mid a-b<p<a,\  a+b<2p} \\
    C_3 &:= \CB{(a,b;c) \in X^*(T)\otimes \R \mid b<p,\ a-b<p,\  2p<a+b}
\end{align*}
and $C_0$ is called the dominant base $p$-alcove. Similarly, we denote by $\uC_0$ the dominant base $p$-alcove in $X^*(\uT)\otimes \R$.  We define
\begin{align*}
    \tilW^+ &:= \{\tilw \in \tilW \mid \tilw(A_0) \text{ is dominant}\}\\
     \tilW^+_1 &:= \{\tilw \in \tilW \mid \tilw(A_0) \text{ is restricted}\}
\end{align*}
and $\utilW^+:= \tilW^{+,\cJ}$ and $\utilW^+_1 := \tilW_1^{+,\cJ}$.

We use the above notations for $G=\GSp_4$. For $G=\GL_n$, we again add subscript $n$ when we need to distinguish them from the case of $\GSp_4$. For example, we write $\tilW_n$, $\tilW^+_n$, and $\tilW^+_{1,n}$ for $\tilW, \tilW^+$, and $\tilW^+_1$ respectively.

\subsubsection{Transfer map}
We write $\std: \GSp_4^\vee \xrightarrow{} \GL_4^\vee$ for the standard representation. It induces a map between the tori $\uT^\vee \xrightarrow{} \uT_4^\vee$. We also write the induced map between cocharacter groups by $\std :X_*(\uT^\vee) \xrightarrow{} X_*(\uT_4^\vee)$. Explicitly, we have
\begin{align*}
    (a,b;c) \xmapsto{\std} (a,b,c-b,c-a).
\end{align*} 
We also define $\cT: X^*(\uT) \xrightarrow{} X^*(\uT_4)$ to be the unique map which makes the following diagram commute
\[
\begin{tikzcd}
X^*(\uT) \arrow[r, "\cT"] \arrow[d, "\phi"] & X^*(\uT_4) \arrow[d, equal] \\
X_*(\uT^\vee) \arrow[r, "\std"] & X_*(\uT_4^\vee)
\end{tikzcd}
\]
Explicitly, $\cT$ maps $(a,b;c) \in X^*(\uT) \mapsto (a+b+c,a+c,b+c,c) \in X^*(\uT_4)$.

The map $\std$ also induces a map between $\uW^\vee$ and $\uW_4^\vee$. We again denote by $\cT: \uW \ra \uW_4$ the composition $\std\circ \phi$ followed by the identification $\uW_4^\vee = \uW_4$. We have $\cT(w\lambda) = \cT(w)\cT(\lambda)$ for all $\lambda \in X^*(\uT)$, $w\in \uW$. As a result, $\cT$ extends to a map between $\tilW$ and $\tilW_4$. Note that $\cT$ is equal to the composition
\begin{align*}
    \tilW \xra{(-)^*} \tilW^\vee \xra{\std} \tilW_4^\vee \xra{(-)^*} \tilW_4.
\end{align*}

The image of $\cT$ in $\utilW_4$ can be interpreted as the fixed-point set of a certain involution. Let $w'_0 \in \uW_4$ be the longest element. Let $\utilW_{4,a,\Q} = (X^*(T_4)\otimes_\Z\Q) \rtimes W_4$.  Define $\Theta:\utilW_{4,a,\Q} \xrightarrow{} \utilW_{4,a,\Q}$ to be the unique group homomorphism satisfying
\begin{align*}
    \Theta(s) = w'_0 s w'_0, \ \ \Theta(a,b,c,d) = \frac{a+b+c+d}{2}(1,1,1,1) - w'_0(a,b,c,d)
\end{align*}
for $s\in \uW_4$ and $(a,b,c,d)\in X^*(\uT_4)$. 
Then $\tilw \in \utilW_4$ is in the image of $\cT$ if and only if $\Theta(\tilw)=\tilw$.

\begin{lemma}\label{lem:transfer-preserves-bruhat-order}
The  map $\cT: \utilW \ra \utilW_4$ respects the Bruhat order, i.e. for all $\tilw_1, \tilw_2 \in \utilW$, 
\begin{align*}
    \text{$\tilw_1\le \tilw_2$ in $\utilW$ if and only if $\cT(\tilw_1)\le\cT(\tilw_2)$ in $\utilW_4$}.
\end{align*}
Moreover, for $\lam \in X^*(\uT)$, we have $\Adm(\lam) \overset{\cT}{\simeq} \Adm(\cT(\lam))^\Theta$.
\end{lemma}

\begin{proof}
The first claim is the result of Kottwitz--Rapoport \cite[Proposition 2.3]{KR-Manuscr-2000-Minuscule-alcoves-MR1785323}. The second claim is the result of Haines--Ng\^o \cite[Proposition 5]{HN-AJM-2002-alcoves-MR1939783}. Note that in \loccit, the set $\mathrm{Perm}(\mu)$  is equal to $\Adm(\mu)$ (for the group $\GL_n$) by Theorem 1 in \loccit.
\end{proof}

\subsubsection{Genericity}\label{subsub:genericity}

For a character $\lam \in X^*(T)$ (resp.~a cocharacter $\lam \in X_*(T^\vee)$), we define $h_\lam = \max_{\al \in \Phi}\{\RG{\lam,\al^\vee}\}$ (resp.~$h_\lam = \max_{\al \in \Phi}\{\RG{\al,\lam}\})$. In particular, $h_\eta = 3$ (resp.~$n-1$) if $G=\GSp_4$ (resp.~$G=\GL_n$).

\begin{defn}[cf.~Definition 2.1.10 in \cite{LLHLM-2020-localmodelpreprint}]\label{def:genericty}
Let $G$ be either $\GSp_4$ or $\GL_n$. Let $\lam\in X^*(T)$ be a character and $m\in \Z_{\ge0}$.
\begin{enumerate}
    \item We say that $\lam$ is \emph{$m$-deep in a $p$-alcove $C$ } if 
\begin{align*}
    n_\alpha p + m < \RG{\lam+\eta , \alpha^\vee} < (n_\al+1)p-m
\end{align*}
for all $\al \in \Phi^+$ where $C = \CB{\lam \in X^*(T)\otimes \R \mid n_\alpha p < \RG{\lam+\eta , \alpha^\vee} < (n_\al+1)p, \forall \al\in \Phi^+}$
\item We say $\lam$ is \emph{$m$-deep} if it is $m$-deep in some $p$-alcove $C$.
\item For $\til{w} = wt_\nu\in \til{W}$, we say that $\til{w}$ is \emph{$m$-generic} if $\nu-\eta$ is $m$-deep.
\item For $\til{w} = wt_\nu \in \til{W}$, we say that $\til{w}$ is \emph{$m$-small} if $h_\nu \le m$.
\item For $\til{z} \in \til{W}^\vee$, we say that $\til{z}$ is \emph{$m$-generic} (resp.
~\emph{$m$-small}) if $\til{z}^*$ is $m$-generic (resp.~$m$-small).
\end{enumerate}
Now let $G=\GSp_4$.
\begin{enumerate}
    \item [(6)] For $\bfa = (\bfa_1,\bfa_2,\bfa_3) \in \F^3$, we say that $\bfa$ is \emph{$m$-generic} if
\begin{align*}
    \{\bfa_1,\bfa_2,\bfa_1+\bfa_2,\bfa_1-\bfa_2\}\cap \{-m,-m+1 , \dots, m-1,m\} = \emptyset
\end{align*}
where $\{-m,-m+1 , \dots, m-1,m\}$ is considered as a subset of $\F$ using $\Z \ra \F$. 
\item [(7)] Let $P(X_1,X_2,X_3)\in \Z[X_1,X_2,X_3]$ be a polynomial and let $R$ be a commutative ring. We say that $\bfa\in R^3$ (resp.~$\bfa=(\bfa_j)_{\jj}\in (R^3)^\cJ$) is \emph{$P$-generic} if $P(\bfa) \mod p$ (resp.~$P(\bfa_j) \mod p$) is in $(R/p)^\times$ (resp.~for all $\jj$). 
\item [(8)] We define $P_m(X_1,X_2,X_3)\in \Z[X_1,X_2,X_3]$ to be the polynomial
\begin{align*}
    P_m(X_1,X_2,X_3) = \prod_{a=1}^m (X_1-X_2 -a)(X_2-a)(X_1+X_2+a).
\end{align*}
Note that $\lam-\eta \in C_0$ is $m$-deep if and only if $\lam$ (viewed as an element in $\Z^3$) is $P_m$-generic.
\end{enumerate}
\end{defn}

\begin{lemma}
Let $G=\GSp_4$. If $\lambda\in X^*(T)$ is $m$-deep, then $\cT(\lam)\in X^*(T_4)$ is $m$-deep in the sense of \cite[Definition 2.1.10]{LLHLM-2020-localmodelpreprint}. Similarly, if $\til{w}\in \til{W}$ is $m$-generic (resp. $m$-small), then $\cT(\til{w})\in \til{W}_4$ is $m$-generic (resp. $m$-small).
\end{lemma}
\begin{proof}
This follows from a direct computation.
\end{proof}

\begin{rmk}
    Almost all genericity conditions in this article can be matched with those for analogous results for $\GL_4$ in \cite{LLHLM-2020-localmodelpreprint} using the above lemma. There are a few exceptions. The polynomial $P_{\lam}$ in Theorem \ref{thm:pcrys-local-model} is a priori not related to the polynomial in Theorem 7.3.2(2) in \loccit. The genericity assumption for Theorem \ref{thm:torus-fixed-point} is also proven under a combinatorial genericity rather than a geometric genericity (in the sense of \S1.3 in \loccit) assumed for Proposition 4.7.3 in \loccit~(we remark that the authors of \cite{LLHLM-2020-localmodelpreprint} were aware that such an improvement is possible). 
\end{rmk}

\subsection{Serre weights and Deligne--Lusztig representations}\label{sub:SW&DL}
Let $G$ be either $\GSp_4$ or $\GL_n$. Recall that we have a finite \'etale $\Zp$-algebra $\cO_p$ and $G_0=\Res_{\cO_p/\Zp} G_{/\cO_p}$. We remark that $G_0(\Fp)=G(k)$ when $\cO_p = \cO_K$, and $G_0(\Fp)=G(\cO_F/p)$ when $\cO_p = \cO_F\otimes_\Z \Z_p$ with $F$ a totally real field in which $p$ is unramified.

\subsubsection{Serre weights}
A \emph{Serre weight} of $G_0(\F_p)$ is an isomorphism class of irreducible $\F$-representations of $G_0(\F_p)$. Recall the set of $p$-restricted dominant weights
\begin{align*}
     X^*_1(\ud{T}) = \CB{\lam \in X^*(\uT) \mid \forall \al^\vee \in \uDel^\vee,   0 \le \RG{\lam,\al^\vee} \le p-1  }.
\end{align*}
For $\lam\in X^*_1(\ud{T})$, we write $L(\lam)$ for the irreducible $\uG_{/\F}$-representation unique up to isomorphism of highest weight $\lam$. Then $F(\lam):= L(\lam)\vert_{\uG(\F_p)}$ is a Serre weight. Moreover, we have the following bijection (\cite[Lemma 9.2.4]{GHS-JEMS-2018MR3871496})
\begin{align*}
    \frac{X_1^*(\ud{T})}{(p-\pi)X^0(\uT)} &\xrightarrow{\sim} \CB{\text{Serre weights of $G_0(\F_p)$}}/\simeq \\ 
    \lam & \mapsto F(\lam)
\end{align*}
where $X^0(\ud{T}):= \CB{\lam\in X^*(\ud{T}) \mid \forall \al\in \uPhi, \RG{\lam,\al^\vee}=0}$. For an integer $m\ge 0$, we say a Serre weight $\sig$ is \emph{$m$-deep} if $\sig \simeq F(\lam)$ for some $m$-deep $\lam$.

Let $X_\reg(\ud{T})\subset X^*_1(\ud{T})$ be the subset of $\lam$ such that $ 0 \le \RG{\lam,\al^\vee} < p-1$ for all $\al^\vee \in \uDel^\vee$. 
We say a Serre weight $\sigma$ is \emph{regular} if $\sigma \simeq F(\lam)$ for some $\lam\in X_\reg(\underline{T})$.

Let $\til{w}_h := w_0t_{-\eta}$. We denote by $\cR$ an endomorphism on $X_\reg(\uT)$ given by $\lambda \mapsto \til{w}_h \cdot \lam$. It preserves $X^0(T)$ and induces a map on the set of regular Serre weights by $\cR(F(\lam)) := F(\cR(\lam))$.

\begin{defn}
Let $\omega-\eta \in \uC_0 \cap X^*(\uT)$ and $\tilw_1 \in \utilW_1^+$. We define
\begin{align*}
    F_{(\tilw_1,\omega)} := F(\pi^\mo(\tilw_1)\cdot(\omega-\eta)).
\end{align*}
Consider the equivalence relation $(\tilw_1,\omega)\sim (t_\nu \tilw_1, \omega-\nu)$ for all $\nu \in X^0(\ud{T})$. The map $(\tilw_1,\omega)\mapsto F_{(\tilw_1,\omega)}$ sends equivalent pairs to the same Serre weight. We call the equivalence class of $(\tilw_1,\omega)$ a \emph{lowest alcove presentation} of $F_{(\tilw_1,\omega)}$. Let $m\in \Z_{\ge 0}$. If $\omega-\eta$ is $m$-deep in $\uC_0$,  we say that $(\tilw_1,\omega)$ is an \emph{$m$-generic} lowest alcove presentation of  $F_{(\tilw_1,\omega)}$.
\end{defn}

\begin{defn}
Let $\sigma$ and $\kappa$ be Serre weights. We write $\sig \uparrow \kappa$ if there exist $\lam,\lam'\in X^*_1(\uT)$   such that $\sig \simeq F(\lam)$, $\kappa \simeq F(\lam')$, and $\lam \uparrow \lam'$. 
\end{defn}

\subsubsection{Deligne--Lusztig representations} 
Let  $(s,\mu)\in \underline{W}\times X^*(\underline{T})$. By \cite[Proposition 9.2.1 and 9.2.2]{GHS-JEMS-2018MR3871496}, we can attach a Deligne--Lusztig representation $R_s(\mu)$ of $G_0(\F_p)$. By \cite[Proposition 10.10]{DL-Annals-1976MR393266}, it is a genuine representation of $G_0(\F_p)$ if $(s,\mu)$ is a \textit{good pair} (see \cite[\S2.2]{LLHL-Duke-2019MR4007598}).

\begin{defn}
Let $R$ be a Deligne--Lusztig representation and $m\in \Z_{\ge0}$.
\begin{enumerate}
    \item We say that $(s,\mu-\eta)$ is an \emph{$m$-generic lowest alcove presentation} of $R$ if $R\simeq R_s(\mu)$ and $\mu-\eta$ is $m$-deep in $\ud{C}_0$. 
    \item We say that $R$ is \emph{$m$-generic} if there exists an $m$-generic lowest alcove presentation of $R$.
\end{enumerate}
\end{defn}

\begin{rmk}
If $\mu-\eta$ is $0$-deep in $\ud{C}_0$, then $(s,\mu)$ is a good pair. Moreover, if $\mu-\eta$ is $1$-deep in $\ud{C}_0$, then $R_s(\mu)$ is  irreducible by \cite[Theorem 6.8]{DL-Annals-1976MR393266}. 
\end{rmk}

The following proposition gives the Jordan--H\"older factors of the reduction of Deligne--Lusztig representations in generic cases.
\begin{prop}[Proposition 2.3.6 in \cite{LLHLM-2020-localmodelpreprint}]\label{prop:JH-DL}
Let $(s,\mu-\eta)$ be a lowest alcove presentation of $R_s(\mu)$ that is $2h_\eta$-generic. For $\lam \in X^*_1(\ud{T})$, $F(\lam) \in \JH(\ov{R_s(\mu)})$ if and only if there exists $\tilw=w t_{-\nu} \in \utilW$ such that
\begin{align*}
    \tilw \cdot (\mu -\eta + s\pi(\nu)))\uparrow \tilw_h \cdot \lam
\end{align*}
and $\tilw \cdot \uC_0 \uparrow \tilw_h \cdot \uC_0$.
\end{prop}

\begin{defn}
Let $w\in \uW$. Let $\nu\in X^*(\uT)$ be an element uniquely determined up to $X^0(\uT)$  by the condition $wt_{-\nu} \in \utilW_1$. We define $\hat{w}:=wt_{-\nu} \mod X^0(\uT)$.
\end{defn}
By abuse of notation, we also let $\hat{w}$ denote a representative in its class.

\begin{defn}\label{def:outer-weight}
Let $R\simeq R_s(\mu)$ be a Deligne--Lusztig representation that is 6-generic if $G=\GSp_4$ and $2(n-1)$-generic if $G=\GL_n$. By Proposition \ref{prop:JH-DL}, we can define a function
\begin{align*}
    F_{R}: \uW &\ra \JH(\ov{R_s(\mu)}) \\
    w & \mapsto F(\tilw_h^\mo \pi^\mo(\hat{w}) \cdot (\mu-\eta + s(\hat{w}^\mo(0)))).
\end{align*}
\end{defn}
Note that this does not depend on the choice of $\hat{w}$. The map $F_R$ depends on the choice of the lowest alcove presentation, which will be clear in the context. So we suppress the dependence on it in the notation.

Serre weights in the image of $F_{R}$ are called \textit{outer weights} in \cite{LLHLM2}.

\subsection{Tame inertial $L$-parameters}\label{sub:inertial-L}
Let $G=\GSp_4$. Recall that we take $\cO_p$ to be a finite \'etale $\Z_p$-algebra. It is isomorphic to $\prod_{v\in S_p}\cO_v$ where $S_p$ is a finite set and $\cO_v$ is the ring of integers in some finite unramified extension $F_v/\Q_p$. Following \cite[\S1.8.2]{LLHLM-2020-localmodelpreprint}, we have the following definitions.

\begin{defn}
Let $A$ be a topological $\cO$-algebra. 
\begin{enumerate}
    \item An \emph{$L$-homomorphism over $A$} is a continuous homomorphism $W_{\Qp} \ra \LuG(A)$. An \emph{$L$-parameter over $A$} is a $\uG^\vee(A)$-conjugacy class of $L$-homomorphisms. When $A$ is finite, any $L$-homomorphism extends to $G_{\Qp}$.
    \item An \emph{inertial $L$-homomorphism over $A$} is a continuous homomorphism $I_{\Q_p} \ra \uG^\vee(A)$ which has open kernel and extends to an $L$-homomorphism over $A$.  An \emph{inertial $L$-parameter over $A$} is a $\uG^\vee(A)$-conjugacy class of inertial $L$-homomorphisms.
    \item An \emph{inertial type for $K$ over $A$} is a $G^\vee(A)$-conjugacy class of homomorphisms $I_K \ra G^\vee(A)$ which has open kernel and extends to a homomorphism $W_K \ra G^\vee(A)$.
\end{enumerate}
\end{defn}
Note that a choice of an isomorphism $\ov{F}_v\simeq \Qpbar$ for  each $v\in S_p$ induces an embedding $G_{F_v}\mono G_{\Qp}$. This induces a  bijection between $L$-homomorphisms over $A$ and collections of   continuous homomorphisms $W_{F_v} \ra G^\vee(A)$ indexed by $S_p$. Once we take $\uG^\vee(A)$-conjugacy classes and $G^\vee(A)$-conjugacy classes respectively, the bijection is independent of the choice of isomorphisms. The same bijection holds for inertial $L$-parameters and collections of inertial types for $F_v$ indexed by $S_p$. In particular, when $\cO_p = \cO_K$, the inertial $L$-parameter over $A$ is equivalent to an inertial type for $K$ over $A$. When $A$ is finite, one can replace Weil groups in the definition with Galois groups.

Let $(s,\mu)\in \underline{W}\times X^*(\underline{T})$ and $s'=\cT(s)$, $\mu'=\cT(\mu)$. Then \cite[Definition 2.2.1]{LLHL-Duke-2019MR4007598} defines a tame inertial $L$-parameter
\begin{align*}
    \tau(s',\mu'+\eta') : I_{\Qp} \ra \uT_4^\vee(\cO). 
\end{align*}
By our choice of $(s',\mu')$, this factors through $\uT^\vee(\cO)\subset \uT_4^\vee(\cO)$, and we let $\tau(s,\mu+\eta)$ denote the induced tame inertial $L$-parameter valued in $\uT^\vee(\cO)$. 

Let $F^*$ denote the endomorphism $p\pi^\mo$ on $X_*(\ud{T}^\vee)$, and $d\ge 1$ be an integer such that $(F^*\circ s^\mo)^d = p^d$. By \cite[\S2.4]{LLHLM-2020-localmodelpreprint}, we have the following explicit description:
\begin{align*}
   \tau(s,\mu+\eta):= \PR{\sum_{i=0}^{d-1} (F^*\circ \phi(s)^\mo)^i(\phi(\mu+\eta))}(\omega_d): I_{\Q_p} \xrightarrow{} \ud{T}^\vee(\cO).
\end{align*}
When $\cO_p = \cO_K$, we also write $\tau(s,\mu+\eta)$ to denote the corresponding tame inertial type for $K$. Following  Example 2.4.1 of \loccit, we write $s_\tau := s_0s_1\dots s_{f-1}\in W$, $r:=\abs{s_\tau}$, and $\textbf{a}^{(0)} := (\sum^{f-1}_{j=0}(F^*\circ s^\mo)^j (\mu+\eta))_0 \in X^*(T)$. Then we have
\begin{align*}
    \tau(s,\mu+\eta) = \PR{\sum_{k=0}^{r-1} p^{fk} \phi(s_\tau)^{-k}\PR{\phi\PR{\textbf{a}^{0}}}}(\omega_{fr}) : I_K \ra T^\vee(\cO).
\end{align*}

Note that base change $-\otimes_\cO E$ and $-\otimes_\cO \F$ induce bijections between tame inertial $L$-parameters over $\cO$, $E$, and $\F$. For  a tame inertial $L$-parameter $\tau$  over $\cO$ or $E$, we write $\ov{\tau}$ for the corresponding tame inertial $L$-parameter over $\F$.

\begin{defn}
Let $\tau$ (resp. $\ov{\tau}$) be a tame inertial $L$-parameter over $E$ (resp. $\F$). Let $m\in \Z_{\ge0}$.
\begin{enumerate}
    \item A pair $(s,\mu)\in \ud{W}\times X^*(\ud{T})$ is an  \emph{$m$-generic lowest alcove presentation} of $\tau$ (resp.~$\ov{\tau}$) if $\mu$ is $m$-deep in $\ud{C}_0$ and $\tau\simeq \tau(s,\mu+\eta)$ (resp.~$\ov{\tau}\simeq \ov{\tau}(s,\mu+\eta)$). 
    
    \item We say  $\tau$ (resp. $\ov{\tau}$) is \emph{$m$-generic} if there is an $m$-generic lowest alcove presentation $(s,\mu)$ of $\tau$ (resp. $\ov{\tau}$).
\end{enumerate}
\end{defn}

\subsection{Inertial local Langlands for $\GSp_4$}\label{sub:ILLC} 
In this subsection, we take $G=\GSp_4$  and $\cO_p=\cO_K$.

In \cite{GT-Annals-LLC-GSp4-MR2800725}, Gan and Takeda established the local Langlands correspondence for $\GSp_4$, which we denote by $\recgt$.  
It is a surjective finite-to-one map that takes equivalence classes of smooth irreducible representations of $\GSp_4(K)$ to $\GSp_4$-conjugacy classes of Weil--Deligne representations of $W_K$ valued in $\GSp_4(\C)$. Fix once and for all an isomorphism $\iota: \C \simeq \Qpbar$. This induces a correspondence $\rec_{\mathrm{GT},\iota}$ over $\Qpbar$. We define a normalized local Langlands correspondence by
\begin{align*}
    \recGTp(\pi) := \rec_{\mathrm{GT},\iota}(\pi \otimes \abs{\simc}^{-3/2})
\end{align*}
for any smooth irreducible $\Qpbar$-representation $\pi$ of $\GSp_4(K)$.

Let $\tau$ be a tame inertial $L$-parameter with 1-generic lowest alcove presentation $(s,\mu)$. We attach a tame type $\sigma(\tau) := R_s(\mu+\eta)$ to $\tau$. We often view $\sig(\tau)$ as $\GSp_4(\cO_K)$-representation by inflation.

\begin{thm}\label{thm:inertial-LLC}
Let $\pi$ be a smooth irreducible $\Qpbar$-representation of $\GSp_4(K)$ and $\tau$ be a tame inertial type for $K$ with 1-generic lowest alcove presentation $(s,\mu)$.  If $\sigma(\tau)\subset \pi\vert_{\GSp_4(\cO_K)}$, then $\recGTp(\pi)\vert_{I_K} \simeq \tau$. Moreover, $\Hom_{\Qpbar[\GSp_4(\cO_K)]}(\sig(\tau), \pi|_{\GSp_4(\cO_K)})$ is 1-dimensional.
\end{thm}

Recall that given a pair $(s,\mu)\in \ud{W}\times X^*(\ud{T})$, we defined $(s_\tau, \textbf{a}^{(0)}) \in W\times X^*(T)$ by
\begin{align*}
    s_\tau = s_0s_1 \dots s_{f-1}, \ \textbf{a}^{(0)} = \sum_{j=0}^{f-1} ((F^* \circ s^\mo)^j(\mu+\eta))_0.
\end{align*}
For such  $(s_\tau, \textbf{a}^{(0)})$, we associate a pair $(\T,\theta):=(T_{s_\tau},\theta_{s_\tau,\bfa\idx{0}})$ and the Deligne--Lusztig representation $\epsilon_{\GSp_4}\epsilon_\T R_\T^\theta$ following \cite[Lemma 4.2]{Herzig-Duke-2009SerreweightMR2541127}. Using \cite[Corollary 10.5 and Proposition 12.2]{DM} and \cite[Proposition A.5.15(1)]{pseudo-red-grps}, 
one can see that $\epsilon_{\GSp_4}\epsilon_\T R_\T^\theta$ is isomorphic to $R_s(\mu+\eta)$ as $\GSp_4(k)\simeq G_0(\F_p)$-representation.
We say that $R_s(\mu+\eta)$ is \emph{cuspidal} if the torus $\T$ is \emph{not} contained in any proper Levi subgroup of $\GSp_4$. Otherwise, it is called \emph{non-cuspidal}.

Let $P$ be the minimal standard parabolic subgroup of $\GSp_4$ containing $\T$ with Levi factor $M$ and the unipotent radical $N$. By \cite[Proposition 8.2]{DL-Annals-1976MR393266} we have
\begin{align*}
    R_\T^\theta = \Ind_{P(\F_q)}^{G(\F_q)} (R_{\T,P}^\theta).
\end{align*}
Note that $M$ decomposes into a product $\prod_{i=1}^r M_i$ (see \cite[\S2.1]{RS-GSp4-bookMR2344630}). We also write $\T = \prod_{i=1}^r \T_i$ where
$\T_i\subset M_i$ is a maximal torus and $\theta_i = \theta|_{\T_i}$. By the K\"unneth Theorem, we have an isomorphism between $M(k)$-representations
\begin{align*}
    R^\theta_{\T,P} = \bigotimes_{i=1}^r R^{\theta_i}_{\T_i, M_i}.
\end{align*}
We give explicit descriptions of $\T_i$ and $\theta_i$. Let us write $\bfa^{(0)}=(a_1,a_2;a_3)$. Recall that we write $T_2$ for the diagonal torus of $\GL_2$ and $W_2$ for the Weyl group for $\GL_2$. Only in this subsection, we write $W_2=\CB{1,s}$.
\begin{enumerate}
    \item When $s_\tau = 1$, we have $P=B$, $r=3$, $M_i=\T_i = \G_m$ for $1 \le i \le 3$. Then $\theta_i$ is given by $a_i\in X^*(\G_m)$.
    \item When $s_\tau \in \CB{s_1, s_2s_1s_2}$, we have $P=S$, $r=2$, $M_1 \simeq \GL_2$ and $M_2 \simeq \G_m$. 
    Then $(\T_1,\theta_1) = (T_{s}, \theta_{s, \mu})$  where $\mu = (a_1,a_2)$ (resp.\,$(a_1,-a_2)$) in $X^*(T_2)$ when $s_\tau = s_1$ (resp.\,$s_2s_1s_2$), and $\theta_2$ is given by $a_3\in X^*(\G_m)$.
    \item When $s_\tau \in \CB{s_2, s_1s_2s_1}$, we have $P=Q$, $r=2$, $M_1 \simeq \G_m$, and $M_2 \simeq \GL_2$. Then $\theta_1$ is given by $a_1$ (resp.\,$a_2$) in $X^*(\G_m)$, and $(\T_2,\theta_2)=(T_{s}, \theta_{s, \mu})$ where $\mu=(a_2+a_3,a_3)$ (resp.\,$(a_1+a_3,a_3)$) in $X^*(T_2)$ when $s_\tau = s_2$ (resp.\,$s_1s_2s_1$).
    \item When $s_\tau \in \CB{s_1s_2,s_2s_1, w_0}$, we have $P=\GSp_4$. This is the only case in which $R_\T^\theta$ is cuspidal.
\end{enumerate}

When $\sig(\tau)$ is non-cuspidal, we can use the above data to describe $\pi$ as a parabolically induced representation.

\begin{prop}\label{prop:non-cuspidal-pi}
Let $\pi$ and $\sig(\tau)$ be as in Theorem \ref{thm:inertial-LLC} and assume that $\sig(\tau)$ is non-cuspidal. 
Following the above notations, there exists a $M_i(\cO_K)K^\times$-representation $\til{R}_{\T_i,M_i}^{\theta_i}$ extending $\eps_{M_i}\eps_{\T_i} R_{\T_i,M_i}^{\theta_i}$ such that \begin{align*}
    \pi \simeq \Ind_{P(K)}^{\GSp_4(K)} \PR{\bigotimes_{i=1}^r \ind_{M_i(\cO_K)K^\times}^{M_i(K)}\til{R}_{\T_i,M_i}^{\theta_i}}.
\end{align*}
Moreover, $\pi$ contains $\sig(\tau)$ with multiplicity one.
\end{prop}

\begin{proof}

Let $\fP\subset \GSp_4(\cO_K)$ be the parahoric subgroup defined as the inverse image of $P(k)\subset \GSp_4(k)$.  We let $\sig_\fP$ and $\sig_M$ denote the inflation of $\eps_P \eps_\T R^\theta_{\T,P}$ to $\fP$ and $M(\cO_K)$ respectively. By \cite[Lemma 3.6]{Morris-Comp-levelzeroMR1713308}, we have an isomorphism $\pi^{\sigma_\fP} \simeq (\pi_{N})^{\sigma_M}$
where $\pi_N$ is the unnormalized Jacquet module. Since $\sigma(\tau)\subset \pi\vert_{\GSp_4(\cO_K)}$, we get $\sigma_M \mono \pi_N\vert_{M(\cO_K)}$. The pair $(M(\cO_K), \sigma_M)$ is an $M(K)$-type. This implies that there is a quotient map of $M(K)$-representations $\pi_N \risom \tau_{\sigma_M}$ where $\tau_{\sig_M}$ is an irreducible supercuspidal representation of $M(K)$ defined by
\begin{align*}
    \tau_{\sigma_M} = \bigotimes_{i=1}^r \ind_{M_i(\cO_K)K^\times}^{M_i(K)} \til{R}_{\T_i,M_i}^{\theta_i}
\end{align*}
for some $M_i(\cO_K)K^\times$-representation $\til{R}_{\T_i,M_i}^{\theta_i}$ extending  $\eps_{M_i}\eps_{\T_i} R_{\T_i,M_i}^{\theta_i}$ (see \cite[Proposition 4.1]{Morris-Comp-levelzeroMR1713308}). By Frobenius reciprocity, we get a nonzero map
\begin{align*}
    \pi \ra{} \Ind_{P(K)}^{\GSp_4(K)}\tau_{\sigma_M}.
\end{align*}
This is an isomorphism by \cite[Lemma 5.1.(a), 5.2.(b)]{GTtheta-RepTheory-2011MR2846304} (for $P=S$ or $Q$) and \cite[Proposition 2.4.6]{BCGP-ab_surf-2018abelian} (for $P=B$).

Finally, the multiplicity-one assertion follows from
\begin{align*}
    \Hom_{\GSp_4(\cO_K)}(\sig(\tau),\pi|_{\GSp_4(\cO_K)}) & \simeq \Hom_{\fP}(\sig_{\fP},\pi|_{\fP}) \simeq \Hom_{M(\cO_K)}(\sig_M,\pi_N|_{M(\cO_K)})
    \\
    &\simeq \Hom_{M(\cO_K)}(\sig_M,\tau_{\sig_M}|_{M(\cO_K)}) \simeq \Hom_{M(K)} (\tau_{\sig_M},\tau_{\sig_M})
\end{align*}
where the first isomorphism is induced by Frobenius reciprocity, the second from $\pi^{\sig_\fP}\simeq (\pi_N)^{\sig_M}$, the third from the geometrical lemma (\cite[\S2.12]{BZ}) and 1-genericity of $\tau$, and the last from Frobenius reciprocity for compact induction.
\end{proof}

Now suppose that $\sig(\tau)$ is cuspidal. We write $\til{\sigma}(\tau)$ for the $G(\cO_K)K^\times$-representation extending $\sigma(\tau)$ by letting $K^\times$ act by central character of $\pi$. By \cite[Lemma 4.5.1]{DBR-Annals-2009MR2480618}, we have an isomorphism of $G(K)$-representations 
\begin{align*}
    \pi\simeq \ind_{G(\cO_K)K^\times}^{G(K)} (\til{\sigma}(\tau)).
\end{align*}

We prove Theorem \ref{thm:inertial-LLC} using an explicit description of the local Langlands correspondence. In the non-cuspidal case, we use the explicit theta correspondence in \cite{GTtheta-RepTheory-2011MR2846304}. In the cuspidal case, we use the explicit construction of the local Langlands correspondence for tame regular semisimple elliptic Langlands parameters in \cite{DBR-Annals-2009MR2480618}. Note that the compatibility between the local Langlands correspondence for $\GSp_4$  of DeBacker--Reeder and Gan--Takeda is proven by Lust (\cite[Theorem 1.1]{Lust-J.Alg-2013MR3065991}).

\subsubsection{Explicit theta correspondences}\label{subsec:theta-correspondence} 
In \cite{GTtheta-RepTheory-2011MR2846304}, Gan--Takeda established theta correspondences between the group $\GSp_4$ and various similitude orthogonal groups $\GSO_{2,2}, \GSO_{3,3}$, and $\GSO_{4,0}$. This was used to prove the local Langlands conjecture for $\GSp_4$ (\cite{GT-Annals-LLC-GSp4-MR2800725}). Parabolically induced representations of $\GSp_4$ have a nonzero theta lift to $\GSO_{3,3}$. 
The group $\GSO_{3,3}$ admits an exceptional isomorphism 
\begin{align*}
    \GSO_{3,3} \simeq (\GL_4\times \GL_1)/\CB{(z,z^{-2})\mid z \in \GL_1)}.
\end{align*}
Using this isomorphism, we view a representation of $\GSO_{3,3}$ as a representation of $\GL_4\times \GL_1$.

Let $\pi$ be an  irreducible smooth $\C$-representation of $\GSp_4(K)$. If the theta lifting of $\pi$ to $\GSO_{3,3}$ is given by (nonzero) $\Pi\boxtimes \chi$, the $L$-parameter of $\Pi \btimes \chi$ (which is valued in $\GL_4(\C)\times \GL_1(\C)$) factors through the map 
\begin{align*}
    \GSp_4(\C) \xrightarrow{\std\times \simc}\GL_4(\C)\times \GL_1(\C)
\end{align*}
which provides the $L$-parameter $\recgt(\pi)$ of $\pi$. 

In the following theorem, we write $\phi_\tau$ for the $L$-parameter attached to a smooth irreducible $\Qpbar$-representation $\tau$ of $\GL_2(K)$ by \cite{HT} conjugated by $\iota$. We write $\omega_\tau$ for the central character of $\tau$.

\begin{prop}\label{prop:L-parameter-non-cuspidal}
Let $\pi$ be a smooth irreducible $\Qpbar$-representation of $\GSp_4(K)$.
\begin{enumerate}
    \item ($P=B$ the Borel subgroup) If $\pi \simeq \Ind_{B(K)}^{\GSp_4(K)}(\chi_1,\chi_2;\chi)$, then
    \begin{align*}
        \recGTp(\pi) = \chi_1\chi_2\chi \abs{\cdot}^{-3}  \oplus \chi_1\chi\abs{\cdot}^{-2} \oplus \chi_2 \chi\abs{\cdot}^{-1} \oplus \chi : W_K \xrightarrow{} T(E) \subset \GSp_4(E).
    \end{align*}
    \item ($P=Q$ the Klingen parabolic) If $\pi \simeq \Ind_{Q(K)}^{\GSp_4(K)} (\chi\otimes \tau)$, then
    \begin{align*}
        \recGTp(\pi) = \phi_\tau \abs{\cdot}^{-1/2} \oplus \phi_\tau \chi \abs{\cdot}^{-5/2}  : W_K \xrightarrow{} M_Q(E) \subset \GSp_4(E). 
    \end{align*}
    \item ($P=S$ the Siegel parabolic) If $\pi \simeq \Ind_{P(K)}^{\GSp_4(K)} (\tau \otimes \chi)$, then
    \begin{align*}
        \recGTp(\pi) = \chi  \oplus \phi_\tau \chi \abs{\cdot}^{-3/2} \oplus \chi \omega_\tau \abs{\cdot}^{-3} : W_K \xrightarrow{} M_S (E) \subset \GSp_4(E).
    \end{align*}
\end{enumerate}
\end{prop}
\begin{proof}
This is a special case of \cite[Proposition 13.1]{GTtheta-RepTheory-2011MR2846304} (vi), (iv), and (v) (for (1), (2), and (3) respectively). Note that the induction in \loccit~is normalized.
\end{proof}

\begin{proof}[Proof of Theorem \ref{thm:inertial-LLC} in the non-cuspidal case]
This follows from Proposition \ref{prop:non-cuspidal-pi} and \ref{prop:L-parameter-non-cuspidal}, and the inertial local Langlands correspondence for $\GL_2$ and $\GL_1$ (e.g.~\cite[Proposition 2.5.5]{LLHLM-2020-localmodelpreprint}).
\end{proof}

\subsubsection{Depth zero regular supercuspidal local Langlands}\label{subsubsec:DBLLC}

Let $\tau\simeq \tau(s,\mu+\eta)$ be a 1-generic tame inertial type for $K$ over $E$ and $\psi:G_K \ra \GSp_4(E)$ be a continuous representation extending $\tau$. We also assume that $\sig(\tau)$ is cuspidal. After conjugating, we can assume that $\psi(I_K)$ is contained in $T(E)$ and $\psi(\Frob_K) \in N_G(T)$. We can write $\psi(\Frob_K)=wt$ for a unique $w\in W$ and some $t\in T(E)$. Note that $t$ gives a well-defined class in $T/(1-w)T$, and thus  $\simc(t)$ is independent of the choice of $t$. By the construction of $\tau$, we must have $w = \phi(s_\tau)$. The cuspidality of $\sig(\tau)$ implies $w\in\CB{s_1s_2,s_2s_1,w_0}$. Then $\psi$ is TRSELP (tame regular semisimple elliptic $L$-parameter) in the sense of \cite[\S 4.1]{DBR-Annals-2009MR2480618}.

For any TRSELP $\psi$, DeBacker--Reeder constructed the $L$-packet of depth-zero supercuspidal representations associated with it. These representations are distributed among the pure inner forms of $\GSp_4$. The pure inner forms are parameterized by Galois cohomology $H^1(K,\GSp_4)$. Since $H^1(K,\Sp_4)=1$ (because $\Sp_4$ is simply-connected) and $H^1(K,\G_m)=1$, we have $H^1(K,\GSp_4)=1$. Thus, all these are representations of $\GSp_4$.

Now we explain the construction in \cite[\S 4]{DBR-Annals-2009MR2480618} in a special case.  The $L$-packet of $\psi$, denoted by $\Pi(\psi)$, is parameterized by $\mathrm{Irr}(C_\psi)$, the (finite) set of irreducible representations of $C_\psi = \pi_0 (Z_G(\Im\psi))$. 
One can check that $C_\psi$ is trivial when $s_\tau = s_1s_2$ or $s_2s_1$ and is isomorphic to $\Z/2\Z$ when $s_\tau = w_0$. In both cases, we simply take the trivial representation of $C_\psi$.  Then the corresponding element in $\Pi(\psi)$ is given by 
\begin{align*}
    \ind^{G(K)}_{G(\cO_K)K^\times} \til{R}_{\T}^{\theta} 
\end{align*}
where $(\T,\theta)=(T_{s_\tau},\theta_{s_\tau,\bfa\idx{0}})$, $\til{R}_{\T}^{\theta} |_{G(\cO_K)} \simeq \eps_G\eps_{\T}{R}_{\T}^{\theta} $, and $\til{R}_{\T}^{\theta} |_{K^\times}$ sends $p$ to $\simc(t)$.

\begin{proof}[Proof of Theorem \ref{thm:inertial-LLC} in the cuspidal case]
We know that $\sig(\tau)\subset \pi|_{G(\cO_K)}$ implies that $\pi\simeq \ind_{G(\cO_K)K^\times}^{G(K)} (\til{\sig}(\tau))$ for some $\til{\sig}(\tau)$ extending $\sig(\tau)$. Let $\omega_\pi$ be the central character of $\pi$. By the above construction, $\pi$ is contained in the $L$-packet of $\psi$ such that $\psi|_{I_K}\simeq \tau$ and $\simc(\psi)(\Frob_K)=\omega_\pi(p)$. The multiplicity one assertion follows from
\begin{align*}
    \Hom_{G(\cO_K)}(\sig(\tau),\pi|_{G(\cO_K)}) \simeq \Hom_{G(\cO_K)K^\times}(\til{\sig}(\tau),\pi|_{G(\cO_K)K^\times}) \simeq \Hom_{G(K)}(\pi,\pi)
\end{align*}
where the first isomorphism exists because $\pi$ admits a central character, and the second isomorphism follows from the Frobenius reciprocity for compact inductions.
\end{proof}

\subsubsection{Serre weights of a tame inertial $L$-parameter}
We define the conjectural set of Serre weights associated to a tame inertial $L$-parameter following \cite[Definition 9.2.5]{GHS-JEMS-2018MR3871496}.
\begin{defn}\label{def:W?}
Let $\rhobar$ be a tame inertial $L$-parameter over $\F$. We define $W^?(\rhobar)$ to be the set $\cR (\JH (\osig([\rhobar])))$. 
\end{defn}

\begin{defn}\label{def:obvious-weight}
Let $\rhobar\simeq \ov{\tau}(s,\mu)$ be a $2h_\eta$-generic tame inertial $L$-parameter over $\F$. We define a function
\begin{align*}
    F_{\rhobar} : \uW & \ra W^?(\rhobar) \\
    w &\mapsto \cR(F_{\sig([\rhobar])}(w)).
\end{align*}
We define $W_{\obv}(\rhobar)$ to be the image of $F_{\rhobar}$ and call its elements \emph{obvious weights} of $\rhobar$.
\end{defn}
Note that $W_{\obv}(\rhobar)$ does not depend on the choice of the lowest alcove presentation of $\rhobar$ and coincides with $W_{\obv}(\rhobar)$ defined in \cite[Definition 2.6.3]{LLHLM-2020-localmodelpreprint}. In \loccit, $F_{\rhobar}(w)$ is called the obvious weight of $\rhobar$ corresponding to $w$.

\subsubsection{Transfer to $\GL_4$}
Recall that we have a map $\cT$ defined in \S\ref{sub:prelim} which maps $\uW$ to $\uW_4$ and $X^*(\uT)$ to $X^*(\uT_4)$.
Using this, we define the \emph{transfer} of Deligne--Lusztig representations and Serre weights of $G_0(\F_p)$ to $(\GL_4)_0(\F_p)$.
\begin{prop}\label{prop:transfer-DL-SW}
The map $\cT$ induces a well-defined assignment from Deligne--Lusztig representations (resp.\,Serre weights) of $G_0(\F_p)$ to the Deligne--Lusztig representations (resp.\,Serre weights) of $(\GL_4)_0(\F_p)$ given by 
\begin{align*}
    R_s(\mu) &\mapsto \cT(R_s(\mu)):= R_{\cT(s)}(\cT(\mu)) \\
    F(\lam) &\mapsto \cT(F(\lam)):= F(\cT(\lam)).
\end{align*}
Suppose that $\mu-\eta\in \ud{C}_0$ is $2h_\eta$-deep. Then $F(\lam)\in \JH(\ov{R_s(\mu)})$ implies $\cT(F(\lam))\in \JH(\ov{\cT(R_s(\mu))})$. The converse holds if, in addition, $\mu-\eta = (a_j,b_j;c_j)_\jj$ and $\abs{a_j - p/2}>3/2$ for each $\jj$.
\end{prop}

\begin{proof}
The first claim for Deligne--Lusztig representations follows from the fact that the map $\cT$ respects the Weyl group action on the character lattice and the $p$-dot action. For Serre weights, it follows from the fact that $\cT$ maps $X^*_1(\ud{T})$  and $X^0(\ud{T})$  into $X^*_1(\ud{T}_4)$ and  $X^0(\ud{T}_4)$ respectively.

Suppose that $F(\lam)\in \JH(\overline{R_{s}(\mu)})$. By Proposition \ref{prop:JH-DL}, there exists $w\in \uW$ and $\tilw\in \utilW_1$ such that $\hat{w}\uparrow \tilw$  and
\begin{align*}
    F(\lam)\simeq F\PR{ \tilw_h^\mo \tilw  \cdot(\mu-\eta + s\pi(\hat{w}^\mo(0)))} \in \JH(\ov{R_s(\mu)}).
\end{align*}
Let $\cT(\tilw_h)= \tilw_{h}'$. By applying $\cT$ to the above equation, we have
\begin{align*}
    \cT(F(\lam))\simeq  F\PR{ \tilw_h^{\prime\mo} \cT(\tilw)  \cdot(\cT(\mu)-\eta' + \cT(s\pi(\hat{w}^\mo(0))))}.
\end{align*}
Since the two orders $\le$ and $\uparrow$ coincide on $\utilW^+$, we have $\cT(\hat{w})\uparrow \cT(\tilw)$  by Lemma \ref{lem:transfer-preserves-bruhat-order}.  Then the above equality implies $\cT(F(\lam))\in \JH( \ov{\cT(R_s(\mu))})$ by Proposition \ref{prop:JH-DL}.

For the converse, suppose $\cT(F(\lam))\in \JH( \ov{\cT(R_s(\mu))})$. Then there exists $w'\in \uW_4$ and $\tilw'\in \utilW_1$ such that $\hat{w}' \uparrow \tilw'$ and
\begin{align*}
    \cT(F(\lam))\simeq F\PR{ \tilw_h^{\prime\mo} \tilw'  \cdot(\cT(\mu)-\eta' + \cT(s)\pi(\hat{w}^{\prime\mo}(0)))}. 
\end{align*}
This shows that $\tilw_h^{\prime\mo} \tilw'  \cdot(\cT(\mu)-\eta' + \cT(s)\pi(\hat{w}^{\prime\mo}(0)))$ is fixed by $\Theta$ (defined in \S\ref{sub:prelim}). Suppose that $\tilw'$ is fixed by $\Theta$. Then $\hat{w}^{\prime\mo}(0)$ is fixed by $\Theta$, and a simple computation shows that $\hat{w}'$ is fixed by $\Theta$ as well. Thus $\tilw'=\cT(\tilw)$ and $\hat{w}'=\cT(\hat{w})$ for some $\tilw\in \utilW_1$ and $w\in \uW$. As in the previous paragraph, this shows that $F(\lam)\in \JH(\ov{R_s(\mu)})$.

We finish the proof by showing that $\tilw'$ is fixed by $\Theta$. Let $\tilw_0\in \utilW$ be an element such that $\lam \in \tilw_0\cdot \uC_0$. Then we can write $\tilw'=\cT(\tilw_0)\del'$ for some $\del'\in \uOm_4$. Suppose that $\del'$ is not fixed by $\Theta$. Since $\Omega_4/\cT(\Omega)$ is the cyclic group of order 2, we can assume that $\del'_j$ is the generator of $\Omega_4$ sending $(a,b,c,d)$ to $(b,c,d,a-p)$ for at least one $j\in \cJ$. Let us write $\mu-\eta = (a_j,b_j;c_j)_\jj$ and $\cT(s)\pi(\hat{w}^{\prime\mo})=(x_j,y_j,z_j,w_j)_\jj$. Then 
\begin{align*}
    \PR{\del'\cdot (\cT(\mu-\eta)+\cT(s)\pi(\hat{w}^{\prime\mo}))}_j & = \del'_j \cdot (a_j+b_j+c_j +x_j, a_j + c_j+ y_j, b_j+c_j+z_j, c_j+w_j) \\
    & = (a_j + c_j+ y_j, b_j+c_j+z_j, c_j+w_j, a_j+b_j+c_j +x_j - p). 
\end{align*}
Since $\del'\cdot (\cT(\mu-\eta)+\cT(s)\pi(\hat{w}^{\prime\mo}))$ is fixed by $\Theta$, we have $2a_j  - p = z_j+ w_j -  y_j - x_j$. However, $\abs{z_j+w_j-y_j-x_j} \le 3$ by \cite[Remark 4.1.4]{LLHL-Duke-2019MR4007598}. This leads to a contradiction; hence $\del'$ is fixed by $\Theta$.
\end{proof}

\begin{example}
The condition $\abs{a_j-p/2}>3/2$ in the preceding lemma may seem dubious but it is necessary. For example, let $K=\Qp$ and $(s,\mu-\eta)\in W\times X^*(T)$ be a lowest alcove presentation of $R_s(\mu)$ where $s=e$, $\mu-\eta=(a,b;c)$. We take $a= (p-1)/2$. Let $w' \in W_4$ be the element sending $(x,y,z,w)$ to $(y,z,w,x)$. Then $\hat{w}'=w't_{-\nu}$ where $\nu= (1,0,0,0)$. Let
\begin{align*}
    \lam' &:= \tilw_h^{\prime\mo} \hat{w}\cdot ( \cT(\mu)-\eta' + \nu) \\
    &= \tilw_h^{\prime\mo} \cdot \cT(a,a-b;b+c-a)
\end{align*}
so that $F(\lam')\in \JH(\ov{
\cT(R_e(\mu))})$. However, one can easily check that
\begin{align*}
    F(\tilw_h^\mo\cdot (a,a-b;b+c-a))\notin \JH(\ov{R_e(\mu)}).
\end{align*}
\end{example}

\begin{cor}\label{cor:transfer-W?}
Let $\rhobar\simeq \tau(s,\mu)$ be a $2h_\eta$-generic tame inertial $L$-parameter. Recall the set $W^?(\std(\rhobar))$ defined in \cite[Definition 2.6.1]{LLHLM-2020-localmodelpreprint}.  Suppose that $\mu-\eta=(a_j,b_j;c_j)_\jj$ and $\abs{a_j-p/2}>3/2$ for each $\jj$. Then 
\begin{align*}
    \cT(W^?(\rhobar)) = W^?(\std(\rhobar)) \cap \{ \cT(F(\lam)) \mid \lam\in X_1^*(\uT) \}.
\end{align*}
\end{cor}

The following Corollary shows the compatibility between the inertial local Langlands correspondence and the transfer of Deligne--Lusztig representations. 

\begin{cor}\label{LLC-type_transfer-compatibility}
Let $\pi$ be a smooth irreducible $\Qpbar$-representation of $\GSp_4(K)$ and $\psi = \recGTp(\pi)$. Let $\Pi$ be a smooth irreducible $\Qpbar$-representation of $\GL_4(K)$ corresponding to $\std(\psi)$ under the local Langlands correspondence of \cite{HT} (conjugated by $\iota$). Let $\tau$ be a 1-generic tame inertial $L$-parameter. If $\sigma(\tau)\subset \pi \vert_{\GSp_4(\cO_K)}$, then $\cT(\sigma(\tau)) \subset \Pi\vert_{\GL_4(\cO_K)}$.
\end{cor}
\begin{proof}
By Theorem \ref{thm:inertial-LLC}, we have $\psi\vert_{I_K} \simeq \tau$. Then the claim follows from the construction of $\cT(\sig(\tau))$ and \cite[Theorem 3.7 (2)]{Shotton-Duke-BMconj-MR3769675}. (Note that although the local Langlands correspondence used in \loccit~is  normalized, it only differs by an unramified twist.)
\end{proof}

\begin{rmk}
The converse of the preceding corollary is not true, as the $L$-packet of $\psi$ can have two elements (e.g.~if $\psi|_{I_K}\simeq \tau(s,\mu)$ and $s_\tau = w_0$).
\end{rmk}

\subsection{Combinatorics between types and weights}\label{sub:comb-type-wt}
Let $G$ be either $\GSp_4$ or $\GL_n$. We first introduce some notation and definitions.
\begin{notation}
Let $(s,\mu)$ be a lowest alcove presentation of a Deligne--Lusztig representation $R$ (resp.~a tame inertial $L$-parameter $\tau$). We let $\tilw(R)$ (resp.~$\tilw(\tau)$) denote  $t_{\mu+\eta}s \in \utilW$.
If $\rhobar$ is another tame inertial $L$-parameter with a lowest alcove presentation $(s',\mu')$, we write
\begin{align*}
    \tilw(\rhobar,\tau):=\tilw(\tau)^\mo \tilw(\rhobar) \in  \utilW.
\end{align*}
\end{notation}

\begin{rmk}\label{rmk:central-char-LAP}
When working with tame inertial $L$-parameters and Serre weights, we need to choose their lowest alcove presentations in a compatible way. This can be done by fixing an algebraic central character (lifting a given central character) in the sense of \cite[\S2.2]{LLHLM-2020-localmodelpreprint} (see also Lemma 2.2.4 and 2.3.2 therein). Since such compatible lowest alcove presentations always exist in our context, we assume that all lowest alcove presentations we choose are compatible for simplicity.
\end{rmk}

For a dominant character $\lam\in X^*(\uT)$, let $W(\lam)_{/\cO}$ be the irreducible algebraic $\uG_{/\cO}$-representation unique up to isomorphism of highest weight $\lam$. Let $V(\lam)$ be the restriction of $W(\lam)_{/\cO}$ to $G_0(\Z_p)$.
We define a \emph{type} (of $G_0(\Z_p)$) to be a pair $(\lam+\eta,\tau)$ where $\lam\in X_*(\uT^\vee)$ is a dominant cocharacter and $\tau$ is a 1-generic tame inertial $L$-parameter. To a type $(\lam+\eta,\tau)$, we associate a locally algebraic representation of $G_0(\Z_p)$
\begin{align*}
    \sigma(\lam,\tau) := \sigma(\tau) \otimes_{\cO} V(\phi^\mo(\lam)).
\end{align*}

Let $W(\lam)_{/\F}$ be the dual Weyl module of highest weight $\lam$ for the algebraic group $\uG_{/\F}$ and $W(\lam)$ be the restriction of $W(\lam)_{/\F}$ to $G_0(\F_p)$. Note that $V(\lam)\otimes_\cO \F \simeq W(\lam)$.

Let $\lam\in X^*(\uT)$ be a dominant character. We have the set of \emph{admissible pairs} defined in \cite[\S2.1]{LLHLM-2020-localmodelpreprint}
\begin{align*}
    \AP(\lam+\eta) := \CB{(\tilw_1,\tilw_2) \in (\utilW_1^+ \times \utilW^+)/X^0(\uT) \mid \tilw_1 \uparrow t_\lam \tilw_h^\mo \tilw_2}.
\end{align*}

\begin{prop}[Proposition 2.3.7 in \cite{LLHLM-2020-localmodelpreprint}]\label{prop:JH-factor-lam-tau}
Let $\lam \in X^*(\uT)$ be a dominant character. Let $m$ be an integer such that $m\ge \max\{h_{\lam+\eta}, 2h_\eta\}$. For an $m$-generic Deligne--Lusztig representation $R$, we have the following bijection
\begin{align*}
\begin{split}
\AP(\lam+\eta) & \simeq \JH( \ov{R}\otimes_\F W(\lam) ) \\ 
(\tilw_1,\tilw_2) &\mapsto F_{(\tilw_1, \tilw(R)\tilw_2^\mo(0))}.
\end{split}
\end{align*}
Moreover, every Jordan--H\"older factor is $(m-h_{\lam+\eta})$-deep.
\end{prop}

\begin{rmk}
Suppose $\lam=0$ in Proposition \ref{prop:JH-factor-lam-tau}. By \cite[Proposition 2.1.6]{LLHLM-2020-localmodelpreprint}, the condition $\tilw_1 \uparrow \tilw_h^\mo \tilw_2$ is equivalent to $\tilw_2 \uparrow \tilw_h \tilw_1$. Let us write
\begin{align*}
    \nu = \pi^\mo(\tilw_1)\cdot(\mu-\eta + s\tilw_2^\mo(0))
\end{align*}
for $(\tilw_1,\tilw_2)\in \AP(\eta)$ so that $F_{(\tilw_1,\tilw(R)\tilw_2^\mo(0))} = F(\nu)$. 
Then
\begin{align*}
    \pi^\mo(\tilw_2) \cdot (\mu-\eta + s\tilw_2^\mo(0)) \uparrow \tilw_h \cdot \nu
\end{align*}
as in Proposition \ref{prop:JH-DL}.
\end{rmk}

\begin{prop}[Proposition 2.6.2 in \cite{LLHLM-2020-localmodelpreprint}]\label{prop:JH-factor-W?}
Let $m$ be an integer such that $m\ge 2h_\eta$. Let $\rhobar$ be a tame inertial $L$-parameter over $\F$ with an $m$-generic lowest alcove presentation. Then there is a bijection
\begin{align*}
    \{ (\tilw,\tilw_2)\in (\utilW^+_1 \times \utilW^+)/X^0(\uT) \mid \tilw_2 \uparrow \tilw    \} &\risom W^?(\rhobar) \\
    (\tilw,\tilw_2) & \mapsto F_{(\tilw,\tilw(\rhobar)\tilw_2^\mo(0))}.
\end{align*}
Moreover, every Jordan--H\"older factor is $(m-3)$-deep.
\end{prop}

\section{The theory of local models}\label{sec:local-models}
In this section, we generalize the theory of local models in \cite{LLHLM-2020-localmodelpreprint} to the group $\GSp_4$. Note that when we write $\GSp_4$ in this section, we mean the dual group $\GSp_4^\vee$. In particular, the set of coroots $\Phi^\vee$ of $\GSp_4$ is identified with a set of roots of $\GSp_4^\vee$ by the duality isomorphism $\phi$. If  $\al^\vee\in \Phi^\vee$, we let $U_{\al^\vee}\subset \GSp_4^\vee$ denote the root subgroup associated to the root $\phi(\al^\vee)$ of $\GSp_4^\vee$. In other words, $U_\al^\vee \subset \GSp_4^\vee$ is a subgroup such that  $U_{\al^\vee}\simeq \G_a$ and $tut^\mo = \phi(\al^\vee)(t)u$ for any $t\in T^\vee$ and $u\in U_{\al^\vee}$. We let $G=\GSp_4$ for the remainder of this paper.

Let $X= \A^1_\Z$ be an affine line with coordinate function $v$. We denote by $X_0=\spec \Z$ the zero section of $X$ and $X^0 = \spec \Z[v,v^\mo]$. We often write $t:\spec R \ra X$ to denote a $\Z[v]$-algebra $R$ such that $v$ is mapped to $t\in R$.

\subsection{Global affine Grassmannians}\label{sub:global-affine-Grass}
Let $\cG$  be the Neron blowup of $\GSp_{4/X}$ in $B_{/X}$ along $X_0$ defined in \cite[Definition 3.1]{MRR}. By Theorem 3.2 of \loccit, it is a smooth affine group scheme over $X$ with connected fibers. For $t:\spec R \ra X$ with $t$ regular in $R$, the set of $R$-points is given by
\begin{align*}
    \cG(R) = \{g \in \GSp_4(R) \mid \ g \mod t \in B(R/t) \}.
\end{align*}
There is a morphism of $X$-group schemes $\cG \ra \GSp_{4/X}$. If $g\in \cG(R)$, we denote by $\ov{g}$ its image in $\GSp_4(R)$. We also have a similitude character $\simc: \cG \ra \G_m$ sending $g$ to $\simc(\ov{g})$.

The base change $\cG\times_X X^0$ is isomorphic to $\GSp_{4/X^0}$ and the base change along $\varpi: \spec \cO \ra X$ is isomorphic to an Iwahori group scheme inside $\GSp_{4/\cO}$. In particular, $\cG\times_{X,\varpi} \spec \cO$ coincides with the group scheme constructed in \cite[Corollary 4.2]{PZ13-Inv-local_model-MR3103258} for $\GSp_4$ and the Iwahori group scheme (see also \cite[Example 3.3]{MRR}).

We also define a functor $L^+\cM$ whose $R$-points, for $t: \spec R \ra X$, are given by
\begin{align*}
    L^+\cM(R):= \{g\in \Lie \GSp_4 (R\DB{v-t}) \mid \text{$g$ is upper triangular modulo $v$}  \}.
\end{align*}

Let $\cG_4$ be the Bruhat--Tits group for $\GL_4$ defined as in \cite[\S3.1]{LLHLM-2020-localmodelpreprint}. Note that the map $\std: \GSp_4 \ra \GL_4$ induces a morphism $X$-group schemes $\cG \ra \cG_4$ (see \cite[\S2.4]{MRR}) which we denote by $\std$ as well. It is easy to see that $\std: \cG \ra \cG_4$ is a closed immersion.  

We write $L\cG$ for the loop group and $L^+\cG$ for the positive loop group of $\cG$. For $t: \spec R \ra X$, their $R$-points are given by
\begin{align*}
    L\cG (R) &=  \cG (R\DP{v-t}) \\
    L^+\cG (R) &= \cG(R\DB{v-t})
\end{align*}
where we consider $R\DB{v-t}$ and $R\DP{v-t}$ as $\Z[v]$-algebra by sending $v$ to $v$. It is known that $L^+\cG$ is representable by a (not finite type) group scheme and $L\cG$ is representable by an ind-group scheme.

\begin{rmk}\label{rmk:LG-valued-in-noetherian-R}
When $R$ is Noetherian, $v$ is regular in both $R\DB{v-t}$ and  $R\DP{v-t}$. Thus we have the following description:
\begin{align*}
    L\cG(R) & = \CB{g \in \GSp_4(R\DP{v-t}) \mid g  \mod v \in B(R\DP{v-t}/v) } \\ 
    L^+\cG(R) & = \CB{g \in \GSp_4(R\DB{v-t}) \mid g \mod v \in B(R\DB{v-t}/v) }.
\end{align*}
\end{rmk}

We define $\Gr_{\cG,X}$ to be the fpqc quotient sheaf $L^+\cG \bss L\cG$. By \cite[Proposition 6.5]{PZ13-Inv-local_model-MR3103258}, $\Gr_{\cG,X}$ is representable by an ind-projective ind-scheme. 
By the properties of $\cG$, the generic fiber $\Gr_{\cG,X}\times_X X^0$ is isomorphic to $\Gr_{\GSp_4}\times_\Z X^0$, a constant family of affine Grassmannian for the group $\GSp_4$ over $X^0$, and its special fiber $\Gr_{\cG,X}\times_X X_0$ is the affine flag variety $\Fl:=\cI \bss\GSp_{4}$.

Let $d \in \Z$ and $h\in \Z_{\ge0}$. We define subfunctors  $L\cG^{\simc=d}, L\cG^{\simc=d,\le h}\subset L\cG$ by
\begin{align*}
        L\cG^{\simc=d}(R) &= \CB{ g \in L\cG(R) \mid \simc(\ov{g}) \in (v-t)^d (R\DB{v-t})^\times} \\
        L\cG^{\simc=d,\le h}(R) &= \CB{g\in L\cG^{\simc=d}(R) \mid \ov{g}\in \frac{1}{(v-t)^h}M_4(R\DB{v-t})}.
\end{align*}
Both are stable under left multiplication by $L^+\cG$ and induce fpqc quotient subsheafs
\begin{align*}
    \Gr_{\cG,X}^{\simc=d,\le h} := L^+\cG \bss L\cG^{\simc=d,\le h} \subset \Gr_{\cG,X}^{\simc=d}:= L^+\cG \bss L\cG^{\simc=d} \subset \Gr_{\cG,X}.
\end{align*}
The sheaf $\Gr_{\cG,X}^{\simc=d,\le h}$ is representable by a projective scheme over $X$ and $\Gr_{\cG,X}^{\simc=d} = \varinjlim_{h} \Gr_{\cG,X}^{\simc=d,\le h}$.

Our next goal is to describe affine open charts of the projective scheme $\Gr_{\cG,X}^{\simc=d,\le h}$. We define the negative loop group  to be a subfunctor $L^\mm\cG\subset L \cG$ such that for $t:\spec R \ra X$,
{\small\begin{align*}
    L^\mm\cG(R) = \CB{g\in \cG(R\DP{v-t}) \middle\vert 
    \begin{array}{c}
        \ov{g}\in \GSp_4\PR{R\BR{\frac{1}{v-t}}}\\
        \ov{g} \mod {\frac{1}{v-t}} \in \ov{U}(R)  \\
         \ov{g} \mod\frac{v}{v-t} \in B\PR{R\BR{\frac{1}{v-t}}/\PR{\frac{v}{v-t}}} 
    \end{array}
    }.
\end{align*}}

\begin{lemma}\label{lem:mult-formally-et}
The multiplication map
\begin{align*}
    L^+\cG\times_X L^\mm\cG \ra L\cG
\end{align*}
is a formally \'etale (in the sense of \cite[Definition 3.2.4]{LLHLM-2020-localmodelpreprint}) monomorphism, and so is the map $L^\mm \cG \ra \Gr_{\cG,X}$.
\end{lemma}
\begin{proof}
Since this map is a restriction of a monomorphism $L^+\cG_4 \times_X L^\mm \cG_4 \ra L\cG_4$ (\cite[Lemma 3.2.2]{LLHLM-2020-localmodelpreprint}), it is a monomorphism. Then we can show that it is formally \'etale following the argument in \cite[Lemma 3.2.6]{LLHLM-2020-localmodelpreprint} using a version of \cite[Lemma 3.2.3]{LLHLM-2020-localmodelpreprint} for our setup and formal smoothness of $L^+\cG$ and $L^\mm \cG$. Note that formal smoothness easily follows from smoothness of $\cG$.
\end{proof}

Let $\til{z}=wt_\nu \in \til{W}^{\vee}$. We define $\cU(\til{z})$ to be the subfunctor of $L\cG$ given by
{\small\begin{align*}
    \cU(\til{z})(R) = \CB{g \in L\cG(R) \middle\vert \begin{array}{cl}
        \ov{g}(v-t)^{-\nu} \in \GSp_4(R[\frac{1}{v-t}]),    \\ 
         \ov{g}(v-t)^{-\nu}w^\mo \mod \frac{1}{v-t} \in \ov{U}(R), \\
         \ov{g}(v-t)^{-\nu} \mod \frac{v}{v-t} \in B(R[\frac{1}{v-t}]/(\frac{v}{v-t}))
    \end{array}}.
\end{align*}}

\begin{lemma}\label{lem:U(z)-for-et}
The subfunctor $\cU(\tilz) \subset L\cG$ is stable under left multiplication by $L^\mm \cG$ and is a left $L^\mm \cG$-torsor. The natural map $\cU(\til{z}) \ra \Gr_{\cG,X}$ is formally \'etale monomorphism.
\end{lemma}
\begin{proof}
The first claim follows from the definition of $\cU(\tilz)$. The map is monomorphism by Lemma \ref{lem:mult-formally-et}. 
To prove the formal \'etaleness, we show that $\cU(\tilz)$ is formally smooth over $X$ as in the proof of \cite[Lemma 3.2.7]{LLHLM-2020-localmodelpreprint}.

Let $R\epi S$ be a square-zero nilpotent thickening with kernel $J$. Given $A_S \in \cU(\tilz)(S)$, we want to find $A_R\in \cU(\tilz)(R)$ lifting $A_S$.  Since $A_S(v-t)^{-\nu} \in \GSp_4(S[{\frac{1}{v-t}}])$, we can find  $A_R' \in L\cG(R)$  lifting $A_S$ and $A_R'(v-t)^{-\nu} \in \GSp_4(R[{\frac{1}{v-t}}])$ by the smoothness of $\GSp_4$, and the set of such lifts of $A_S$ is $(I+\Lie \GSp_4(J[\frac{1}{v-t}]))A_R'$, where $I$ is the identity matrix. 
To show that there exists $A_R\in \cU(\tilz)(R)$  lifting $A_S$, we need to find $A_R$ satisfying two conditions:
\begin{align*}
    A_R(v-t)^{-\nu}w^\mo \mod \frac{1}{v-t} \in \ov{U}(R),\ \   A_R(v-t)^{-\nu} \mod \frac{v}{v-t} \in B(R[\frac{1}{v-t}]/(\frac{v}{v-t})).
\end{align*}
Again, the existence of a lift satisfying each condition follows from the smoothness of $\ov{U}$ and $B$. The existence of a lift satisfying all three conditions follows from the surjectivity of the quotient map
\begin{align*}
    \Lie \GSp_4(J[\frac{1}{v-t}]) \ra \Lie \GSp_4(J) \times \Lie \GSp_4(J[\frac{1}{v-t}]/ (\frac{v}{v-t})). &\qedhere
\end{align*} 
\end{proof}

For $h\in \Z_{\ge0}$,  we define $\cU(\til{z})^{\simc,\le h}$ to be the base change $\cU(\til{z})\times_{L\cG} L\cG^{\simc=d,\le h} $ where $d= \simc(\nu)$. The following Proposition shows that $\cU(\til{z})^{\simc,\le h}$ is represented by a finite type affine scheme over $X$.

\begin{prop}\label{prop:moduli-interpretation-open-chart}
For a Noetherian $\Z[v]$-algebra $R$, $\cU(\til{z})^{\simc,\le h} (R)$ is the set of matrices $A\in \GSp_4(R[(v-t)^{\pm 1}])$ satisfying:
\begin{itemize}
    \item For $1\le i,j \le 4$,
    \begin{align*}
        A_{ij} = v^{\delta_{i>j}}\PR{\sum_{k=-h}^{\nu'_j-\delta_{i>j} - \delta_{i<w'(j)}} c_{ij,k}(v-t)^k }
    \end{align*}
    and $c_{w'(j) j, \nu'_j - \delta_{w'(j)>j}} = 1$ where $(\nu'_1,\nu'_2,\nu'_3,\nu'_4)=\std(\nu)$ and $w'=\std(w)$,
    \item $\simc (A) = \simc(w) (v-t)^d$.
\end{itemize}
\end{prop}
\begin{proof}
This follows from the fact that $\cU(\til{z})^{\simc,\le h} = \cU_4(\std(\tilz))^{\det,\le h}\times_{L\cG_4} L\cG$ and Proposition 3.2.8 in \cite{LLHLM-2020-localmodelpreprint} (where $\cU_4(\std(\tilz))^{\det,\le h}$ denotes the affine chart defined in \S3.2 of \loccit).
\end{proof}

\begin{cor}\label{prop:tilU-open-in-Gr}
The map $\cU(\til{z})^{\simc,\le h} \ra \Gr_{\cG,X}^{\simc=d,\le h}$ is an open immersion. 
\end{cor}

\begin{proof}
Since a formally \'etale monomorphism between finite type schemes is an open immersion (\cite[Remark 3.2.5]{LLHLM-2020-localmodelpreprint} and \cite[\href{https://stacks.math.columbia.edu/tag/025G}{Tag 025G}]{stacks-project}), this follows from Lemma \ref{lem:U(z)-for-et}.
\end{proof}

\subsection{Geometry of universal local models}\label{sub:uni-local-models}
We introduce universal local models and discuss their basic properties.

\subsubsection{Schubert varieties}\label{subsec:schubert-varities}
Given a dominant cocharacter $\lam\in X_*(T^\vee)$, we denote by $s_\lam$ the section  $X\ra \Gr_{\cG,X}$ induced by the element $(v-t)^\lam\in L\cG(R)$ 
for any $\Z[v]$-algebra $R$. A \emph{global Schubert variety} $\cS_{X}(\lam)$ is defined as the minimal irreducible closed subscheme of $\Gr_{\cG,X}$ containing the section $s_\lam$ and stable under the right multiplication by $L^+\cG$ (cf.~\cite[Definition 3.1]{Zhu14-Annals-coherence_conj-MR3194811}). The map $S_X(\lam) \ra X$ is proper. We also write $\cS_{X^0}(\lam)= \cS_X(\lam)\times_X X^0$. As $\Gr_{\cG,X}\times_X X^0 \simeq \Gr_{\GSp_4}\times_\Z X^0$, $\cS_{X^0}(\lam)$ is the constant family of Schubert varieties in $\Gr_{\GSp_4}$ for $\lam$  over $X^0$. 

Let $\Conv(\lam)$ be the convex hull of the subset $W\lam \subset X_*(T^\vee)$. 
An \emph{open Schubert cell} $\cS^\circ_X(\lam)$ is defined as an open subscheme
\begin{align*}
    \cS^\circ_X(\lam) = \cS_X(\lam)\backslash \cup_{\lam'\in \Conv(\lam), \lam'\notin W\lam} \cS_X(\lam') \subset \cS_X(\lam).
\end{align*}
Again, the base change $\cS^\circ_{X^0}(\lam) = \cS^\circ_X(\lam)\times_X X^0$ is the constant family of open Schubert cells of $\Gr_{\GSp_4}$ for $\lam$  over $X^0$.

We have a map $L^+\cG \ra \Gr_{\cG,X}$ given by the orbit map $g \mapsto s_\lam g$. Note that it factors through a subscheme $\Gr_{\cG,X}^{\simc=d,\le h}$ with $d=\simc(\lam)$ and $h$ sufficiently large. The stabilizer subgroup scheme $L^+\cG_\lam \subset L^+\cG$ of $s_\lam$  is given by
\begin{align*}
    L^+\cG_\lam(R) = L^+\cG(R) \cap \Ad ((v-t)^{-\lam})(L^+\cG(R))
\end{align*}
for $t: \spec R \ra X$.  Thus we have a monomorphism
\begin{align*}
    L^+\cG_\lam \bss L^+\cG \mono \Gr_{\cG,X}^{\simc=d,\le h}
\end{align*}
whose scheme-theoretic image is $\cS_\lam(X)$. Over $X^0$, we have an isomorphism $(L^+\cG_\lam\bss L^+\cG) \times_X X^0 \simeq \cS^\circ_{X^0}(\lam)$. Note that there is a map from $L^+\cG_\lam \times_X X^0$ to $P_\lam \times X^0$ sending $g \mapsto g \mod (v-t)$ where  $P_\lam\subset \GSp_4$ is the parabolic subgroup associated with $\lam$. Thus, we have a natural map $(L^+\cG_\lam\bss L^+\cG) \times_X X^0 \ra P_\lam \bss \GSp_4 \times_\Z X^0$ given by $g \mapsto g \mod (v-t)$. Composing this map with the preceding isomorphism, we get a map
\begin{align*}
    \pi_\lam: \cS^\circ_{X^0}(\lam) \xrightarrow{} (P_\lam \backslash \GSp_4)\times_\Z X^0.
\end{align*}

\subsubsection{Universal local models}
For convenience, we let $\std: \A^3 \ra \A^4$ be the morphism sending $(a_1,a_2,a_3) $ to $(a_1,a_2,a_3-a_2, a_3-a_1)$. Note that this matches with the description of $\std: X_*(T^\vee) \ra X_*(T^\vee_4)$.

We define a subfunctor $L\cG^\nba$ of $L\cG \times_\Z \A^3$ given by
\begin{align*}
    L\cG^\nba(R) = \CB{(g,\bfa) \mid g\in L\cG(R), \bfa\in \A^3, v\ddv{\ov{g}} \ov{g}^\mo + \ov{g}\diag(\std(\bfa)) \ov{g}^\mo \in \frac{1}{v-t}L^+\cM(R)}.
\end{align*}
for $t: \spec R \ra X$.  Since $L\cG^\nba$ is stable under left multiplication by $L^+\cG$, it defines a closed sub-ind-scheme $\Gr^\nba_{\cG,X}:= L^+\cG\bss L\cG^\nba \subset \Gr_{\cG,X}\times_\Z \A^3$ which is ind-proper over $X\times_\Z \A^3$.

\begin{defn}
Let $\lam\in X_*(T^\vee)$ be a dominant cocharacter.  We define the \emph{naive universal local model} $\cM^\nv_X (\le \lam, \nba)$ as  $\Gr^\nba_{\cG,X}\cap (\cS_X(\lam)\times_\Z \A^3)$.
\end{defn}

We write $\cM^\nv_{X^0}(\le \lam,\nba) = \cM^\nv_{X}(\le \lam,\nba)\times_X X^0$. It is a proper scheme over $X^0\times \A^3$.

\begin{prop}\label{prop:schubert-cell-reduction}
Let $\lam\in X_*(T^\vee)$ be a dominant cocharacter. The map $\pi_\lam$ induces an isomorphism
\begin{align*}
    \PR{\cM^\nv_{X^0}(\le\lam,\nba) \cap (\cS^\circ_{X^0}(\lam)\times_\Z \A^3)}\BR{\frac{1}{h_{\lam} !}} \simeq (P_{\lam}\bss \GSp_4) \times_\Z X^0 \times_\Z \A^3\BR{\frac{1}{h_{\lam}!}}.
\end{align*}
\end{prop}

\begin{proof}
The proof of \cite[Proposition 3.3.4]{LLHLM-2020-localmodelpreprint} carries over verbatim. As in \loccit, we can describe $\pi_{\lam}$ on open charts where the open charts are described by a certain unipotent group scheme. Then the computation in \loccit~easily generalizes to our setting using the decomposition of the unipotent group scheme by root subgroups instead of matrix entries. 
\end{proof}

\begin{rmk}\label{rmk:generic-fiber-conn-comp}
The preceding proposition implies that $\cM^\nv_{X^0}(\le \lam, \nba) \cap (\cS^\circ_{X^0}(\lam) \times_\Z \A^3)\BR{\frac{1}{h_\lam !}}$ is indeed proper over $X^0\times_\Z \A^3 \BR{\frac{1}{h_\lam !}}$ (by the properness of partial flag varieties). Thus the map
\begin{align*}
    \cM^\nv_{X^0}(\le\lam, \nba) \cap (\cS^\circ_{X^0}(\lam) \times_\Z \A^3)\BR{\frac{1}{h_\lam !}} \mono  \cM^\nv_{X^0}(\le \lam,\nba)\BR{\frac{1}{h_\lam !}}
\end{align*}
is a proper open immersion. Since a proper morphism is universally closed, this map is an inclusion of a connected component. Using this inductively, we obtain an isomorphism (cf.~\cite[Corollary 3.3.5]{LLHLM-2020-localmodelpreprint})
\begin{align*}
    \PR{\cM^\nv_{X^0}(\le\lam, \nba)\BR{\frac{1}{h_\lam !}}}_{\red} \risom \coprod_{\lam'\le\lam,\lam'\in X_*^+(T^\vee)} (P_{\lam'}\bss \GSp_4)\times_\Z X^0 \times_\Z \A^3\BR{\frac{1}{h_\lam !}}.
\end{align*}
\end{rmk}

 \begin{defn}
    Let $\lam\in X_*(T^\vee)$ be a dominant cocharacter. The \emph{universal local model} $\cM_X(\lam,\nba)$ is the closure of the connected component $\cM_{X^0}^\nv (\le\lam,\nba) \cap (\cS^\circ_{X^0}(\lam)\times_\Z \A^3)$ of $\cM_{X^0}^\nv (\le\lam,\nba)$ inside $\cM_X^{\nv}(\le \lam, \nba)$. In particular, it is $v$-flat.
\end{defn}

\begin{prop}\label{prop:nv-local-model-smooth}
 Let $\lam\in X_*(T^\vee)$ be a dominant cocharacter. Then $\cM_{X^0}^\nv (\le \lam, \nba)\BR{\frac{1}{(2h_\lam)!}}$ is smooth over $X^0\times_\Z \A^3\BR{\frac{1}{(2h_\lam)!}}$.
\end{prop}
\begin{proof}
The proof of \cite[Proposition 3.3.8]{LLHLM-2020-localmodelpreprint} generalizes to our situation straightforwardly using the Lie algebra of $\GSp4$ instead of the matrix algebra in the tangent space dimension computation.  
\end{proof}    

We finish this subsection with two lemmas regarding the normalization of the universal local model. For an integral scheme $Y$, we let $Y^\nm \ra Y$  denote the normalization of $Y$.

Let $l>0$ be an integer. It will be useful to consider the base change \begin{align*}
    \cM_X(\lam,\nba)_l := \cM_X(\lam,\nba)\times_{X, v\mapsto v^l} X
\end{align*}
and its normalization $(\cM_X(\lam,\nba)_l)^\nm$. When $l=e$, if we take $\varpi= (-p)^{1/e}$ as the uniformizer in $\cO$, this has the advantage that the map $\spec \cO \ra X\times \A^3$ sending $v$ to $-p$ is the composite of the maps $\spec \cO \ra X\times \A^3$ sending  $v \mapsto \varpi$ and $X\times \A^3 \ra X\times \A^3$ sending $v\mapsto v^e$. In particular, we have 
\begin{align*}
    \cM_X(\lam,\nba) \times_{X\times \A^3, v\mapsto -p} \spec \cO = \cM_X(\lam,\nba)_e \times_{X\times \A^3, v\mapsto\varpi} \spec \cO.
\end{align*}

\begin{lemma}\label{lem:univ-local-model-flat}
There is an open subscheme $U\subset \A^3$ depending only on $\lam$ and $l$, such that $\cM_X(\lam,\nba)_l\times_{\A^3} U \ra X \times_\Z U$ and $(\cM_X(\lam,\nba)_l)^\nm \times_{\A^3} U \ra X \times_\Z U$ are flat.
\end{lemma}
\begin{proof}
This follows from \cite[Lemma 3.5.5, Remark 3.5.6]{LLHLM-2020-localmodelpreprint}.
\end{proof}

\begin{lemma}\label{lem:spreading-normality}
There is an open subscheme $U\subset \A^3$ only depending on $\lam$ and $l$, such that if $R$ is a complete DVR and $f: \spec R \ra X\times \A^3$ is a morphism sending $v$ to a uniformizer of $R$ and factors through $X\times U$, the base change 
\begin{align*}
    (\cM_X(\lam,\nba)_l)^\nm_R := (\cM_X(\lam,\nba)_l)^\nm \times_{X\times \A^3} \spec R 
\end{align*}
is flat over $\spec R$, and  $(\cM_X(\lam,\nba)_l)^\nm_R$ is normal.
\end{lemma}
\begin{proof}
This follows from \cite[Proposition 3.5.2]{LLHLM-2020-localmodelpreprint}. To satisfy Setup 3.5.1 of \emph{loc.~cit.}, take $S= \A^3[\frac{1}{(2h_\lam)!}]$, $M= (\cM_X(\lam,\nba)_l)^\nm \times_{\A^3} S$, and use Proposition \ref{prop:nv-local-model-smooth}.
\end{proof}

\subsection{Local models in mixed characteristic}\label{sub:mixed-char-local-models}
We specialize to objects over $\cO$ by taking base change $\spec \cO \ra X$ sending $v$ to $-p$. For an fpqc sheaf $Y \ra X$, we write $Y_\cO = Y\times_X \spec \cO$. For example, we have $L\cG_\cO= L\cG \times_X \spec \cO$ and $L^+\cG_\cO, L^+\cM_\cO$ similarly.

We have the global affine Grassmannian $\Gr_{\cG,\cO}= L^+\cG_\cO \bss L\cG_\cO$. Its generic fiber $\Gr_{\cG,E}$ is equal to the usual affine Grassmannian associated to $\GSp_4$ over $E$. For $\lambda\in X_*(T^\vee)$ a dominant cocharacter, we write $S_E^\circ(\lam)\subset \Gr_{\GSp_4,E}$ for the open affine Schubert cell and $S_E(\lam)$ for its reduced closure. The Zariski closure of $S_E(\lam)$ in $\Gr_{\cG,\cO}$ is the Pappas--Zhu local model $M(\le \lam)$ associated to the group $\GSp_4$, the conjugacy class of $\lam$, and the Iwahori subgroup $\cI$ (\cite{PZ13-Inv-local_model-MR3103258}). It is known that $M(\le\lam)$ is projective over $\spec \cO$ (\S7.1 in \loccit).

Let $\bfa \in \cO^3$. We define $L\cG_\cO^{\nba_\bfa}\subset L\cG_\cO$ to be the subfunctor given by
\begin{align*}
    L\cG_\cO^{\nba_\bfa} (R) := \CB{ g \in L\cG_\cO(R) \mid v\ddv{\ov{g}}\ov{g}^\mo + \ov{g} \diag(\std(\bfa))\ov{g}^\mo \in \frac{1}{v+p} L^+\cM_4(R)}
\end{align*}
for an $\cO$-algebra $R$. It is stable under left $L^+\cG_{\cO}$ multiplication and induces a closed sub-ind-scheme $\Gr^{\nba_\bfa}_{\cG,\cO} \subset \Gr_{\cG,\cO}$.
\begin{defn}\label{def:local-model}
Let $\bfa \in \cO^3$. We define the \emph{naive local model} as $M^\nv(\le\lam, \nba_\bfa)= M(\le\lam) \cap \Gr_{\cG,\cO}^{\nba_\bfa}$. The \emph{(mixed characteristic) local model} $M(\lam,\nba_\bfa)$ is defined to be the Zariski closure of $S^\circ_E(\lam)\cap \Gr_{\cG,\cO}^{\nba_\bfa}$ in $M(\le \lam)$. By their construction, $M^\nv(\le\lam, \nba_\bfa)$ and $M(\lam,\nba_{\bfa})$ are projective over $\spec \cO$, and the latter is also flat.
\end{defn}

For later application, we define $M(\le\lam,\nba_{\bfa})$ to be the union of $M(\lam',\nba_{\bfa})$ over all dominant cocharacters $\lam'\le\lam$. Since $S^\circ_E(\lam') \subset S_E(\lam)$, $M(\lam',\nba_{\bfa})$ and $M(\le\lam,\nba_{\bfa})$ are closed subschemes of $M^\nv(\le\lam, \nba_\bfa)$. 

\begin{prop}\label{prop:generic-fiber-M(=<lam)}
We have an equality 
\begin{align*}
    M(\le\lam,\nba_{\bfa})\times_{\spec \cO} \spec E = M^\nv(\le\lam, \nba_\bfa)\times_{\spec \cO} \spec E.
\end{align*}
\end{prop}
\begin{proof}
   By Remark \ref{rmk:generic-fiber-conn-comp} (and by taking base change $\spec E \ra X$), for dominant $\lam'\le\lam$, $M(\lam',\nba_{\bfa}) \times_{\spec \cO} \spec E$ is a connected component of $M^\nv(\le\lam, \nba_\bfa)\times_{\spec \cO} \spec E$. The claim follows by taking the union over all dominant $\lam'\le\lam$.
\end{proof}

We define $U(\tilz) := \cU(\tilz) \times_{X}\spec \cO$ and  $U(\tilz,\lam,\nba_{\bfa}) := U(\tilz) \cap M(\lam,\nba_\bfa)$ where the intersection is taken inside $\Gr_{\cG,\cO}$. Note that the latter is equal to $(\cU(\tilz)^{\simc,\le h}\times_X \spec \cO)\cap M(\lam,\nba_\bfa)$ for large enough $h\ge 0$. Thus, it is an open affine subscheme of $M(\lam,\nba_\bfa)$ by Corollary \ref{prop:tilU-open-in-Gr}.

Our main interest in the geometry of $M(\lam,\nba_\bfa)$ is whether a completed local ring $\mathcal{O}_{M(\lam,\nba_\bfa),x}^\wedge$ is a domain for some $x \in M(\lam,\nba_\bfa)$ . To understand this property more geometrically, we introduce the following definition. 

\begin{defn}
Let $Y$ be a scheme. A point $y \in Y$ is called \emph{unibranch} if the normalization of the local ring $(\cO_{Y,y})_\red$ is local. 
\end{defn}

\begin{rmk}
Suppose that $Y$ is Noetherian and excellent. The following conditions are equivalent:
\begin{enumerate}
    \item $y\in Y$ is unibranch;
    \item the fiber above $y$ of the normalization map $Y^\nm \ra Y$ is a single point (\cite[\href{https://stacks.math.columbia.edu/tag/0C3B}{Tag 0C3B}]{stacks-project});
    \item the completed local ring $\cO_{Y,y}^\wedge$ is a domain (\cite[\href{https://stacks.math.columbia.edu/tag/0C2E}{Tag 0C2E}]{stacks-project}).
\end{enumerate}
\end{rmk}

Let $\til{z}=w t_{\nu}\in \tilW^{\vee}$. There is a constant section $\tilz: \spec \Z \mono \cU(\til{z})^{\simc,\le h} \times_X X_0$ for $h$ large enough,  given by $w v^\nu \in \GSp_4(\Z\DP{v})$. We denote its composition with $\cU(\til{z})^{\simc,\le h}\times_X X_0 \mono \Gr_{\cG,X} \times_X X_0$ 
again by $\til{z}$.

Let $\bfa\in \A^3(\cO)$. We also denote the induced $\F$-point by $\bfa$. We write $\til{z}_{\F,\bfa}$ for the $\F$-point $ \spec \F \ra (\Gr_{\cG,X} \times_X X_0)\times_\Z \A^3$ given by $(\til{z},\bfa)$.

Recall that $e$ denotes the ramification index of the extension $\cO/\Zp$.

\begin{thm}\label{thm:unibranch-local-model}
There exists a non-empty open subscheme $U\subset \A^3$ depending only on $\lam$ and $e$, such that if $\bfa \in U(\cO)$, then $M(\lam, \nba_\bfa)$ is unibranch at any point $\til{z}_{\F,\bfa}$ contained in the special fiber. In addition, $\cO(U(\til{z},\lam,\nba_\bfa))^\pcp$ is a domain. 
\end{thm}

This is a preliminary version of the main unibranch result (Theorem \ref{thm:unibranch-product}). Their proofs follow from the proof of \cite[Theorem 3.7.1]{LLHLM-2020-localmodelpreprint}, but we provide a detailed argument in the hope that it will be useful to the reader. We first introduce the local model in equal characteristic and prove its unibranch property.

\begin{defn}
Let $\bfa\in \A^3(\F)$. We define $\cM^\nv(\le \lam,\nba_\bfa):= \cM_X^\nv(\le \lam,\nba)\times_{\A^3}\bfa$. 
We define the \emph{equal characteristic local model} $\cM(\lam,\nba_\bfa)$ by the Zariski closure of
\begin{align*}
    \PR{\cM^\nv_{X^0}(\le\lam,\nba)\cap (\cS^\circ_{X^0}(\lam)\times_\Z \A^3)}\times_{\A^3} \bfa
\end{align*}
in $\cM^\nv(\le\lam,\nba_\bfa)$. Note that $\cM^\nv(\le\lam,\nba_\bfa)$ and $\cM(\lam,\nba_\bfa)$ are projective over $\A^1_\F$ and the latter is furthermore flat.
\end{defn}

\begin{rmk}\label{rmk:comparing-generic-fiber-of-various-local-models}
Although mixed or equal characteristic local models are generally \emph{not} equal to base changes of universal local models, their generic fibers can be obtained by taking base changes of the generic fibers of universal local models. Suppose that $h_\lam < p $. 
Taking fiber product over the $E$-point $(-p,\bfa)\in X \times \A^3$ (resp.~at $\F$-point $\bfa$), we have
\begin{align*}
    \cM_{X^0}(\lam,\nba) \times_{X\times \A^3} \spec E  &=  M(\lam,\nba_\bfa)\times_\cO \spec E
    \\
    \cM_{X^0}(\lam,\nba) \times_{\A^3} \spec \F &= \cM(\lam,\nba_\bfa) \times_{X}X^0.
\end{align*}
This is because the generic fibers on the right-side are connected components of the generic fibers of the corresponding naive local models by Remark \ref{rmk:generic-fiber-conn-comp}.  Moreover, $\cM_{X}(\lam,\nba) \times_{X\times \A^3} \spec \cO$ (resp.~$\cM_{X}(\lam,\nba) \times_{\A^3} \spec \F$) is closed in $M^\nv(\lam,\nba_{\bfa})$ (resp.~$\cM^\nv(\lam,\nba_\bfa)$). Thus, $M(\lam,\nba_\bfa)$ (resp.~$\cM(\lam,\nba_\bfa)$) is characterized as $p$-flat (resp.~$v$-flat) closure of $\cM_{X^0}(\lam,\nba)\times_{X\times \A^3} \spec E$ inside $\cM_X(\lam,\nba) \times_{X\times \A^3} \spec \cO$ (resp.~$\cM_{X^0}(\lam,\nba)\times_{X\times \A^3} \spec E$ inside $\cM_X(\lam,\nba) \times_{\A^3} \spec \F$).
\end{rmk}

For $l\in \Z_{> 0}$ and an $X$-scheme $Y$, we write $Y_l$ for the base change $Y\times_{X, v\mapsto v^l} X$.

\begin{prop}[cf.~Proposition 3.4.4 in \cite{LLHLM-2020-localmodelpreprint}]\label{prop:unibranch-equal-char}
Let $l>0$ be an integer. Then $\cM(\lam,\nba_\bfa)_l$ is unibranch at any $\til{z}_{\F,\bfa}$ contained in its special fiber. Moreover, the preimage of $\cU(\tilz)$ in $(\cM(\lam,\nba_\bfa)_l)^\nm \times_X X_0$ is connected.
\end{prop}

\begin{proof}
This mainly follows from  \cite[Lemma 3.4.7, 3.4.8]{LLHLM-2020-localmodelpreprint}. Indeed, we can take a one-parameter subgroup in Lemma 3.4.7 in \emph{loc.\;cit.}~valued in $T^{\vee}\times \G_m$ (instead of $T^\vee_4 \times \G_m$) whose induced action on $\cU(\tilz)^{\simc,\le h}$ (instead of $\cU_4(\tilz)^{\det,\le h}$) satisfies the conditions stated in Lemma 3.4.7 in \emph{loc.\;cit.}. Using this, the proof of \cite[Proposition 3.4.4]{LLHLM-2020-localmodelpreprint} can be applied in our case too.
\end{proof}

\begin{proof}[Proof of Theorem \ref{thm:unibranch-local-model}]
The novel idea in \cite{LLHLM-2020-localmodelpreprint} is to compare the mixed characteristic and equal characteristic local models inside the universal local model. Let $U$ be an open subscheme $U \subset \A^3$ satisfying Lemmas \ref{lem:univ-local-model-flat}  and \ref{lem:spreading-normality} for $l=e$ and $(2h_\lam)!e!$ is invertible. Note that this implies $U(\cO) = \emptyset$ unless $p$ does not divide $(2h_\lam)!e!$. Thus we can assume that $\cO/\Z_p$ is tame and take $\varpi = (-p)^{1/e}$ (after enlarging the residue field if necessary). Let $\bfa \in U(\cO)$.  The following diagram explains how to make such a comparison.
\[
\begin{tikzcd}
& M(\lam,\nba_\bfa)^\nm \arrow[ld, dashed, "\sim"]  \arrow[d] & \\
(\cM_X(\lam,\nba)_e)^\nm\times_{X\times U} \spec \cO  \arrow[r] \arrow[d] & M(\lam,\nba_\bfa) \arrow[r] \arrow[d] & \spec \cO \arrow[d, "{(v\mapsto \varpi,\bfa)}"] \\
(\cM_X(\lam,\nba)_e)^\nm \arrow[r] & \cM_X(\lam,\nba)_e \arrow[r] & X\times U \\
(\cM_X(\lam,\nba)_e)^\nm\times_{X\times U} \A^1_\F \arrow[u] \arrow[r] & \cM(\lam,\nba_\bfa)_e \arrow[u] \arrow[r] & \A^1_\F \arrow[u, "{(v\mapsto v,\bfa)}"'] \\
& ( \cM(\lam,\nba_\bfa)_e)^\nm  \arrow[lu, dashed, two heads] \arrow[u] & 
\end{tikzcd}
\]
By Lemma \ref{lem:univ-local-model-flat}, $\cM_X(\lam,\nba)_e$ is flat over $X\times U$. Thus, all rectangles are cartesian by Remark \ref{rmk:comparing-generic-fiber-of-various-local-models}. Moreover, two base changes of normalization map $(\cM_X(\lam,\nba)_e)^\nm \ra \cM_X(\lam,\nba)_e$ are finite and birational. Finiteness is obvious, and birationality is preserved by base change  because the dense open subscheme $\cM_{X^0}(\lam,\nba)_e\subset \cM_{X}(\lam,\nba)_e$, which is already normal by Proposition \ref{prop:nv-local-model-smooth}, is still dense after each base change. This induces two surjective dashed arrows by \cite[\href{https://stacks.math.columbia.edu/tag/035Q}{Tag 035Q}]{stacks-project}, and the top dashed arrow is an isomorphism by Lemma \ref{lem:spreading-normality}.

Note that $M(\lam,\nba_\bfa)$ and $\cM(\lam,\nba_\bfa)$ share the same special fiber. Suppose that $M(\lam,\nba_\bfa)$ is \emph{not} unibranch at $\til{z}_{\F,\bfa}$. Then there are at least two points in the preimage of $\til{z}_{\F,\bfa}$ in $M(\lam,\nba_\bfa)^\nm$. Therefore, the preimage of $\til{z}_{\F,\bfa}$ in $(\cM(\lam,\nba)_e)^\nm$ contains at least two points. In turn, this implies that the preimage of $\til{z}_{\F,\bfa}$ in $(\cM_X(\lam,\nba)_e)^\nm\times_{X\times U} \A_\F^1$ contains at least two points. By the surjectivity of the bottom dashed map, this contradicts Proposition \ref{prop:unibranch-equal-char}. This proves that $M(\lam,\nba_\bfa)$ is unibranch at $\tilz_{\F,\bfa}$.

Note that $\cO(U(\tilz,\lam,\nba_\bfa))[1/\varpi]$ is a regular domain by Proposition \ref{prop:nv-local-model-smooth}. Also, the preimage of $U(\tilz,\lam,\nba_\bfa)$ in the special fiber of $M(\lam,\nba_\bfa)^\nm$ is connected by Proposition \ref{prop:unibranch-equal-char} and the above diagram. Then $\cO(U(\tilz,\lam,\nba_\bfa))^\pcp$ is a domain by \cite[Lemma 3.7.2]{LLHLM-2020-localmodelpreprint}. 
\end{proof}

\subsection{Products of local models}\label{sub:products-local-models}
We now generalize Theorem \ref{thm:unibranch-local-model} to products of local models. Let $\cJ$ be a finite set. Let $\lam = (\lam_j)_{j\in\cJ} \in X_*(T^\vee)^\cJ$ be a dominant cocharacter and $\til{z} = (\til{z}_j)_\jj \in \tilW^{\vee,\cJ}$. We define
\begin{align*}
    \cM_{X,\cJ}(\lam,\nba) = \prod_{\jj} \cM_X(\lam_j,\nba) \subset (\Gr_{\cG,X}\times_\Z \A^3)^\cJ.
\end{align*}
Let $\bfa = (\bfa_j)_\jj \in (\A^3)^\cJ(\cO)$. 
We also define  $M_\cJ(\lam,\nba_\bfa) := \prod_{\jj}M(\lam_j,\nba_{\bfa_j})$ and  $U(\tilz,\lam,\nba_\bfa)) := \prod_\jj U(\tilz_j,\lam_j,\nba_{\bfa_j})$.
The following is the main result of this chapter.

\begin{thm}\label{thm:unibranch-product}
There exists a non-empty open subscheme $U\subset (\A^3)^\cJ$ depending only on $\lam$ and $e$, such that if $\bfa \in U(\cO)$, then $M_\cJ (\lam, \nba_\bfa)$ is unibranch at any point $\til{z}_{\F,\bfa}$ contained in its special fiber. In addition, $\cO(U(\til{z},\lam,\nba_\bfa))^\pcp$ is a domain.
\end{thm}

\begin{proof}
Let $M_j:= \cM_X(\lam_j,\nba) \ra X\times \A^3$ and $\tilz_j \in \tilW^\vee$. The special fiber $M_j\times_X X_0$ intersect with a section $\tilz_j$ only if $\tilz_j = wt_\nu$ for some $w\in W$ and $\nu \in \Conv(\lam_j)$. Note that there are finitely many $\tilz_j$ satisfying these conditions. 

Let $\eta$ be the generic point of $\A^3$. We define $Fix_j$ to be the set of $\tilz_j \in \tilW^\vee$ which intersects with $M_j\times_{X\times \A^3} (X_0 \times \eta)$. For any $\tilz_j$ such that $\tilz_j = wt_\nu$ for some $w\in W$ and $\nu \in \Conv(\lam_j)$, consider a map
\begin{align*}
    f_{\tilz_j}: (M_j\times_X X_0)  \cap  \tilz \ra \A^3.
\end{align*}
(Here, the intersection is taken inside ${\Gr_{\cG,X}\times_X X_0\times \A^3}$.) 
The image of $f_{\tilz_j}$ is constructible, and it contains $\eta$ if and only if $\tilz_j \in Fix_j$. Note that any constructible set containing the generic point of an irreducible scheme contains an open neighborhood of the generic point. Thus, there exists an open neighborhood $V'_j$ of $\eta$ which is contained in the image of $f_{\tilz_j}$ for all $\tilz_j \in Fix_j$ and the complement of the image of $f_{\tilz_j}$ for all $\tilz_j \notin Fix_j$.

Let $V_j$ be the intersection of $V'_j$ with the open subscheme of $\A^3$ satisfying the conclusion of Lemma \ref{lem:univ-local-model-flat} for $\lam_j$ and $e$.  
Then $M_j|_{V_j} \ra X\times V_j$ and $\tilz_j\in Fix_j$ satisfy the assumptions of \cite[Corollary 3.6.2]{LLHLM-2020-localmodelpreprint}. As a result, there exists an integer $e'$ and an open subscheme $U'_j\subset V_j$ satisfying the following: for any $\tilz_j\in Fix_j$ and $\bfa_j \in U_j(\cO)$, there exists a finite DVR extension $\cO'$ of $\cO$ of degree $\le e'$ and an $\cO'$-point of $M_j$ lifting $(\tilz_j,\bfa_j) \in M(\lam_j,\nba_{\bfa_j})(\F)$.

Let $U_j\subset U_j'$ be the open subscheme in which Theorem \ref{thm:unibranch-local-model} holds for $\lam_j$ and $e e'!$. Then we define $U= \prod_\jj U_j$.  Let $\cO'/\cO$ be the extension obtained by adjoining an $e'!$-th root of uniformizer and $e'!$-th root of unity. Note that $\cO'$ contains any extension of $\cO$ of degree $\le e'$.
For $\tilz\in \prod_\jj Fix_j$ and $\bfa \in U(\cO)$, $M_\cJ(\lam,\nba_\bfa)$ is unibranch at $\tilz_{\F,\bfa}$ if and only if the completed local ring of $M_\cJ(\lam,\nba_\bfa)$ at $\tilz_{\F,\bfa}$ is a domain, which is a subring of the completed local ring of $M_\cJ(\lam,\nba_\bfa) \times_\cO \spec \cO'$ at $\tilz_{\F,\bfa}$. Similarly, $\cO(U_\cJ(\tilz,\lam,\nba_{\bfa}))^\pcp$ is a subring of $\cO(U_\cJ(\tilz,\lam,\nba_{\bfa}\times_\cO \spec \cO'))^\pcp$. Thus, it suffices to prove the claim after replacing $\cO$ with $\cO'$. Note that both the completed local ring of $M_\cJ(\lam,\nba_\bfa) \times_\cO \spec \cO'$ at $\tilz_{\F,\bfa}$ and $\cO(U_\cJ(\tilz,\lam,\nba_{\bfa}\times_\cO \spec \cO'))^\pcp$ are completed tensor products of the completed local ring of $M(\lam_j,\nba_{\bfa_j})\times_\cO \spec \cO'$ at $\tilz_{j,\F,\bfa_j}$ and $\cO(U_\cJ(\tilz_j,\lam_j,\nba_{\bfa_j})\times_\cO \spec \cO')^\pcp$, which are known to be domain by Theorem \ref{thm:unibranch-local-model}. By the choice of $\cO'$, each $M_j\times_{\cO}\spec \cO'$ has an $\cO'$-point lifting $(\tilz_j,\bfa_j)$ in its special fiber. Then the two claims follow from \cite[Proposition 2.2]{KW2}, and \cite[Lemma A.1.1]{BLGGT-MR3152941} respectively (as explained in the last two paragraphs of the proof of \cite[Theorem 3.7.1]{LLHLM-2020-localmodelpreprint}).
\end{proof}

\subsection{Special fiber of naive local models in mixed characteristic}\label{sub:special-fiber-local-models}
In the remainder of this section, we study the special fiber of the naive local models in mixed characteristic. Recall that the naive local models in mixed characteristic are defined by  imposing the monodromy condition on the Pappas--Zhu local models inside $\Gr_{\cG,\cO} = L^+\cG_\cO \bss L \cG_\cO$. The special fiber $\Gr_{\cG,\cO}\times_\cO \spec \F$ is the affine flag variety $\Fl := \cI_\F \bss L(\GSp_4)_\F$. Here $\cI_F:= L^+\cG_{\F}$ is the usual Iwahori group scheme. Given $\bfa\in \cO^3$, we define
$\Fl^{\nba_\bfa}:= \Fl \times_{\Gr_{\cG,\cO}} \Gr_{\cG,\cO}^{\nba_\bfa}$.

It is known by \cite[Theorem 9.3]{PZ13-Inv-local_model-MR3103258}  that the special fiber $M(\le \lam)_\F$ is the reduced union of the affine Schubert cells $S^\circ_\F(\til{z})$ for $\til{z}\in \Adm^\vee(\lam)$. Thus, the underlying reduced closed subscheme of $M^\nv(\le\lam,\nba_\bfa)_\F$ is equal to the reduced union of $S^\circ_\F(\til{z})\cap \Fl^{\nba_\bfa}$ for $\til{z}\in \Adm^\vee(\lam)$. 

For $\al\in \Phi$, we define $H_\al^{(0,1)} = \{x \in X^*(T)\times_\Z \R \mid 0<\RG{x,\al^\vee}<1\}$.
The following is the main result of this subsection.

\begin{thm}[{cf.~\cite[Theorem 4.2.4]{LLHLM-2020-localmodelpreprint}}]\label{thm:schubert-cell-with-monodromy}
Let $h>0$ be an integer, $\til{w}\in \tilW$, and $\bfa\in \cO^3$. Suppose that $\til{w}$ is $h$-small and $\bfa \mod \varpi \in \F^3$ is $h$-generic (Definition \ref{def:genericty}). Then $S^\circ_\F(\til{w}^*)\cap \Fl^{\nba_\bfa}$ is an affine space of dimension $4 - \# \{\al \in \Phi^+ \mid \til{w}(A_0) \subset H_\al^{(0,1)}\}$.
\end{thm} 

As a Corollary, we get a classification of top-dimensional (i.e.~$4$-dimensional) irreducible components of $M^\nv(\le\lam,\nba_\bfa)_\F$. 
If $d\in \Z_{\ge 0}$ and $X$ is a reduced scheme, we let $\Irr_d X$ denote the set of $d$-dimensional irreducible closed subschemes of $X$.

\begin{defn}
We say that $\tilw \in \tilW$ is \emph{regular} if $\tilw^*(A_0)\nsubseteq H^{(0,1)}_\al$ for any $\al\in \Phi^+$. 
Let $\lam\in X_*(T^\vee)$. We define $\Adm_\reg^\vee(\lam)\subset \Adm^\vee(\lam)$ to be the subset of $\tilz\in \Adm^\vee(\lam)$ such that $\tilz^*$ is regular. We similarly define $\Adm_\reg(\nu)$ for $\nu\in X^*(T)$. Then $(-)^*$ induces a bijection between $\Adm_\reg(\nu)$ and $\Adm_\reg^\vee(\phi(\nu))$.
\end{defn}

\begin{cor}\label{cor:naive-spec-fiber-irred-comp}
Let $\lam\in X^*(T^\vee)$ be a dominant cocharacter. Suppose that $\bfa \mod \varpi \in \F^3$ is $h_\lam$-generic. There is a bijection
\begin{align*}
    \Adm_\reg^\vee(\lam) &\xrightarrow{\sim} \Irr_{4}M^\nv(\le\lam, \nba_\bfa)_\F    \\
    \til{z} &\mapsto \ov{(S^\circ_\F(\til{z}) \cap \Fl^{\nba_\bfa})}.
\end{align*}
\end{cor}
\begin{proof}
This follows from \cite[Theorem 9.3]{PZ13-Inv-local_model-MR3103258} and Theorem \ref{thm:schubert-cell-with-monodromy}. To apply Theorem \ref{thm:schubert-cell-with-monodromy}, note that $\tilz \in \Adm^\vee_\reg(\lam)$ is $h_\lam$-small.
\end{proof}

To prove Theorem \ref{thm:schubert-cell-with-monodromy}, we describe $S^\circ_\F(\tilz)$ using explicit coordinates. A version of Bruhat decomposition says that we have a double coset decomposition
\begin{align*}
    L(\GSp_4)_\F = \cup_{\til{z}\in \tilW^\vee} \cI_\F \til{z} \cI_\F
\end{align*}
and the open Schubert cell  $S^\circ_\F (\til{z})$  can be identified with $\cI_\F\bss \cI_\F \til{z} \cI_\F $.

Let $L^\mm \cG_\F \subset L\cG_\F $ be the subfunctor given by
\begin{align*}
    L^\mm\cG_\F(R) = \CB{g\in \GSp_4 \PR{R\BR{\frac{1}{v}}} \mid g \mod \frac{1}{v} \in \ov{B}(R)}
\end{align*}
for any $\F$-algebra $R$. We define $N_{\til{z}} := \til{z}^\mo L^\mm \cG_\F \tilz  \cap \cI_\F$.

Recall that the duality isomorphism $\phi$ identifies a coroot $\al^\vee$ of $G$ and a root $\phi(\al^\vee)$ of $G^\vee$, and we write $U_{\al^\vee}$ for the root subgroup of $G^\vee$ associated to $\phi(\al^\vee)$. Given $(\al^\vee,m)\in \Phi^\vee \times \Z$, we let $U_{\al^\vee,m}\subset LU_{\al^\vee}$ denote the subfunctor such that for any $\F$-algebra $R$, $U_{\al^\vee,m}(R) \subset  U_{\al^\vee}(R\DP{v})$ is identified with $v^mR \subset R\DP{v}$ under the isomorphism $U_{\al^\vee}\simeq \G_a$.

In what follows, we fix $x \in A_0$. Note that  $\ceil{\RG{x, \al^\vee}} = \del_{\al^\vee>0}$ and $\floor{\RG{x, \al^\vee}} = -\del_{\al^\vee<0}$.

\begin{prop}
Let $\tilz\in \tilW^\vee$.
\begin{enumerate}
    \item We have the following  decomposition
\begin{align*}
    N_{\til{z}} \simeq \prod_{(\al^\vee,m)\in \Phi^\vee_{\til{z}}} U_{-\al^\vee,m}
\end{align*}
where the product runs over the set $\Phi^\vee_{\til{z}}= \CB{(\al^\vee,m)\in \Phi^\vee\times \Z \mid \RG{x,\al^\vee} < m < \RG{\til{z}^*(x),\al^\vee}}$. 
    \item  There is an isomorphism
\begin{align*}
    \cI_\F \times_\F \CB{\til{z}} \times_\F N_{\til{z}} \simeq \cI_\F \til{z} \cI_\F
\end{align*}
given by multiplication.  This induces an isomorphism $\til{z} N_{\til{z}} \simeq S^\circ_\F(\til{z})$.
\end{enumerate}
\end{prop}

\begin{proof}
\begin{enumerate}
    \item We claim that $U_{-\al^\vee,m} \subset N_{\tilz}$ if and only if $(\al^\vee,m)\in \Phi^\vee_{\tilz}$.
    It is easy to check that  $U_{-\al^\vee,m} \subset \cI_\F$ if and only if $\del_{\al^\vee >0} \le m$.  Let $\tilz^* = st_{\nu}\in \tilW$. Direct computation shows that $\tilz U_{-\al^\vee,m} \tilz^\mo = U_{-s^\mo(\al^\vee),m-\RG{\nu,\al^\vee}}$. Thus $U_{-\al,m} \subset \tilz^\mo L^\mm \cG_\F \tilz$ if and only if $m-\RG{\nu,\al^\vee} < \RG{x,s^\mo(\al^\vee)}$, or equivalently $m < \RG{s(x) + \nu, \al^\vee}$. This proves the claim.
    
    Observe that $U_{-\al^\vee,m}\subset N_{\tilz}$ only if $\RG{\tilz^*(x),\al^\vee} >0$. Let $w\in W$ be the unique element such that $w \tilz^* \in \tilW^+$. If $\RG{\tilz^*(x),\al^\vee} = \RG{w\tilz^*(x),w(\al^\vee)} >0$, then $w(\al^\vee)>0$. In other words, $\phi(w)N_{\tilz}\phi(w)^\mo \subset L\ov{U}$. Then the claimed decomposition follows from the decomposition of $\ov{U}$.
    
    \item The injectivity of the multiplication map is obvious. The surjectivity follows from $\cI_\F =  ((\tilz^\mo\cI_\F \tilz)\cap \cI_\F) N_{\tilz}$.
\end{enumerate}
\end{proof}

Let $d_{\al^\vee, \til{z}} = \floor{\RG{\til{z}^*(x) , \al^\vee}} - \ceil{\RG{x, \al^\vee}}$.

\begin{cor}\label{cor:schubert-cell-description}
Let $R$ be a Noetherian $\F$-algebra. If $M\in N_{\til{z}}(R)$, we can write $M = \prod_{\al^\vee \in \Phi^\vee} v^{\del_{\al^\vee>0}}M_{\al^\vee}$ where $M_{\al^\vee}\in U_{-\al^\vee}(R[v])\simeq R[v]$ is a polynomial of degree at most $d_{\al^\vee,\til{z}}$ when $d_{\al^\vee,\til{z}}\ge 0$ and is zero when $d_{\al^\vee,\til{z}}<0$.
\end{cor}

\begin{proof}[Proof of Theorem \ref{thm:schubert-cell-with-monodromy}]
Let $\tilz= \tilw^*$ and $w\in W$ be the unique element such that $w \tilw \in \tilW^+$.
We follow the proof of \cite[Theorem 4.2.4]{LLHLM-2020-localmodelpreprint} using Corollary \ref{cor:schubert-cell-description} instead of Corollary 4.2.12 of \emph{loc.~cit.}. Indeed, \loccit~shows that the monodromy condition on $\tilz M$ for some $M\in N_{\tilz}(R)$ implies  that coefficients in the polynomial $M_{\al^\vee}$ are determined by its top degree coefficient and coefficients of $M_{\al^{\prime\vee}}$ such that $w(\al^{\prime \vee}) < w(\al^\vee)$. Inductively, this shows that $S^\circ_{\F}(\tilz) \cap \Fl^{\nba_\bfa}$ is an affine space with coordinates given by the top degree coefficients of $M_{\al^{\vee}}$. 
Thus, the dimension of $S^\circ_{\F}(\tilz) \cap \Fl^{\nba_\bfa}$ is equal to
\begin{align*}
    \# \CB{\al^\vee \in \Phi^\vee \mid U_{-\al^\vee,m}\subset N_{\tilz} \text{ for some $m\ge 0$}}.
\end{align*}
This is equal to the number of $\al^\vee\in \Phi^\vee$ such that $\RG{\tilw(x),\al^\vee}>0$ and $\tilz^*(x)$ and $x$ do not lie in the same $\al$-strip. The first condition says that $w(\al^\vee)\in \Phi^{\vee,+}$, which has size 4. For such $\al$, the second condition holds unless   $\al^\vee\in\Phi^{\vee,+}$ and $\tilw(A_0)\subset H^{(0,1)}_\al$. This completes the proof. 
\end{proof}

We now make comparison between irreducible components in  $M^\nv(\le\lam,\nba_\bfa)_\F$ for different choices of $\lam$ and $\bfa$. 
We define $L(\GSp_4)_\F^\nbaz \subset L(\GSp_4)_\F$ as the subsheaf given by
\begin{align*}
    R \mapsto \CB{A \in L\GSp_4(R) \mid v \ddv{A}A^\mo \in \frac{1}{v}L^+\cM_\F(R)}
\end{align*}
for an $\F$-algebra $R$. Let  $\Fl^\nbaz$ denote the fpqc-quotient sheaf $\cI_\F \bss L(\GSp_4)_\F^\nbaz$.

\begin{rmk}\label{rmk:translation-and-monodromy}
Let $\bfa\in \cO^3$ and $\til{z} = s^\mo t_\mu \in \tilW^\vee$ such that $\bfa \equiv s^\mo(\mu) \mod \varpi$. Choose a dominant cocharacter $\lam\in X_*(T^\vee)$ and $\tilw\in \tilW$. A direct computation shows that (cf.~\cite[Proposition 4.3.1]{LLHLM-2020-localmodelpreprint})
\begin{align*}
    M^\nv(\le\lam,\nba_\bfa)_\F \til{z} &= M(\le \lam)_\F \til{z}\cap \Fl^\nbaz \\
    (S^\circ_\F (\til{w}^*)\cap \Fl^{\nba_\bfa})\til{z} &= (S^\circ_\F (\til{w}^*)\til{z})\cap \Fl^{\nbaz}.
\end{align*}
This allows us to compare $M^\nv(\le\lam,\nba_\bfa)_\F$ for different $\lam$ and $\bfa$ and its irreducible components inside $\Fl^\nbaz$.
\end{rmk}

\begin{defn}
Let $\til{s} \in \tilW$ and $\til{w}_1, \til{w}_2 \in \tilW^+$. We define
\begin{align*}
    S^\circ_\F(\til{w}_1,\til{w}_2, \til{s}) &:= S^\circ_\F((\til{w}_2^\mo w_0 \til{w}_1)^*)\til{s}^* \subset \Fl \\ 
    S^\circ_\F(\til{w}_1,\til{w}_2, \til{s})^\nbaz &:= S^\circ_\F(\til{w}_1,\til{w}_2, \til{s}) \cap \Fl^\nbaz \subset \Fl^\nbaz \\
    S^\nbaz_\F(\til{w}_1,\til{w}_2, \til{s}) &:= \ov{S^\circ_\F(\til{w}_1,\til{w}_2, \til{s})^\nbaz}
\end{align*}
where the closure is taken in $\Fl^\nbaz$. 
\end{defn}

\begin{lemma}\label{lem:schubert-cell-nba-zero-dimension}
Let $\til{s}=t_{\mu}s$, $\tilw_1$, and $\tilw_2$ be as above. If, in addition, $\tilw_2^\mo w_0 \tilw_1$ is $m$-small and $\til{
s}$ is $m$-generic for some integer $m$, then $S^\nbaz_\F(\til{w}_1,\til{w}_2, \til{s})$ is an irreducible closed subschemes of $\Fl^\nbaz$ of dimension 4.
\end{lemma}
\begin{proof}
This follows from Remark \ref{rmk:translation-and-monodromy} and Theorem \ref{thm:schubert-cell-with-monodromy} (cf.~\cite[Proposition 4.3.4]{LLHLM-2020-localmodelpreprint}).
\end{proof}

In fact, many of $S^\nbaz_\F(\til{w}_1,\til{w}_2, \til{s})$ for different $\til{s},\tilw_1,$ and $\tilw_2$ are equal.   

\begin{prop}[{\cite[Proposition 4.3.5 and 4.3.6]{LLHLM-2020-localmodelpreprint}}]\label{prop:schubert-cell-identification}
Let $\til{s} \in \tilW$ and $\til{w}_1, \til{w}_2 \in \tilW^+$.
\begin{enumerate}
    \item Suppose that, for $i=1,2$, $\til{w}_i$ is $m_i$-small for some integer $m_i$ and $\til{s}$ is $(m_1+m_2)$-generic. Then we have 
\begin{align*}
    S^\nbaz_\F(\til{w}_1,\til{w}_2, \til{s}) = S^\nbaz_\F(\til{w}_1,e, \til{s}\til{w}_2^\mo).
\end{align*}
    \item Suppose that $\tilw_1$ is in $\tilW_1^+$ and $\til{s}$ is $h_\eta$-generic. For all $w\in W$, we have
    \begin{align*}
    S^\nbaz_\F(\til{w}_1,e, \til{s}) = S^\nbaz_\F(\til{w}_1,e, \til{s}w).
\end{align*}
\end{enumerate}
\end{prop}

\begin{defn}
Let $(\til{w}_1,\omega) \in \tilW_1^+ \times X^*(T)$ with $t_\omega$ being $h_\eta$-generic. We define $C_{(\til{w}_1,\omega)}:= S^\nbaz_\F(\til{w}_1,e, \til{s})$ for any $\til{s}\in \tilW$ such that $\til{s}(0)=\omega$. 
\end{defn}
This is well-defined by Proposition \ref{prop:schubert-cell-identification}. Since $\til{w}_1$ is $h_\eta$-small, $C_{(\til{w}_1,\omega)}$ is irreducible of dimension 4 by Lemma \ref{lem:schubert-cell-nba-zero-dimension}.

Let $\lam\in X_*(T^\vee)$ be a regular dominant cocharacter. Recall that the set of admissible pairs
\begin{align*}
    \AP(\phi^\mo(\lam)) = \CB{(\til{w}_1,\til{w}_2) \in (\tilW_1^+ \times \tilW^+)/X^0(T) \mid \til{w}_1 \uparrow t_{\phi^\mo(\lam)-\eta} \tilw_h^\mo \til{w}_2}.
\end{align*}
By \cite[Corollary 2.1.7]{LLHLM-2020-localmodelpreprint}, there is a bijection
\begin{align}\label{eqn:AP-AdmReg-bij}
\begin{split}
    \AP(\phi^\mo(\lam)) & \xrightarrow{\sim} \Adm^\vee_\reg(\lam) \\ 
    (\til{w}_1, \tilw_2) & \mapsto (\tilw_2^\mo w_0 \tilw_1)^*.
\end{split}
\end{align}

\begin{thm}\label{thm:AP-irred-comp-bij}
Let $\lam\in X_*(T^\vee)$ be a regular dominant cocharacter and $\bfa \in \cO^3$. Let $\til{z}= s^\mo t_\mu  \in \tilW^\vee$ be $h_{\lam+\eta}$-generic such that $\bfa \equiv \phi(s^\mo(\mu)) \mod \varpi$. We have a bijection 
\begin{align*}
    \AP(\phi^\mo(\lam)) & \xrightarrow{\sim} \Irr_4 (M^\nv(\le\lam,\nba_\bfa)_\F \til{z}) \\
    (\tilw_1,\tilw_2) & \mapsto C_{(\tilw_1, \til{z}^* \tilw_2^\mo(0))}.
\end{align*}
\end{thm}
\begin{proof}
This follows from Corollary \ref{cor:naive-spec-fiber-irred-comp}, Remark \ref{rmk:translation-and-monodromy}, Proposition \ref{prop:schubert-cell-identification}, and \eqref{eqn:AP-AdmReg-bij} as in the proof of \cite[Theorem 4.3.10]{LLHLM-2020-localmodelpreprint}. 
\end{proof}

\subsection{Matching irreducible components and Serre weights}\label{sub:matching-irred-comp-SW}
Let $\cI_{1,\F}\subset \cI_\F$ be the subfunctor given by 
\begin{align*}
    R \mapsto \CB{A\in \GSp_4(R\DB{v}) \mid A \mod v \in U(R)} 
\end{align*}
for $\F$-algebra $R$. We define $\til{\Fl}$ as the fpqc-quotient sheaf $\cI_{1,\F} \bss L\GSp_4$. 
Then the natural quotient map $\Psi: \til{\Fl} \ra \Fl$ is a $T^\vee_\F$-torsor. 

If $Y\subset \Fl$ is a closed subscheme, we let $\til{Y}$ denote the pullback $Y\times_{\Fl}\til{\Fl}$. Then $\til{Y} \ra Y$ is again a $T^\vee_\F$-torsor. Let $\lam\in X_*(T^\vee)$ be a dominant cocharacter and $\bfa \in \cO^3$. Then we have the following $T^\vee_\F$-torsors
\begin{align*}
    \til{M}(\leq\lam)_\F & \ra M(\le\lam)_\F \\
    \til{M}^\nv(\le\lam, \nba_\bfa)_\F & \ra M^\nv(\le\lam, \nba_\bfa)_\F \\
    \til{M}(\lam, \nba_\bfa)_\F & \ra M(\lam, \nba_\bfa)_\F.
\end{align*}
For  $(\tilw,\omega)\in \tilW_1^+ \times X^*(T)$ such that $t_\omega$ is $h_\eta$-generic, we also have a $T^\vee_\F$-torsor $\til{C}_{(\til{w},\omega)} \ra C_{(\tilw,\omega)}$.

Let $\cJ= \Hom(k,\F)$. For $\jj$, if $X_j\subset \Fl$ (resp.~$\til{X}_j\subset \til{\Fl}$) is a closed subscheme, we define $X_\cJ:= \prod_\jj X_j$ (resp.~$\til{X}_\cJ := \prod_\jj \til{X}_j$) to be the closed subscheme in $\Fl^\cJ$ (resp.~$\til{\Fl}^\cJ$). In particular, given a dominant cocharacter $\lam = (\lam_j)_\jj \in X_*(\ud{T}^{\vee})$ and $\bfa = (\bfa_j)_\jj \in (\cO^{3})^\cJ$, we have
\begin{align*}
    M^\nv_\cJ (\le\lam,\nba_\bfa)_\F &= \prod_\jj M(\le\lam_j,\nba_{\bfa_j})_\F \\
    \til{M}_\cJ^\nv (\le\lam,\nba_\bfa)_\F &= \prod_\jj \til{M}(\le\lam_j,\nba_{\bfa_j})_\F.
\end{align*}
Let $\tilz = (\tilz_j)_\jj \in \utilW^\vee$ such that $\tilz_j = s_j^\mo t_{\mu_j}$ and $\bfa_j \equiv s_j^\mo(\mu_j) \mod \varpi$.
We have the following cartesian diagram
\[
\begin{tikzcd}
\til{M}_\cJ^\nv (\le\lam,\nba_\bfa)_\F \arrow[r, hook, "r_{\til{z}}"] \arrow[d] & \til{\Fl}_\cJ^\nbaz \arrow[d] \\
M_\cJ^\nv (\le\lam,\nba_\bfa)_\F \arrow[r, hook, "r_{\til{z}}"]  & \Fl_\cJ^\nbaz
\end{tikzcd}
\]
where $r_{\til{z}}$ is right translation by $\til{z}$.  Let $(\tilw_1,\tilw_2) \in \AP(\phi^\mo(\lam)) = \prod_\jj \AP(\phi^\mo(\lam_j))$. If $\tilz$ is $h_{\lam+\eta}$-generic, by Theorem \ref{thm:AP-irred-comp-bij}, we have a $7f$-dimensional irreducible component of $\til{M}^\nv_\cJ (\le\lam,\nba_\bfa)_\F \til{z}$ given by
\begin{align*}
    \til{C}_{(\tilw_1,\til{z}^*\tilw_2^\mo(0))} := \prod_\jj  \til{C}_{(\tilw_{1,j},\til{z}_j^*\tilw_{2,j}^\mo(0))}.
\end{align*}

\begin{defn}
Let $\sigma$ be a $h_\eta$-deep Serre weight. 
Let $(\tilw_1,\omega)\in \utilW_1^+ \times X^*(\uT)$ be the lowest alcove presentation of $\sigma$ corresponding to $\zeta$. We define $C_\sigma^\zeta:= C_{(\tilw_1,\omega)}$ and $\tilC^\zeta_\sigma := \tilC_{(\tilw_1,\omega)}$.
\end{defn}

\begin{thm}\label{thm:nv-irred-comp-SW-bij}
Let $\lam \in X_*(\uT^\vee)$ be a regular dominant cocharacter. Let $R$ be a Deligne--Lusztig representation with $\max\{h_{\lam+\eta},2h_\eta\}$-generic lowest alcove presentation $(s,\mu)$. 
Let  $\bfa \in (\cO^3)^\cJ$ such that $\bfa_j \equiv \phi(s_j^\mo(\mu_j+\eta_j)) \mod \varpi$. Then we have
\begin{align*}
    \Irr_{4f}\PR{(M_\cJ^\nv(\le\lam,\nba_\bfa)_\F) \phi(s^\mo t_{\mu+\eta})} = \CB{C_\sigma^\zeta \mid \sigma \in \JH(\ov{R}\otimes W(\phi^\mo(\lam)-\eta))}.
\end{align*}
\end{thm}

\begin{proof}
This follows from Proposition \ref{prop:JH-factor-lam-tau} and Theorem \ref{thm:AP-irred-comp-bij}.
\end{proof}

\subsection{Torus-fixed points}\label{sub:torus-fixed-point}

The torus $T^{\vee,\cJ}$ acts on $\Fl^\cJ$ by right translation and stabilizes $\Fl_\cJ^\nbaz$. The natural inclusion ${\tilW}^{\vee,\cJ} \subset \Fl^\cJ$ (sending $wt_{\nu} \in \tilW^{\vee,\cJ}$ to $w v^{\nu} \in \Fl^\cJ(\F)$) identifies ${\tilW}^{\vee,\cJ}$ and the set of $T^{\vee,\cJ}$-fixed points in $\Fl^\cJ$. Note that $\tilW^{\vee,\cJ}$ is contained in $\Fl_\cJ^\nbaz$. In this subsection, we prove the following result on the $T^{\vee,\cJ}$-fixed points of irreducible components in $\Fl_\cJ^\nbaz$.

\begin{thm}\label{thm:torus-fixed-point}
Let $\sigma$ be a $h_\eta$-deep Serre weight with a lowest alcove presentation $(\tilw_1,\omega)$. 
Then, the set of $T^{\vee,\cJ}$-fixed points in $C_{(\tilw_1,\omega)}$ is given by
\begin{align*}
    C_{(\tilw_1,\omega)}^{T^{\vee,\cJ}} = \{(t_{\omega}\tilw)^* \mid \tilw\in \tilW^\cJ, \tilw \le w_0 \tilw_1 \}.
\end{align*}
\end{thm}
\begin{proof}
It suffices to prove the claim for $\cJ=\{*\}$. By definition, $C_{(\tilw_{1},\omega)} = S_\F^{\nba_0}(\tilw_{1},e,t_{\omega})$, and the claim follows from Theorem \ref{thm:BA-GSp4} and Remark \ref{rmk:translation-and-monodromy}.
\end{proof}

\begin{cor}\label{cor:geom-SW-local-model}
Let $\rhobar$ be a $2h_\eta$-generic tame inertial $L$-parameter over $\F$. Let $W^\zeta_g(\rhobar)$ be the set of $h_\eta$-deep Serre weights $\sigma$ such that $\tilw^*(\rhobar) \in C^\zeta_\sigma$. Then, we have $W^\zeta_g(\rhobar) = W^?(\rhobar)$.
\end{cor}

\begin{proof}
This follows from a direct computation using Theorem \ref{thm:torus-fixed-point} as in the proof of \cite[Theorem 4.7.6]{LLHLM-2020-localmodelpreprint}.
\end{proof}

We end this section by recording the following proposition, which will be used in \S\ref{sub:SWC}.

\begin{prop}\label{prop:unibranch-T-fixed-point}
$C_{(\tilw_1,\omega)}$ is unibranch at its $\uT^\vee$-fixed points.
\end{prop}
\begin{proof}
This can be proven as the corresponding result for $\GL_4$ \cite[Proposition 4.7.5]{LLHLM-2020-localmodelpreprint} similarly to Proposition \ref{prop:unibranch-equal-char}. 
\end{proof}

\section{Moduli stacks in $p$-adic Hodge theory}\label{sec:moduli-stacks}

In this section, we construct a symplectic variant of the moduli stacks of \'etale $\pgma$-modules constructed in \cite{EGstack} and prove its properties including the existence of closed substacks parameterizing potentially crystalline representations. We also construct symplectic variants for the $p$-adic formal algebraic stacks considered in \cite[\S 5]{LLHLM-2020-localmodelpreprint}, namely, the moduli stacks of Breuil--Kisin modules, and \'etale $\phi$-modules. Then we study the geometric properties of potentially crystalline substacks using local models in \S\ref{sec:local-models}.

\subsection{Symplectic \'etale $\pgma$-modules}\label{sub:pgma-stack}
Only in this subsection, we allow $K/\Qp$ to be ramified.

For a $p$-adically complete $\cO$-algebra $R$, we write $\cX_n(R)$ for the groupoid of projective \'etale $\pgma$-modules of rank $n$ with $R$-coefficients, in the sense of \cite[Definition 2.7.2]{EGstack}. Then $\cX_n$ is a Noetherian formal algebraic stack over $\spf \cO$ (Corollary 5.5.17 in \emph{loc.~cit.}). When $R$ is a finite $\cO$-algebra, there is an equivalence of categories between the category of projective \'etale $(\varphi,\Gamma)$-modules with $R$-coefficients and the category of continuous representations of $G_K$ on finite projective $R$-modules. The goal of this subsection is to construct a moduli stack of \emph{symplectic} $\pgma$-modules over $\cO$ using the Emerton--Gee stack.

We start by defining a symplectic variant of \'etale $\pgma$-modules.

\begin{defn}\label{def:sym-et-pgma}
Let $R$ be a $p$-adically complete $\cO$-algebra. A \emph{symplectic} projective \'etale $\pgma$-module (of rank 4) with $R$-coefficients is a triple $(M,N,\al)$ where $M\in \cX_4(R)$, $N\in \cX_1(R)$, and $\al: M \simeq M^\vee \otimes N$ is an isomorphism between rank 4 \'etale $\pgma$-modules satisfying the alternating condition $(\al^\vee\otimes N)^\mo \circ \al=-1_{M}$. Note that $\al$ induces a non-degenerate bilinear pairing
\begin{align*}
   p_\al: M\times M &\ra N
    \\
    (m_1, m_2) &\mapsto \al(m_1)(m_2)
\end{align*}
that is alternating, i.e.~$p_\al(m_1,m_2) = -p_\al(m_2,m_1)$. Then a morphism $F: (M,N,\al) \ra (M',N',\al')$ of symplectic projective \'etale $\pgma$-modules is a pair of morphisms $f:M\ra M'$ and $g:N\ra N'$ of \'etale $\pgma$-modules that commutes with the bilinear pairings induced by $\al$ and $\al'$. 
We write $\cX_\Sym(R)$ for the groupoid of projective symplectic \'etale $\pgma$-modules with $R$-coefficients.
\end{defn}

\begin{lemma}\label{lem:sym-gal-rep-equiv}
Let $R$ be a commutative ring and $G$ be an abstract group. There is an equivalence of groupoids
\begin{align*}
    \CB{\rho: G \ra \GSp_4(R)} &\simeq \CB{(\rho',\chi, \al)\mid \begin{array}{cc}
         \rho':G \ra \GL_4(R), \chi: G \ra \GL_1(R), \\
        \al: \rho' \simeq (\rho')^\vee\otimes \chi, (\al^\vee\otimes \chi)^\mo \circ \al = -1
    \end{array} } \\
    \rho &\mapsto  (\std(\rho), \simc(\rho), \al_\rho)
\end{align*} 
where morphisms in the latter groupoid are given by isomorphisms of $\rho'$ and $\chi$ commuting with $\al$, and the isomorphism $\al_\rho$ is given by the matrix $J$ with respect to the standard basis of $\std(\rho)$ and its dual basis.
\end{lemma}

\begin{proof}
Note that $\al$ defines a non-degenerate alternating $R$-bilinear pairing $\RG{-,-}_{\al}:\rho'\times \rho' \ra R$ given by $\RG{x,y}_{\al}=\al(y)(x)$. Then the essential surjectivity follows from the presentation of standard symplectic space. More precisely, one can find two vectors $v_1,v_4 \in \rho'$ such that $\RG{v_1,v_4}_\al=1$, and two vectors $v_2, v_3$ such that $\RG{v_2,v_3}_\al=1$ in the complement of the span of $v_1,v_4$. Then the  basis $\{v_1,v_2,v_3,v_4\}$ gives the standard symplectic space. The proof of full faithfulness is elementary.
\end{proof}

\begin{cor}\label{cor:sym-pgma-gal-equiv}
Let $R$ be a finite local $\cO$-algebra. There is an equivalence of groupoids
\begin{align*}
    \CB{\rho: G_K \ra \GSp_4(R)} \simeq \cX_\Sym(R).
\end{align*}
\end{cor}
\begin{proof}
This follows from Lemma \ref{lem:sym-gal-rep-equiv} and  Fontaine's equivalence between Galois representations and \'etale $\pgma$-modules.
\end{proof}

We recall some definitions from \cite{Emerton-formal-alg-stk}. A formal algebraic stack is \textit{quasi-compact} if it admits a morphism from a quasi-compact formal algebraic space which is representable by algebraic spaces, smooth, and  surjective. A morphism $\cX \ra \cY$ of formal algebraic stacks is \textit{quasi-compact} if for every morphism $Z \ra \cY$ whose source is an affine scheme, the fiber product $\cX\times_{\cY} Z$ is a quasi-compact formal algebraic stack. A formal algebraic stack $\cX$ is \textit{quasi-separated} if its diagonal morphism (which is representable by algebraic spaces) is quasi-compact and quasi-separated. We have the following generalization.

\begin{defn}
A  morphism $f: \cX \ra \cY$ of formal algebraic stacks is \emph{quasi-separated} if the relative diagonal $\Del_{f} : \cX \ra \cX \times_\cY \cX$ is quasi-compact and quasi-separated.
\end{defn}

It is easy to see that if $f: \cX \ra \cY$ is a  quasi-separated morphism of formal algebraic stacks whose target is quasi-separated, then $\cX$ is quasi-separated.

A formal algebraic stack $\cX$ is \emph{locally Noetherian} if it admits a morphism from a disjoint union of Noetherian affine formal algebraic spaces which is representable by algebraic spaces, smooth, and surjective. Moreover, $\cX$ is \emph{Noetherian} if it is locally Noetherian, quasi-compact, and quasi-separated.

We denote by $\std: \cX_\Sym \ra \cX_4$ and $\simc: \cX_\Sym \ra \cX_1$  morphisms sending $(M,N,\al)$ to $M$ and $N$ respectively.

\begin{thm}\label{thm:sym-pgma-stack-reptbl}
The morphism $\std \times \simc: \cX_{\Sym} \ra \cX_4 \times_{\spf \cO} \cX_1$ is representable by  algebraic spaces of finite presentation and quasi-separated. In particular, the category fibered in groupoids $\cX_\Sym$ is a Noetherian formal algebraic stack over $\spf \cO$.
\end{thm}

\begin{proof}
We first prove that $\std\times\simc$ is representable by algebraic spaces and quasi-compact. Let $S\ra \spf \cO$ be a test scheme and $S \ra \cX_4 \times_{\spf \cO} \cX_1$ be a morphism corresponding to \'etale $\pgma$-modules $M_S$ and $N_S$ (of rank 4 and 1, respectively). If $S'$ is a $S$-scheme, we write $(M_{S'},N_{S'})$ for the pullback of $(M_S,N_S)$ to $S'$. The fiber product
\begin{align*}
    \ud{\al}_S := \cX_\Sym \times_{\cX_4\times \cX_1} S
\end{align*}
is given by the following subsheaf of $\ud{\Isom}(M_S,M_S^\vee\otimes N_S)$
\begin{align*}
    \ud{\al}_S: S' \mapsto \CB{\al' \in \Isom(M_{S'},M_{S'}^\vee\otimes N_{S'}) \mid ((\al')^\vee\otimes N_{S'})\circ \al' = -1_{M_S}}.
\end{align*}
We have the following cartesian diagram
\[
\begin{tikzcd}
\ud{\al}_S \arrow[rr] \arrow[d] & & S \arrow[d , "{-1_{M_S}}"] \\
\ud{\Isom}(M_S,M_S^\vee\otimes N_S) \arrow[r] & \ud{\Isom}(M_S,M_S^\vee\otimes N_S)  \times_S \ud{\Isom}(M_S^\vee\otimes N_S, M_S)  \arrow[r, "c"] & \ud{\Aut}(M_S)
\end{tikzcd},
\]
where the right vertical map is a constant section given by $-1_{M_S} \in \ud{\Aut}(M_S)$, the left bottom horizontal map sends $\al$ to $(\al, (\al^\vee \otimes N_S)^\mo)$, and $c$ is the composition. Since $\cX_4$ is quasi-separated, $\ud{\Aut}(M_S) \ra S$ is quasi-compact and quasi-separated. This implies that the section $-1_{M_S}$ is quasi-compact, and so is $\ud{\al}_S \ra \ud{\Isom}(M_S,M_S^\vee\otimes N_S)$. In particular, $\ud{\al}_S \ra S$ is quasi-compact. 
This proves that $\std\times \simc$ is representable by algebraic spaces locally of finite presentation and quasi-compact.

We now prove that $\std\times \simc$ is quasi-separated. By the following cartesian diagram
\[
\begin{tikzcd}
\ud{\Isom}((M_1,N_1,\al_1),(M_2,N_2,\al_2)) \arrow[r] \arrow[d] & \ud{\Isom}(M_1,M_2) \times \ud{\Isom}(N_1,N_2) \arrow[r] \arrow[d]  & S \arrow[d] \\
\cX_\Sym \arrow[r] & \cX_\Sym \times_{\cX_4\times \cX_1} \cX_\Sym \arrow[r] & \cX_\Sym \times \cX_\Sym
\end{tikzcd},
\]
it suffices to show that
\begin{align*}
    \ud{\Isom}((M_1,N_1,\al_1),(M_2,N_2,\al_2)) \ra \ud{\Isom}(M_1,M_2) \times_S \ud{\Isom}(N_1,N_2)
\end{align*}
is quasi-compact and quasi-separated. It is clearly a monomorphism and is therefore quasi-separated. Note that there is a morphism
\begin{align*}
   d: \ud{\Isom}(M_1,M_2) \times_S \ud{\Isom}(N_1,N_2) & \ra \ud{\Isom}(M_1,M_2) \times_S \ud{\Isom}(N_1,N_2) \\
    (f,g) &\mapsto (\al_2^\mo \circ ((f^\vee)^\mo\otimes g ) \circ \al_1 ,g)
\end{align*}
which makes the following diagram cartesian
\[
\begin{tikzcd}
\ud{\Isom}((M_1,N_1,\al_1),(M_2,N_2,\al_2)) \arrow[r] \arrow[d] & \ud{\Isom}(M_1,M_2) \times_S \ud{\Isom}(N_1,N_2) \arrow[d, "\PR{\id, d}" ] \\
\ud{\Isom}(M_1,M_2) \times_S \ud{\Isom}(N_1,N_2) \arrow[r, "\Del"] & (\ud{\Isom}(M_1,M_2) \times_S \ud{\Isom}(N_1,N_2))^2
\end{tikzcd}.
\]
Then the top horizontal arrow is quasi-compact because $\ud{\Isom}(M_1,M_2) \times_S \ud{\Isom}(N_1,N_2)$ is  quasi-separated.

Finally, $\cX_\Sym$ is a formal algebraic stack over $\spf \cO$ by \cite[Lemma 5.19]{Emerton-formal-alg-stk}. It is Noetherian because $\std\times \simc$ is representable by algebraic spaces of finite presentation. 
\end{proof}

Let $\tau:I_K \ra T^\vee(E)$ be an inertial $K$-type and   $\lam\in X_*(\uT^{\vee})$ be a dominant cocharacter.  Recall the closed substacks $\cX_4^{\lam',\tau'}$ (resp.~$\cX_4^{\ss,\lam',\tau'}$)  of $\cX_4$ called the potentially crystalline (resp.~semistable) substack with Hodge type $\lam':=\std(\lam)$ and inertial type $\tau':=\std(\tau)$ (\cite[Theorem 4.8.12]{EGstack}). It is an $\cO$-flat $p$-adic formal algebraic substack. For a finite flat $\cO$-algebra $A$, $\cX_4^{\lam',\tau'}(A)$ (resp.~$\cX_4^{\ss,\lam',\tau'}(A)$) is the full subgroupoid of $\cX_4(A)$ consisting of lattices in potentially crystalline (resp.~semistable) $G_K$-representations of Hodge type $\lam'$ and inertia type $\tau'$.  For convenience, we write $\cX_4^{\lam,\tau}$ and $\cX_4^{\ss,\lam,\tau}$ instead of $\cX_4^{\lam',\tau'}$ and $\cX_4^{\ss,\lam',\tau'}$. We define $\cX_{\Sym}^{\lam,\tau}$ (resp.~$\cX_{\Sym}^{\ss,\lam,\tau}$) to be the $\cO$-flat part (see \cite[Example 9.11]{Emerton-formal-alg-stk}) of  $\cX_\Sym \times_{\cX_4} \cX_4^{\lam',\tau'}$ (resp.~$\cX_\Sym \times_{\cX_4} \cX_4^{\ss,\lam',\tau'}$). 

\begin{prop}\label{prop:crys-ss-stacks}
Let $\lam, \tau$ be as above. 
\begin{enumerate}
    \item The stack $\cX_\Sym^{\lam,\tau}$ (resp.~$\cX_{\Sym}^{\ss,\lam,\tau}$) is a $p$-adic formal algebraic stack and uniquely determined as the $\cO$-flat closed substack of $\cX_\Sym$ such that for a finite flat $\cO$-algebra $A$, $\cX_\Sym^{\lam,\tau}(A)\subset \cX_\Sym(A)$ (resp.~$\cX_\Sym^{\ss,\lam,\tau}(A)\subset \cX_\Sym(A)$) is precisely the subgroupoid consisting of lattices in potentially crystalline (resp.~semistable) $G_K$-representations of Hodge type $\lam$ and inertia type $\tau$. 
    \item The algebraic stacks $\cX_\Sym^{\lam,\tau}\times_{\spf \cO} \spec \F$ (resp.~$\cX_{\Sym}^{\ss,\lam,\tau}\times_{\spf \cO} \spec \F$)  are equidimensional of dimension $d_\lam:= \sum_{j\in\cJ} \dim P_{\lam_j}\bss \GSp_4$.
\end{enumerate}
\end{prop}
\begin{proof}
The first claim follows from the construction. For the second claim, the proof is identical to that  of \cite[Theorem 4.8.14]{EGstack}, using $\GSp_4$ instead of $\GL_4$ and the dimension formula for the potentially crystalline (or semistable) symplectic deformation ring (\cite[Theorem A]{BG}).
\end{proof}

\subsubsection{Irreducible components in $\cX_{\Sym,\red}$}\label{subsub:irred-comp-pgma-stack}
Let $\cX_{4,\red}$ be the underlying reduced substack of $\cX_4$. It is an algebraic stack of finite presentation over $\F$ and equidimensional of dimension $6[K:\Qp]$ (\cite[Theorem 5.5.11 and 6.5.1]{EGstack}). Moreover, its irreducible components are labeled by Serre weights (of $\GL_4(k)$). Our goal is to prove analogous results for the underlying reduced substack $\cX_{\Sym,\red}$ of $\cX_{\Sym}$.

Let $\sigma$ be a Serre weight (of $\GSp_4(k)$). Then there exists $\lam\in X^*_1(\uT)$ such that $\sigma \simeq F(\lam)$. For each $\jj$, we identify $\lam_j$ with a triple of integers $(\lam_{j,1}, \lam_{j,2} ; \lam_{j,3})$ such that $0 \le \lam_{j,1}- \lam_{j,2}, \lam_{j,2} \le p-1$ as explained in \S\ref{sub:prelim}.

\begin{defn}\label{def:rhobar-max-nonsplit-weight-sig}
We say that $\rhobar: G_K \ra \GSp_4(\F)$ is \emph{maximally non-split of niveau 1 and of weight $\sigma$} if 
\begin{align}\label{eqn:max-non-split-rhobar}
    \rhobar = \pma{\chi_1 & *& *& *\\ 
    0& \chi_2& *& *\\ 
    0&0 & \chi_3&* \\ 
    0& 0&0 & \chi_4 }
\end{align}
where \begin{itemize}
    \item $\rhobar$ is maximally non-split of niveau 1, i.e.~it has a unique $G_K$-stable complete flag;
    \item $\bigoplus_{i=1}^4  \chi_i|_{I_K} = \prod_{\jj} \ov{\omega}_{K,\sigma_j}^{\phi(\lam_j+\eta_j)}$ (as $T^\vee(\F)$-valued representations);
    \item If $\chi_1 \chi_2^\mo|_{I_K} = \ov{\veps}$ (resp.~$\chi_2 \chi_3^\mo|_{I_K} = \ov{\veps}$), then $\lam_{j,2}=p-1$ (resp.~$\lam_{j,1}-\lam_{j,2}=p-1$) for all $\jj$ if and only $\chi_1 \chi_2^\mo = \ov{\veps}$ (resp.~$\chi_2 \chi_3^\mo = \ov{\veps}$) and the element of $\Ext^1_{G_K}(\chi_1,\chi_{2})$ (resp.~$\Ext^1_{G_K}(\chi_2,\chi_3)$) determined by $\rhobar$ is tr\`es ramifi\'ee. Otherwise,  $\lam_{j,2} = 0$ (resp.~$\lam_{j,1} - \lam_{j,2} = 0$) for all $\jj$. (Note that $\chi_1 \chi_2^\mo = \chi_3 \chi_4^\mo$.)
\end{itemize}
\end{defn}

The following is the main theorem of this subsection.

\begin{thm}\label{thm:irred-comp-pgma-stack}
The stack $\cX_{\Sym,\red}$ is an algebraic stack over $\spec \F$ of finite presentation and equidimensional of dimension $4[K:\Q_p]$. For each Serre weight $\sigma$, there exists an irreducible component $\cC_\sigma \subset \cX_{\Sym,\red}$ containing a dense locus of $\rhobar$ maximally non-split of niveau 1 and of weight $\sigma$. This induces a bijection between the set of isomorphism classes of Serre weights and the set of irreducible components of $\cX_{\Sym,\red}$.
\end{thm}

Before we proceed, we define \emph{potentially diagonalizable representations}. 

\begin{defn}\label{defn:pd-local}
A continuous representation $\rho: G_K \ra \GSp_4(\cO)$ is \emph{potentially diagonalizable} if $\rho\otimes_\cO E$ is potentially crystalline and $\std(\rho)$ is potentially diagonalizable in the sense of \cite[\S1.4]{BLGGT-MR3152941}.
\end{defn}

\begin{example}\label{ex:pdreps}
Suppose that $\rho:G_K \ra \GSp_4(\cO)$ is potentially crystalline and ordinary (in the sense of \cite[\S 1.4]{BLGGT-MR3152941}). Then $\rho$ is potentially diagonalizable by Lemma 1.4.3 of \loccit.
\end{example}

The following two results will be proven in \S\ref{subsub:families-of-extensions} and \S\ref{subsub:crys-lift} respectively. We write $\cX_{\Sym,\red,\Fpbar}$ for $\cX_{\Sym,\red}\times_{\F}\spec \Fpbar$.

\begin{prop}\label{prop:EGstack-gsp4-weak-irred-comp}
For each Serre weight $\sigma$, there exists an algebraic stack $\cC_{\sig,\Fpbar}\subset \cX_{\Sym,\red,\Fpbar}$ irreducible of dimension $4[K:\Qp]$ containing a dense locus of $\rhobar$ maximally non-split of niveau 1 and of weight $\sigma$. Furthermore, there exists a closed substack $\cC_{small}\subset \cX_{\Sym,\red,\Fpbar}$ of dimension  strictly less than $4[K:\Qp]$ such that 
\begin{align*}
    \cX_{\Sym,\red,\Fpbar} = \bigcup_{\sig} \cC_{\sig,\Fpbar} \cup \cC_{small}.
\end{align*}
\end{prop}

\begin{thm}\label{thm:crys-lift}
Any continuous representation $\rhobar: G_K \ra \GSp_4(\Fpbar)$ admits a lift $\rho: G_K \ra \GSp_4(\ov{\Z}_p)$ such that $\rho\otimes_{\ov{\Z}_p}\Qpbar$ is crystalline of regular Hodge type (i.e.~given by a regular dominant $\lam\in X_*(\uT^\vee)$). Furthermore, $\rho$ can be taken to be potentially diagonalizable.
\end{thm}

Granting Proposition \ref{prop:EGstack-gsp4-weak-irred-comp} and Theorem \ref{thm:crys-lift}, we prove:

\begin{proof}[Proof of Theorem \ref{thm:irred-comp-pgma-stack}]
If a closed immersion of reduced algebraic stacks that are locally of finite type over $\spec \Fpbar$ is surjective on finite type points (see \cite[\S6.6]{EGstack}), then it is an isomorphism. Thus, if we prove that any $\rhobar:G_K \ra \GSp_4(\Fpbar)$ is contained in some $\cC_{\sig,\Fpbar}$, then the closed immersion $\cup_\sig \cC_{\sig,\Fpbar} \mono \cX_{\Sym,\red,\Fpbar}$ from Proposition \ref{prop:EGstack-gsp4-weak-irred-comp} is an isomorphism. By Theorem \ref{thm:crys-lift}, $\rhobar$ is contained in the reduction of a crystalline stack $\cX_\Sym^{\lam}\times_{\spf \cO} \spec \Fpbar$ for some regular cocharacter $\lam$. By Proposition \ref{prop:crys-ss-stacks}, the algebraic stack $\cX_\Sym^{\lam}\times_{\spf \cO} \spec \Fpbar$ is  equidimensional of dimension $4[K:\Qp]$. So its underlying space is a union of $\cC_{\sig,\Fpbar}$, and one of them contains $\rhobar$.

It remains to show that $\cC_{\sig,\Fpbar}$ descends to $\cC_{\sig}$ over $\spec \F$. We need to show that $\Gal(\Fpbar/\F)$ stabilizes each component $\cC_{\sig,\Fpbar}\subset \cX_{\Sym,\red,\Fpbar}$. This follows because  $\Gal(\Fpbar/\F)$-action preserves the property of being maximally non-split of niveau 1 and of weight $\sig$. 
\end{proof}

\subsubsection{Families of extensions}\label{subsub:families-of-extensions}
We prove Proposition \ref{prop:EGstack-gsp4-weak-irred-comp}. The essential ingredient is generalizations of the Proposition 5.4.4 in \cite{EGstack} which computes the dimension of families of extensions inside $\cX_{n,\red,\Fpbar}$. We start by introducing several notations.

Recall that $S$ (resp.~$Q$) denotes the Siegel (resp.~Klingen) parabolic subgroup of $\GSp_4$. We write $U_S$ (resp.~$U_Q$) for its unipotent radical and $L_S$ (resp.~$L_Q$) for its Levi component. Then $U_S\simeq \G_a^{\oplus 3}$ and $U_Q$ is an extension of $\G_a^{\oplus 2}$ by $\G_a$. 

We also need the following auxiliary groups.  Let $Q_4$ be the minimal parabolic of $\GL_4$ containing $\std(Q)$ and $U_{Q_4}$ be its unipotent radical. We define the following subgroups of $U_{Q_4}$;
\begin{align*}
    U_{Q_3} = \CB{\pma{1 & * & *& 0 \\
    0 & 1 & 0 & 0 \\
    0 & 0 & 1 & 0 \\
    0 & 0 & 0 & 1}}, \ 
    U_{Q_3'} = \CB{\pma{1 & 0 & 0& 0 \\
    0 & 1 & 0 & * \\
    0 & 0 & 1 & * \\
    0 & 0 & 0 & 1}}, \ 
    U_{P_4} = \CB{\pma{ 1& 0 & 0& * \\
    0 & 1 & 0 & * \\
    0 & 0 & 1 & * \\
    0 & 0 & 0 & 1}}.
\end{align*}
Let $\iota: \GL_4 \ra \GL_4$ be an involution given by $\iota(A) = J^\mo A^{-\top}J$. Then $(U_{Q_4})^\iota = U_Q$  and  $\iota: U_{Q_3} \risom U_{Q'_3}$.

Let $R$ be a finite $\Zpbar$-algebra. Let ${\theta}:G_K \ra \GL_2(R)$ be a continuous representation and $\chi: G_K \ra R^\times$ be a continuous character. Then ${\theta} \oplus ({\theta}^\vee \otimes \chi)$ is an $L_S$-valued representation of $G_K$ and $\chi \oplus {\theta} \oplus \det({\theta})\chi^\mo$ is an $L_Q$-valued representation of $G_K$. We write 
\begin{align*}
    \Ad_S({\theta},\chi) : G_K & \xra{{\theta} \oplus ({\theta}^\vee \otimes \chi)} L_S(R) \xra{\Ad} \Aut(U_S(R))
    \\
    \Ad_Q(\chi,{\theta}) : G_K & \xra{\chi \oplus {\theta} \oplus \det({\theta})\chi^\mo} L_Q(R) \xra{\Ad} \Aut(U_Q(R))
    \\
    \Ad_{Q_4}(\chi,\theta) :  G_K & \xra{\chi \oplus {\theta} \oplus \det({\theta})\chi^\mo} L_{Q_4}(R) \xra{\Ad} \Aut(U_{Q_4}(R)).
\end{align*}
Note that $G_K$-actions on  $U_{Q_3}$ and $U_{Q'_3}$ induced by $\Ad_{Q_4}(\chi,\theta)$ define subrepresentations of  $\Ad_{Q_4}(\chi,\theta)$, which  we denote  by $\Ad_{Q_3}(\chi,\theta)$ and $\Ad_{Q'_3}(\chi,\theta)$ respectively. The isomorphism $\iota: U_{Q_3} \risom U_{Q'_3}$ induces an isomorphism of $G_K$-modules
\begin{align*}
    \Ad_{Q_3}(\chi,\theta) \risom \Ad_{Q'_3}(\chi, \theta).
\end{align*}
There is an obvious quotient map $U_{P_4} \epi U_{Q'_3}.$  We can view $U_{P_4}$ as the unipotent radical of a standard parabolic subgroup of $\GL_4$ whose Levi factor is $\GL_3 \times \GL_1$. If $\psi: G_K \ra \GL_3(R)$ is a continuous representation, we write
\begin{align*}
    \Ad_{P_4}(\psi,\chi) : G_K \xra{\psi \oplus \chi} \GL_3(R) \times \GL_1(R) \xra{\Ad} \Aut(U_{P_4}(R)).
\end{align*}

Finally, we remark that there is an auxiliary embedding
\begin{align*}
\iota_{\mathrm{aux}}:\{(g,h)\in \GL_2\times \GL_2 \mid \det(g)=\det(h)\} &\ra \GSp_4 \\
   \left(\pma{a & b \\ c& d}, \pma{x & y \\ z &w}\right) & \mapsto \pma{a & & & b \\ 
    & x & y & \\
    & z& w & \\
    c& & & d}.
\end{align*}
We denote by $G_{\ax}$ the source of $\iota_{\mathrm{aux}}$ and view it as a subgroup scheme of $\GSp_4$ via $\iota_{\mathrm{aux}}$.

\begin{rmk}\label{rmk:GSp4-rep-classification}
    We classify $\rhobar: G_K \ra \GSp_4(\Fpbar)$ into the following categories. Let $\chi = \simc(\rhobar)$. 
    \begin{enumerate}[label=(\Alph*)]
        \item We have $\rhobar$ whose image is contained in $B(\Fpbar)$ after conjugation. Then 
        we can write
        \begin{align*}
            \rhobar \simeq \pma{\chi_1 & * & * & * \\ 0 & \chi_2 & * & * \\ 0 & 0 & \chi\chi_2^\mo  & * \\ 
            0 & 0 & 0 & \chi \chi_1^\mo }.
        \end{align*}
        \item We have $\rhobar$ whose image is contained in $S(\Fpbar)$ whose composition with $S \epi L_S$ is $L_S$-irreducible. Then we can write 
        \begin{align*}
            \rhobar \simeq \pma{\ov{\theta} & * \\ 0 & \ov{\theta}^\vee \chi}
        \end{align*}
        where $\ov{\theta}$ is an irreducible 2-dimensional $G_K$-representation.
        \item Similarly, we have $\rhobar$ whose image is contained in $Q(\Fpbar)$ whose composition with $Q \epi L_Q$ is $L_Q$-irreducible. Then we can write 
        \begin{align*}
            \rhobar \simeq \pma{\chi_1 & * & *  \\ 0 & \ov{\theta} & * \\ 0 & 0 & \chi \chi_1^\mo }
        \end{align*}
        where $\ov{\theta}$ is an irreducible 2-dimensional $G_K$-representation such that $\det(\ov{\theta}) = \chi$.
        \item We have $\rhobar$ that is $G$-irreducible such that $\std(\rhobar)$ is also irreducible. These are precisely the unramified twists of representations of the form $\Ind_{G_{K_4}}^{G_K}\ov{\omega}_{K_4,\iota}^{a}$ where $K_4/K$ is the unramified extension of degree $4$, $\iota: K_4 \mono E$ is an embedding, and $a$ is an integer divisible by $q+1$ and not divisible by $q^2+1$. 
        \item Finally, we have $G$-irreducible $\rhobar$ such that $\std(\rhobar)$ is reducible. These are precisely the representations valued in $G_{\ax}(\Fpbar)$ after conjugation and $\std(\rhobar)\simeq \ov{\theta}_1\oplus \ov{\theta}_2$ where $\ov{\theta}_i$ are 2-dimensional $G_K$-representations and $\det(\ov{\theta}_1)=\det(\ov{\theta}_2)$. To see this, note that when $\std(\rhobar)$ is reducible, if it contains a 1-dimensional $G_K$-stable subspace, it belongs to either (A) or (C), and if it contains a 2-dimensional $G_K$-stable subspace that is Lagrangian (i.e.~the symplectic bilinear form restricts to zero), it belongs to (B), and the only remaining case is when it contains a 2-dimensional $G_K$-stable subspace that is symplectic (i.e.~the symplectic bilinear form restricts to a non-degenerate bilinear form), it is valued in $G_{\ax}(\Fpbar)$.  
    \end{enumerate}
    Note that these items are disjoint. This can be easily seen by considering $\std(\rhobar)^{\ss}$, except for items (B) and (E). For split $\rhobar$ in item (B) (i.e.~valued in $L_S(\Fpbar)$), there is $\rhobar'$ in item (E) satisfying $\std(\rhobar) \simeq \std(\rhobar')$, but they are non-isomorphic as $\GSp_4$-valued representations because the bilinear forms restricted to irreducible constituents are Lagrangian and symplectic, respectively.
\end{rmk}

We now consider families of extensions valued in the Siegel parabolic subgroup. 
Let $T$ be a reduced finite type $\Fpbar$-scheme with a morphism $T \ra \cX_{2,\red,\Fpbar}$ whose scheme-theoretic image is of pure dimension $d$. We let $\chi: G_K \ra \Fpbar^\times$ be a continuous character. We write $\ov{\theta}_T$ for the family of $G_K$-representations corresponding to  $T \ra \cX_{2,\red}$. Following the convention of \cite{EGstack}, we write $H^2(G_K, \Ad_S(\ov{\theta}_T,\chi))$ for the coherent sheaf on $T$ defined as the cohomology group of the \emph{Herr complex} $H^2(\cC^\bullet(M))$ on $T$ where $M$ is the rank 3 \'etale $\pgma$-module corresponding to the family $\Ad_S(\ov{\theta}_T,\chi)$ (see Remark 5.1.30 of \loccit). 

We suppose that $H^2(G_K, \Ad_S(\ov{\theta}_T,\chi))$ is locally free of constant rank $r$. Following the discussion before Proposition 5.4.4 in \loccit, we find a complex of finite rank locally free $\cO_T$-modules
\begin{align*}
    C^0_T \ra Z^1_T
\end{align*}
such that $\coker(C^0_T \ra Z^1_T)\simeq H^1(G_K,\Ad_S(\ov{\theta}_\cT,\chi))$. Let $V= \ud{Spec}(\Sym ((Z_T^1)^\vee))$. Then $V$ parameterizes a universal family of extension
\begin{align*}
    0 \ra \ov{\theta} \ra \rhobar_V \ra \ov{\theta}^\vee \otimes \chi \ra 0
\end{align*}
such that $\rhobar_V$ is valued in $\GSp_4$. This induces a morphism $V \ra \cX_{\Sym,\red,\Fpbar}$.

\begin{lemma}\label{lem:dim-ext-family-Siegel}
Let $T\ra \cX_{2,\red,\Fpbar}, \chi$, and $V$ be as above. In particular, the scheme-theoretic image of $T \ra \cX_{2,\red,\Fpbar}$ is of pure dimension $d$, and $H^2(G_K, \Ad_S(\ov{\theta}_T,\chi))$ is locally free of constant rank $r$. Then the scheme-theoretic image of $V \ra \cX_{\Sym,\red,\Fpbar}$ is of dimension $\le d + 3[K:\Qp]+ r-1$. Furthermore, if $T$ is generically maximally non-split of niveau 1, then the equality holds.
\end{lemma}
\begin{proof}
This follows from the proof of \cite[Proposition 5.4.4]{EGstack}.
\end{proof}

\begin{rmk}\label{rmk:H2-Siegel}
    We record an elementary result on the unipotent radical $L_S$ of $S$. Let $X$ be a 4-dimensional vector space over $\Fpbar$ equipped with a non-degenerate alternating bilinear form $\RG{-,-}$. Given a choice of Siegel parabolic subgroup $S$, there is a decomposition $X\simeq Y \oplus Z$ where $Y,Z$ are 2-dimensional subspaces such that the bilinear form induces a perfect pairing $Y\times Z \ra \Fpbar$, and $U_S(\Fpbar)$ can be identified with a subspace of $f\in \Hom(Z,Y)$ such that
    \begin{align*}
        \RG{z_1, f(z_2)} = \RG{z_2, f(z_1)}, \ \ \forall z_1,z_2\in Z.
    \end{align*}
    
    There is an isomorphism $\al: Z\simeq Z^\vee$ given by a unique up to scaling non-degenerate alternating bilinear form on $Z\simeq \Fpbar^2$. Note that this bilinear form is completely unrelated to the one in the previous paragraph. If $\{e_1,e_2\}$ is a basis for $\Fpbar^2$ with a dual basis $\{e_1^\vee, e_2^\vee\}$, it maps (up to scaling) $e_1\mapsto -e_2^\vee$ and $e_2 \mapsto e_1^\vee$. Let $f_\al$ be the composition $Z\stackrel{\al}{\simeq}Z^\vee \simeq Y$. Then
    \begin{align*}
        \RG{e_1, f_\al(e_2)} = - \RG{e_2, f_{\al}(e_1)} = 1
    \end{align*}
    and thus $f_\al\notin U_S(\Fpbar)$. The same is true when $\al$ is replaced by its composition with a unipotent automorphism of $Z$.
\end{rmk}

We now consider the Klingen parabolic case. Let $\ov{\theta}: G_K \ra \GL_2(\Fpbar)$ be a continuous irreducible representation and let $\del = \det(\ov{\theta})$. Let $T$ be a reduced finite type $\Fpbar$-scheme with a morphism $T \ra \cX_{1,\red,\Fpbar}$ whose scheme-theoretic image is of pure dimension $d$. We furthermore assume that  $H^2(G_K, \chi_T^2 \del^\mo)$ is locally free of constant rank $r$. For each $t\in T(\Fpbar)$, $H^1(G_K, \Ad_{Q}(\chi_t,\ov{\theta}))$ parameterizes $\rhobar: G_K \ra Q(\Fpbar)$ whose projection onto $L_Q(\Fpbar)$ is $\chi_t \oplus \ov{\theta} \oplus \det(\ov{\theta})\chi^\mo$. However, the unipotent radical $U_{Q}$ is non-abelian. To avoid using non-abelian cohomology, we construct a family of extensions in two steps.

Since $\theta$ is irreducible, we have $H^2(G_K, \Ad_{Q_3}(\chi_T, \ov{\theta}))=0$. Similar to Siegel parabolic case, there exists a locally free $\cO_T$-module $Z^1_T$ of constant rank with a surjection onto $H^1(G_K,\Ad_{Q_3}(\chi_T, \ov{\theta}))$. The vector bundle $V := \ud{Spec}(\Sym((Z_T^1)^\vee))$ parameterizes a universal family of extension 
\begin{align*}
    0 \ra \chi_T \ra \ov{\psi}_V \ra \ov{\theta} \ra 0.
\end{align*}
By \cite[Proposition 5.4.4]{EGstack}, the scheme-theoretic image of $V \ra \cX_{3,\red,\Fpbar}$ is of dimension $\le d+ 2[K:\Qp]-1$.

We can and do replace $V$ by its open dense locus on which 
$H^2(G_K, \Ad_{P_4}(\ov{\psi}_V,\chi_V^\mo \del))$ vanishes.  
Let $\chi_V$ be the pullback of $\chi_T$ to $V$. There is a locally free $\cO_V$-module $Z^1_V$ of constant rank with a surjection onto $H^1(G_K,\Ad_{P_4}(\ov{\psi}_V,\chi_V^\mo \del))$, and  the vector bundle $W:= \ud{Spec}(\Sym((Z^1_V)^\vee))$ parameterizes a universal family of extensions
\begin{align*}
    0 \ra \ov{\psi}_V \ra \rhobar_W \ra \det(\rhobar)\chi_T \ra 0.
\end{align*}
This induces a morphism $f': W \ra \cX_{4,\red,\Fpbar}$ whose scheme-theoretic image is of dimension $\le d+ 5[K:\Qp]-2$ by \loccit. We denote by $f:W \ra \cX_{4,\red,\Fpbar} \times \cX_{1,\red,\Fpbar}$ given by $f' \times \delta$. Then the scheme-theoretic image of $f$ is of dimension $\le d+5[K:\Qp] - 3$. 
Then
\begin{align*}
    \tilW= W \times_{\cX_{4,\red,\Fpbar}\times \cX_{1,\red,\Fpbar}} \cX_{\Sym,\red,\Fpbar}
\end{align*}
is an algebraic space by Theorem \ref{thm:sym-pgma-stack-reptbl}. We write $\til{f}:\tilW \ra  \cX_{\Sym,\red,\Fpbar}$  for the base change of $f$ along $\std\times \simc$.

\begin{lemma}\label{lem:dim-ext-family-Klingen}
Let $\ov{\theta}, T\ra \cX_{1,\red,\Fpbar}$, and $\til{f}: \til{W} \ra \cX_{\Sym,\red,\Fpbar}$ be as above. In particular, the scheme-theoretic image of $T\ra \cX_{1,\red,\Fpbar}$ is of pure dimension $d$, and $H^2(G_K, \chi_T^2 \del^\mo)$ is locally free of constant rank $r$.  Then the scheme-theoretic image of $\til{f}$ is of dimension $\le d + 3[K:\Q_p] + r -1$.
\end{lemma}

\begin{proof}

For $w: \spec \Fpbar \ra \til{W}$, we write $\til{W}_{\til{f}(w)}$ (resp.~$W_{f(w)}$) for the fiber of $\til{f}$ at $\til{f}(w)$ (resp.~of $f$ at $\std\times \simc \circ \til{f}(w)$). 
We have the following commutative diagram
\[
\begin{tikzcd}
\til{W}_{\til{f}(w)}  \arrow[rd]  \arrow[d, equal] \arrow[r]& \spec \Fpbar \arrow[rd, "\til{f}(w)"] 
\\
W_{f(w)} \arrow[rd]& \til{W} \arrow[r, "\til{f}"] \arrow[d] & \cX_{\Sym,\red,\Fpbar} \arrow[d, "\std\times \simc"] 
\\
 & W \arrow[r, "f"] & \cX_{4,\red,\Fpbar} \times \cX_{1,\red,\Fpbar}.
\end{tikzcd}
\]
Let $\til{\cZ}$ (resp.~$\cZ$) be the scheme-theoretic image of $\til{f}$ (resp.~$f$). By \cite[\href{https://stacks.math.columbia.edu/tag/0DS4}{Tag 0DS4}]{stacks-project}, we have
\begin{align*}
    \dim \til{\cZ} &\ge \dim \til{W} - \dim \til{W}_{\til{f}(w)} \\
    \dim {\cZ} &\ge \dim {W} - \dim {W}_{{f}(w)}
\end{align*}
and both equalities hold for $w$ whose image is contained in an dense open subset of $\tilW$. Thus, it suffices to show that 
\begin{align*}
   \dim \til{W} - \dim \til{W}_{\til{f}(w)}  \le d + 3[K:\Qp] + r- 1
\end{align*}
after possibly shrinking $\tilW$ by its open dense subspace. We do this by replacing $V$ by its subspace on which $\Hom_{G_K}(\ov{\theta},\ov{\psi}_V)$ vanishes and $W$ and $\tilW$ by its preimages.

Let $U$ be the scheme-theoretic image of $\tilW \ra W$. Since $\tilW_{\til{f}(w)} = W_{f(w)} $, we have 
\begin{align*}
    \dim \til{W} - \dim \til{W}_{\til{f}(w)} = (\dim \tilW - \dim U) + (\dim U - \dim W) + (\dim W - \dim W_{f(w)}).
\end{align*}
We claim that
\begin{align}
    \dim \tilW - \dim U &\le 1 \label{eqn:dim1}\\
    \dim U - \dim W &\le -2[K:\Qp] +r+1 \label{eqn:dim2}\\
    \dim W - \dim W_{f(w)} & \le d + 5[K:\Qp]  -3. \label{eqn:dim3}
\end{align}
By combining them, we get the desired inequality. For \eqref{eqn:dim3}, this follows from \cite[\href{https://stacks.math.columbia.edu/tag/0DS4}{Tag 0DS4}]{stacks-project} and the upper bound for $\dim \cZ$. 

For \eqref{eqn:dim2}, we compute $\dim U - \dim V$ and $\dim V - \dim W$. The latter is given by $-\rank Z^1_V$. For the former, again using \cite[\href{https://stacks.math.columbia.edu/tag/0DS4}{Tag 0DS4}]{stacks-project}, we have
\begin{align*}
    \dim U - \dim V \le \dim U_v
\end{align*}
for any $v\in V$. By our construction, $v$ corresponds to a non-split extension $\ov{\psi}_v$ of $\ov{\theta}$ by $\chi_t$, where $t$ is the image of $v$ under $V \ra T$. Then $\ov{\psi}_v$ determines a class $[c] \in H^1(G_K, \Ad_{Q_3}(\chi_t, \ov{\theta}))$. By the isomorphism $\iota:  \Ad_{Q_3}(\chi_t, \ov{\theta}) \risom  \Ad_{Q_3'}(\chi_t, \ov{\theta})$, we get $[c'] := \iota([c])$.  Any $w\in W_v(\Fpbar)$ determines a class $[b]\in H^1(G_K, \Ad_{P_4}(\ov{\psi}_v,\chi_t^\mo \del))$. Then $w$ is in $U_v(\Fpbar)$ if and only if $[b]$ is mapped to $\Fpbar^\times [c']$ under the map
\begin{align*}
    g: H^1(G_K,\Ad_{P_4}(\ov{\psi}_v,\chi_t^\mo \del)) \ra H^1(G_K,\Ad_{Q_3'}(\chi_t,\ov{\theta})).
\end{align*}
We apply Galois cohomology to the following short exact sequence
\begin{align*}
    0 \ra \chi_t^2 \del^\mo \ra \Ad_{P_4}(\ov{\psi}_v,\chi_t^\mo \del) \ra \Ad_{Q_3'}(\chi_t,\ov{\theta}) \ra 0.
\end{align*}
Since $\dim_{\Fpbar} H^2(G_K, \chi_t^2 \del^\mo) = r$, it shows that $g^\mo(\Fpbar[c'])$ is of codimension $2[K:\Qp]-r-1$. 
Let $Z'_v$ (resp.~$Z''_v$) be the preimage of $\Fpbar[c']$ (resp.~$0$) under the composition of $Z^1_v:= v^*(Z^1_V) \epi H^1(G_K,\Ad_{P_4}(\ov{\psi}_v,\chi_t^\mo \del))$ and $g$, which are of codimension $2[K:\Qp]-r-1$ (resp.~$2[K:\Qp]-r-2$). Then $U_v = \ud{Spec}((Z'_v)^\vee) \backslash \ud{Spec}((Z''_v)^\vee)$, whose dimension is $\rank Z^1_v - (2[K:\Qp]-r-1)$. Since $\rank Z^1_v = \rank Z^1_V$, this completes the proof of \eqref{eqn:dim2}.

Finally, for \eqref{eqn:dim1}, we again use \cite[\href{https://stacks.math.columbia.edu/tag/0DS4}{Tag 0DS4}]{stacks-project} to obtain
\begin{align*}
    \dim \tilW - \dim U \le \dim \tilW_{u}
\end{align*}
for $u\in U(\Fpbar)$. Note that the fiber $\tilW_{u}$ fits into the following base change diagram
\[
\begin{tikzcd}
    \tilW_{u} \arrow[r, "\til{f}"] \arrow[d] & \cX_{\Sym,\red,\Fpbar} \arrow[d, "\std\times \simc"] \\ 
    \spec \Fpbar \arrow[r, "f(u)"] & \cX_{4,\red,\Fpbar} \times \cX_{1,\red,\Fpbar}.
\end{tikzcd}
\]
As explained in the proof of Theorem \ref{thm:sym-pgma-stack-reptbl}, $\tilW_u$ is given by the space of isomorphisms $\Isom(\rhobar_u, \rhobar^\vee_u \otimes \delta)$ satisfying the alternating condition, where $\rhobar_u$ is the 4-dimensional Galois representation corresponding to $f(u)$. Since $\rhobar_u$ admits a unique filtration with graded pieces given by $\chi_t$, $\ov{\theta}$, and $\chi_t^\mo \delta$, we have $\Aut(\rhobar_u)\simeq \G_m$, and every element in $\Isom(\rhobar_u, \rhobar^\vee_u \otimes \delta)$ satisfies the alternating condition. This shows that $\tilW_u\simeq \G_m$. 
\end{proof}

\begin{proof}[Proof of Proposition \ref{prop:EGstack-gsp4-weak-irred-comp}]
Let $\sigma$ be a Serre weight. Then $\sigma \simeq F(\lam)$ for $\lam \in X^*_1(\uT)$ well-defined up to $(p-\pi)X^0(\uT)$. We identify $\lam_j$ with a triple $(\lam_{j,1}, \lam_{j,2}; \lam_{j,3})$ as in \S\ref{sub:prelim}. Then $\lam':= (\lam_{j,1}+ \lam_{j,2} +  \lam_{j,3}+2, \lam_{j,1} + \lam_{j,3}+2)$ defines an element in $X^*_1(\uT_2)/(p-\pi)X^0(\uT_2)$, and thus $\sigma':= F(\lam')$ is a well-defined Serre weight of $\GL_2(k)$. 

There exists an irreducible component $\cC_{\sigma',\Fpbar} \subset \cX_{2,\red,\Fpbar}$ of dimension $[K:\Qp]$ characterized as the closure of the locus $\cU_{\sig'}$ of $G_K$-representations of the form
\begin{align*}
    \ov{\theta}=\pma{\chi_1 & * \\ 0 & \chi_2}
\end{align*}
such that $\chi_1 \ur_{t_1^\mo} \oplus \chi_2 \ur_{t_{2}^\mo} = \prod_\jj \ov{\omega}_{K,\sig_j}^{\lam_{j}'+(1,0)}$ for some $t_1, t_2 \in \Fpbar^\times$, and moreover, if $\chi_1 \chi_2^\mo|_{I_K}  = \ov{\veps}$, then  $\lam_{j,2}=p-1$ for all $\jj$ if and only if   $\chi_1\chi_2^\mo =\ov{\veps}$ and the extension class is \tres~(see \cite[Theorem 5.5.12]{EGstack} (corrected in \cite{EGerrata}) and \cite[\S 7.4]{LLHLM-2020-localmodelpreprint}). When $\lam_{j,1}-\lam_{j,2}=0$ for all $\jj$, we assume that $t_2^2\neq 1$ by replacing $\cU_{\sig'}$ by its dense open substack.  We call such a $\ov{\theta}$  nonsplit of weight $\sigma'$. 
There exists a codimension 1 closed substack $\cC_{\sigma'}^{\mathrm{fixed}} \subset \cC_{\sigma',\Fpbar}$ containing a dense locus of nonsplit $\ov{\theta}$ of weight $\sigma'$ as above with $t_2=1$. We let $\cU_{\sigma'}^{\mathrm{fixed}} = \cC_{\sigma'}^{\mathrm{fixed}} \cap \cU_{\sigma'}$. We can and do assume that $\cC_{\sigma',\Fpbar}$ (resp.~$\cU_{\sigma'}$) is obtained from $\cC_{\sigma'}^{\mathrm{fixed}}$ (resp.~$\cU_{\sigma'}^{\mathrm{fixed}}$) by taking unramified twists (cf.~\cite[Lemma 5.3.2]{EGstack}).

Let $T$ be an $\Fpbar$-scheme smoothly covering $\cU_{\sig'}^{\fixed}$ (resp.~$\cU_{\sig'}$) if $
\lam_{j,1}-\lam_{j,2}=p-1$ for all $\jj$, which we call the tr\`es ramifi\'ee case (resp.~otherwise). 
We write $\chi_{\sigma}$ for the continuous character $\prod_{\jj} \ov{\omega}^{\simc(\phi(\lam + \eta))}_{K,\sig_j}: G_K \ra \Fpbar^\times$. Let $\ov{\theta}_T$ be the family $G_K$-representations parameterized by $T\ra \cX_{2,\red,\Fpbar}$. We claim that $H^2(G_K, \Ad_{S}(\ov{\theta}_T, \chi_{\sigma}))$ is of constant rank 1 in the tr\`es ramifi\'ee case and $0$ otherwise. 

By local Tate duality, $H^2(G_K, \Ad_{S}(\ov{\theta}_T, \chi_{\sigma}))\simeq H^0(G_K, \Ad_{S}(\ov{\theta}_T^\vee, \chi_{\sigma}^\mo (1))^\vee$ and $H^0(G_K, \Ad_{S}(\ov{\theta}_T^\vee, \chi_{\sigma}^\mo (1))) \subset \Hom_{G_K}(\ov{\theta}_T, \ov{\theta}_T^\vee \otimes \chi_{\sigma}(1))$. If $f\in \Hom_{G_K}(\ov{\theta}_T, \ov{\theta}_T^\vee \otimes \chi_{\sigma}(1))$ is an isomorphism, it is induced by $\al$ in Remark \ref{rmk:H2-Siegel} composed with a unipotent automorphism of $\ov{\theta}_T$ (which is only possible when $\chi_1=\chi_2$). However, such an isomorphism is not contained in $H^0(G_K, \Ad_{S}(\ov{\theta}_T^\vee, \chi_{\sigma}^\mo (1))$ by \loccit. Thus, any element in $H^0(G_K, \Ad_{S}(\ov{\theta}_T^\vee, \chi_{\sigma}^\mo (1))$ factors through $\Hom_{G_K}(\chi_2, \chi_2^\mo \chi_{\sig}(1))$, which is of rank 1 if and only if we are in the \tres~case.

By applying Lemma \ref{lem:dim-ext-family-Siegel} and the preceding discussion, we obtain a vector bundle $V$ over $T$ parameterizing extensions $\rhobar_V$ of $\ov{\theta}_T^\vee \otimes \chi_{\sigma}$ by $\ov{\theta}_T$, and the induced morphism $V\ra \cX_{\Sym,\red,\Fpbar}$ has the scheme-theoretic image of dimension $4[K:\Qp]-1$. We let $\cC_{\sigma,\Fpbar}$ be the scheme-theoretic image of the unramified twists of $\rhobar_V$. By Lemma 5.3.2 in \cite{EGstack}, $\cC_{\sigma,\Fpbar}$ is of dimension $4[K:\Qp]$.

It remains to construct $\cC_{small}$ and prove that $\cX_{\Sym,\red,\Fpbar} = \bigcup_{\sigma} \cC_{\sigma,\Fpbar} \cup \cC_{small}$. If the union $\bigcup_{\sigma} \cC_{\sigma,\Fpbar} \cup \cC_{small}$ exhausts all $\Fpbar$-points of $\cX_{\Sym,\red,\Fpbar}$, then the equality follows. Thus, we construct finitely many closed substacks of $\cX_{\Sym,\red,\Fpbar}$ containing $\Fpbar$-points of $\cX_{\Sym,\red,\Fpbar}$ which are (a priori) not contained in all $\cC_{\sigma,\Fpbar}$ and take $\cC_{small}$ to be the union of such closed substacks. We follow the classification in Remark \ref{rmk:GSp4-rep-classification}.
\begin{enumerate}[label=(\Alph*)]
    \item Let $\cU\subset \cX_{2,\red,\Fpbar}$ be the complement of $\cU_{\sig'}$ for all $\sig'$ and its 0-dimensional locus of irreducible representations. For $r\ge 0$ and $s\in \Z/(q-1)$, let $\cU_{r,s}\subset \cU$ be the locus of representations $\ov{\theta}$ such that  $\dim_{\Fpbar}H^2(G_K,\Ad_S(\ov{\theta},\ov{\omega}_{K,\sig_0}^{s})=r$. By the description of $\cU_{\sig'}$, we have $\dim \cU_{0,s} \le [K:\Qp]-1$ and $\dim \cU_{1,s} \le [K:\Qp]-2$. For $r \ge 2$, $\cU_{r,s}$ consists of representations of the form $\ov{\theta}=\chi_1\oplus \chi_2$ such that $\chi_i^2 = \ov{\omega}_{K,\sig_0}^s(1)$ for $i=1,2$ and moreover $\chi_1\neq \chi_2$ (resp.~$\chi_1=\chi_2$) when $r=2$ (resp.~$r=3$), and its dimension is -2 (resp.~-4). For all $r,s$, we can apply Lemma \ref{lem:dim-ext-family-Siegel} to construct a universal family of extensions of $\ov{\theta}^\vee\otimes \ov{\omega}_{K,\sig_0}^s$ by $\ov{\theta}$ with unramified twists whose scheme-theoretic image in $\cX_{\Sym,\red,\Fpbar}$ is at most $4[K:\Qp]-1$.  
    \item Let $\cU\subset \cX_{2,\red,\Fpbar}$ be the 0-dimensional locus of irreducible $G_K$-representations. By Remark \ref{rmk:H2-Siegel}, $H^2(G_K, \Ad_S(\ov{\theta},\chi)) = 0$ for any $\ov{\theta}\in \cU$ and a character $\chi$. For $s\in \Z/(q-1)$ and $\chi = \ov{\omega}_{K,\sig_0}^{s}$, we can apply Lemma \ref{lem:dim-ext-family-Siegel} to construct a universal family of extensions of $\ov{\theta}^\vee\otimes \ov{\omega}_{K,\sig_0}^s$ by $\ov{\theta}$ with unramified twists whose scheme-theoretic image in $\cX_{\Sym,\red,\Fpbar}$ is at most $3[K:\Qp]$.

    \item For $s\in \Z/(q-1)$, let $T_{s}$ be an $\Fpbar$-scheme smoothly covering the irreducible component of $\cX_{1,\red,\Fpbar}$ given by the unramified twists of $\ov{\omega}_{K,\sig_0}^s$. Let $\chi_{T_s}$ be the universal character parameterized by $T_s$.  Let $\{\ov{\theta}_i\}$ be a finite set of 2-dimensional irreducible $G_K$-representations over $\Fpbar$ whose unramified twists exhaust all 2-dimensional irreducible $G_K$-representations over $\Fpbar$. Let $\del_i:=\det(\ov{\theta}_i)$. There is a non-empty open (resp.~positive codimension closed) subscheme $T_{s,i}^\circ$ (resp.~$T_{s,i}^c$) of $T_s$ such that $H^2(G_K,\chi_{T_s}\del^\mo)$ is of rank 0 (resp.~rank 1). Then we can apply Lemma \ref{lem:dim-ext-family-Klingen} for $\ov{\theta} = \ov{\theta}_i$ and $T=T_{s,i}^\circ$ and $T_{s,i}^c$ to construct families of extensions with unramified twists whose scheme-theoretic image in $\cX_{\Sym,\red,\Fpbar}$ is at most $3[K:\Qp]$.

    Note that the construction before Lemma \ref{lem:dim-ext-family-Siegel} excludes the locus of $V$ where $H^2(G_K, \Ad_{P_4}(\ov{\psi}_V,\chi_V^\mo\del))\neq 0$. By the local Tate duality, such a locus is contained in the locus on which the extension of $\ov{\theta}$ by $\chi_V$ given by $\ov{\psi}_V$ splits, or in other words, $\Hom_{G_K}(\ov{\theta},\ov{\psi}_V)\neq 0$. Then the $\GSp_4$-valued $G_K$-representations of the form in Remark \ref{rmk:GSp4-rep-classification}(C) are of the form $\iota_{\ax}(\rhobar, \ov{\theta})$ where $\rhobar$ is an extension of $\chi_V^\mo \delta$ by $\chi_V$. Such representations will be covered in item (E) below.

    \item For the triples $(\rhobar,\chi,\al)\in \cX_{\Sym,\red,\Fpbar}(\Fpbar)$ such that $\std(\rhobar)$ is irreducible, there are only finitely many pairs $(\rhobar,\chi)$ up to unramified twists. Fix a pair $(\rhobar,\chi)$. Then the space of isomorphisms $\al: \rhobar\simeq \rhobar^\vee\otimes \chi$ is parameterized by $\G_m$. The automorphism group of $(\rhobar,\chi)$ is $\G_m\times \G_m$, and it acts naturally on the the copy of $\G_m$ parameterizing $\al$ where the first factor acts by weight $-2$ and the second factor acts by weight $1$. The copy of $\G_m$ parameterizing triples $(\rhobar,\chi,\al)$ with fixed $(\rhobar,\chi)$ maps into $\cX_{\Sym,\red,\Fpbar}$ and factors through a quotient $\G_m \epi [\G_m/(\G_m\times \G_m)] \simeq [*/\G_m]$ followed by a monomorphism. Thus, after unramified twists, its scheme-theoretic image is of dimension 0.

    \item We define an algebraic stack $\cX_{\ax}$ over $\Fpbar$ by the following cartesian diagram
    \[
    \begin{tikzcd}
        \cX_{\ax} \arrow[r] \arrow[d] & \cX_{1,\red,\Fpbar} \arrow[d, "\Delta"] \\ 
        \cX_{2,\red,\Fpbar} \times \cX_{2,\red,\Fpbar} \arrow[r, "\det\times \det"] & \cX_{1,\red,\Fpbar} \times \cX_{1,\red,\Fpbar}.
    \end{tikzcd}
    \]
    For an $\Fpbar$-algebra $R$, $\cX_{\ax}(R)$ is the groupoid whose objects are triples
    \begin{align*}
        (\ov{\theta}_1,\ov{\theta}_2,\del)\in (\cX_{2,\red,\Fpbar} \times \cX_{2,\red,\Fpbar} \times \cX_{1,\red,\Fpbar})(R)
    \end{align*}
    together with isomorphisms $\det(\ov{\theta}_1) \simeq \det(\ov{\theta}_2) \simeq \del$. Its morphisms are morphisms in $\cX_{2,\red,\Fpbar} \times \cX_{2,\red,\Fpbar} \times \cX_{1,\red,\Fpbar}$ commuting with the isomorphisms between the determinants $\det(\ov{\theta}_i)$ and $\del$. In particular, an automorphism of $(\ov{\theta}_1,\ov{\theta}_2,\del)$ can be viewed (Zariski locally, after choosing bases) as a triple $(g_1,g_2,c)\in (\GL_2\times \GL_2\times \GL_1)(R)$ satisfying $\det(g_1)=\det(g_2)=c$. By its construction, $\dim \cX_{\ax} = 2[K:\Qp]-1$. 
    
    We define
    \begin{align*}
        \iota_{\ax}:\cX_{\ax}&\ra \cX_{\Sym,\red,\Fpbar} \\
        (\ov{\theta}_1,\ov{\theta}_2,\del) & \mapsto (\ov{\theta}_1\oplus \ov{\theta}_2,\del,\al)
    \end{align*}
    where $\al = \al_1\oplus \al_2$ and $\al_i:\ov{\theta}_i \simeq \ov{\theta}_i^\vee \otimes \del$ defined in the following way: Zariski locally on $R$, we can choose a basis $\{e_1,e_2\}$ of $\ov{\theta}_i$ with a dual basis $\{e_1^\vee,e_2^\vee\}$ of $\ov{\theta}_i^\vee$ and a basis $\{e_1\wedge e_2\}$ of $\del$ using the isomorphism $\det(\ov{\theta}_i)\simeq \del$. Then $\al_i$ maps $e_1$ to $-e_2^\vee \otimes e_1 \wedge e_2$ and $e_2$ to $e_1^\vee \otimes e_1 \wedge e_2$. Note that $\al$ is independent of the choice of basis. 

    We claim that the scheme-theoretic image of $\iota_{\ax}$ is of at most $2[K:\Qp]-1$ dimension. By \cite[\href{https://stacks.math.columbia.edu/tag/0DS4}{Tag 0DS4}]{stacks-project}, it suffices to show that for $x= (\ov{\theta}_1,\ov{\theta}_2,\del)\in \cX_{\ax}(\Fpbar)$, the fiber $\cX_{\ax,\iota_{\ax}(x)}$ is at least 0-dimensional. Since $\iota_{\ax}$ is faithful, it is representable by an algebraic space \cite[\href{https://stacks.math.columbia.edu/tag/04SZ}{Tag 04SZ}]{stacks-project}. Thus, the fiber is at least 0-dimensional.
\end{enumerate} 

We take $\cC_{small}$ to be the union of finitely many closed substacks constructed as scheme-theoretic images in above items. 
It has dimension $< 4[K:\Qp]$, and $\bigcup_{\sigma} \cC_{\sigma,\Fpbar} \cup \cC_{small}$ exhausts all $\Fpbar$-points of $\cX_{\Sym,\red,\Fpbar}$. This completes the proof.
\end{proof}

\subsubsection{Existence of crystalline lifts}\label{subsub:crys-lift} 
In his thesis \cite{Lin}, Zhongyipan Lin developed an obstruction theory for lifting $G_K$-representations valued in reductive groups with mod $p$ coefficients and applied it to prove the existence of crystalline lifts of $G_K$-representations valued in the exceptional group $G_2$. As already mentioned in \loccit, the existence of crystalline lifts for classical groups follows from certain codimension estimates on the moduli stack of $\pgma$-modules. We briefly recall Lin's results and prove Theorem \ref{thm:crys-lift}.

Recall that $G=\GSp_4$. We say that a $G_K$-representation $\rhobar: G_K \ra \GSp_4(\F)$ is \emph{$G$-completely reducible} if for any parabolic subgroup $P$ of $G$ containing the image of $\rhobar$, a Levi subgroup $L$ of $P$ also contains the image of $\rhobar$. These are precisely representations in Remark \ref{rmk:GSp4-rep-classification} item (A--E) with all extensions being trivial in item (A--C). 
In these cases, we can find explicit crystalline lifts using the crystalline lifts of irreducible $G_K$-representations (e.g.~the proof of \cite[Theorem 6.4.4]{EGstack}).

It remains to prove the existence of crystalline lifts of $\rhobar: G_K \ra \GSp_4(\F)$ that factors through either $S(\F)$ or $Q(\F)$. Suppose that $\rhobar$ factors through $S(\F)$ and denote its Levi factor by $\ov{\theta}\oplus \ov{\chi}\ov{\theta}^\vee$ where $\ov{\theta}: G_K \ra \GL_2(\F)$ is a continuous representation and $\ov{\chi} : G_K \ra \F^\times$ is a character.  Then $\rhobar$ corresponds to an extension class $[\ov{c}]\in H^1(G_K, \Ad_{S}(\ov{\theta},\ov{\chi}))$. In this case, the $G_K$-module $\Ad_{S}(\ov{\theta},\ov{\chi})$ is abelian. By \cite[Theorem 6.4.4]{EGstack}, we can find crystalline lifts $\theta: G_K \ra \GL_2(\cO)$ of $\ov{\theta}$ and $\chi: G_K \ra \cO^\times$ of $\ov{\chi}$ with Hodge type given by regular dominant cocharacters $\lam = (\lam_{j,1},\lam_{j,2})_{\jj}\in X_*(\uT_2^\vee)$ and  $\mu=(\mu_{j})_\jj \in X_*(\uT_1^\vee)$ respectively. Then $\chi \theta^\vee$ has Hodge type $\lam'=(\mu_j - \lam_{j,2}, \mu_j - \lam_{j,1})_{\jj}$. Twisting $\theta$ by a crystalline character with sufficiently large Hodge type with trivial mod $p$ reduction, we may assume that $\lam$ is \emph{slightly greater than} $\lam'$ (i.e.~
$\lam_{j,2} \ge \mu_{j}-\lam_{j,2} +1$ for all $\jj$, and the inequality is strict for at least one $j$). Let $R= R_{\ov{\theta}}^{\lam}$ be the crystalline framed deformation ring of $\ov{\theta}$ with Hodge type $\lam$. Then the proof of Theorem 6.3.2 in \loccit~easily generalizes to our setup and shows that there exists a crystalline lift $\theta':G_K \ra \GL_2(\cO)$ of $\ov{\theta}$ with Hodge type $\lam$ lying on the same irreducible component of $\spec R$ that $\theta$ does, and a lift $\rho: G_K \ra S(\cO)$ of $\rhobar$ with Levi factor $\theta' \oplus \chi \theta^{\prime\vee}$. By Lemma 6.3.1 in \loccit, $\rho$ is crystalline.

Finally, we suppose that  $\rhobar: G_K \ra \GSp_4(\F)$ factors through $Q(\F)$ with Levi factor $\rhobar^\ss = \ov{\chi} \oplus \ov{\theta} \det(\ov{\theta}) \ov{\chi}^\mo$. We may assume $\ov{\theta}$ is irreducible; otherwise $\rhobar$ factors through $S(\F)$ and thus its crystalline lift exists by the previous paragraph. As in the previous paragraph, we can find crystalline lifts $\theta: G_K \ra \GL_2(\cO)$ of $\ov{\theta}$ and $\chi: G_K \ra \cO^\times$ of $\ov{\chi}$ with Hodge type given by regular dominant cocharacters $\lam = (\lam_{j,1},\lam_{j,2})_{\jj}\in X_*(\uT_2^\vee)$ and  $\mu=(\mu_{j})_\jj \in X_*(\uT_1^\vee)$ respectively. Then $\Ad_{Q}({\chi},{\theta})/{\chi}^2\det({\theta})^\mo \simeq {\chi}{\theta}^\vee$ has Hodge type given by $(\mu_{j}-\lam_{j,2},\mu_{j}-\lam_{j,1})_{\jj}$. By twisting $\chi$, we may assume that $\mu_{j}\ge \lam_{j,1}$ for all $\jj$ and the inequality is strict for at least one $j$.

Let $R_{\ov{\theta}}^{\lam}$ (resp.~$R_{\ov{\chi}}^\mu$) be the crystalline framed deformation ring of $\ov{\theta}$ (resp.~$\ov{\chi}$) of Hodge type $\lam$ (resp.~$\mu$). We let $X=\spec R$ be an irreducible component of $\spec R_{\ov{\theta}}^{\lam} \ctimes_{\cO} R_{\ov{\chi}}^\mu$. We let $\theta^\univ: G_K \ra \GL_2(R)$ and $\chi^\univ: G_K \ra R^\times$ be the universal families.

\begin{thm}[Theorem A in \cite{Lin}]\label{thm:Lin}
Let $[\ov{c}]\in H^1(G_K, \Ad_{Q}(\ov{\chi},\ov{\theta}))$ be the class corresponding to $\rhobar$. Recall that we assume $p>2$. Suppose that
\begin{enumerate}
    \item for each $s \ge 1$, the locus
    \begin{align*}
        X_s:=\{x\in \spec R \mid \dim k(x)\otimes_R H^2(G_K, \chi^\univ (\theta^\univ)^\vee) \ge s\}
    \end{align*}
    has codimension $\ge s+1$ in $X$;
    \item there exists a finite extension $K'/K$ of degree prime-to-$p$ such that $\rhobar^\ss(G_{K'})$ has $p$-power order and;
    \item there exists a $\Zpbar$-point of $\spec R$ whose restriction to $G_{K'}$ is mildly regular in the sense of \cite[Definition 3.0.1]{Lin}. 
\end{enumerate}
Then there exists a $\Zpbar$-point of $\spec R$ corresponding to $\theta': G_K \ra \GL_2(\Zpbar)$ and $\chi': G_K \ra \Zpbar^\times$ and a class $[c]\in H^1(G_K, \Ad_{Q}(\chi',{\theta'}))$ lifting $[\ov{c}]$.
\end{thm}

We now explain how the assumptions in the previous Theorem hold in our setting. By the irreducibility of $\ov{\theta}$,  $H^2(G_K, \chi^\univ (\theta^\univ)^\vee) =0$ and item (1) follows. Item (2) follows from the fact that both $\ov{\chi}$ and $\ov{\theta}$ have images of prime-to-$p$ order. For item (3), we take the $\cO$-point of $\spec R$ given by $\theta$ and $\chi$. By the condition on Hodge type, we have $H^0(G_{K'}, \chi^2 \theta^\vee) = 0$ (this is the condition (MR1) in \cite[Definition 3.0.1]{Lin}). To verify the condition (MR2) in \loccit, which asserts the non-degeneracy of the quadratic form defining the cup product on Lyndon--Demu\v{s}kin complex, we follow \S A in \loccit. Since $\rhobar^\ss|_{G_{K'}}$ is a trivial representation, the computation is much simpler. If we identify the underlying 2-dimensional vector space of $\chi \theta^\vee$ as
\begin{align*}
    \G_a \oplus \G_a &\simeq U_Q/Z(U_Q) \\
    (x,y) &\mapsto \pma{1&x&y&  \\ & 1 & & y \\ & & 1 &-x \\ & & & 1} \mod Z(U_Q),
\end{align*}
then the quadratic form defining the cup product on Lyndon--Demu\v{s}kin complex (see \S2 and \S A in \loccit) is given by a matrix $\pma{& M_{12} \\ M_{21} & }$
where
\begin{align*}
    M_{12} = \pret M_{21} = \pma{0 & 1   & & \\
    -1 & 0 & & \\ 
    & & 0& 1 \\ 
    & & -1 &0 \\ 
     & & & & \dots \\ 
     & & & & & 0 & 1 \\ 
     & & & & & -1 & 0} .
\end{align*}
Thus, the quadratic form is non-degenerate. By Theorem \ref{thm:Lin}, there exists a lift of  $\rhobar$ which is crystalline by the condition on Hodge type.

In all cases of $\rhobar$, the crystalline lift $\rho$ is ordinary when restricted to $G_{K'}$ for some finite unramified extension $K'/K$. Thus $\rho$ is potentially diagonalizable as explained in Example \ref{ex:pdreps}.  \hfill $\square$

\subsection{Symplectic Breuil--Kisin modules}\label{sub:BK-stack}

We start by recalling the definition of Breuil--Kisin modules with tame descent data. We refer to \cite[\S5.1]{LLHLM-2020-localmodelpreprint} for further detail. Throughout this subsection, let $\tau: I_{\Qp} \ra \uT_n^\vee(\cO)$ be a tame inertial $L$-parameter with 1-generic lowest alcove presentation $(s,\mu)$. We let $R$ be a $p$-adically complete $\cO$-algebra.

Let $r$ be the order of $s_\tau$. We write $f'=fr$ and let $K'/K$ be the unramified extension of degree $r$. We fix a choice $\pi'= (-p)^{1/(p^{f'}-1)}$. We let $L' = K'(\pi')$ and $\Del = \Gal(L'/K)$. Define $\fS_{L',R} := (\cO_{K'}\otimes_{\Zp} R)\DB{u'}$. It is equipped with an endomorphism $\varphi$ and a $\Del$-action such that
\begin{align*}
    \fS_{L',R}^\Del = \fS_{K,R} := (\cO_K \otimes_{\Zp} R)\DB{v} 
\end{align*}
where $v = (u')^{p^{f'}-1}$. Let $\cJ' = \Hom_{\Zp}(\cO_{K'},\cO)$. We have a decomposition $\fS_{L',R} = \bigoplus_{j'\in \cJ'} R\DB{u'}$ where $\cO_{K'}$ acts on $j'$-summand through the map $j'$.
For any $\fS_{L',R}$-module $M$, we let $M^{(j')}$ be the $R\DB{u'}$-submodule of $M$ such that $\cO_{K'}$ acts by $j'$.

Let $h\ge 0$ be an integer. A Breuil--Kisin module  of rank $n$, height in $[0,h]$, and type $\tau$ with $R$-coefficients is a  projective $\fS_{L',R}$-module $\fM$ of rank $n$ equipped with
\begin{itemize}
    \item an injective $\fS_{L',R}$-linear map $\phi_\fM : \varphi^*(\fM) \ra \fM$ whose cokernel is annihilated by $E(v)^h$ where $E(v) = v+p$; 
    \item a semilinear $\Del$-action commuting with $\phi_\fM$ such that
    \begin{align*}
    \fM^{(j')} \mod u' \simeq \tau^\vee \otimes_\cO R
\end{align*}
as $\Del':= \Gal(L'/K')$-representation  for each $j' \in \cJ'$.
\end{itemize}
We write $Y_n^{[0,h],\tau}(R)$ for the groupoid of Breuil--Kisin modules  of rank $n$, height in $[0,h]$, and type $\tau$ with $R$-coefficients.

\begin{prop}[Proposition 3.1.3 and 3.3.5 in \cite{cegsB}]\label{prop:BK-stack-GLn-finite-type}
    The category fibered in groupoids $Y_n^{[0,h],\tau}$ is a $p$-adic formal algebraic stack of finite type over $\spf \cO$ with affine diagonal.
\end{prop}

\begin{rmk}\label{rmk:BK-mod-base-change}
Note that we can change the $r$ above  by any integer divisible by the order of $s_\tau$. Using this, we can always interpret Breuil--Kisin modules with different descent data as modules over a same base.
\end{rmk}

\begin{defn}[cf. Definition 2.1.9 in \cite{EL}]
For $\fM \in Y_n^{[0,h],\tau}(R)$, we define $\fM^\vee \in Y_n^{[0,h],\tau^\vee}(R)$ by setting
\begin{align*}
    \fM^\vee = \Hom_{\fS_{L',R}}(\fM, \fS_{L',R})
\end{align*}
with induced semilinear $\Del$-action and $\phi_{\fM^\vee}: \varphi^*(\fM^\vee) \ra \fM^\vee$ given by the formula
\begin{align*}
    \phi_{\fM^\vee}(f)(m) = \varphi (f ( \phi_\fM^\mo ( E(v)^h m)))
\end{align*}
for all $f\in \fM^\vee$ and $m\in \fM$. Here $\phi_\fM^\mo(E(v)^hm)$ makes sense because of the height condition and that $\phi_\fM$ is injective.
\end{defn}

\begin{defn}
For $i=1,2$, let $n_i\ge 1$ and $h_i\ge 0$ be integers and $\tau_i$ be a tame inertial $L$-parameter valued in $\uT_{n_i}^\vee$. For $\fM_i \in Y_{n_i}^{[0,h_i],\tau_i}(R)$, we define a Breuil--Kisin module
\begin{align*}
    \fM_1\otimes_{\fS_{L',R}} \fM_2 \in Y_{n_1n_2}^{[0,h_1+h_2],\tau_1\otimes \tau_2}(R)
\end{align*}
by defining $\phi_{\fM_1\otimes\fM_2} := \phi_{\fM_1} \otimes \phi_{\fM_2}$ and taking $\Del$-action induced by  $\Del$-actions on $\fM_1$ and $\fM_2$.
\end{defn}

When $n=1$ and $h=0$, we write $Y_1^{0,\tau} := Y_1^{[0,0],\tau}$. For $\fN\in Y_1^{0,\tau}$, the evaluation map $\fN \otimes \fN^\vee \ra \fS_{L',R}$ is an isomorphism of Breuil--Kisin modules.

Let $\tau:I_{\Qp}\ra \uT^\vee(\cO)$ be a tame inertial $L$-parameter. For simplicity, we write $Y_4^{[0,h],\tau}$ to denote $Y_4^{[0,h],\std(\tau)}$.

\begin{defn}[cf.~Definition 2.1.17 in \cite{EL}]
 A \emph{symplectic} Breuil--Kisin module of  height in $[0,h]$, and type $\tau$ with $R$-coefficients is a triple $(\fM,\fN,\al)$ where $\fM \in Y_4^{[0,h],\tau}(R)$, $\fN\in Y_1^{0,\simc(\tau)}$, and $\al: \fM \simeq \fM^\vee \otimes \fN$ satisfying the \emph{alternating condition}
\begin{align*}
    (\al^\vee\otimes \fN)^\mo \circ \al = -1_\fM.
\end{align*}
We write $Y_\Sym^{[0,h],\tau}(R)$ for the groupoid of symplectic Breuil--Kisin modules of height in $[0,h]$ and type $\tau$ with $R$-coefficients.
\end{defn}

\begin{prop}\label{prop:BKstack-alg}
The map $Y_\Sym^{[0,h],\tau} \ra  Y_4^{[0,h],\tau}\times_{\cO} Y_1^{0,\simc(\tau)}$ sending $(\fM,\fN,\al)$ to $(\fM,\fN)$ is representable by schemes of finite presentation. In particular, $Y_\Sym^{[0,h],\tau}$ is a $p$-adic formal algebraic stack of finite type over $\spf \cO$. 
\end{prop}

\begin{proof}
Let $S = \spec R$. Consider a $S$-point $(\fM_S, \fN_S) \in Y_4^{[0,h],\tau}\times_{\cO} Y_1^{0,\simc(\tau)}(S)$. For any $S$-scheme $S'$, we write $(\fM_{S'},\fN_{S'})$ for the pullback of $(\fM_S, \fN_S)$ to $S'$. The fiber product $Y_\Sym^{[0,h],\tau} \times_{ Y_4^{[0,h],\tau}\times_{\cO} Y_1^{0,\simc(\tau)}} S$ is the subsheaf $\ud{\al}_S$ of $\ud{\Isom}(\fM_{S}, \fM_S^\vee \otimes \fN_S)$ given by
\begin{align*}
    \ud{\al}_S : S' \mapsto \CB{\al' \in \Isom(\fM_{S'} , \fM_{S'}^\vee \otimes \fN_{S'}) \mid ((\al')^\vee \otimes \fN_{S'})^\mo \circ \al'  = -1_{\fM_{S'}}}
\end{align*}
for any $S$-scheme $S'$. The sheaf $\ud{\Isom}(\fM_{S}, \fM_S^\vee \otimes \fN_S)$ is representable by affine schemes by Proposition \ref{prop:BK-stack-GLn-finite-type}.  There is a cartesian square
\[
\begin{tikzcd}
\ud{\al}_S \arrow[rr] \arrow[d] &  & S \arrow[d, "-1_{\fM_S}"] \\
\ud{\Isom}(\fM_{S}, \fM_S^\vee \otimes \fN_S)  \arrow[r]  & \ud{\Isom}(\fM_{S}, \fM_S^\vee \otimes \fN_S) \times_S \ud{\Isom}(\fM_{S}^\vee \otimes \fN_S, \fM_S) \arrow[r, "c"] & \ud{\Aut}(\fM_{S})
\end{tikzcd}
\]
where the right vertical map is a constant section given by $-1_{\fM_S} \in \ud{\Aut}(\fM_{S})$, the left bottom horizontal map sends  $\al$ to $(\al, (\al^\vee\otimes \fN)^\mo)$, and $c$ is the composition. This shows that $\ud{\al}_S$ is representable by schemes of finite presentation. The claim that $Y_\Sym^{[0,h],\tau}$ is a $p$-adic formal algebraic stack follows from \cite[Lemma 5.19]{Emerton-formal-alg-stk} and Proposition \ref{prop:BK-stack-GLn-finite-type}.
\end{proof}

We recall several notions regarding Breuil--Kisin modules. We refer \cite[\S5]{LLHLM-2020-localmodelpreprint} and the references therein for a more detailed discussion.

Let $\be= (\be^{(j')})_{j'\in \cJ'}$ be an eigenbasis of  $\fM\in Y_n^{[0,h],\tau}$ (in the sense of \cite[Definition 5.1.6]{LLHLM-2020-localmodelpreprint}). Then the dual basis $\be^\vee = ((\be^{(j')})^\vee)_{j'\in \cJ'}$ is an eigenbasis of $\fM^\vee$.

\begin{defn}\label{defn:eigenbasis}
Let $(\fM,\fN,\al)\in Y_\Sym^{[0,h],\tau}(R)$. An \emph{eigenbasis} of $(\fM,\fN,\al)$ is a pair $(\be,\ga)$ of eigenbases $\be= (\be^{(j')})_{j'\in \cJ'}$ of $\fM$ and $\ga = (\ga^{(j')})_{j'\in \cJ'}$ of $\fN$   satisfying
\begin{align*}
    \al(\be^{(j')}) = ((\be^{(j')})^\vee \otimes \ga^{(j')})J.
\end{align*}
\end{defn}

Let $(\fM,\fN,\al)\in Y_\Sym^{[0,h],\tau}(R)$ be a symplectic Breuil--Kisin module with eigenbasis $(\be,\ga)$. For $j'\in \cJ'$, we write $C^{(j')}_{\fM,\be} \in M_4(R\DB{u'})$ for the matrix representation of $\phi_\fM^{(j')}$ with respect to $\varphi^*(\be^{(j'-1)})$ and $\be^{(j')}$.  The height condition implies $E(u')^h (C^{(j')}_{\fM,\be})^\mo\in M_4(R\DB{u'})$.  

Recall that  $(s,\mu)$ is a  1-generic lowest alcove presentation of $\tau$. For $j'\in \cJ'$, we write $j'=j_0 + fk$ for unique $0\le j \le f-1$ and $0\le k \le r-1$. Define 
\begin{align}\label{eqn:dd-parameters}
\begin{split}
    \bm{\alpha}'_{j'}&= \left\{\begin{array}{ll}
        s_\tau^{-k} (s_{f-1}^\mo s_{f-2}^\mo \dots s_{f-j}^\mo) (\mu_{f-j} + \eta_{f-j}) & \text{if $j'\neq 0$} \\
        \mu_0 + \eta_0 & \text{if $j'= 0$}
    \end{array}\right.
    \\
 (\bfa')\ix{j'} &= \sum_{i=0}^{f'-1} \bm{\alpha}'_{-j'+i}p^i\\
s'_{\rmor,j'} &=s_{\tau}^{k+1}(s_{f-1}^\mo s_{f-2}^\mo \dots s_{j+1}^\mo).
\end{split}
\end{align}
Then we ``remove the descent datum" and get the matrix
\begin{align*}
A_{\fM,\be}^{(j')} := \Ad( \phi(s'_{\rmor,j'})^\mo\phi(-(\bfa')\ix{j'})(u')) (C_{\fM,\be}\ix{j'})  \in M_4(R\DB{v})
\end{align*}
that is upper triangular modulo $v$ for each $j'\in \cJ'$. By the height condition, we have $E(v)^h (A_{\fM,\be}^{(j')})^\mo \in M_4(R\DB{v})$ which is upper triangular modulo $v$. Both $C_{\fM,\be}^{(j')}$ and $A_{\fM,\be}\ix{j'}$ depend only on $j' \mod f$. Since $\al(\be^{(j')})=((\be^{(j')})^\vee\otimes\ga^{(j')})J$, we have the following symplecticity property of matrices $C_{\fM,\be}^{(j')}$ and $A_{\fM,\be}\ix{j'}$.

\begin{lemma}\label{lem:BK-Frob-GSp4}
Let $(\fM,\fN,\al)\in Y_\Sym^{[0,h],\tau}(R)$ with an eigenbasis $(\be,\ga)$.
\begin{enumerate}
    \item The partial Frobenius matrix $C^{(j')}_{\fM,\be}$ satisfies
    \begin{align*}
        (C^{(j')}_{\fM,\be})^\top J C^{(j')}_{\fM,\be} = C^{(j')}_{\fN,\ga} E(v)^h J
    \end{align*}
    In particular, we have $C^{(j')}_{\fM,\be} \in \GSp_4(R\DB{u'}[\frac{1}{E(v)}])$.
    \item Similarly, we have
    \begin{align*}
        (A^{(j')}_{\fM,\be})^\top J A^{(j')}_{\fM,\be} = A^{(j')}_{\fN,\ga} E(v)^h J.
    \end{align*}
    In particular, for $p$-adically complete $R$, we have $A\ix{j'}_{\fM,\be} \in \GSp_4(R\DB{v}[\frac{1}{E(v)}])$ and is upper-triangular modulo $v$, and is therefore in $L\cG_\cO(R)$.
\end{enumerate}
\end{lemma}
\begin{proof}
For item (1), we can show that $(\be^\vee\otimes\ga)J$ is an eigenbasis of $\fM^\vee\otimes \fN$ and
\begin{align*}
    C\ix{j'}_{\fM^\vee\otimes \fN, (\be^\vee\otimes  \ga) J} = C\ix{j'}_{\fN,\ga} E(v)^h J^\mo (C^{(j')}_{\fM,\be})^{-\top} J.
\end{align*}
Then the claim follows from the definition of a symplectic Breuil--Kisin module and its eigenbasis. 

For item (2), note that $X:= \phi(s'_{\rmor,j'})^\mo\phi(-(\bfa')\ix{j'})(u')$ is a diagonal matrix such that $X(w_0 X w_0)$ is a scalar matrix. Since $X = X^{\top}$, the claimed equation is equivalent to
\begin{align*}
    X^\mo (C^{(j')}_{\fM,\be})^{\top} X J X C^{(j')}_{\fM,\be} X^\mo = C^{(j')}_{\fN,\ga} E(v)^h J.
\end{align*}
Then it follows from item (1) because
\begin{align*}
    XJX = X(Jw_0) (w_0Xw_0) w_0 = X (w_0 X w_0) J
\end{align*}
is a scalar multiple of $J$ (and similarly for $X^\mo J X^\mo$).

To realize $A\ix{j'}_{\fM,\be}$ as an element in $L\cG_\cO(R)$, note that since $Y^{[0,h],\tau}_{\Sym}$ is finite type over $\spf \cO$, the morphism $\spf R \ra Y^{[0,h],\tau}_{\Sym}$ factors through $\spf R \ra \spf R'$ for some finite type $\cO$-subalgebra $R' \subset R$. Thus $A\ix{j'}_{\fM,\be}$ is an element in $\GSp_4(R'\DB{v}[\frac{1}{E(v)}])$, which defines an element in $L\cG_{\cO}(R')$ by Remark \ref{rmk:LG-valued-in-noetherian-R}. By composing $\spf R \ra \spf R'$, we get the desired element in $L\cG_{\cO}(R)$.
\end{proof}

We record how $A^{(j')}_{\fM,\be}$ changes under a change of basis.

\begin{prop}\label{prop:BK-change-of-basis}
Let $(\fM,\fN,\al)\in Y_\Sym^{[0,h],\tau}(R)$ with an eigenbasis $(\be_1,\ga_1)$. Suppose that $(\be_2,\ga_2)$ is another eigenbasis such that, for each $j'\in \cJ'$, $\be_2\ix{j'} =  \be_1\ix{j'}D\ix{j'}$ and $\ga\ix{j'}_2 =  \ga_1\ix{j'}\simc(D\ix{j'})$ for some $D\ix{j'}\in \GSp_4(R\DB{u'})$. Define
\begin{align*}
    I\ix{j'}:= \Ad( \phi(s'_{\rmor,j'})^\mo\phi(-(\bfa')\ix{j'})(u'))(D\ix{j'}).
\end{align*}
Then $I\ix{j'} \in \cI(R)$ and it satisfies 
\begin{align*}
    A_{\fM,\be_1}\ix{j'} = I\ix{j'}  A_{\fM,\be_2}\ix{j'} (\Ad(\phi(s_j^\mo) \phi(\mu_j + \eta_j)(v)) (\varphi(I\ix{j'-1})^\mo))
\end{align*}
where $j\in \cJ$, $j = j' \mod f$.

Conversely, for any $I\ix{j'}\in \cI(R)$ only depending on $j'\mod f$, $(\be_2,\ga_2)$ given by $\be_{2,j'}=\be_{1,j'}D\ix{j'}$ and $\ga_{2,j'}=\ga_{1,j
}\simc(D\ix{j'})$ is an eigenbasis where $D\ix{j'}$ is defined by 
\begin{align*}
    D\ix{j'}:= \Ad( \phi((\bfa')\ix{j'})(u')\phi(s'_{\rmor,j'}))(I\ix{j'}) \in \GSp_4(R\DB{u'}).
\end{align*}
\end{prop}
\begin{proof}
This essentially follows from \cite[Proposition 5.1.8]{LLHLM-2020-localmodelpreprint}. Note that the conditions $\be_2\ix{j'} =  \be_1\ix{j'}D\ix{j'}$ and $\ga\ix{j'}_2 =  \ga_1\ix{j'}\simc(D\ix{j'})$ imply that $(\be_2,\ga_2)$ is an eigenbasis.
\end{proof}

\begin{defn}
Let $(\fM,\fN,\al)\in Y_\Sym^{[0,h],\tau}(\F)$. We define the \emph{shape} of $(\fM,\fN,\al)$ to be the unique element $\tilz \in \tilW^{\vee,\cJ}$ such that $A\ix{j}_{\fM,\be}\in \cI(\F) \tilz_j \cI(\F)$ for some eigenbasis $(\be,\ga)$ and for each $\jj$. By Proposition \ref{prop:BK-change-of-basis}, this is independent of the choice of eigenbasis.  
\end{defn}

For integers $a\le b$, we define $L^{[a,b]}\cG_\cO$ to be the subfunctor of $L\cG_\cO$ given by
\begin{align*}
    L^{[a,b]}\cG_\cO(R) = \CB{ g\in L\cG_\cO(R) \mid E(v)^{-a}\ov{g}, E(v)^b \ov{g}^\mo \in M_{4}(R\DB{v+p}) }
\end{align*}
for an $\cO$-algebra $R$.  It is preserved by left and right multiplication by $L^+\cG_\cO$. We write
\begin{align*}
    \Gr^{[a,b]}_{\cG,\cO} := L^+\cG_\cO \bss L^{[a,b]}\cG_\cO
\end{align*}
for the induced sub-ind-scheme of $\Gr_{\cG,\cO}$.  For $(s,\mu) \in W^\cJ \times X^*(T)^\cJ$, we define $(s,\mu)$-twisted $\varphi$-conjugation action of $(L^+\cG_\cO)^\cJ$ on $(L^{[a,b]}\cG_\cO)^\cJ$ by 
\begin{align*}
     (I\ix{j})_\jj \cdot (A\ix{j})_\jj := \PR{I\ix{j}  A\ix{j} ((\Ad(\phi(s_j^\mo) \phi(\mu_j + \eta_j)(v)) (\varphi(I\ix{j-1})^\mo)) }_\jj.
\end{align*}
Similarly, we define $(s,\mu)$-twisted conjugation action by the above formula, but with $\varphi$ omitted. 
We denote by $[(L^{[a,b]}\cG_\cO)^\cJ /_{\varphi,(s,\mu)} (L^+\cG_\cO)^\cJ]$ the fpqc quotient stack using $(s,\mu)$-twisted $\varphi$-conjugation action.

Let $Y_\Sym^{[0,h],\tau,\be}$ be the category fibered in groupoids over $\spf \cO$ sending $R$ to a groupoid of tuples $(\fM,\fN,\al,\be,\ga)$ where $(\fM,\fN,\al) \in Y_{\Sym}^{[0,h],\tau}(R)$ and $(\be,\ga)$ is an eigenbasis of $(\fM,\fN,\al)$. By Proposition \ref{prop:BK-change-of-basis}, $Y_\Sym^{[0,h],\tau,\be}$ is an $L^+\cG_\cO^\cJ$-torsor over $Y_\Sym^{[0,h],\tau}$.

\begin{prop}\label{prop:presenting-BK-stack}
The morphism $Y_\Sym^{[0,h],\tau,\be} \ra (L^{[0,h]}\cG_\cO)^\cJ$ sending $(\fM,\fN,\al,\be,\ga)$ to $(A\ix{j}_{\fM,\be})_\cJ$ induces an isomorphism of $p$-adic formal algebraic stacks $Y_\Sym^{[0,h],\tau} \simeq [(L^{[0,h]}\cG_\cO)^\cJ /_{\varphi,(s,\mu)} (L^+\cG_\cO)^\cJ]^\pcp$. Here, $\pcp$ denotes the $p$-adic completion.
\end{prop}
\begin{proof}
This can be proven as \cite[Proposition 5.2.1]{LLHLM-2020-localmodelpreprint} using Proposition \ref{prop:BK-change-of-basis}.
\end{proof}

Recall that $\Gr_{\cG,\F}$ is equal to the affine flag variety $\Fl = \cI_\F \bss L(\GSp_4)_\F$. For integers $a\le b$, let  $L^{[a,b]}(\GSp_4)_\F$ denote the functor defined on $\F$-algebras $R$ by
\begin{align*}
    L^{[a,b]}(\GSp_4)_\F : R \ra \CB{A \in L\GSp_4(R) \mid Av^{-a}, A^\mo v^b \in M_4(R\DB{v})}.
\end{align*}
The fpqc quotient $[\cI_\F \bss L^{[a,b]}(\GSp_4)_\F]$ is isomorphic to a Noetherian closed subscheme $\Fl^{[a,b]}\subset \Fl$. Note that $\Gr_{\cG,\F}^{[a,b]} := \Gr_{\cG,\cO}^{[a,b]}\times_{\spec \cO} \spec \F \subset \Fl^{[a,b]}$ where the inclusion is strict.

We define
\begin{align*}
    \cI_{1,\F} : R \mapsto \CB{A\in \GSp_4(R\DB{v}) \mid A \mod v \in U(R)}
\end{align*}
and $\til{\Gr}_{\cG,\F}^{[a,b]}:= [\cI_{1,\F} \bss L^{[a,b]}\cG_\F]$. It is a $T^\vee_\F$-torsor over $\Gr_{\cG,\F}^{[a,b]}$. Similarly, we have $\til{\Fl}^{[a,b]} := \cI_{1,\F} \bss L^{[a,b]}(\GSp_4)_\F$ which is a $T^\vee_\F$-torsor over $\Fl^{[a,b]}$.

\begin{prop}\label{prop:presenting-BK-stack-mod-p}
Suppose that  the lowest alcove presentation $(s,\mu)$ of $\tau$ is $(h+1)$-generic. There is an isomorphism of algebraic stacks  $\pi_{(s,\mu)}: Y_{\Sym,\F}^{[0,h],\tau} \risom [(\til{\Gr}^{[0,h]}_{\cG,\F})^\cJ /_{(s,\mu)} T_\F^{\vee,\cJ}]$.
\end{prop}
\begin{proof}
This follows from Proposition \ref{prop:presenting-BK-stack} and \cite[Lemma 5.2.2]{LLHLM-2020-localmodelpreprint}. Note that $(I\ix{j})_\jj$ in \loccit~is an element in $\GSp_4(R\DB{v})^\cJ$ if and only if $(X_j)_\jj$ in \loccit~is an element in  $\GSp_4(R\DB{v})^\cJ$.
\end{proof}

Recall that, for each $\tilz \in \tilW^{\vee,\cJ}$, there is a  subfunctor $\cU(\tilz) = \prod_{\jj} \cU(\tilz_j) \subset L\cG^\cJ$. For any integers $a \le b$, we define
\begin{align*}
    U^{[a,b]}(\tilz) := \cU(\tilz)_\cO \cap L^{[a,b]}\cG_\cO^\cJ.
\end{align*}
The natural projection map $U^{[a,b]}(\tilz) \ra (\Gr_{\cG,\cO}^{[a,b]})^\cJ$ is an open immersion by Proposition \ref{prop:tilU-open-in-Gr}. Since $\Gr_{\cG,\F}^\cJ = \Fl^\cJ$ is ind-proper and its $T^{\vee,\cJ}$-fixed points are exactly $\tilz \in \tilW^{\vee,\cJ}$, $\cU(\tilz)_\F$ form an open cover of $\Gr^\cJ_{\cG,\F}$. 
Also we have $T^{\vee,\cJ}$-torsors
\begin{align*}
    \til{\cU}(\tilz) &:=T^{\vee,\cJ} \cU(\tilz) \ra \cU(\tilz) \\
    \til{U}^{[a,b]}(\tilz) &:=T^{\vee,\cJ}_\cO U(\tilz)^{[a,b]} \ra U(\tilz)^{[a,b]}.
\end{align*}
The images of ${U}(\tilz)^{[a,b]}_\F$ (resp.~$\til{U}(\tilz)^{[a,b]}_\F$) form an open cover of $({\Gr}^{[a,b]}_{\cG,\F})^\cJ$ (resp.~$(\til{\Gr}^{[a,b]}_{\cG,\F})^\cJ$).

Recall that $Y_\Sym^{[0,h],\tau}$ is a $p$-adic formal algebraic stack, which implies $Y_\Sym^{[0,h],\tau} = \varinjlim_{i} Y_\Sym^{[0,h],\tau}\times_{\spf\cO} \spec \cO/\varpi^i$. The algebraic stack $Y_\Sym^{[0,h],\tau}\times_{\spf\cO} \spec \cO/\varpi^i$ has the same underlying topological space for all $i\in \Z_{>0}$. There is a bijection between open substacks  of a given algebraic stack and open subsets of its underlying topological space (see \cite[\href{https://stacks.math.columbia.edu/tag/06FJ}{Tag 06FJ}]{stacks-project}). Thus there is a bijection between open substacks of $Y_{\Sym,\cO/\varpi^i}^{[0,h],\tau} := Y_\Sym^{[0,h],\tau}\times_{\spf\cO} \spec \cO/\varpi^i$ and open substacks of $Y_{\Sym,\F}^{[0,h],\tau}$.

\begin{defn}[cf. Definition 5.2.4 in \cite{LLHLM-2020-localmodelpreprint}]\label{defn:gauge}
Let $\tilz \in \tilW^{\vee,\cJ}$. 
\begin{enumerate}
    \item We define $Y_{\Sym,\F}^{[0,h],\tau}(\tilz)$ to be the open substack of $Y_{\Sym,\F}^{[0,h],\tau}$ corresponding to
    \begin{align*}
        [\til{U}^{[0,h]}(\tilz)_\F /_{(s,\mu)} T_\F^{\vee,\cJ}] \subset [(\til{\Gr}^{[0,h]}_{\cG,\F})^\cJ /_{(s,\mu)} T_\F^{\vee,\cJ}]
    \end{align*}
    under the isomorphism $\pi_{(s,\mu)}$. We write $Y_{\Sym, \cO/\varpi^i}^{[0,h],\tau}(\tilz)$ to be the open substack of $Y_{\Sym,\cO/\varpi^i}^{[0,h],\tau}$ induced by $Y_{\Sym,\F}^{[0,h],\tau}(\tilz)$.
    \item We define the  $p$-adic formal open substack $Y_\Sym^{[0,h],\tau}(\tilz) := \varinjlim_i Y_{\Sym, \cO/\varpi^i}^{[0,h],\tau}(\tilz)$ of $Y_\Sym^{[0,h],\tau}$. 
    \item Let $R$ be a $p$-adically complete $\cO$-algebra. We say that $(\fM,\fN,\al)\in Y_\Sym^{[0,h],\tau}(R)$ \emph{admits a $\tilz$-gauge} if $(\fM,\fN,\al) \in Y_\Sym^{[0,h],\tau}(\tilz)(R)$.
    \item For $(\fM,\fN,\al) \in Y_\Sym^{[0,h],\tau}(\tilz)(R)$, we define a \emph{$\tilz$-gauge basis} of $(\fM,\fN,\al)$ as an eigenbasis $(\be,\ga)$ such that $A_{\fM,\be}\ix{j}\in \til{\cU}(\tilz_j)(R)$ for all $\jj$.
\end{enumerate}
\end{defn}

Note that since $\til{U}^{[0,h]}(\tilz)_\F$  form an open cover of $(\til{\Gr}^{[0,h]}_{\cG,\F})^\cJ$, $Y_{\Sym}^{[0,h],\tau}(\tilz)$ form an open cover of $Y_{\Sym}^{[0,h],\tau}$. In Corollary \ref{cor:BK-stack-nonempty}, we show that when these open charts are non-empty.

\begin{prop}\label{prop:gauge-basis}
Let $(s,\mu)$ be a $(h+1)$-generic lowest alcove presentation of $\tau$. Let $(\fM,\fN,\al)\in Y_\Sym^{[0,h],\tau}(\tilz)(R)$. If $\fM$ admits an eigenbasis, then $\fM$ admits a $\tilz$-gauge basis, and the set of $\tilz$-gauge bases is a $T^{\vee,\cJ}(R)$-torsor.
\end{prop}

\begin{proof}
Note that if $(\be,\ga)$ is a $\tilz$-gauge basis for $(\fM,\fN,\al)$, then $\be$ is a $\tilz$-gauge basis for $\fM$ in the sense of \cite[Definition 5.2.6]{LLHLM-2020-localmodelpreprint}. By Proposition 5.2.7 of \emph{loc.~cit.}, $\fM$ admits a $\tilz$-gauge basis $\be$, 
and the set of such bases is $T_4^{\vee,\cJ}$-torsor. Let $\ga$ be any eigenbasis of $\fN$. 
Since $(\be^\vee\otimes \ga)J$ and $\al(\be)$ are $\tilz$-gauge basis of $\fM^\vee\otimes \fN$, there exists $t\in T_4^{\vee,\cJ}(R)$ such that $\al(\be_j)= (\be_j^\vee\otimes\ga_j) J t_j$ for $\jj$. Then the alternating condition implies that $t_j J t_j^\mo = J$. Let $\be'= \be t'$ for some $t'\in T^{\vee,\cJ}_4(R)$ and $\ga' = \ga c$ for some $c\in (R\DB{v}^\times)^\cJ$. Then $(\be')^\vee = \be^\vee (t')^\mo$ and
 \begin{align*}
     \al(\be'_j) = ((\be'_j)^\vee\otimes \ga'_j)t'_jJt'_jt_jc_j^\mo \text{ for each $\jj$}.
 \end{align*}
The condition on $t$ implies that there exists $t'_j$ and $c_j$ such that $t'_jJt'_j = c_jJt^\mo_j$ for each $\jj$, and the solution gives a $\tilz$-gauge basis $(\be',\ga')$ of $(\fM,\fN,\al)$. Since the set of solutions is a $T^{\vee,\cJ}$-torsor, the set of $\tilz$-gauge bases is a  $T^{\vee,\cJ}$-torsor.
\end{proof}

Let $\lam \in X_*(T^\vee)^\cJ$ be a dominant cocharacter such that $\std(\lam_j) \in [0,h]^4$ for all $\jj$. There is an  $\cO$-flat closed $p$-adic formal substack $Y_4^{\le\std(\lam),\tau}\subset Y_4^{[0,h],\tau}$ (e.g.~\cite[\S5.3]{LLHLM-2020-localmodelpreprint}). For simplicity, we write $Y_4^{\le\lam,\tau}$ instead of $Y_4^{\le\std(\lam),\tau}$. We have its symplectic variant $Y_{\Sym}^{\le\lam,\tau}$ which is defined as an  $\cO$-flat part  of $Y_\Sym^{[0,h],\tau} \times_{Y^{[0,h],\tau}_4} Y_4^{\le\lam,\tau}$. For any finite extension  $E'/E$ with ring of integers $\cO'$, $(\fM,\fN,\al)\in Y_{\Sym}^{[0,h],\tau}(\cO')$ belongs to $(\fM,\fN,\al)\in Y_{\Sym}^{\le\lam,\tau}(\cO')$ if and only if the elementary divisors of $A_{\fM,\be}\ix{j} \in \GSp_4(E'\DP{v+p})$ (for any eigenbasis $(\be,\ga)$) is bounded by $E(v)^{\lam_j}$ for each $\jj$. In particular, $(A_{\fM,\be}\ix{j})_\jj$ gives a $E'$-point in $S_E(\lam)\subset \Gr_{\cG,E}^\cJ$. We have its open substack
\begin{align*}
    Y^{\le\lam,\tau}_\Sym(\tilz) = Y^{\le\lam,\tau}_\Sym \times_{Y^{[0,h],\tau}_\Sym} Y^{[0,h],\tau}_\Sym(\tilz).
\end{align*}

\begin{rmk}
For any finite extension $\F'/\F$, $Y_4^{\le \lam,\tau}(\F')$ is the full subgroupoid of $Y_4^{[0,h],\tau}(\F')$ consisting of Breuil--Kisin modules whose shapes lie in the set $\Adm^\vee(\std(\lam))$ by \cite[Proposition 5.4]{CL-ENS-2018-Kisinmodule-MR3764041}. By Lemma \ref{lem:transfer-preserves-bruhat-order}, $Y_{\Sym}^{\le \lam,\tau}(\F')$ is the full subgroupoid of $Y_{\Sym}^{[0,h],\tau}(\F')$ consisting of Breuil--Kisin modules whose shapes lie in the set $\Adm^\vee(\lam) \simeq \Adm^\vee(\std(\lam))^\Theta$. 
\end{rmk}

On the local model side, recall the Pappas--Zhu local model $M_\cJ(\le\lam) \subset \Gr^{[0,h],\cJ}_{\cJ,\cO}$. We have the open neighborhood at $\tilz$ defined by $U(\tilz,\le\lam) := M_\cJ(\le\lam) \cap \cU(\tilz)_\cO$. We also write  $\til{U}(\tilz,\le\lam) := T^{\vee,\cJ}_\cO \times_{\spec \cO} U(\tilz,\le\lam)$.

\begin{thm}\label{thm:BK-local-model}
Let $(s,\mu)$ be a $(h+1)$-generic lowest alcove presentation of $\tau$ and let $\tilz\in \tilW^{\vee,\cJ}$. We have a local model diagram of $p$-adic formal algebraic stacks over $\cO$
\[
\begin{tikzcd}
& \til{U}(\tilz, \le \lam)^\pcp \arrow[ld, "T^{\vee,\cJ}_\cO"'] \arrow[rd, "T^{\vee,\cJ}_\cO"] & \\
Y_\Sym^{\le \lam,\tau}(\tilz) = \BR{\til{U}(\tilz, \le \lam) /_{(s,\mu)} T^{\vee,\cJ}_\cO}^\pcp 
& & U(\tilz, \le \lam)^\pcp 
\end{tikzcd}
\]
where diagonal arrows are $T^{\vee,\cJ}_\cO$-torsors. The superscript $\wedge_p$ means taking $p$-adic completion.
\end{thm}

\begin{proof}
This follows from the proof of \cite[Theorem 5.3.1 and 5.3.3]{LLHLM-2020-localmodelpreprint} using Proposition \ref{prop:BK-change-of-basis} and \ref{prop:gauge-basis}.  
\end{proof}

\begin{cor}\label{cor:BK-stack-nonempty}
Under the hypothesis in Theorem \ref{thm:BK-local-model}, $Y^{\le \lam,\tau}_\Sym(\tilz) \neq \emptyset$ if and only if $\tilz\in \Adm^\vee(\lam)$.
\end{cor}

\begin{proof}
By Theorem \ref{thm:BK-local-model}, the stack $Y^{\le\lam,\tau}_\Sym(\tilz)$ is nonempty if and only if $U(\tilz,\le\lam)\neq \emptyset$. By \cite[Theorem 9.3]{PZ13-Inv-local_model-MR3103258}, $M_\cJ(\le\lam)_\F =\cup_{\til{s}\in \Adm^\vee(\lam)} S^\circ_\F(\til{s})$. Then $U(\tilz,\le\lam)\neq \emptyset$ if and only if  $S^\circ_\F(\til{s}) \cap \cU(\tilz)_\F$ for some $\til{s}\in \Adm^\vee(\lam)$. The last condition is equivalent to $\tilz \in S_\F(\til{s})$ for some $\til{s}\in \Adm^\vee(\lam)$ by \cite[Lemma 4.7.1]{LLHLM-2020-localmodelpreprint} (which easily generalizes to our setup). Then the claim follows from the standard description of torus-fixed points of affine Schubert varieties.
\end{proof}

\subsection{Symplectic \'etale $\varphi$-modules }\label{sub:et-phi-stack}
In this subsection, we fix a dominant cocharacter $\lam\in X_*(T)^\cJ$ such that $\std(\lam_j) \in [0,h]^4$ and a $1$-generic inertial tame type $\tau$. We continue to denote by $R$ a $p$-adically complete $\cO$-algebra.

Let $n>0$ be an integer. The ring $\cO_{\cE,K}:= W(k)\DP{v}^\pcp$  is equipped with a Frobenius endomorphism $\varphi$ extending usual Frobenius on $W(k)$ and sending $v$ to $v^p$. For $p$-adically complete Noetherian $\cO$-algebra $R$, $\Phi\text{-}\mathrm{Mod}_K^{\et,n}(R)$ is defined as the groupoid of rank $n$ \'etale $\varphi$-modules with $R$-coefficients. It is known that $\Phi\text{-}\mathrm{Mod}_K^{\et,n}$ is an  ind-algebraic fppf stack over $\spf \cO$ \cite[Corollary 3.1.5]{EGstack}.   

Objects of $\Phi\text{-}\mathrm{Mod}_K^{\et,n}(R)$ are given by rank $n$ projective modules $\cM$ over $\cO_{\cE,K}\ctensor_{\Z_p} R$ equipped with an isomorphism $\phi_\cM : \varphi^*(\cM) \risom \cM$. For each $\jj$, we have an induced morphism   $\phi_\cM\ix{j}: \cM\ix{j-1} \ra \cM\ix{j}$. We also define $\cM^\vee \in \Phi\text{-}\mathrm{Mod}_K^{\et,n}(R)$  the dual \'etale $\varphi$-module of $\fM$ whose underlying module is $\Hom_{\cO_{\cE,K}\ctensor_{\Z_p} R}(\cM, \cO_{\cE,K}\ctensor_{\Z_p} R)$ and  $\phi_{\cM^\vee} : \varphi^*(\cM^\vee) \ra \cM^\vee$ given by a formula
\begin{align*}
\phi_{\cM^\vee}(f)(m) = \varphi(f ( \phi_{\cM}^\mo(m))  
\end{align*}
for all $f\in \varphi^*(\cM^\vee) =\Hom_{\cO_{\cE,K}\ctensor_{\Z_p} R}(\varphi^*(\cM), \cO_{\cE,K}\ctensor_{\Z_p} R)$ and $m\in \cM$.

We define $\catphiMsym$ as the moduli stack of \emph{symplectic} \'etale $\varphi$-modules whose objects are given by triples $(\cM,\cN,\al)\in \catphiMsym(R)$ where $\cM \in \catphiMfour(R)$, $\cN\in \catphiMone(R)$, and $\al: \cM \risom \cM^\vee\otimes_{\cO_{\cE,K}\ctensor_{\Z_p}R} \cN$ satisfying alternating condition $(\al^\vee \otimes \cN)^\mo \circ \al = -1$. 

\begin{prop}\label{prop:et-phi-mod-stack-alg}
The map $\catphiMsym \ra \Phi\text{-}\mathrm{Mod}_K^{\et,4} \times_\cO \Phi\text{-}\mathrm{Mod}_K^{\et,1}$ sending $(\cM,\cN,\al)$ to $(\cM,\cN)$ is representable by algebraic spaces. In particular, $\catphiMsym$ is an  ind-algebraic fppf stack over $\spf \cO$.
\end{prop}
\begin{proof}
This can be proven as Proposition \ref{prop:BKstack-alg} using \cite[Corollary 3.1.5]{EGstack}.
\end{proof}

\begin{defn}
Let $(\cM,\cN,\al)\in \catphiMsym(R)$. A \emph{basis} of $(\cM,\cN,\al)$ is a pair $(\be,\ga) = ((\be\ix{j})_{\jj} , (\ga\ix{j})_{\jj})$ where for each $\jj$, $\be\ix{j}$ (resp.~$\ga\ix{j}$) is a basis for a rank 4 (resp.~1) free $R\DB{v}[1/v]^\pcp$-module $\cM\ix{j}$ (resp.~$\cN\ix{j}$)   such that
\begin{align*}
    \al(\be\ix{j}) = ((\be\ix{j})^\vee \otimes \ga\ix{j})J.
\end{align*}
\end{defn}

\begin{lemma}\label{lem:et-phi-Frob-GSp4}
Let $(\cM,\cN,\al)\in \catphiMsym(R)$.
\begin{enumerate}
    \item If $(\be,\ga)$ is a basis of $(\cM,\cN,\al)$, then the matrix representation of $\phi_\cM\ix{j}$ (resp.~$\phi_\cN\ix{j}$) with respect to the basis $\be\ix{j-1}$ and $\be\ix{j}$ (resp.~$\ga\ix{j-1}$ and $\ga\ix{j}$) is given by a matrix $A\ix{j}\in \GSp_4(R\DP{v}^\pcp)$ (resp.~$\simc(A)\in (R\DP{v}^\pcp)^\times$).
    \item If $(\be_1,\ga_1)$ and $(\be_2,\ga_2)$ are basis of $(\cM,\cN,\al)$ such that $\be_2\ix{j} = \be_1\ix{j}A_j$ and $\ga_2\ix{j} = \ga_1 \ix{j}c_j$ for some $A_j \in \GL_4(R\DP{v}^\pcp)$ and $c_j \in (R\DP{v}^\pcp)^\times$, then $A_j \in \GSp_4(R\DP{v}^\pcp)$ and $c_j = \simc(A_j)$.
\end{enumerate}
\end{lemma}
\begin{proof}
Item (1) can be shown as Lemma \ref{lem:BK-Frob-GSp4}. Item (2) follows from a direct computation using the condition $\al(\be\ix{j})=( (\be\ix{j})^\vee \otimes \ga\ix{j})J$.
\end{proof}

For a tame inertial type $\tau'$ valued in $T_n^\vee(E)$, there is a morphism $\varepsilon_{\tau'} : Y^{[0,h],\tau'}_n \ra \Phi\text{-}\mathrm{Mod}_K^{\et,n}$   representable by algebraic spaces, proper, and of finite presentation (\cite[Proposition 5.4.1]{LLHLM-2020-localmodelpreprint}).

Given $\fM \in Y^{[0,h],\tau'}_n(R)$ and an integer $m\ge 0$, we define $\fM(m)\in Y^{[m,h+m],\tau'}_n(R)$ to be the Breuil--Kisin module whose underlying module is $\fM$ and the Frobenius endomorphism is given by $\phi_{\fM(m)} = E(v)^m \phi_{\fM}$.

We define a morphism of ind-algebraic fppf stacks over $\spf \cO$ 
\begin{align*}
    \veps_\tau : Y_\Sym^{[0,h],\tau} &\ra \catphiMsym \\
    (\fM,\fN,\al) &\mapsto (\varepsilon_{\std(\tau)}(\fM), \varepsilon_{\simc(\tau)}(\fN(h)), \varepsilon_{\std(\tau)} (\al)).
\end{align*}

\begin{rmk}
Note that the dual of Breuil--Kisin module and \'etale $\varphi$-module are not compatible. This is because of $E(v)^h$ in the formula defining the Frobenius of dual Breuil--Kisin modules. This is why we use $\varepsilon_{\simc(\tau)}(\fN(h))$ instead of $\varepsilon_{\simc(\tau)}(\fN)$ in the definition of $\veps_\tau$, so that we still have $\veps_{\tau}(\fM^\vee \otimes \fN) = \veps_{\std(\tau)}(\fM)^\vee \otimes  \varepsilon_{\simc(\tau)}(\fN(h))$. In particular, this explains that $\veps_{\std(\tau)}(\al)$ is well-defined.
\end{rmk}

\begin{prop}\label{prop:eps-tau-closed-immersion}
Suppose that $\tau$ is $(h+1)$-generic. The map $\varepsilon_\tau :Y_\Sym^{[0,h],\tau} \ra \catphiMsym$ is a closed immersion.
\end{prop}
\begin{proof}
By \cite[Proposition 5.4.3]{LLHLM-2020-localmodelpreprint}, $\varepsilon_{\std(\tau)} : Y^{[0,h],\tau}_4 \ra \catphiMfour$ is a monomorphism. Then it is fully faithful by \cite[\href{https://stacks.math.columbia.edu/tag/04ZZ}{Tag 04ZZ}]{stacks-project}. As a result, the diagram
\[
\begin{tikzcd}
Y_\Sym^{[0,h],\tau} \arrow[rr, "\varepsilon_\tau"] \arrow[d] & &  \catphiMsym \arrow[d] \\
Y^{[0,h],\tau}_4 \times_{\spf\cO} Y^{0,\simc(\tau)}_1 \arrow[rr, "\varepsilon_{\std(\tau)} \times \veps_{\simc(\tau)}"] & & \catphiMfour\times_{\spf \cO} \catphiMone
\end{tikzcd}
\]
is cartesian, and the claim follows from \cite[Proposition 5.4.3]{LLHLM-2020-localmodelpreprint}.
\end{proof}

\begin{lemma}\label{lem:eps_tau-formula}
Let $(\fM,\fN,\al) \in Y^{[0,h],\tau}_\Sym(R)$ and $(\cM,\cN,\veps_\tau(\al)) = \veps_\tau((\fM,\fN,\al))$. If $(\be,\ga)$ is an eigenbasis of $(\fM,\fN,\al)$, there exists a basis $(\fb,\fc)$ of $(\cM,\cN,\veps_\tau(\al))$ determined by $(\be,\ga)$ such that $\phi_{\cM}\ix{j}$ with respect to $\fb$ is given by $A\ix{j}_{\fM,\be}s_{j}^\mo v^{\mu_j+\eta_j}$.
\end{lemma}
\begin{proof}
The existence of $\fb$ follows from \cite[Proposition 5.4.2]{LLHLM-2020-localmodelpreprint}. Then there is a unique $\fc$ such that $(\fb,\fc)$ is a basis of  $(\cM,\cN,\veps_\tau(\al))$.
\end{proof}

Let $a \le b$ be integers. There is a natural map
\begin{align*}
    \iota'_{\tilz}: \prod_\jj (L^{[a,b]}\GSp_4)_\F \tilz_j &\ra \catphiMsym  \\
    (A\ix{j}\tilz_j)_\jj & \mapsto (\cM,\cN,\al)
\end{align*}
where $\cM$ (resp.~$\cN$) is a free rank $4$ (resp.~1) \'etale $\varphi$-module such that $\phi_\cM\ix{j}$ (resp.~$\phi_\cN\ix{j}$) with respect to the standard basis is given by $A\ix{j}\tilz_j$ (resp.~$\simc(A\ix{j}\tilz_j)$) and $\al$ is given by the matrix $J$ with respect to the standard basis of $\cM$ and its dual basis. 
We also define a closed subscheme
\begin{align*}
    \til{\Fl}^{[a,b]}_{\cJ,\tilz} := \prod_\jj (\cI_{1,\F} \bss (L^{[a,b]}\GSp_4)_\F \tilz_j) \subset \til{\Fl}^\cJ.
\end{align*}
We denote by $T^{\vee,\cJ}_\F\mathrm{\dash conj}$ a $T^{\vee,\cJ}_\F$-action on  $ \til{\Fl}^{[a,b]}_{\cJ,\tilz}$ given by
\begin{align*}
    (D_j)_\jj\cdot (\cI_{1,\F} A\ix{j}\tilz_j)_\jj  = (D_j \cI_{1,\F} A\ix{j}\tilz_j D_{j-1}^\mo)_\jj
\end{align*}
for $(D_j)_\jj \in T^{\vee,\cJ}_\F$ and $(\cI_{1,\F} A\ix{j}\tilz_j)_\jj \in \til{\Fl}^{[a,b]}_{\cJ,\tilz}$.

\begin{prop}\label{prop:et-phi-mod-finite-height}
Suppose that $\tilz= \sig^\mo t_{\nu+\eta}$ where $\nu$ is $(b-a+1)$-deep in $\uC_0$. Then the morphism $\iota'_{\tilz}$ induces  a monomorphism
\begin{align*}
    \iota_{\tilz} : [\til{\Fl}^{[a,b]}_{\cJ,\tilz} / T_\F^{\vee,\cJ}\mathrm{\dash conj}] \mono \catphiMsym.
\end{align*}
\end{prop}
\begin{proof}
By Lemma \ref{lem:et-phi-Frob-GSp4} and \cite[Lemma 5.4.4]{LLHLM-2020-localmodelpreprint}, the morphism $\iota'_{\tilz}$ factors through a monomorphism
\begin{align*}
    [\prod_\jj (L^{[a,b]}\GSp_4)_\F \tilz_j/_{\varphi}\cI_\F] \mono \catphiMsym,
\end{align*}
and the source is isomorphic to $[\til{\Fl}^{[a,b]}_{\cJ,\tilz} / T_\F^{\vee,\cJ}\mathrm{\dash conj}]$ by Lemma 5.2.2 in \loccit. 
\end{proof}

By combining Proposition \ref{prop:presenting-BK-stack-mod-p}, Lemma \ref{lem:eps_tau-formula}, and Proposition \ref{prop:et-phi-mod-finite-height},  we obtain the following result.

\begin{prop}\label{prop:diagram-BK-and-et-phi-stacks}
Let $a\le b$, $h\ge 0$ be integers  and $\tilz = \sig^\mo t_{\nu+\eta}\in \tilW^{\vee,\cJ}$ such that $\nu$ is $(b-a+1)$-deep in $\uC_0$. Suppose that $(\Gr_{\cG,\F}^{[0,h],\cJ})\tilw^*(\tau) \subset \Fl^{[a,b]}_{\cJ,\tilz}$. Then we have a commutative diagram
\[
\begin{tikzcd}
 \til{M}_\cJ(\le\lam)_\F \arrow[r, hook] \arrow[d] & \til{\Gr}_{\cG,\F}^{\BR{0,h},\cJ} \arrow[r, hook, "r_{\tilw^*(\tau)}"] \arrow[d, "\pi_{(s,\mu)}"] & \til{\Fl}^{\BR{a,b}}_{\cJ,\tilz} \arrow[r] & \BR{\til{\Fl}^{\BR{a,b}}_{\cJ,\tilz} / T_\F^{\vee,\cJ}\dash\mathrm{conj}}\ \arrow[d, "\iota_{\tilz}"] \\
  Y_{\Sym,\F}^{\le\lam,\tau} \arrow[r, hook] & Y_{\Sym,\F}^{\BR{0,h},\tau} \arrow[rr, hook, "\veps_\tau"] & & \catphiMsym
\end{tikzcd}
\]
where all hooked arrows are closed immersions. Here, $r_{\tilw^*(\tau)}$ is the map given by the right multiplication by $\tilw^*(\tau)$.
\end{prop}

\subsection{Local models for potentially crystalline stacks}\label{sub:local-models-pcrys}

We have a canonical morphism $\varepsilon_\infty: \cX_n \ra \catphiMn$ constructed in \cite[Proposition 3.7.2]{EGstack} which corresponds to restricting $G_K$-representations to $G_{K_\infty}$-representations when evaluated at complete Noetherian finite local $\cO$-algebras. We have an induced map $\varepsilon_\infty : \cX_\Sym \ra \catphiMsym$.

\begin{prop}\label{prop:eps-inf-mono}
Suppose $\tau$ is $(h+2)$-generic. Then the map $\varepsilon_\infty: \cX_\Sym^{[0,h],\tau} \ra \catphiMsym$ is a monomorphism.
\end{prop}
\begin{proof}
This follows from \cite[Proposition 7.2.11]{LLHLM-2020-localmodelpreprint} using the argument given in the proof of Proposition \ref{prop:eps-tau-closed-immersion}
\end{proof}

We define $\cX_{\Sym}^{\le\lam,\tau}$ (resp.~ $\cX_{\Sym,\reg}^{\le\lam,\tau}$) to be the scheme-theoretic union of $\cX_{\Sym}^{\lam',\tau}$ for all dominant (resp.~regular dominant) cocharacters $\lam'\le \lam$. Similarly, we have $\cX_4^{\le\lam',\tau'}$.

Recall the $\cO$-flat closed substack $Y^{[0,h],\tau,\nba_\infty}_4$ (resp.~$Y^{\le\lam,\tau,\nba_\infty}_4$) of $Y^{[0,h],\tau}_4$ (resp.~$Y^{\le\lam,\tau}_4$) defined in \cite[\S7.2]{LLHLM-2020-localmodelpreprint}. Its key property is that if $\tau$ is $(h+2)$-generic, then $\cX^{[0,h],\tau}_4 \ra \catphiMfour$ factors through $Y^{[0,h],\tau,\nba_\infty}_4$ and induces isomorphisms $\cX^{[0,h],\tau}_4 \simeq Y^{[0,h],\tau,\nba_\infty}_4$ and  $\cX^{\le\lam,\tau}_4 \simeq Y^{\le\lam,\tau,\nba_\infty}_4$ (\cite[Prop 7.2.3]{LLHLM-2020-localmodelpreprint}).  We define $Y^{[0,h],\tau,\nba_\infty}_\Sym$ (resp.~$Y^{\le\lam,\tau,\nba_\infty}_\Sym$) to be the $\cO$-flat part of $Y^{[0,h],\tau}_\Sym \times_{Y^{[0,h],\tau}_4} Y^{[0,h],\tau,\nba_\infty}_4$ (resp.~$Y^{\le\lam,\tau}_\Sym \times_{Y^{\le\lam,\tau}_4} Y^{\le\lam,\tau,\nba_\infty}_4$). We have the following result.

\begin{prop}\label{prop:EG-stack=BK-stack}
Suppose that $\tau$ is $(h+2)$-generic. The map $\veps_\infty$ factors through $Y_\Sym^{[0,h],\tau,\nba_\infty}$ and induces isomorphisms $\cX_\Sym^{[0,h],\tau} \simeq Y_\Sym^{[0,h],\tau,\nba_\infty}$ and $\cX_\Sym^{\le\lam,\tau}\simeq Y_\Sym^{\le\lam,\tau,\nba_\infty}$.
\end{prop}

\begin{proof}
Since $\veps_\infty$ and $\veps_{\tau}$ are fully faithful, we have
\begin{align*}
    \cX_\Sym^{[0,h],\tau} & \simeq \PR{(\cX_4^{[0,h],\tau} \times \cX_1^{[0,2h],\simc(\tau)}) \times_{\catphiMfour \times \catphiMone} \catphiMsym}_{\cO\mathrm{\dash flat}} \\ 
    Y_\Sym^{[0,h],\tau,\nba_\infty} & \simeq \PR{(Y_4^{[0,h],\tau,\nba_\infty}\times Y_1^{[0,2h],\simc(\tau)})\times_{\catphiMfour \times \catphiMone} \catphiMsym}_{\cO\mathrm{\dash flat}}
\end{align*}
where the subscript ${\cO\mathrm{\dash flat}}$ means taking the $\cO$-flat part. Then the claim follows from the corresponding result for $\GL_4 \times \GL_1$ in \cite[Prop 7.2.3]{LLHLM-2020-localmodelpreprint}.
\end{proof}

We assume that $\tau$ is $(h+2)$-generic until the end of this subsection. Using the preceding proposition, we consider $\cX_{\Sym}^{[0,h],\tau}$ as a closed substack of $Y_{\Sym}^{[0,h],\tau}$.

Before stating the main result of this subsection, we introduce some notations.  Recall that $Y_{\Sym}^{\le\lam,\tau}$ is the union of $Y_{\Sym}^{\le\lam,\tau}(\tilz)$ for $\tilz\in \Adm^\vee(\lam)$. We define
\begin{align*}
    Y_{\Sym}^{\le\lam,\tau,\nba_\infty}(\tilz)\subset Y_{\Sym}^{\le\lam,\tau,\nba_\infty},\ \  \cX_{\Sym}^{\lam,\tau}(\tilz)\subset \cX_{\Sym}^{\lam,\tau}, \ \  \cX_{\Sym,\reg}^{\le\lam,\tau}(\tilz)\subset \cX_{\Sym,\reg}^{\le\lam,\tau}
\end{align*}
to be the open substacks obtained by taking the intersection with $Y_{\Sym}^{\le\lam,\tau}(\tilz)$ inside $Y_{\Sym}^{\le\lam,\tau}$. We also define $T^{\vee,\cJ}_\cO$-torsors
\begin{align*}
    \til{U}(\tilz, \le \lam, \nba_{\tau,\infty}) \ra Y^{\le\lam,\tau,\nba_\infty}_\Sym(\tilz), \ \ 
    \til{\cX}_{\Sym}^{\lam,\tau}(\tilz) \ra \cX_{\Sym}^{\lam,\tau}(\tilz), \ \ 
    \til{\cX}_{\Sym,\reg}^{\lam,\tau}(\tilz)  \ra \cX_{\Sym,\reg}^{\lam,\tau}(\tilz)
\end{align*}
by taking pullback along the $T^{\vee,\cJ}_\cO$-torsor $\til{U}(\tilz,\le\lam)^\pcp \ra Y^{\le\lam,\tau}_\Sym(\tilz)$.

Let $\bfa\in \cO^3$. We define  $M_{\cJ,\reg}(\le\lam,\nba_{\bfa})$ to be the union of $M_\cJ(\lam',\nba_{\bfa})$ for all regular dominant  cocharacters $\lam'\le\lam$. Note that by Proposition \ref{prop:generic-fiber-M(=<lam)}, $M_{\cJ,\reg}(\le\lam,\nba_{\bfa})$ is a maximal $\cO$-flat closed subvariety equidimensional of dimension $1+4\#\cJ$ inside $M_\cJ^\nv(\le\lam,\nba_{\bfa})$ (which is in turn contained in $M_\cJ(\le\lam)$).  We define open subschemes
\begin{align*}
    U^\nv(\tilz,\le\lam,\nba_{\bfa}) \subset M_\cJ^\nv(\le\lam,\nba_{\bfa})\ \  \text{and} \ \ U_{\reg}(\tilz,\le\lam,\nba_{\bfa}) \subset M_{\cJ,\reg}(\le\lam,\nba_{\bfa})
\end{align*}
by taking the intersection with the open subscheme $U(\tilz,\le\lam)\subset M_\cJ(\le\lam)$. We also define $T^{\vee,\cJ}_\cO$-torsors
\begin{align*}
    \til{U}^\nv(\tilz,\le\lam,\nba_{\bfa}) \ra U^\nv(\tilz,\le\lam,\nba_{\bfa}) \ \ \text{and} \ \ 
    \til{U}_{\reg}(\tilz,\le\lam,\nba_{\bfa}) \ra U_{\reg}(\tilz,\le\lam,\nba_{\bfa}) 
\end{align*}
by taking pullback along the trivial $T^{\vee,\cJ}_\cO$-torsor $\til{U}(\tilz,\le\lam) \ra U(\tilz,\le\lam)$.

To the lowest alcove presentation $(s,\mu)$ of $\tau$, we associate tuples of integer $\bfa_\tau = (\bfa_{\tau,j'})_{j' \in \cJ'} \in (\cO^3)^{\cJ'}$ given by $\bfa_{\tau,j'} := (s'_{\rmor,j'})^\mo (\bfa')\ix{j'}/(1-p^{f'})$. We remark that if $\jj$  and  $j \equiv j' \mod f$, then $\bfa_{\tau,j'} \equiv s_j^\mo(\mu_j+\eta_j) \mod \varpi$ (see \cite[Lemma 7.3.1]{LLHLM-2020-localmodelpreprint}).

\begin{thm}\label{thm:pcrys-local-model}
\begin{enumerate}
    \item We have the following commutative diagram of $p$-adic formal algebraic stacks
\[
\begin{tikzcd}
\til{\cX}_{\Sym,\reg}^{\le\lam,\tau}(\tilz) \arrow[rrrr, bend left=15, dashed] \arrow[r, hook]& \til{U}(\til{z},\le\lam,\nba_{\tau,\infty}) \arrow[r, hook] \arrow[ld, "T^{\vee,\cJ}_\cO"'] \arrow[rr, bend left=15, dashed] & \til{U}(\til{z},\le\lam)^{\wedge_p} \arrow[ld, "T^{\vee,\cJ}_\cO"']  \arrow[rd, "T^{\vee,\cJ}_\cO"]  & \til{U}^\nv(\til{z}, \le\lam,\nba_{\bfa_\tau})^{\wedge_p} \arrow[rd, "T^{\vee,\cJ}_\cO"]  \arrow[l, hook']&  \til{U}_\reg(\tilz,\le\lam,\nba_{\bfa_\tau})^\pcp \arrow[l, hook']
\\
 Y_\Sym^{\le\lam,\tau,\nba_\infty}(\til{z}) \arrow[r, hook] & Y_\Sym^{\le\lam,\tau}(\til{z}) &  & U(\til{z},\le\lam)^{\wedge_p} &U^\nv(\til{z},\le\lam,\nba_{\bfa_\tau})^{\wedge_p} \arrow[l, hook'].
\end{tikzcd}
\]
where the two parallelograms are cartesian, hooked arrows are closed immersions, and diagonal arrows are $T^{\vee,\cJ}_\cO$-torsors. Moreover, there exists an integer $N_\mathrm{sing} = N(\CB{\lam_j}_{j\in \cJ})$ depending only on $\CB{\lam_j}_{j\in \cJ} \subset \Z^3$ such that if $\mu$ is $N_\mathrm{sing}$-deep in $\ud{C}_0$, then the dotted arrows exist and are closed immersions for all $\til{z}\in \Adm^\vee(\lam)$.

\item There exists a nonzero polynomial $P= P_{\CB{\lam_j}_{j\in\cJ}}(X_1,X_2,X_3)\in \Z[X_1,X_2, X_3]$ depending only on $\CB{\lam_j}_{j\in \cJ} \subset \Z^3$ and the ramification index $e$ of $\cO$ such that if $\mu$ is $P$-generic, then the longer dotted arrow is an isomorphism. In particular, we obtain the following local model diagram
\[
\begin{tikzcd}
 & \til{\cX}_{\Sym,\reg}^{\le\lam,\tau}(\tilz) \simeq \til{U}_\reg(\tilz,\le\lam,\nba_{\bfa_\tau})^\pcp \arrow[ld, "T^{\vee,\cJ}_\cO"']  \arrow[rd, "T^{\vee,\cJ}_\cO"]& \\
 {\cX}_{\Sym,\reg}^{\le\lam,\tau}(\tilz) & & {U}_\reg(\tilz,\le\lam,\nba_{\bfa_\tau})^\pcp
\end{tikzcd}.
\]
Furthermore, for any tame $\rhobar\in \cX_{\Sym}(\F)$ and regular dominant $\lam'\le\lam$, the versal rings of $\cX_{\Sym}^{\lam',\tau}$ at $\rhobar$ are domains (if nonzero).
\end{enumerate}
\end{thm}

For the proof of Theorem \ref{thm:pcrys-local-model}, we need the following lemma. Its proof is given in \S\ref{sub:patching-argument}.

\begin{lemma}\label{lem:pcrys-stack-nonempty}
Suppose that $\tau$ is $(h+4(h_\eta+1))$-generic. Then $\cX_{\Sym}^{\lam,\tau}(\tilz) \neq \emptyset$ if and only if $\tilz \in \Adm^\vee(\lam)$. 
\end{lemma}

\begin{proof}[Proof of Theorem \ref{thm:pcrys-local-model}]
The proof is essentially identical to that of \cite[Theorem 7.3.2]{LLHLM-2020-localmodelpreprint}. For item (1), note that Proposition 3.3.9 in \loccit, which uses Elkik's approximation theorem, easily generalizes to our setup; its proof only uses the fact that open charts of universal local models are affine and smooth after inverting $v$ which follows from Proposition \ref{prop:nv-local-model-smooth} in our case. For item (2), we use Theorem \ref{thm:unibranch-product} and Lemma \ref{lem:pcrys-stack-nonempty} instead of Theorem 3.7.1 and Lemma 7.3.5 in \loccit. 
\end{proof}

\subsection{Potentially crystalline stacks modulo $p$}\label{sub:pcrys-modp}
Let $\lam\in X_*(T^\vee)^\cJ$ be a regular dominant cocharacter such that $\std(\lam)\subset([0,h]^4)^\cJ$. Let $\tau$ be a tame inertial type with $(h+2)$-generic lowest alcove presentation $(s,\mu)$. 
Recall that both $\cX_{\Sym,\red}$ and $\cX_{\Sym,\F}^{\lam,\tau}$ are equidimensional algebraic stacks over $\F$ of dimension $4f$. Thus $\cX_{\Sym,\F}^{\lam,\tau}$ is topologically a union of irreducible components of $\cX_{\Sym,\red}$, which are labeled by Serre weights. 

Suppose that $\cC_\sig \subset \cX_{\Sym,\F}^{\le\lam,\tau}$. We define algebraic stacks $\til{\cC}_\sig$  and $\til{\cX}_{\Sym,\F}^{\le\lam,\tau}$ by the following cartesian diagram
\begin{equation}\label{eqn:diagram-EG-BK-stacks}
    \begin{tikzcd}
\til{\cC}_\sig \arrow[r, hook]  \arrow[d] & \til{\cX}_{\Sym,\F}^{\le\lam,\tau} \arrow[r, hook] \arrow[d] & \til{M}_\cJ(\le\lam)_\F  \arrow[d, "\pi_{\PR{s,\mu}}"] \\
\cC_\sig \arrow[r, hook] & {\cX}_{\Sym,\F}^{\le\lam,\tau} \arrow[r, hook] & Y_{\Sym,\F}^{\le\lam,\tau}.
\end{tikzcd}
\end{equation}
If $\mu$ is $(2h-2)$-deep in $\uC_0$, then the closed immersion $\til{\cX}_{\Sym,\F}^{\le\lam,\tau} \mono \til{M}_\cJ(\le\lam)_\F$ factors through $\til{M}_\cJ^\nv (\le\lam,\nba_{\bfa_\tau})_\F$ by \cite[Proposition 7.1.10]{LLHLM-2020-localmodelpreprint}. 
Recall that by Theorem \ref{thm:nv-irred-comp-SW-bij}, $7f$-dimensional irreducible components of $\til{M}_\cJ^\nv (\le\lam,\nba_{\bfa_\tau})_\F$ are exactly $\til{C}_\sigma^\zeta \tilw(\tau)^*$ for all $\sigma \in \JH(W(\phi^\mo(\lam)-\eta)\otimes \ov{\sig}(\tau))$.

The following theorem describes the underlying reduced substack of $\cX_{\Sym,\F}^{\le\lam,\tau}$.

\begin{thm}\label{thm:naive-BM}
Let $\lam,\tau$ be as above. We assume that $\mu$ is $\max\{2h_\eta,(h+4(h_\eta+1))\}$-deep in $\uC_0$. Then we have
\begin{align*}
    \PR{\cX_{\Sym,\reg}^{\le\lam,\tau}}_\red = \PR{\cX_{\Sym}^{\lam,\tau}}_\red = \bigcup_{\sig \in \JH(W(\phi^\mo(\lam)-\eta)
    \otimes \ov{\sig}(\tau))} \cC_\sig.
\end{align*}
\end{thm}

We need the following lemma, whose proof is given in \S\ref{sub:patching-argument}.

\begin{lemma}\label{lem:lowerbound-mod-p-crys-stack}
Suppose that $\mu$ is $(h+4(h_\eta+1))$-deep in $\uC_0$. Then for each $\sig \in \JH(W(\phi^\mo(\lam)-\eta)
    \otimes \ov{\sig}(\tau))$, we have $\cC_\sig \subset \cX_{\Sym,\F}^{\lam,\tau}\subset \cX_{\Sym,\F}^{\le \lam,\tau}$.
\end{lemma}

\begin{proof}[Proof of Theorem \ref{thm:naive-BM}]
Since $\mu$ is $(h+4(h_\eta+1))$-deep in $\uC_0$, Lemma \ref{lem:lowerbound-mod-p-crys-stack} shows that $\cX_{\Sym,\F}^{\lam,\tau}$ (resp.~$\til{\cX}_{\Sym,\F}^{\lam,\tau}$) has at least $\# \JH(W(\phi^\mo(\lam)-\eta)
    \otimes \ov{\sig}(\tau))$ irreducible components of dimension $4f$ (resp.~$7f$). On the other hand, the number of $7f$-dimensional irreducible components of $\til{\cX}_{\Sym,\F}^{\le\lam,\tau}$ is at most that of $\til{M}_\cJ^\nv(\le\lam,\nba_{\bfa_\tau})_\F$, which is $\# \JH(W(\phi^\mo(\lam)-\eta)
    \otimes \ov{\sig}(\tau))$ by Theorem \ref{thm:nv-irred-comp-SW-bij}. This proves our claim.
\end{proof}

\begin{thm}\label{thm:mod-p-local-model}
Let $(s,\mu)$ be a $\max\{2h-2,2h_\eta\}$-generic lowest alcove presentation of $\tau$. 
If $\sig$ is a Serre weight such that $\cC_{\sig}\subset \cX_{\Sym}^{\lam,\tau}$, then $\sig \in \JH(W(\phi^\mo(\lam)-\eta)
    \otimes \ov{\sig}(\tau))$ and we have a commutative diagram
    \[
    \begin{tikzcd}
     & & & & \til{C}_\sig^\zeta \arrow[d, hook] \\
    \til{\cC}_\sig \arrow[r, hook]   \arrow[rrrru, bend left=10, "\simeq"] \arrow[d, "T^{\vee,\cJ}_\F"] & \til{\cX}_{\Sym,\F}^{\le\lam,\tau} \arrow[r, hook] \arrow[d, "T^{\vee,\cJ}_\F"]  & \til{M}_\cJ^\nv(\le\lam,\nba_{\bfa_\tau})_\F \arrow[r, hook]& \til{M}_\cJ(\le\lam)_\F \arrow[r, hook, "r_{\tilw^*\PR{\tau}}"] \arrow[d, "T^{\vee,\cJ}_\F"]  & \til{\Fl}_{\cJ,\tilw^*\PR{\tau}}^{[0,h]} \arrow[d, "T^{\vee,\cJ}_\F"] \\
    \cC_\sig \arrow[r, hook] \arrow[rrrrd, bend right=10]&  {\cX}_{\Sym,\F}^{\le\lam,\tau} \arrow[rr, hook]& & Y_{\Sym,\F}^{\le\lam,\tau} \arrow[r, hook]& \BR{\til{\Fl}_{\cJ,\tilw^*\PR{\tau}}^{[0,h]} / T^{\vee,\cJ}_\F \dash \mathrm{conj}}  \arrow[d]\\
     & & & & \Phi\text{-}\mathrm{Mod}_{K,\F}^{\et,\Sym}
    \end{tikzcd}
    \]
    where all rectangles are cartesian, all vertical arrows labeled by $T^{\vee,\cJ}_\F$ are $T^{\vee,\cJ}_\F$-torsors, and all hooked arrows are closed immersions. The bottom diagonal map is the composition of canonical morphisms $\cC_\sig \mono \cX_{\Sym,\red} \xra{\varepsilon_\infty} \Phi\text{-}\mathrm{Mod}_{K,\F}^{\et,\Sym}$. Furthermore, if $(s,\mu)$ is $(h+4(h_\eta+1))$-generic, then the above diagram holds for all $\sig \in \JH(W(\phi^\mo(\lam)-\eta)
    \otimes \ov{\sig}(\tau))$.
\end{thm}
\begin{proof}
By Proposition \ref{prop:diagram-BK-and-et-phi-stacks} and the discussion after \eqref{eqn:diagram-EG-BK-stacks}, we get the diagram except for the top diagonal arrow. The image of $\til{\cC}_\sig$ in $\til{M}_\cJ^\nv(\le\lam,\nba_{\bfa_\tau})_\F$ is a top dimensional irreducible component. By Theorem \ref{thm:nv-irred-comp-SW-bij}, it is equal to $\til{C}_{\sig'}^\zeta \tilw^*(\tau)^\mo$ for some $\sig' \in  \JH(W(\phi^\mo(\lam)-\eta)
    \otimes \ov{\sig}(\tau))$ and we can show that $\sig=\sig'$ following the argument in \cite[Theorem 7.4.2]{LLHLM-2020-localmodelpreprint}. Then the last assertion follows from Theorem \ref{thm:naive-BM}.    
\end{proof}

\begin{prop}\label{prop:geom-SW-EG-stack}
Let $\rhobar$ be a $2h_\eta$-generic tame $L$-parameter over $\F$.  Let $W_g(\rhobar)$ be the set of $3h_\eta$-deep Serre weights $\sigma$ such that $\rhobar \in \cC_\sigma(\F)$. Then, we have $W_g(\rhobar) = W^?(\rhobar|_{I_K})$.
\end{prop}

\begin{proof}
This follows from Corollary \ref{cor:geom-SW-local-model} using Theorem \ref{thm:mod-p-local-model} as in the proof of \cite[Proposition 7.4.7]{LLHLM-2020-localmodelpreprint}. The $3h_\eta$-deepness of $\sig$ is required to show that $\cC_{\sig}\subset \cX^{\eta,\tau}_{\Sym,\F}$ for $2h_\eta$-generic $\tau=\tau(1,\kappa)$  where $\sig=F(\kappa)$.
\end{proof}

\section{Global setup}\label{sec:global}

\subsection{Some local Galois deformation rings}\label{sub:local-def-rings}
Let $\CNL_\cO$ be the category of complete Noetherian local $\cO$-algebras with residue field $\F$.

Let $F$ be either a number field or a local field. 
Let $\rhobar: G_F \ra \GSp_4(\F)$ be a continuous representation. We let $R_{\rhobar}^\square$ be the framed deformation ring representing the functor $D^\square_{\rhobar}$ taking $A\in \CNL_\cO $ to the set of $\GSp_4(A)$-valued lifts $\rho$ of $\rhobar$. If $\psi: G_F \ra \cO^\times$ is a character lifting $\simc(\rhobar)$, we write $R_{\rhobar}^{\square,\psi}$ for the fixed similitude deformation ring representing $D^{\square,\psi}_{\rhobar}$ taking $A\in \CNL_\cO $ to the set of $\GSp_4(A)$-valued lifts $\rho$ with $\simc(\rho)=\psi\otimes_\cO A$.

\subsubsection{Local deformations:  $l=p$}
Suppose that $F=K$. Given a type $(\lam+\eta,\tau)$, we denote by $R_{\rhobar}^{\lam+\eta,\tau}$ the unique $\cO$-flat quotient of $R_{\rhobar}^{\square}$ whose $A$-points, for any $\cO$-flat $A\in \CNL_\cO$, are lattices in potentially crystalline representations with Hodge type $\lam+\eta$ and tame inertial type $\tau$. We have its version with fixed similitude character   $R_{\rhobar}^{\lam+\eta,\tau,\psi}$. Note that $R_{\rhobar}^{\lam+\eta,\tau,\psi}$ is nonzero only if $\psi$ is potentially crystalline of type $(\simc(\lam+\eta),\simc(\tau))$. We record the following Lemma relating a potentially crystalline deformation ring to its fixed similitude variants.

\begin{lemma}\label{lem:fixed-sim-formally-sm}
Recall that $p>2$. Twisting by the universal unramified twist
\begin{align*}
    \ur_{1+X}: G_K \ra \cO\DB{X} 
\end{align*}
sending $\Frob_K$ to $1+X$ induces an isomorphism $R_{\rhobar}^{\lam+\eta,\tau} \simeq R_{\rhobar}^{\lam+\eta,\tau,\psi} \DB{X}$.
\end{lemma}
\begin{proof}
This can be proven as \cite[Lemma 4.3.1]{EG-geom_BM-MR3134019}.
\end{proof}

We also write $R_{\std(\rhobar)}^{\lam+\eta,\tau}$ to denote the potentially crystalline ($\GL_4$-)deformation ring of $\std(\rhobar)$ and type $(\std(\lam+\eta),\std(\tau))$.

We prove that certain potentially crystalline deformation ring of tame $\rhobar:G_K \ra \GSp_4(\F)$ is nonzero and a domain. We have the following necessary condition for $R_{\rhobar}^{\lam+\eta,\tau} \neq 0$ under a mild assumption on $\rhobar$ and $\tau$.

\begin{prop}\label{prop:pcrys-nonzero}
Let $\rhobar:G_K \ra \GSp_4(\F)$ be a continuous representation and $(\lam+\eta,\tau)$ be a type.
\begin{enumerate}
    \item Suppose that $\rhobar$ is semisimple and $m$-generic. If $R_{\rhobar}^{\lam+\eta,\tau}\neq 0$ for a 1-generic tame inertial type $\tau$, then $\tau$ is $(m-(h_\eta+1))$-generic. 
    \item If $\tau$ has a lowest alcove presentation $(s,\mu)$ with $\mu$ $(h_{\lam+\eta}+1)$-deep in $\uC_0$ and $R_{\rhobar}^{\lam+\eta,\tau} \neq 0$, then $\rhobar^\ss$ has a 
    lowest alcove presentation such that $\tilw(\rhobar,\tau)\in \Adm(\phi^\mo(\lam)+\eta)$.
\end{enumerate}
\end{prop}
\begin{proof}
By \cite[Lemma 5]{Enns19}, if $R_{\rhobar}^{\lam+\eta,\tau} \neq 0$, then $R_{\rhobar^\ss}^{\lam+\eta,\tau} \neq 0$. Thus, we may assume that $\rhobar=\rhobar^\ss$ for item (2). If $R_{\rhobar}^{\lam+\eta,\tau} \neq 0$, then $R_{\std(\rhobar)}^{\lam+\eta,\tau}\neq 0$. Then item (1) (resp.~item (2)) follow from the corresponding result for $\GL_4$ \cite[Proposition 3.3.2]{LLHL-Duke-2019MR4007598} (resp.~\cite[Corollary 5.5.8]{LLHLM-2020-localmodelpreprint}) and Lemma \ref{lem:transfer-preserves-bruhat-order}. 
\end{proof}

There is a Coxeter length function $l$ on $\uW_a$ which can be extended to  $\utilW\simeq \uW_a \rtimes \uOm$ by setting $l(\tilw\delta)=l(\tilw)$ for $\tilw\in \uW_a$ and $\delta \in \uOm$. It is expected that if $R_{\rhobar}^{\lam+\eta,\tau} \neq 0$, the complexity of $R_{\rhobar}^{\lam+\eta,\tau}$ increases as the length of $\tilw(\rhobar,\tau)$ decreases. In the special case that $\tilw(\rhobar,\tau)$ has the \emph{maximal} length, i.e.~$\tilw(\rhobar,\tau)\in \uW(\phi^\mo(\lam)+\eta)$, we can compute $R_{\rhobar}^{\lam+\eta,\tau}$ explicitly under a genericity assumption.

\begin{thm}\label{thm:longest-shape-def-ring}
Let $\rhobar:G_K \ra \GSp_4(\F)$ be a tame representation. Let $(\lam+\eta,\tau)$ be a type and $(s,\mu)$ be a $(2h_{\lam+\eta}+1)$-generic lowest alcove presentation of $\tau$. If there is a lowest alcove presentation $(s_{\rhobar},\mu_{\rhobar})$ of $\rhobar|_{I_K}$ such that $\tilw(\rhobar,\tau) \in \uW(\lam+\eta)$, then $R_{\rhobar}^{\lam+\eta,\tau}$ is a power series ring over $\cO$ with $4f+11$ variables. Moreover, any $\rho: G_K \ra \GSp_4(\cO)$ of type $(\lam+\eta,\tau)$ lifting $\rhobar$ is potentially diagonalizable.
\end{thm}

\begin{proof}
We briefly explain the changes needed for the proof of \cite[Theorem 3.4.1]{LLHL-Duke-2019MR4007598}. As in \loccit~(cf.~the diagram (3.16)), we compute a related deformation ring for a Breuil--Kisin module $(\ov{\fM},\ov{\fN},\al')$ such that
\begin{align*}(\rhobar|_{G_{K_\infty}},\simc(\rhobar)|_{G_{K_\infty}},\al) \simeq T_{dd}^*((\ov{\fM},\ov{\fN},\al')).
\end{align*}
We first consider deformations of $\ov{\fM}$. As in Proposition 3.4.8 in \loccit, the universal partial Frobenius matrix $A^{(j),\univ}$ at $\jj$ of the universal deformation of $\ov{\fM}$ over $R_{\ov{\fM}}^{\le \lam+\eta',\tau,\ov{\be}}$ can be factorized into $T^{(j),\univ} U^{(j),\univ}$ where $T^{(j),\univ}$ and $U^{(j),\univ}$ are a diagonal matrix and a lower triangular matrix (after conjugation). The entries of $U^{(j),\univ}$ are in $R_{\ov{\fM}}^{\le \lam+\eta',\tau,\ov{\be}}[v]$. Then Proposition 3.4.12 in \loccit~computes the ring  $R_{\ov{\fM}}^{\le \lam+\eta',\tau,\ov{\be},\nba}$ as a quotient of $R_{\ov{\fM}}^{\le \lam+\eta',\tau,\ov{\be}}$ obtained by imposing equations solving the non-leading coefficients of $U^{(j),\univ}_{\al}$ in terms of the leading coefficient of $U^{(j),\univ}_{\al'}$ for $0>\al'>\al$ (this is where we need $(2h_{\lam+\eta}+1)$-genericity of $\tau$). As a result, $R_{\ov{\fM}}^{\le \lam+\eta',\tau,\ov{\be},\nba}$ is a power series ring generated by $10f$ variables, $4$ coming from $T^{(j),\univ}$ and $6=\#\Phi^-_4$ from $U^{(j),\univ,\nba}$ for each $\jj$.

The universal deformation ring $R_{(\ov{\fM},\ov{\fN},\al')}^{\le\lam+\eta,\tau,\ov{\be},\nba}$ of $(\ov{\fM},\ov{\fN},\al')$ is obtained by imposing the condition
\begin{align*}
    T^{(j),\univ}, U^{(j),\univ, \nba} \in \GSp_4(R_{\ov{\fM}}^{\le \lam+\eta',\tau,\ov{\be},\nba}).
\end{align*}
Thus, it is a quotient of $R_{\ov{\fM}}^{\le \lam+\eta',\tau,\ov{\be},\nba}$ obtained by imposing $t_1t_4=t_2t_3$ where $t_1,\dots,t_4$ are the entries of $T^{(j),\univ}$ and $U^{(j),\univ, \nba}_{\al}\equiv \pm U^{(j),\univ, \nba}_{\al'} \mod U^{(j),\univ, \nba}_{-\al_{23}}$ with an appropriate sign for $\al,\al'\in \Phi^{-}_4$ such that the images of $\al^\vee,(\al')^\vee$ under the restriction map $\Phi^-_4 \epi \Phi^-$ are the same (see \cite[pg.~28]{RS-GSp4-bookMR2344630} for the structure of the unipotent subgroup $U$ of $\GSp_4$). Since it is of relative dimension $7f$ over $\cO$, it is a formal power series ring over $\cO$ with $7f$ generators. This implies the assertion on $R_{\rhobar}^{\lam+\eta,\tau}$. 
The potential diagonalizability can be proven using Example \ref{ex:pdreps}. 
\end{proof}

We also record a lifting result for certain  non-tame $\rhobar$.

\begin{lemma}\label{lem:max-nonsplit-rhobar-lifting}
Let $\kappa\in X^*_1(\uT)$ be 0-deep.  Let $\rhobar: G_K \ra \GSp_4(\F)$ be maximally non-split of niveau 1 and weight $\sig=F(\kappa)$. Let $a$ be an integer and let $k$ be the unique integer such that $a- 2k=:b\in \{3,4\}$. Let $(\lam_a+\eta,\tau_a)$ be a type such that 
\begin{align*}
    \lam_a = \left\{ \begin{array}{cc}
        k(1,1;2) &  \text{if $b=3$}\\
        (1,1;1)+ k(1,1;2) & \text{if $b=4$}
    \end{array} \right.
\end{align*}
and $\tau_a = \tau(1,\kappa-\phi^\mo(\lam_a))$. 
Then there is a potentially crystalline lift  $\rho:G_K \ra \GSp_4(\cO)$ of $\rhobar$ with Hodge type $\lam_a+\eta$ and tame inertial type $\tau_a$ and of the form
\begin{align*}
    \rho = \pma{\chi_1 & * & * & * \\ 
    0 & \chi_2 & * & * \\ 
    0 & 0 & \chi_3 & * \\
    0 & 0& 0 & \chi_4}
\end{align*}
where $\oplus_{i=1}^4\chi_i = \eps_p^{\lam_a+\eta} \prod_{\jj} {\omega}_{K,\sig_j}^{\phi(\kappa_j)-\lam_{a,j} }$. 
\end{lemma}

\begin{proof}
 It is clear that $\chi_i$ lifts $\ov{\chi}_i$. Since $\kappa$ is 0-deep, for $i<j$, $\ov{\chi}_i \ov{\chi}_j^\mo \neq \ov{\eps}_p$. Thus, $\chi_i \chi_j^\mo \neq \eps_p$. Then the existence of a lift $\rho$ follows from the vanishing of $H^2(G_K, \chi_i \chi_j^\mo)$. It is potentially crystalline by \cite[Lemma 6.3.1]{EGstack}. To take $\rho$ valued in $\GSp_4$, we must match the extension classes at the $(1,2)$ and $(1,3)$ matrix entries with those at the $(3,4)$ and $(2,4)$ matrix entries (as in the proof of Lemma \ref{lem:dim-ext-family-Klingen}). Again, this can be done because $H^2(G_K, \chi_i \chi_j^\mo)=0$. 
\end{proof}

\subsubsection{Local deformations: $l\neq p$}\label{subsec:lneqp}
We record deformation problems for Ihara avoidance argument. Let $l$ be a prime and $l\neq p$. Let $F/\Q_l$ be a finite extension with the ring of integers $\cO_F$, a uniformizer $\varpi_F$, and the residue field $k_F$ of size $q_F$. We assume that $q_F\equiv 1 \mod p$. Let $\rhobar: G_F \ra \GSp_4(\F)$ be a trivial representation and $\psi:G_F \ra \cO^\times$ be a continuous character trivial modulo $\varpi$.

Let $\zeta=(\zeta_1,\zeta_2)$ be a pair of continuous characters $\zeta_i: \cO_F^\times \ra \cO^\times$ that are trivial modulo $\varpi$. We let $D_{\rhobar}^\zeta$ be the functor  taking $A\in \CNL_\cO$ to the set of $A$-valued lifts $\rho: G_F \ra \GSp_4(A)$ such that for any $\sigma\in I_{F}$, the characteristic polynomial of $\rho(\sig)$ is
\begin{align*}
    (X - \zeta_1(\Art_F^\mo(\sigma)))(X - \zeta_2(\Art_F^\mo(\sigma)))(X - \zeta_2(\Art_F^\mo(\sigma)))^\mo(X - \zeta_1(\Art_F^\mo(\sigma)))^\mo.
\end{align*}
Then $D_{\rhobar}^\zeta$ is a local deformation problem. We let $R^\zeta_{\rhobar}\in \CNL_\cO$ be an object representing $D_{\rhobar}^\zeta$. We record the following results on $R^\zeta_{\rhobar}$.

\begin{prop}[Proposition 7.4.7 and 7.4.8 in \cite{BCGP-ab_surf-2018abelian}]\label{prop:Ihara-avoid-def}
\begin{enumerate}
    \item Suppose that $\zeta=1$ is the pair of trivial characters. Then $\spec R^1_{\rhobar}$ is equidimensional of dimension 11 and every generic point has characteristic zero. Moreover, every generic point of $\spec R^1_{\rhobar}/\varpi$ is the specialization of a unique generic point of $\spec R^1_{\rhobar}$.
    \item Suppose that $\zeta_1,\zeta_2\neq 1$ and $\zeta_1 \neq \zeta_2^{\pm 1}$. Then $\spec R^{\zeta}_{\rhobar}$ is irreducible of dimension 11, and its generic point has characteristic zero.
\end{enumerate}
\end{prop}

\subsection{Weak patching functors}\label{subsec:patching-functors}
Recall the finite \'etale $\Z_p$-algebra $\cO_p$. It can be written as a finite product of finite \'etale local $\Z_p$-algebras $\prod_{v\in S_p}\cO_v$. We write $F_p = \cO_p[1/p] \simeq \prod_{v\in S_p}F_v$. Let $\cJ = \Hom_{\Z_p}(\cO_p, \cO)$.

Let $\rhobar:G_{\Q_p} \ra \LuG(\F)$ be a tame $L$-parameter. We consider $\psip:G_{\Q_p} \ra \prescript{L}{}{\underline{\G_m}(\cO)}$  lifting $\simc(\rhobar)$. Note that $\psip$ is equivalent to a collection $\CB{\psi_v : G_{F_v} \ra \cO^\times}_{v\in S_p}$. We define a set of $\psip$ that is relevant to our applications in \S\ref{sec:applications}.

\begin{defn}
\begin{enumerate}
    \item We define $\Psi(\rhobar)$ to be the set of $\psi_p: G_{\Qp} \ra \prescript{L}{}{\underline{\G_m}(\cO)}$ lifting $\simc(\rhobar)$ and $\psi_v\eps_p^{-w}$ has finite image and is tamely ramified for some integer $w$ and for all $v \in S_p$. 
    \item We say a type $(\lam+\eta,\tau)$ is \textit{compatible with $\psi_p$} if $\simc(\lam_j+\eta_j)=w$ for all $\jj$ and 
    \begin{align*}
        \simc(\tau)\eps_p^w|_{I_{\Qp}} = \psi_p|_{I_{\Qp}}.
    \end{align*}
\end{enumerate}
\end{defn}

\begin{rmk}
In applications, we take $\psi_v$ to be a restriction of a global character $\psi: G_F \ra \cO^\times$ for a totally real field $F$. In the above definition, the constancy of Hodge type and the finiteness of the image of $\psi_v\eps_p^{-w}$ are necessary because any $\psi$ satisfies these conditions (which essentially follow from the compactness of the norm-one subgroup of $ \A_F^\times/F^\times$ and \cite[Theorem 2.43]{GeeAWS}). Also, note that given a tame $L$-parameter $\rhobar$ and a type $(\lam+\eta,\tau)$, $\psip\in \Psi(\rhobar)$ compatible with $(\lam+\eta,\tau)$ always exists and is determined up to twisting by an unramified character of $p$-power order.
\end{rmk}

Let $\psi_p\in \Psi(\rhobar)$. We define 
\begin{align*}
    R_{\rhobar}^{\psi_p}:= \ctimes_{v\in S_p, \cO} R_{\rhobar_v}^{\square,\psi_v}, \ \ R_\infty^{\psi_p}:= R_{\rhobar}^{\psi_p}\ctimes_\cO R^p
\end{align*}
where  
$R^p$ is a complete Noetherian equidimensional flat $\cO$-algebra. For a type $(\lam+\eta,\tau)$ compatible with $\psi_p$, we define
\begin{align*}
    R_{\rhobar}^{\lam+\eta,\tau,\psi_p} := \ctimes_{v\in S_p,\cO} R_{\rhobar_v}^{\lam_v+\eta_v,\tau_v,\psi_v}, \ \ R_{\infty}^{\lam+\eta,\tau,\psi_p}:= R_\infty^{\psi_p} \otimes_{R_{\rhobar}^{\psi_p}} R_{\rhobar}^{\lam+\eta,\tau,\psi_p}.
\end{align*}

Let $\Mod(R_{\infty}^{\psi_p})$ be the category of finitely generated modules over $R_{\infty}^{\psi_p}$ and $\Rep_\cO^{\psi_p}(\GSp_4(\cO_p))$ be the category of topological $\cO[\GSp_4(\cO_p)]$-modules, which are finitely generated over $\cO$, with fixed central character given by $\otimes_{v\in S_p}(\psi_v\eps_p^{-3}|_{I_{F_v}})\circ (\Art_{F_v}|_{\cO_{F_v}})$. Let $\Rep_\F^{\psi_p}(\GSp_4(\cO_p))$ be the full subcategory of $\Rep_\cO^{\psi_p}(\GSp_4(\cO_p))$ consisting of the objects whose underlying modules are $\F$-vector spaces.

\begin{defn}\label{def:patching-functors}
A \emph{(fixed similitude) weak patching functor} for $(\rhobar,\psip)$ is 
a nonzero covariant exact functor
\begin{align*}
    M_\infty^{\psi_p}: \Rep^{\psi_p}_\cO(\GSp_4(\cO_p)) \ra \Mod(R_\infty^{\psi_p})
\end{align*}
satisfying the following conditions: 
\begin{enumerate}[(1)]
    \item for any type $(\lam+\eta,\tau)$ compatible with $\psi_p$ and a  $\GSp_4(\cO_p)$-stable $\cO$-lattice  $\sigma^\circ(\lam,\tau)$ in $\sig(\lam,\tau)$, $M_\infty^\psip(\sigma^\circ(\lam,\tau))$ is a maximal Cohen--Macaulay module over  $R_\infty^{\lam+\eta,\tau,\psip}$ if it is nonzero; and 
    \item for all $\sigma \in \JH(\ov{\sigma}(\lam,\tau))$, $M_\infty^\psip(\sigma)$ is a maximal Cohen--Macaulay module over $R_\infty^{\lam+\eta,\tau,\psip}/\varpi$ if it is nonzero.
\end{enumerate}
Furthermore, we say that
\begin{enumerate}[(i)]
    \item $M_\infty^{\psi_p}$ is \emph{minimal} if $R^p$ is formally smooth over $\cO$ and $M_\infty^{\psi_p}(\sig^\circ(\lam,\tau))[p^\mo]$ is locally free of rank at most one over $R_\infty^{\lam+\eta,\tau,\psip}[p^\mo]$; and
    \item $M_\infty^{\psi_p}$ is \emph{potentially diagonalizable} if $M_\infty^\psip(\sigma^\circ(\lam,\tau))$ is nonzero for all $(\lam+\eta,\tau)$ such that $\rhobar_v$ has a potentially diagonalizable lift of type $(\lam_v+\eta_v,\tau_v)$ for each $v\in S_p$. 
\end{enumerate}
\end{defn}

\subsection{Algebraic automorphic forms}\label{subsec:alg-aut-forms} In this subsection, we recall the global setup in \cite[\S4]{EL}. Let $F$ be a totally real field.
Suppose that $[F:\Q]$ is even. We also assume that $p$ is unramified in $F$. Let  $D$ be a quaternion algebra over $F$ ramified at all infinite places and split at all finite places. Such a $D$ exists because $[F:\Q]$ is even. Choose a maximal order $\cO_D$ of $D$. Then $\cO_D$ is stable under the canonical involution of $D$. We define $\cG$ to be the $\cO_F$-group given by
\begin{align*}
    \cG(R) = \CB{ g\in \GL_2(R\otimes_{\cO_F}\cO_D) \mid g^*g = \mu(g)I \ \text{for some $\mu(g)\in (R\otimes_{\cO_F}\cO_D)^\times$}}
\end{align*}
for any $\cO_F$-algebra $R$. Here, $*$ denotes the composition of transpose and the canonical involution of $D$.  Then $\cG$ is an inner form of $\GSp_4$. The center $Z_\cG$ is isomorphic to $\G_m$, and $\cG$ is compact modulo center at infinity. For each finite place $v$, we fix an isomorphism $\cO_{D,v} \simeq M_2(\cO_{F_v})$ which induces an isomorphism $\iota_v : \cG_{/\cO_{F_v}} \ra \GSp_{4/ \cO_{F_v}}$.

Let $\chi: \A_F^\times/F^\times \ra \C^\times$ be a Hecke character. We write $\chi_{p,\iota} = \iota^\mo \circ \chi$. Let $U = U_p U^{\infty,p} \le \cG(\cO_p) \times \cG(\A^{\infty,p}_F)$ be a compact open subgroup. For a finite place $v$ of $F$, we write $\Iw(v)$ (resp.~$\Iw_1(v)$) for the Iwahori subgroup (resp.~pro-$p$ Iwahori subgroup) of $\cG(F_v)\simeq \GSp_4(F_v)$. Let $W$ be an $\cO$-module with a continuous action of $U_p$. We define $S_\chi(U,W)$ to be the $\cO$-module of functions $f: \cG(F)\bss \cG(\A_F^{\infty}) \ra W$ such that
\begin{align*}
    f(zgu)= \chi_{p,\iota}(z) u_p^\mo f(g) \ \ \  \forall g\in \cG(\A^\infty_{F}), u\in U, z\in Z_\cG(\A^\infty_F).
\end{align*}

Let $U\le \cG(\A_F^{\infty})$ be a compact open subgroup. We assume that $U$ is \textit{sufficiently small} in the sense that it contains no element  of order $p$. For our local applications, we will fix a  finite place $v_0$ of residue characteristic $>5$ such that $q_{v_0}\not\equiv 1 \mod p$. We assume that $U_{v_0}= \Iw_1(v_0)$. In this case, the choice of $U_{v_0}$ ensures that $U$ is sufficiently small. Let $P_U$ be the finite set of finite places of $F$ at which $U$ is ramified. Let $S_p$ be the set of places of $F$ dividing $p$. For a finite set $P$, we define the   \emph{universal Hecke algebra} $\T^{P,\univ}$ to be the polynomial ring over $\cO$ generated by $S_v, T_{v,1}, T_{v,2}$ for each $v \notin P$ and for $v=v_0$ if $v_0 \in P$. When $P$ contains  $S_p\cup P_U$, there is a natural $\T^{P,\univ}$-action on $S_\chi(U,W)$ where   $S_v, T_{v,1}, T_{v,2}$ act through the double coset operators
\begin{align*}\small{
    \BR{U_v\pma{\varpi_v & & & \\
    & \varpi_v & & \\
    & & \varpi_v & \\
    & & & \varpi_v }U_v}, \ \ \BR{U_v\pma{\varpi_v & & & \\
    & \varpi_v & & \\
    & & 1 & \\
    & & & 1 }U_v}, \ \ \BR{U_v\pma{\varpi_v^2 & & & \\
    & \varpi_v & & \\
    & & \varpi_v & \\
    & & & 1 }U_v}}
\end{align*}
respectively. We denote the image of $\T^{P,\univ}$ in $\End_{\cO}(S_{\chi}(U,W))$ by $\T^P_{\chi}(U,W)$.

Let $\psi:= \chi_{p,\iota}\eps_p^{3}$. Let $\ov{r}: G_F \ra \GSp_4(\F)$ be an absolutely irreducible continuous representation and $\simc(\ov{r}) = {\psi} \mod \varpi$. When we have a fixed place $v_0$, we assume that $\rbar|_{G_{F_{v_0}}}$ is a sum of unramified characters with no two eigenvalues of $\Frob_{v_0}$ have ratio $q_{v_0}$. In applications, we can always choose such $v_0$ (\cite[\S 7.7]{BCGP-ab_surf-2018abelian}). We write
\begin{align*}
    \rbar|_{G_{F_{v_0}}} \simeq \ur_{c_1c_2 c_0}  \oplus \ur_{c_1c_0} \oplus \ur_{c_2c_0}  \oplus  \ur_{c_0}
\end{align*}
for $c_1,c_2,c_0 \in \F^\times$. We denote by $P_{\rbar}$ the set of finite places either dividing $p$ or at which $\ov{r}$ is ramified. For any finite set $P\supset P_{\rbar}$, we define a maximal ideal $\fm_{\rbar,\chi}^P \le \T^{P,\univ}$ with residue field $\F$ by demanding that
\begin{enumerate}
    \item for each $v \notin P$,
\begin{align*}
    S_v \mod \fm_{\rbar,\chi}^P = \ov{\psi}(\Frob_v)
\end{align*}
and the characteristic polynomial of $\rbar(\Frob_v)$ in $\F[X]$ is given by
\begin{equation*}
X^4-T_{v,1}X^3+(q_vT_{v,2}+(q_v^3+q_v)S_v)X^2-q_v^3T_{v,1}S_vX+q_v^6S_v^2 \mod \fm_{\rbar,\chi}^P,
\end{equation*}
\item for $v=v_0$ (if $v_0 \in P$), 
\begin{align*}
    S_v \mod \fm_{\rbar,\chi}^P = c_1c_2c_0^2, \ T_{v,1} \mod \fm_{\rbar,\chi}^P = c_1c_2c_0, \ \ T_{v,2} \mod \fm_{\rbar,\chi}^P = c_1^2c_2 c_0^2
\end{align*}
\end{enumerate}
(cf. \cite[\S2.4.7]{BCGP-ab_surf-2018abelian}; note $\fm_{\rbar,\chi}^P$ is well-defined by symplecticity of $\rbar$).

\begin{defn}
\begin{enumerate}

\item We say that a pair $(\rbar,\chi)$ as above is \emph{automorphic of weight $\mu\in (X^*(T)^+)^{S_p}$ and level $U$} if there is a finite set of finite places $P$ containing $P_U \cup P_{\rbar}$ such that $S_\chi(U,V(\mu)^\vee)_{\fm_{\rbar,\chi}^P}\neq 0$.
\item Let $\sigma$ be a Serre weight of $G_0(\F_p)$. We say that $\rbar$ is \emph{modular of weight $\sigma$} (and level $U$), or equivalently, $\sig$ is a \emph{modular weight} of $\rbar$ (at level $U$) if $S_\chi(U,\sigma^\vee)_{\fm_{\rbar,\chi}^P}\neq 0$ for some $\chi$. Note that $S_\chi(U,\sigma^\vee)_{\fm_{\rbar,\chi}^P}$ does not depend on the choice of $\chi$ as long as $\chi_{p,\iota} \eps_p^{3} \mod \varpi = \simc(\rbar)$.
\item We let $W(\rbar)$ be the set of modular Serre weights of $\rbar$. 
\end{enumerate}
\end{defn}

\begin{rmk}\label{rmk:sim-char-aut}
    We remark that $(\rbar,\chi)$ is automorphic for some $\chi$ if and only if $(\rbar,\chi)$ is automorphic for any $\chi$ such that $\chi_{p,\iota}\eps_p^3 \mod \varpi =\simc (\rbar)$. Equivalently, $\rbar$ is a mod $p$ reduction of the Galois representation attached to a regular algebraic cuspidal automorphic representation of $\cG(\A_F)$ (equivalently, of $\GSp_4(\A_F)$); see \cite[Remark 4.2.4]{EL}.
\end{rmk}

Let $(\rbar,\chi)$ be automorphic of weight $\mu$ and level $U$ and $P$ be as above. Suppose that there is a subset $R\subset P$ disjoint from $S_p$ such that for $v\in R$, $U_v = \Iw_1(v)$, $q_v \equiv 1 \mod p$, and $\rbar|_{G_{F_v}}$ is trivial. For each $v\in R$, we choose a pair of characters $\zeta_v = (\zeta_{v,1}, \zeta_{v,2})$ such that
\begin{enumerate}
    \item for $i=1,2$, $\zeta_{v,i}: \cO^\times_{v} \ra \cO^\times$ is a continuous character trivial modulo $\varpi$;
    \item either $\zeta_{v,1}=\zeta_{v,2}=1$, or $\zeta_{v,1},\zeta_{v,2}\neq 1$ and $\zeta_{v,1}\neq\zeta_{v,2}^{\pm 1}$.
\end{enumerate}
Let $T_\der := \ker (T \xra{\simc} \G_m)$. The projection to the first two diagonal entries induces an isomorphism  $T_\der \simeq \G_m^2$.  We write $\zeta_R$ for the character
\begin{align*}
    \zeta_R=\prod_{v\in R} [\ov{\zeta}_v]: \prod_{v\in R} T_{\der}(k_v) \ra \cO^\times.
\end{align*}
We let $\T^{P}_{\chi}(U,V(\mu))_{\zeta_R}$ to be the image of $\T^{{P},\univ}$ in $\End_{\cO}(S_{\chi}(U,V(\mu))_{\zeta_R})$ where the subscript $\zeta_R$ denotes taking the $\zeta_R$-coinvariant for the action of $\Pi_{v\in R} T_{\der}(k_v)$.

\begin{prop}\label{prop:gal-rep-valued-in-Hecke}
Keep the above notations and assumptions. We further assume that $\rbar: G_F \ra \GSp_4(\F)$ is absolutely irreducible. Then there exists a unique continuous representation
\begin{align*}
    r^{P}_{\chi,\mu}(U): G_F \ra \GSp_4(\T^{P}_{\chi}(U,V(\mu))_{\zeta_R,\fm_{\rbar,\chi}^P})
\end{align*}
such that
\begin{enumerate}
    \item $r^{P}_{\chi,\mu}(U)$ lifts $\rbar$;
    \item $\simc(r^{P}_{\chi,\mu}(U(Q)))=\psi$;
    \item if $v\notin P$, then $r^{P}_{\chi,\mu}(U)$ is unramified at $v$ and the characteristic polynomial of $r^{P}_{\chi,\mu}(U)(\Frob_v)$ in $\T^{P}_{\chi}(U,V(\mu))_{\zeta_R,\fm_{\rbar,\chi}^P}[X]$ is equal to
    \begin{align*}
    X^4-T_{v,1}X^3+(q_vT_{v,2}+(q_v^3+q_v)\chi_v(\varpi_v))X^2-q_v^3T_{v,1}\chi_v(\varpi_v) X+q_v^6 \chi_v(\varpi_v)^2,
    \end{align*}
    \item if $v\in R$, then the $\T^{P}_{\chi}(U,V(\mu))_{\zeta_R,\fm_{\rbar,\chi}^P}$-point of $R_{\rbar|_{G_{F_v}}}^\square$ induced by $r^{P}_{\chi,\mu}(U)|_{G_{F_v}}$ factors through $R_{\rbar|_{G_{F_v}}}^{\zeta_v}$;
    \item if $v_0 \in P$ and $v= v_0$, then
    \begin{align*}
        r^P_{\chi,\mu}(U)|_{G_{F_{v}}} \simeq\ur_{T_{v,1}} \oplus \ur_{T_{v,1}/T_{v,2}}  \oplus  \ur_{S_v T_{v,1}/T_{v,2}} \oplus \ur_{S_v/T_{v,1}};
    \end{align*}
    \item and for every $\cO$-algebra homomorphism $f:\T^{P}_{\chi}(U,V(\mu))_{\zeta_R,\fm_{\rbar,\chi}^P}\ra E'$ where $E'$ is a finite extension of $E$, the representation $f\circ r^P_{\chi,\mu}(U)|_{G_{F_v}}$ is de Rham of Hodge type $\mu_v+\eta_v$ for all $v\in S_p$. 
\end{enumerate}
\end{prop}
\begin{proof}
This is proven in \cite[Proposition 4.2.6]{EL} except item (4) and (5). By the local-global compatibility in \cite[Theorem 3.5]{Mok-Comp-2014-MR3200667}, item (4) follows from  \cite[Proposition 2.4.28]{BCGP-ab_surf-2018abelian} and item (5) follows from Proposition 2.4.3, 2.4.4, and 7.4.2 in \loccit.
\end{proof}

\begin{rmk}\label{rmk:exact}
    For any type $(\lam+\eta,\tau)$, if $\JH(\sig(\lam,\tau))$ contains a modular weight of $\rbar$, then $S_{\chi}(U,\sig^\circ(\lam,\tau)^\vee)_{\fm_{\rbar,\chi}^P}\neq 0$ for some $U$ and $\chi$. Then $\rbar_p$ admits a potentially crystalline lift of type $(\lam+\eta,\tau)$ by Proposition \ref{prop:gal-rep-valued-in-Hecke}.
\end{rmk}

\begin{thm}\label{thm:weight-elim}
Let $\rbar: G_F \ra \GSp_4(\F)$, $\chi$, and $\psi$ be as in Theorem \ref{thm:existence-patching-functor}(1). Let $\sig \in W(\rbar)$. We suppose that either $(\rbar|_{G_{F_v}})^\ss$ is $(6(h_\eta+1)-2)$-generic for each $v|p$ or $\sigma$ is $(2h_\eta+1)$-deep and $(\rbar|_{G_{F_v}})^\ss$ is $4(h_\eta+1)$-generic for each $v|p$. Then $\sigma \in W^?(\rbar^\ss_p)$.
\end{thm}
\begin{proof}
By Remark \ref{rmk:sim-char-aut} and \cite[Lemma 4.4.3]{EL}, we can and do assume that $\psi_v \eps^{-3}$ is finite for $v\in S_p$. By Remark \ref{rmk:exact}, for a type $(\lam+\eta,\tau)$ such that $\JH(\overline{\sig}(\tau))$ contains a modular weight, $\rbar_p$ has a potentially crystalline lift of type  $(\lam+\eta,\tau)$.

Suppose that $\lam$ is $3(h_\eta+1)$-deep. For each $s\in \uW$,  $\tau_s :=\tau (s,\tilw_h\cdot \lam+\eta)$ is $(2h_\eta+3)$-generic and $\sig\in \JH(\osig(\tau_s))$. Thus, $\rbar_p$ admits a potentially crystalline lift of type $(\eta,\tau_s)$.  Following the proof of \cite[Corollary 4.1.12]{LLHL-Duke-2019MR4007598}, we can show that $\sig\in W^?(\rbar_p^\ss)$ using Lemma 4.1.10 in \loccit~(easily generalized to our setting) and Proposition \ref{prop:pcrys-nonzero}(2).

Suppose that $\lam$ is not $3(h_\eta+1)$-deep but $(2(h_\eta+1)-1)$-deep. Then the tame type $\tau_e=\tau (e,\tilw_h\cdot \lam+\eta)$ is $(h_\eta+1)$-generic but not $4(h_\eta+1)$-generic by \cite[Proposition 2.2.16]{LLHL-Duke-2019MR4007598} and its proof. We also have $F(\lam)\in \JH(\ov{\sig}(\tau_e))$. Then $\std(\rbar_p)$ does not admit a potentially crystalline lift of type $(\eta',\std(\tau_e))$ by Proposition \ref{prop:pcrys-nonzero}(1). This contradicts the first paragraph of this proof.

Finally, suppose that $\lam$ is not $(2(h_\eta+1)-1)$-deep and $(\rbar|_{G_{F_v}})^\ss$ is $(6(h_\eta+1)-2)$-generic for each $v|p$. For $v|p$,  $\std(\rbar|_{G_{F_v}})$ does not have a potentially crystalline lift of type $(\eta',\std(\tau(e,\lam)))$ by \cite[Theorem 8]{Enns19} as explained in the proof of \cite[Corollary 4.2.4]{LLHL-Duke-2019MR4007598}. Thus, $\rbar_p$ does not have a potentially crystalline lift of type $(\eta,\tau(e,\lam))$. This again contradicts the first paragraph of this proof.
\end{proof}

\subsection{Patching argument}\label{sub:patching-argument}
In this subsection, we prove the existence of weak patching functors in the global case (i.e.~$\cO_p= \cO_F\otimes_{\Z}\Zp)$ and in the local case (i.e.~$\cO_p=\cO_K$).

\begin{defn}\label{def:odd-TW-conditions}
\begin{enumerate}
    \item We say that a continuous representation $\rbar: G_F \ra \GSp_4(\F)$ is \emph{odd} if for each infinite place $v$ and corresponding choice of complex conjugation $c_v \in G_F$, $\simc(\rbar)(c_v)=-1$.
    
    \item We say that a continuous representation $\rbar: G_F \ra \GSp_4(\F)$  \emph{satisfies Taylor--Wiles conditions} if $\rbar$ is absolutely irreducible, odd, vast, and tidy (in the sense of \cite[\S7.5]{BCGP-ab_surf-2018abelian}).
\end{enumerate}
\end{defn}

\begin{thm}\label{thm:existence-patching-functor}
\begin{enumerate}
    \item Let $\rbar: G_F \ra \GSp_4(\F)$ be a continuous representation. Suppose that $\rbar$ satisfies Taylor--Wiles conditions and $(\rbar,\chi)$ is automorphic (of some Hecke character $\chi$, weight $\mu'$, and level $U'$). 
    Let 
    $\psi:= \chi_{p,\iota}\eps_p^3$. 
    Then for $\psip:=\{\psi|_{G_{F_v}}\}_{v\in S_p}$, there exists a weak patching functor for $(\rbar_p,\psip)$. Moreover, for a Serre weight $\sig$ of $\GSp_4(\cO_F/p)$, $\sig\in W(\rbar)$ if and only if $M_\infty^{\psip}(\sig)\neq 0$.

\item Let $\rhobar: G_K \ra \GSp_4(\F)$ be a $16$-generic continuous representation. Suppose that $\rhobar$ is either tame or maximally non-split of niveau 1 and weight $\sig$. Then for each $\psip\in \Psi(\rhobar)$, there exists a potentially diagonalizable weak patching functor for $(\rhobar,\psip)$. If $\rhobar$ is tame, it can be chosen to be minimal. Moreover, $M^{\psip}_\infty$ is independent of $\psip$ when restricted to $\Rep_\F^{\psi_w}(\GSp_4(\cO_K))$.
\end{enumerate}
\end{thm}

\begin{rmk}\label{rmk:patching-nonzero}
We apply part (2) of Theorem \ref{thm:existence-patching-functor} to prove a version of Breuil--M\'ezard conjecture in \S\ref{sub:BM}. The part (1) will be used to prove Serre weight conjectures (which also requires results in \S\ref{sub:BM}) and a modularity lifting result in \S\ref{sub:SWC} and \ref{sub:modularity-lifting}. 
\end{rmk}

We discuss the structure of the remaining subsection. We first provide a general construction of a weak patching functor by extending the construction in \cite[\S4.4]{EL}, which is based on \cite{6author} and \cite{BCGP-ab_surf-2018abelian}.  Theorem \ref{thm:existence-patching-functor}(1) will follow from this immediately. Theorem \ref{thm:existence-patching-functor}(2) is proven using the general construction together with a globalization of $\rhobar$ and Theorem \ref{thm:weight-elim}. We finish the subsection with the proof of Lemma \ref{lem:pcrys-stack-nonempty} and \ref{lem:lowerbound-mod-p-crys-stack}, thus completing the proof of Theorem \ref{thm:pcrys-local-model} and \ref{thm:mod-p-local-model}.

\subsubsection{Construction of weak patching functors}\label{subsubsec:construction-of-patching-functors}

Let $F$, $\rbar$, and $\chi$ be as in Theorem \ref{thm:existence-patching-functor}(1). 
Let  $S_p$ be the set of places of $F$ dividing $p$ and $S$ be a finite set of finite places containing $S_p$. We define
\begin{align*}
    q = h^1(F_S/F, \ad(\rbar)(1)), \ \ \  g= 2q-4[F:\Q_p] + \abs{S} -1.
\end{align*}
For $T\subset S$, we define $\cT_T:= \cO\DB{y_1, \dots, y_{11\abs{T}-1}}$. We also define $S_\infty:= \cT_S\DB{\Z_p^{2q}}$. 
Viewing $S_\infty$ as an augmented $\cO$-algebra, we let $\fa_\infty$ denote the augmentation ideal of $S_\infty$.

Suppose that $S$ contains $S_{\rbar}\cup S_U$. Let $R$ be a subset of $S$ disjoint from $S_p$. We assume that for each $v\in R$, $U_v = \Iw_1(v)$, $q_v \equiv 1 \mod p$, and $\rbar|_{G_{F_v}}$ is trivial. The presence of $R$ is necessary for the ``Ihara avoidance'' argument used in the proof of Theorem \ref{thm:modularity-lifting} (see Lemma \ref{lem:modularity-lifting-lemma}).    
For each $v\in R$, we choose a pair of characters $\zeta_v = (\zeta_{v,1}, \zeta_{v,2})$ as in \S\ref{subsec:lneqp} which induce a character $\zeta_R : \prod_{v\in R} T_{\der}(k_v) \ra \cO^\times$.

For each $v\in S$, let $\cD_v \subset \cD_v^{\square,\psi|_{G_{F_v}}}$ be a local deformation problem represented by  $R_v\in \CNL_\cO$. If $v\in R$, we take $\cD_v = \cD_v^{\zeta_v}$ defined in \S\ref{subsec:lneqp}. We consider a global deformation problem
\begin{align*}
    \cS=(S,\{\cD_v\}_{v\in S}, \psi).
\end{align*}
If $T$ is a subset of $S$, we write $\cD_\cS^T$ for the functor of $T$-framed deformations of type $\cS$. Since $\rbar$ is absolutely irreducible, $\cD_\cS$ (resp.~$\cD_\cS^T$) is represented by $R_\cS$ (resp.~$R^{T}_\cS$) in $\CNL_\cO$. The choice of a universal lift $r_{\cS} :G_F \ra \GSp_4(R_\cS)$ gives an isomorphism $R_\cS\otimes_\cO \cT_T \simeq R_{\cS}^T$. Let $R_{\cS}^{T,\loc}:=\ctimes_{v\in T} R_v$. Then there exists a natural map $R_{\cS}^{T,\loc} \ra R_{\cS}^{T}$.

Given a Taylor--Wiles datum $(Q,(\ov{\al}_{v,1}, \dots, \ov{\al}_{v,4})_{v\in Q})$ (\cite[Definition 4.4.6]{EL}), we define the augmented deformation problem
\begin{align*}
    \cS_Q=(S\cup Q, \{\cD_v\}_{v\in S} \cup \{\cD_v^{\square,\psi|_{G_{F_v}}}\}_{v\in Q}, \psi).
\end{align*}
For each $v\in Q$, let $\Del_v= k_v^\times(p)^2$ where $k_v^\times(p)$ is the maximal $p$-power quotient of $k_v^\times$. Set $\Del_Q =\prod_{v\in Q}\Del_v$. Let $\fa_Q$ denote the augmentation ideal of $\cO[\Del_Q]$. Then there is a canonical local ring homomorphism $\cO[\Del_Q] \ra R_{\cS_Q}^S$ such that $R_{\cS_Q}^S/\fa_Q \simeq R_{\cS}^S$. 
Since $\rbar$ is odd and vast, \cite[Corollary 7.6.3]{BCGP-ab_surf-2018abelian} shows that for each $n\ge 1$, there exists a Taylor--Wiles datum  $Q_n$ disjoint from $S$ such that
\begin{itemize}
    \item $\abs{Q_n}=q$;
    \item $q_v \equiv 1 \mod p^n$ for each $v\in Q_n$;
    \item there exists a local $\cO$-algebra surjection
    \begin{align*}
        \varphi_n^{\psip}: R_{\cS}^{S,\loc}\DB{x_1,\dots, x_g} \epi R_{S_{Q_n}}^S
    \end{align*}
    with $g:=2q-4[F:\Q] + \abs{S} -1$.
\end{itemize}

We define open compact subgroups $U^{p}_1(Q_n)$ of $\cG(\A^{\infty,p}_F)$ by setting $U^{p}_1(Q_n)_v = U^{p}_v$ if $v\notin Q\cup S$ and $U^{p}_1(Q_n)_v = \Iw_1(v)$ if $v\in Q$. 
We define a $\cG(\cO_p)$-patching datum in the sense of \cite[Definition 4.3.5]{EL}
\begin{align*}
    \PR{S_\infty, R_\infty^{\psip}, (R_n^{\psip}, \varphi_n^{\psip},\{M_r^{\psip}(H)_n\}_{r\ge 1, H\le_{\mathrm{c.o.}}\cG(\cO_p)}, \al_n^{\psip})_{n\ge 1},\{M_r^{\psip}(H)_0\}_{r\ge 1,H\le_{\mathrm{c.o.}}\cG(\cO_p)}  }
\end{align*}
by setting
\begin{itemize}
    \item $R_\infty^\psip:=R_{\cS}^{S,\loc}\DB{X_1,\dots, X_g}$; 
    \item $R_n^\psip:=R_{\cS_{Q_n}}^S$ with $S_\infty$-algebra structure induced by $S_\infty\epi \cO[\Del_{Q_n}]\otimes_\cO \cT_S \ra R^S_{\cS_{Q_n}}$;
    \item $M_r^\psip(H)_0 := [S_\chi(H\cdot U^{p},\cO/\varpi^r)_{\fm_{\rbar,\chi}^{S}, \zeta_R}]^\vee$ where the subscript $\zeta_R$ denotes the $\zeta_R$-coinvariant for the action of $\prod_{v\in R}T_\der(k(v))$; and
    \item for $n\ge 1$,
    \begin{align*}
        M_r^\psip(H)_n &:=[S_{\chi}(H\cdot U^{p}_1(Q_n), \cO/\varpi^r)_{\fm_{\rbar,\chi}^{S\cup Q_n},\fm_{Q_n},\zeta_R}]^\vee\otimes_{R_{\cS_{Q_n}}}R_n^\psip \\
         \al_n^\psip&:M_r^\psip(H)_n /\fa_\infty \simeq M_r^\psip(H)_0
    \end{align*}
    where $S_{\chi}(H\cdot U^{p}_1(Q_n), \cO/\varpi^r)_{\fm_{\rbar,\chi}^{S\cup Q_n},\fm_{Q_n},\zeta_R}$ is viewed as an $R_{\cS_{Q_n}}$-module via the usual Hecke action and Proposition \ref{prop:gal-rep-valued-in-Hecke}, $\fm_{Q_n}$ denotes the kernel of $\otimes_{v\in Q_n}\cO[T(F_v)/T(\cO_v)_1] \ra \F$ sending
    \begin{align*}
       T(\cO_v)/T(\cO_v)_1 \mapsto 1,  \be_0(\varpi_v) \mapsto \chi_v(\varpi_v) , \be_1(\varpi_v) \mapsto \ov{\al}_{v,1}, \textrm{and } \be_2(\varpi_v) \mapsto \ov{\al}_{v,1}\ov{\al}_{v,2},
    \end{align*}
    and the isomorphism $\al_n^\psip$ follows from \cite[\S2.4.29]{BCGP-ab_surf-2018abelian}.
\end{itemize}
Fix a nonprincipal ultrafilter $\cF\subset 2^{\N}$. By \cite[Lemma 4.3.9]{EL}, 
\begin{align*}
    M_{\rbar,\infty}^\psip:= \varprojlim_{r,H}\cU_\cF(\{M_r^{\psip}(H)_n\otimes_{S_\infty}S_\infty/\fm_{S_\infty}^r\}_{n\ge 1})
\end{align*}
is a finitely generated projective $S_\infty\DB{\GSp_4(\cO_p)}$-module with compatible $S_\infty[\GSp_4(F_p)]$-action and the action of $S_\infty$ factors through the map $S_\infty \ra R_\infty^\psip$ induced by $S_\infty \ra R_n^\psip$. We also let $M_{\rbar,\infty}^{\psip}$ denote the covariant functor
\begin{align*}
    \Rep^\psip_\cO(\GSp_4(\cO_p)) & \ra \Mod(R_\infty^\psip) \\
    V &\mapsto \Hom_{\cO[\GSp_4(\cO_p)]}(M_{\rbar,\infty}^{\psip}, V^\vee)^\vee.
\end{align*}

\begin{proof}[Proof of Theorem \ref{thm:existence-patching-functor}(1)]
We can assume that the level $U'$ is sufficiently small by shrinking it if necessary. We apply the above construction for the given $\rbar$, $\chi$, $S=S_{\rbar}\cup S_U$, $R=\emptyset$, and $\cD_v = \cD_v^{\square,\psi|_{G_{F_v}}}$ for each $v\in S$ and set $M_\infty^{\psip} := M_{\rbar,\infty}^{\psip}$. Since $M_\infty^\psip$ is projective over $\cO\DB{\GSp_4(\cO_p)}$, $M_\infty^\psip$ is an exact functor. Let $(\lam+\eta,\tau)$ be a type compatible with $\psip$ and $\sig \in \JH(\osig(\lam,\tau))$. 
It follows from \cite[Theorem A]{BG} and \cite[Proposition 7.4.7, 7.4.8]{BCGP-ab_surf-2018abelian} that $\dim S_\infty = \dim R^{\lam +\eta,\tau,\psip}_\infty$. Then $M_\infty^{\psip}(\sig^\circ(\lam,\tau))$ is defined over $R^{\lam +\eta,\tau,\psip}_\infty$ by Theorem \ref{thm:inertial-LLC} and Proposition \ref{prop:gal-rep-valued-in-Hecke}. It is maximal Cohen--Macaulay by the usual commutative algebra argument (e.g.~\cite[Lemma 4.18]{6author}). The maximal Cohen--Macaulayness of $M_\infty^{\psip}(\sig)$ over $R^{\lam +\eta,\tau,\psip}_\infty/\varpi$ follows similarly. 
Finally, for a Serre weight $\sig$, we have $\sig\in W(\rbar)$ if and only if $M_\infty^{\psip}(\sig)\neq 0$ because by construction,
\begin{align*}
    (M_\infty^{\psip}(\sig)/\fm_\infty)^\vee \simeq S_{\chi}(U,\sig^\vee)_{\fm_{\rbar,\chi}^P}. &\qedhere
\end{align*}
\end{proof}

\begin{proof}[Proof of Theorem \ref{thm:existence-patching-functor}(2)]
    
When $\rhobar$ is tame (resp.~maximally non-split of niveau 1 and weight $\sig$), we call it the tame case (resp.~the maximally non-split case). In both cases, we apply \cite[Lemma 4.4.5]{EL} and the existence of potentially diagonalizable lifts (Theorem \ref{thm:crys-lift}) to obtain a totally real field $F$, a continuous representation $\rbar: G_F \ra \GSp_4(\F)$, and a Hecke character $\chi$ such that
\begin{itemize}
    \item $F_v \simeq K$ for all $v|p$;
    \item $\rbar|_{G_{F_v}}\simeq \rhobar$;
    \item $\rbar$ is unramified at all finite places not dividing $p$;
    \item $\rbar$ satisfies Taylor--Wiles conditions; and
    \item $(\rbar,\chi)$ is potentially diagonalizably automorphic of level unramified outside $p$.
\end{itemize}

Fix a place $w\in S_p$. Let $\psi_w\in \Psi(\rhobar)$ such that $\psi_w \eps^{-a}$ is finite for $a\in \Z$. By Remark \ref{rmk:sim-char-aut} and \cite[Lemma 4.4.3]{EL}, we can and do assume that $\psi:= \chi_{p,\iota}\eps^{3}$ satisfies $\psi_w\simeq \psi|_{G_{F_v}}$ for all $v\in S_p$. Since $\rbar$ is vast and tidy, there exists a place $v_0\notin S_p$ as in \S\ref{subsec:alg-aut-forms} such that 
no two eigenvalues of $\rbar(\Frob_{v_0})$ have ratio $q_{v_0}$ (\cite[\S7.7]{BCGP-ab_surf-2018abelian}). 
We take $U^p\le \cG(\A_F^{\infty,p})$ by setting $U_v = \cG(\cO_v)$ for $v \neq v_0$ and $U_{v_0}= \Iw_1(v_0)$. 
We apply the construction in \S\ref{subsubsec:construction-of-patching-functors} by taking $S= S_p\cup\{v_0\}$, $R=\emptyset$, $\cD_v = \cD_v^{\square,\psi_v}$ for each $v\in S$ and obtain a functor $M_{\rbar,\infty}^{\psip}$. 

For $a\in \Z$ as above, let $k$ be the unique integer such that $a=2k+b$ where $b\in \{3,4\}$. We define $\lam_a\in X_*(\uT^\vee)$ as
\begin{align*}
    \lam_a = \left\{  \begin{array}{cc}
        k(1,1;2) & \text{if $b=3$} \\
        (1,1;1) + k(1,1;2) &  \text{if $b=4$}
    \end{array}\right.
\end{align*}
We let $\tau_a$ be a tame inertial type such that 
$\til{w}(\rhobar^\ss,\tau_a) = t_{u(\lam_a+\eta)}$ for a fixed $u\in \uW$ depending on $\sig$ (in the maximally non-split case, by taking $\tau_a$ as in Lemma \ref{lem:max-nonsplit-rhobar-lifting}). We take the ring $R^p$ in \S\ref{subsec:patching-functors} to be
\begin{align*}
    (\ctimes_{v\in S_p\bss\{w\}} R_{\rbar|_{G_{F_v}}}^{\lam_a+\eta,\tau_a,\psi_v})\ctimes_\cO R_{\rbar|_{G_{F_{v_0}}}}\DB{x_1,\dots, x_g}. 
\end{align*}
Note that in the tame case,  $R^p$  is formally smooth over $\cO$ by Theorem \ref{thm:longest-shape-def-ring} and \cite[Proposition 7.4.2]{BCGP-ab_surf-2018abelian}. Set  $R^{\psi_w}_\infty = R_{\rhobar}^{\square,\psi_w}\ctimes_\cO R^p$.
We define a functor 
\begin{align*}
    M_\infty^{\psi_w}: \Rep_\cO(\GSp_4(\cO_K)) & \ra \Mod(R^{\psi_w}_\infty) \\
     V & \mapsto \Hom_{\cO[\GSp_4(\cO_p)]}(M_{\rbar,\infty}^{\psi_p}, ((\otimes_{v\in S_p\bss \{w\},\cO} \sig^\circ(\lam_a,\tau_a)) \otimes_\cO V)^\vee )^\vee.
\end{align*}
Then $M_\infty^{\psi_w}$ is a weak patching functor by the same argument in the proof of Theorem \ref{thm:existence-patching-functor}(1). 

Note that, for each $v\in S_p\bss\{w\}$, there exists a potentially diagonalizable lift of $\rbar|_{G_{F_v}}$ of type $(\lam_a+\eta,\tau_a)$ by Theorem \ref{thm:longest-shape-def-ring} (for the tame case) and Lemma \ref{lem:max-nonsplit-rhobar-lifting} (in the maximally non-split case). If $\rho$ is a potentially diagonalizable lift of $\rhobar$ of type $(\lam+\eta,\tau)$ compatible with $\psi_w$, we can apply \cite[Theorem 3.4]{patrikis-tang-2020potential} to find a potentially diagonalizable lift $r$ of $\rbar$ with $\simc(r)=\psi$. By \cite[Lemma 4.4.4]{EL}, $r$ is automorphic, which implies that $M_\infty^{\psi_w}(\sig^\circ(\lam,\tau))$ is nonzero. In conclusion, $M_\infty^{\psi_w}$ is potentially diagonalizable.

We show that $M_\infty^{\psi_w}$ is minimal in the tame case. 
We already showed that $R^p$ is formally smooth. 
Since 
$R_\infty^{\lam_a+\eta,\tau_a,\psi_w}[1/p]$ is regular, $M_\infty^{\psi_w}(\sig^\circ(\lam,\tau))[1/p]$ is locally free over its support. Moreover, it has rank one; this can be checked at finite level, which follows from the multiplicity one assertion in Theorem \ref{thm:inertial-LLC}, the choice of $v_0$ and Hecke operators at $v_0$, and \cite[Proposition 2.4.3, 2.4.4]{BCGP-ab_surf-2018abelian}. 

Finally, we show that $M_\infty^{\psi_w}$ is independent of $\psi_w$ when restricted to $\Rep_\F^{\psi_w}(\GSp_4(\cO_K))$. By \cite[Corollary 2.6.5]{LLHLM-2020-localmodelpreprint}, we have
\begin{align*}
    \JH(\osig(\lam_a,\tau_a))\cap W^?(\rhobar^\ss|_{I_K}) = \{F_{\rhobar}(u)\}.
\end{align*}
For $V\in \Rep^{\psi_w}_\F(\GSp_4(\cO_K))$,  Theorem \ref{thm:weight-elim} implies 
\begin{align*}
    M_\infty^{\psi_w}(V) &= \Hom_{\cO[\GSp_4(\cO_p)]}(M_{\rbar,\infty}^{\psi_p}, (\otimes_{v\in S_p\bss \{w\}} \ov{\sig}(\lam_a,\tau_a) )^\vee \otimes_\cO V^\vee)  \\
    &= \Hom_{\cO[\GSp_4(\cO_p)]}(M_{\rbar,\infty}^{\psi_p}/\varpi, (\otimes_{v\in S_p\bss \{w\}} F_{\rhobar}(e))^\vee \otimes_\F V^\vee).
\end{align*} 
Note that the choices of Taylor--Wiles data and the nonprincipal ultrafilter in the construction in \S\ref{subsubsec:construction-of-patching-functors} relies only on $\rbar$ and is independent of $\psi_w$.  As a result, $M_{\rbar,\infty}^{\psip}/\varpi$ is independent of $\psi_w$. Together with the above equation, this implies the claim.
\end{proof}

Before we proceed to prove Lemma \ref{lem:pcrys-stack-nonempty} and \ref{lem:lowerbound-mod-p-crys-stack}, we first prove the following partial converse of Proposition \ref{prop:pcrys-nonzero}.

\begin{prop}\label{prop:pcrys-nonzero-iff}
Suppose that a tame $\rhobar: G_K \ra \GSp_4(\F)$ is 16-generic. Let $(\lam+\eta,\tau)$ be a type with a $(h_{\lam}+1)$-generic lowest alcove presentation of $\tau$. If $\rhobar$ has a 
lowest alcove presentation such that $\tilw(\rhobar,\tau)^* \in \Adm^\vee(\lam+\eta)$, then $R^{\lam+\eta,\tau}_{\rhobar}\neq 0$.
\end{prop}
 
\begin{proof}
By Theorem \ref{thm:existence-patching-functor}(2), there exists a potentially diagonalizable weak patching functor $M_\infty^{\psip}$ for $(\rhobar,\psip)$ where $\psip$ is compatible with $(\lam+\eta,\tau)$. We claim that $M_\infty^{\psip}(\sig)\neq 0$ for all $\sig\in W_{\obv}(\rhobar)$.
Since $\JH(\ov{\sig}(\lam,\tau))\cap W_{\obv}(\rhobar) \neq \emptyset$ by \cite[Proposition 2.6.6]{LLHLM-2020-localmodelpreprint}, the claim implies $M_\infty^{\psip}(\osig^\circ(\lam,\tau)) \neq 0$ and thus $R^{\lam+\eta,\tau}_{\rhobar}\neq 0$. To prove the claim, we can use Corollary 2.6.5 in \loccit, Theorem \ref{thm:longest-shape-def-ring} and \ref{thm:weight-elim}.
\end{proof} 
 
\begin{proof}[Proof of Lemma \ref{lem:pcrys-stack-nonempty}]
The proof is identical to \cite[Lemma 7.3.5]{LLHLM-2020-localmodelpreprint} using Corollary \ref{cor:BK-stack-nonempty} and Proposition \ref{prop:pcrys-nonzero-iff} instead of Corollary 5.3.5 and Proposition 6.2.7 in \loccit.
\end{proof}

\begin{proof}[Proof of Lemma \ref{lem:lowerbound-mod-p-crys-stack}]
This follows from the argument in the proof of Lemma 7.4.4 and Proposition 6.2.9 in \cite{LLHLM-2020-localmodelpreprint} using Theorem \ref{thm:existence-patching-functor}(2) and Lemma \ref{lem:max-nonsplit-rhobar-lifting} instead of Proposition 6.2.4 and Lemma 6.2.8 in \loccit.
\end{proof}

\section{Applications}\label{sec:applications}

\subsection{The Breuil--M\'ezard conjecture}\label{sub:BM}

We first state the \emph{geometric} and \emph{versal} Breuil--M\'ezard conjectures for $\GSp_4$ following \cite[\S 8.1]{LLHLM-2020-localmodelpreprint}.

Recall that the algebraic stack $\cX_{\Sym,\red}$ is equidimensional of dimension $4f$ and its irreducible components are labeled by Serre weights of $\GSp_4(k)$. We write $\Z[\cX_{\Sym,\red}]$ for the free abelian group generated by irreducible components of $\cX_{\Sym,\red}$. We call elements in $\Z[\cX_{\Sym,\red}]$ \emph{cycles} and $\cC_\sigma \in \Z[\cX_{\Sym,\red}]$ for a Serre weight $\sigma$ an \emph{irreducible cycle}. A cycle is \emph{effective} if it is a sum of irreducible cycles with non-negative coefficients.

For a type $(\lam+\eta,\tau)$, we define a cycle $\cZ_{\lam,\tau}:= \sum_{\sigma} \mu_\sigma (\cX_{\Sym,\F}^{\lam+\eta,\tau}) \cC_\sigma \in \Z[\cX_{\Sym,\red}]$ where $\mu_\sigma(\cX_{\Sym,\F}^{\lam+\eta,\tau})$ is the multiplicity of $\cC_\sigma$ as an irreducible component of $\cX_{\Sym,\F}^{\lam+\eta,\tau}$ in the sense of \cite[\href{https://stacks.math.columbia.edu/tag/0DR4}{Tag 0DR4}]{stacks-project}. 

\begin{conj}[$\cS$-restricted geometric Breuil--M\'ezard conjecture for $\GSp_4$]\label{conj:geom-BM} Let $\cS$ be a set of types. For each $\sigma \in \JH(\ov{\sigma}(\cS)):= \cup_{(\lam+\eta,\tau)\in \cS} \JH(\osig(\lam,\tau))$,
there exists an effective cycle $\cZ_\sigma\in \Z[\cX_{\Sym,\red}]$ such that for all $(\lam+\eta,\tau)\in \cS$, 
\begin{align*}
    \cZ_{\lam,\tau} = \sum_{\sigma\in \JH(\ov{\sigma}(\lam,\tau))} [\osig(\lam,\tau):\sig]\cZ_\sig.
\end{align*}
\end{conj}

Let $\rhobar \in \cX_{\Sym}(\F)$. We fix a versal ring $R_{\rhobar}^{\mathrm{ver}}$ for $\cX_{\Sym}$ at $\rhobar$. For example, we can take $R_{\rhobar}^{\mathrm{ver}} = R_{\rhobar}^\square$ the framed deformation ring. For a type $(\lam+\eta,\tau)$, we define
\begin{align*}
    \spf R_{\rhobar}^{\mathrm{ver},\lam+\eta,\tau} & := \spf R_{\rhobar}^{\mathrm{ver}}\times_{\cX_{\Sym}} \cX_{\Sym}^{\lam+\eta,\tau} \\ 
    \spf R_{\rhobar}^{\mathrm{alg}} & := \spf R_{\rhobar}^{\mathrm{ver}}\times_{\cX_{\Sym}} \cX_{\Sym,\red}.
\end{align*}
Note that the versal map $\spf R_{\rhobar}^{\mathrm{alg}} \ra \cX_{\Sym,\red}$ is effective, i.e.~arises from a map $i_{\rhobar}: \spec R_{\rhobar}^{\mathrm{alg}} \ra \cX_{\Sym,\red}$. The map $i_{\rhobar}$ induces a surjective map from the set of irreducible components of $\spec R_{\rhobar}^{\mathrm{alg}}$  to the set of irreducible components of $\cX_{\Sym,\red}$ containing $\rhobar$ (\cite[\href{https://stacks.math.columbia.edu/tag/0DRB}{Tag 0DRB}]{stacks-project}).  Define $\Z[\spec R_{\rhobar}^{\mathrm{alg}}]$ as the free abelian group generated by irreducible components of $\spec R_{\rhobar}^{\mathrm{alg}}$.  We interpret $\Z[\cX_{\Sym,\red}]$ and $\Z[\spec R_{\rhobar}^{\mathrm{alg}}]$ as spaces of functions on sets of irreducible components. Then we have a pullback \begin{align*}
    i^*_{\rhobar} : \Z[\cX_{\Sym,\red}] \ra \Z[\spec R_{\rhobar}^{\mathrm{alg}}].
\end{align*}
We define $\cZ_{\lam,\tau}(\rhobar) := i^*_{\rhobar}(\cZ_{\lam,\tau})$. The cycle $\cZ_{\lam,\tau}(\rhobar)$ is equal to the cycle corresponding to $\spec R_{\rhobar}^{\mathrm{ver},\lam+\eta,\tau}/\varpi$ (\cite[\href{https://stacks.math.columbia.edu/tag/0DRD}{Tag 0DRD}]{stacks-project}).

\begin{conj}[$\cS$-restricted versal Breuil--M\'ezard conjecture for $\GSp_4$]\label{conj:versal-BM} Let $\cS$ be a set of types. For each $\sigma \in \JH(\ov{\sigma}(\cS))$, there exists an effective cycle $\cZ_\sigma(\rhobar)\in \Z[\spec R_{\rhobar}^{
\mathrm{alg}}]$ such that for all $(\lam+\eta,\tau) \in \cS$, 
\begin{align*}
    \cZ_{\lam,\tau}(\rhobar) = \sum_{\sigma\in \JH(\ov{\sigma}(\lam,\tau))} [\osig(\lam,\tau):\sig]\cZ_\sig(\rhobar).
\end{align*}
\end{conj}

Let $\Lam \subset X_*(\uT^\vee)$ be a finite set of dominant cocharacters containing 0. We define $\cS_{\Lam}$ as the set of types $(\lam'+\eta,\tau)$ where $\lam' \le \lam$  and $\tau$ is $P_{\lam+\eta,e}P_{(6(h_\eta+1)-2)+h_{\lam+\eta}}$-generic for some $\lam \in \Lam$ (see Theorem \ref{thm:pcrys-local-model} for the definition of $P_{\lam+\eta}$). Let $\cP_\ss$ be the set of $\rhobar\in \cX_{\Sym}(\F)$ such that $\rhobar|_{I_K}$ is $(6(h_\eta+1)-2)$-generic tame inertial type for $K$ over $\F$. 
The following is the main result of this subsection.

\begin{thm}\label{thm:versal-BM}
Let $\rhobar\in \cP_{\ss}$. 
For each Serre weight $\sig$, there exists an effective cycle $\cZ_\sig(\rhobar) \in \Z[\spec R_{\rhobar}^{\mathrm{alg}}]$ such that
\begin{align*}
    \cZ_{\lam,\tau}(\rhobar) = \sum_{\sigma\in \JH(\ov{\sigma}(\lam,\tau))} [\osig(\lam,\tau):\sig]\cZ_\sig(\rhobar).
\end{align*}
for all $(\lam+\eta,\tau) \in \cS_{\Lam}$.  
\end{thm}

Before we prove the theorem, we discuss cycles in the fixed similitude deformation ring. Let $\psip\in \Psi(\rhobar)$. Suppose that $C\subset \spec R_{\rhobar}^{\square, \psip}$ is a $(4f+10)$-dimensional irreducible reduced closed subscheme contained in the union of $\spec R_{\rhobar}^{\lam+\eta,\tau,\psip}/\varpi$ for all choices of a dominant character $\lam\in X^*(\uT)$ and a (possibly trivial) tame inertial $L$-parameters $\tau$ such that $(\lambda+\eta,\tau)$ is compatible with $\psip\in \Phi(\rhobar)$. By Lemma \ref{lem:fixed-sim-formally-sm}, the unramified twist of $C$ gives a $(4f+11)$-dimensional irreducible reduced closed subscheme $\til{C} \subset \spec R_{\rhobar}^{\square}/\varpi$ contained in the union of $\spec R_{\rhobar}^{\lam+\eta,\tau}/\varpi$.

Let $\Z[R_{\rhobar}^{\mathrm{p\dash crys}}]$ be the free abelian group generated by $(4f+10)$-dimensional cycles in $\spec R_{\rhobar}^\square/\varpi$ supported on the union of $\spec R_{\rhobar}^{\lam+\eta,\tau,\psip}/\varpi$ for all choices of a dominant character $\lam\in X^*(\uT)$, a (possibly trivial) tame inertial $L$-parameters $\tau$, and a character $\psip\in \Psi(\rhobar)$ compatible with $(\lam+\eta,\tau)$. By the previous paragraph, there is a natural injective group homomorphism
\begin{align}\label{eqn:unram-twist-cycle}
    \til{(\cdot)} : \Z[R_{\rhobar}^{\mathrm{p\dash crys}}] \mono \Z[\spec R_{\rhobar}^\alg]
\end{align}
such that $\til{Z}(R_{\rhobar}^{\lam+\eta,\tau,\psip}/\varpi) = \cZ_{\lam,\tau}(\rhobar)$ for all $\lam$, $\tau$, and $\psip$ as above.

Let $R$ be a Noetherian $\F$-algebra. Given a $R$-module $M$, we let $Z(M)$ denote the cycle
\begin{align*}
    \sum_{\cC} \mu_\cC(M)\cC \in \Z[\spec R]
\end{align*}
where $\cC$ ranges over irreducible components of  $\spec R$, and $\mu_\cC(M)$ denotes the length of the module $M_{\fp_\cC}$ over $ R_{\fp_\cC}$ for the prime ideal ${\fp_\cC}$ corresponding to $\cC$.

\begin{proof}[Proof of Theorem \ref{thm:versal-BM}]
By Theorem \ref{thm:existence-patching-functor}(2), there is a minimal weak patching functor $M_\infty^{\psi_p}$ for $(\rhobar,\psip)$ for all $\psip\in \Psi(\rhobar)$.  
For a Serre weight $\sig$, we define $\cZ_\sig(\rhobar):= \til{Z}(M_\infty^{\psip}(\sig))$ for some choice of $\psip\in \Psi(\rhobar)$. Note that this is independent of the choice of $\psip$. 
For any $(\lam+\eta,\tau)\in \cS_\Lam$, $R_{\rhobar}^{\lam+\eta,\tau}$ is a domain by Theorem \ref{thm:pcrys-local-model}. Let $\psi_p\in \Psi(\rhobar)$ be a character   compatible with $(\lam+\eta,\tau)$. By \cite[Lemma 2.2.7, 2.2.10]{EG-geom_BM-MR3134019} and the definition of \eqref{eqn:unram-twist-cycle}, we have $\til{Z}(M_\infty^{\psip}(\osig(\lam,\tau))) = Z_{\lam,\tau}(\rhobar)$ and
\begin{align*}
    \til{Z}(M_\infty^{\psip}(\osig(\lam,\tau))) = \sum_{\sig\in \JH(\osig(\lam,\tau))} [\osig(\lam,\tau):\sig]\til{Z}(M_\infty^{\psip}(\sig)). &\qedhere
\end{align*}
\end{proof}

We explain how one can interpolate Theorem \ref{thm:versal-BM} to prove a version of Conjecture \ref{conj:geom-BM} and \ref{conj:versal-BM}. Since they follow from the axiomatic argument of \cite[\S8.3]{LLHLM-2020-localmodelpreprint}, we keep the discussion short.

Let $\hat{\cS}_{\Lam}$ be the union of $\cS_{\Lam}$ and the set of types of the form $(\eta,\tau)$ where $\tau$ is $2h_\eta$-generic. We also define $\hat{\cS}_{\Lam,\mathrm{elim}}$ to be the union of  $\hat{\cS}_{\Lam}$ and the set of all types of the form $(\eta,\tau)$. 

\begin{defn}[cf.~Definition 8.3.3 in \cite{LLHLM-2020-localmodelpreprint}]
Let $\sig \in \JH(\ov{\sig}(\hcS_\Lam))$ be a Serre weight.
\begin{enumerate}
    \item We say that $\sig$ \emph{$\hcS_\Lam$-covers} $\sig'$ if for $(\lam+\eta,\tau)\in \hcS_\Lam$ such that $\sig \in \JH(\ov{\sig}(\lam,\tau))$, $\cC_{\sig'}$ lies in $\cX_{\Sym,\F}^{\lam+\eta,\tau}$.
    \item We say that  $\sig$ is \emph{$(\cS_\Lam,\hcS_\Lam)$-generic} if for all Serre weights $\sig'$ such that $\sig$ $\hcS_\Lam$-covers $\sig'$, $\cC_{\sig'}$ does not lie in $\cX_{\Sym,\F}^{\lam+\eta,\tau}$ for any $(\lam+\eta,\tau)\in \hcS_\Lam \bss \cS_\Lam$.
\end{enumerate}
\end{defn}

In our case, we have a simple description of $\hcS_\Lam$-covering.

\begin{lemma}\label{lem:covering-vs-uparrow}
Suppose $\sigma$ and $\sig'$ are $3h_\eta$-deep Serre weights, and $\sig$ $\hcS_\Lam$-covers $\sig'$. Then $\sig' \uparrow \sig$.
\end{lemma}
\begin{proof}
Let $\tau$ be a $2h_\eta$-generic tame inertial $L$-parameter such that $\sig \in \JH(\osig(\tau)$.  
Since $\sig$ $\hcS_\Lam$-covers $\sig'$, $\cC_{\sig'} \subset \cX_{\Sym,\F}^{\eta,\tau}$. By Theorem \ref{thm:naive-BM}, we have $\sig'\in \JH(\osig(\tau))$. 
In particular, 
$\sig$ covers $\sig'$ in the sense of \cite[Definition 2.3.10]{LLHLM-2020-localmodelpreprint}. 
Then our claim follows from the equivalence between item (1) and (4) in Proposition 2.3.12 in \loccit~using 9-deepness of $\sig$ and $\sig'$. (Note that in our setup, if $\tilw\in \tilW^+$  and $\tilw_1 \uparrow \tilw$ for some $\tilw\in \tilW^+_1$, then $\tilw_1\in \tilW_1^+$. Thus $L(\pi^\mo(\tilw_1)\cdot (\omega-\eta))|_{\rmG}$ in item (4) in \loccit~is equal to $F(\pi^\mo(\tilw_1)\cdot (\omega-\eta)) = F_{(\tilw_1,\omega)}$ and $F_{(\tilw_1,\omega)} \uparrow F_{(\tilw,\omega)}$.)
\end{proof}

\begin{lemma}\label{lem:BM-system}
For any Serre weight $\sig$, there exists integers $n_{\lam,\tau}^\sig$ such that 
\begin{align*}
    [\sig] - \sum_{(\lam+\eta,\tau)\in \hcS_{\Lam}} n_{\lam,\tau}^\sig [\ov{\sig}(\lam,\tau)]
\end{align*}
is supported on the union of $\cC_{\sig'}$ for $\sig'\notin \cup_{\rhobar\in \cP_{\ss}} W^?(\rhobar|_{I_K})$.  In other words, $\hcS_\Lam$ is a $(\cP_\ss,\hcS_{\Lam,\mathrm{elim}}$)-Breuil--M\'ezard system in the sense of \cite[Definition 8.3.3]{LLHLM-2020-localmodelpreprint}. 
\end{lemma}
\begin{proof}
This can be proven as \cite[Lemma 8.4.4]{LLHLM-2020-localmodelpreprint}. Note that $\sig'\notin \cup_{\rhobar\in \cP_{\ss}} W^?(\rhobar|_{I_K})$ implies that $\sig'$ is $(\cP_{\ss},\hat{\cS}_{\Lam,\mathrm{elim}})$-irrelevant in the sense of Definition 8.3.1 in \loccit~by the proof of Theorem \ref{thm:weight-elim}.
\end{proof}

Let $\cS_{\cP_{\ss},\Lam}\subset \cS_{\Lam}$ be the subset of $(\lam+\eta,\tau)$ such that  all $\sig\in \JH(\ov{\sig}(\lam,\tau))$ are $(\cS_\Lam,\hcS_\Lam)$-generic.

\begin{thm}\label{thm:geom-BM}
For each $(\cS_\Lam,\hcS_\Lam)$-generic $\sig$ in $\JH(\ov{\sig}(\cS_{\cP_\ss,\Lam}))$, there exists a unique effective cycle $\cZ_\sig$ in $\Z[\cX_{\Sym,\red}]$ with the support contained in $\CB{\cC_\kappa| \kappa \uparrow \sig}$ and for each $\rhobar\in \cP_\ss$, $i^*_{\rhobar}(\cZ_\sig) = \cZ_\sig(\rhobar)$. Moreover, for $(\lam+\eta,\tau)\in \cS_{\cP_\ss,\Lam}$,  we have
\begin{align*}
    \cZ_{\lam,\tau} = \sum_{\sig\in \JH(\ov{\sig}(\lam,\tau)) } [\ov{\sig}(\lam,\tau):\sig] \cZ_\sig.
\end{align*}
In particular, Conjecture \ref{conj:geom-BM} holds for $\cS = \cS_{\cP_\ss,\Lam}$.
\end{thm}

\begin{proof}
For $\sig\in \JH(\osig(\hcS_\Lam))$, let $\tr_{\sig,\hcS_\Lam}$ be the idempotent endomorphism of $\Z[\cX_{\Sym,\red}]$ which maps $\cC_{\sig'}$ to itself if $\sig$ $\hcS_\Lam$-covers $\sig'$ and to 0 otherwise. Following the notation in Lemma \ref{lem:BM-system}, we define
\begin{align*}
    \cZ_\sig := \tr_{\sig,\hcS_\Lam}\PR{\sum_{(\lam+\eta,\tau)\in \hcS_\Lam} n^\sig_{\lam,\tau} \cZ_{\lam,\tau}}.
\end{align*}
Then the remaining proof is essentially identical to \cite[Theorem 8.3.5(2)]{LLHLM-2020-localmodelpreprint} using a generalization of \cite[Lemma 8.2.2]{LLHLM-2020-localmodelpreprint} to our setting for $\cP=\cP_{\ss}$. Its assumption is satisfied by a generalization of Lemma 8.4.9 in \loccit, which follows from Theorem \ref{thm:mod-p-local-model} in our case. The claim on the support of $\cZ_{\sig}$ follows from Lemma \ref{lem:covering-vs-uparrow}.
\end{proof}

Given a polynomial $f\in \Z[X_1,X_2,X_3]$ and $\omega \in X^*(T)\simeq \Z^3$, define
\begin{align*}
    f^\omega(X_1,X_2,X_3) := \prod_{\nu = (\nu_1,\nu_2;\nu_3) \in \Conv(\omega)}f(X_1-\nu_1, X_2 -\nu_2, X_3 - \nu_3)
\end{align*}
We also define
\begin{align*}
    P_{\cP_\ss, \Lam,e} =  \prod_{\lam \in \Lam} \prod_{\jj} {P}_{\eta,e}^{(\lam+\eta-w_0(\eta))_j}P_{7(h_\eta+1)-3}^{(\lam+\eta-w_0(\eta))_j}.
\end{align*}

The following two corollaries are immediate from Theorem \ref{thm:geom-BM} using an obvious generalization of Lemmas 8.4.11 and 8.5.1 in \cite{LLHLM-2020-localmodelpreprint} (see also, Corollaries 8.4.12 and 8.5.2 in \loccit).

\begin{cor}\label{cor:geom-BM-poly-gen}
For each Serre weight $\sig$,  there exists an effective cycle $\cZ_\sig$ in $\Z[\cX_{\Sym,\red}]$ with support contained in $\CB{\cC_\kappa| \kappa \uparrow \sig}$ such that for any $\lam\in \Lam$ and  any tame inertial type $\tau$ for $K$  with a lowest alcove presentation $(s,\mu_\eta)$ with $P_{\cP_\ss, \Lam,e}$-generic $\mu$,
\begin{align*}
    \cZ_{\lam,\tau} = \sum_{\sig\in \JH(\ov{\sig}(\lam,\tau)) } [\ov{\sig}(\lam,\tau):\sig] \cZ_\sig.
\end{align*}
\end{cor}

\begin{cor}\label{cor:versal-BM-poly-gen}
Let $\rhobar: G_K \ra \GSp_4(\F)$ be a continuous representation. Suppose that $\rhobar^\ss|_{I_K}$ has a lowest alcove presentation $(s,\mu-\eta)$ with $P_{\cP_\ss, \Lam,e}^{\eta}P_{m}$-generic $\mu$ for $m=\max\{2h_{\Lam}+2h_\eta,6(h_\eta+1)-2\}$. Then for each Serre weight $\sig$, there exists a cycle $\cZ_{\sig}(\rhobar)\in \Z[\spec R_{\rhobar}^{\alg}]$ such that for any $\lam\in \Lam$ and any tame inertial type $\tau$ for $K$,
\begin{align*}
    Z(R^{\lam+\eta,\tau}_{\rhobar}/\varpi) = \sum_{\sig} [\osig(\lam,\tau):\sig] \cZ_\sig(\rhobar).
\end{align*}
\end{cor}

\subsection{The weight part of Serre's  conjecture}\label{sub:SWC}
Let $F,\chi,\rbar$ be as in \S\ref{subsec:alg-aut-forms}. We further assume that $\rbar$ is automorphic of some weight $\mu$ and level $U$. Recall that $W(\rbar)$ is the set of modular Serre weights of $\rbar$. The following conjecture is due to Gee--Herzig--Savitt (\cite{GHS-JEMS-2018MR3871496}).

\begin{conj}\label{conj:SWC-tame}
If $\rbar|_{G_{F_v}}$ is tame and sufficiently generic at $v|p$, then $W(\rbar)=W^?(\rbar_p|_{I_{\Qp}})$.
\end{conj}

We say that a Serre weight $\sigma$ is \emph{generic} if $\sigma$ is $(\cS_{\{0\}}, \hcS_{\{0\}})$-generic. Note that if $\sigma$ is generic, then it is $(\cS_{\Lam},\hcS_\Lam)$-generic for any finite set of dominant cocharacters $\Lam \subset X_*(\uT^\vee)$ (cf.~\cite[Lemma 8.4.8]{LLHLM-2020-localmodelpreprint}). We let $W_{\gen}(\rbar)$ be the subset of generic Serre weights in $W(\rbar)$.

\begin{defn}
Let $\rhobar: G_{\Qp} \ra \LuG(\F)$ be an $L$-parameter. 
We define $W^{\BM}_{\gen}(\rhobar)$ to be a set of $\sig=\otimes_{v\in S_p} \sig_v$ where  $\sig_v$ is a generic Serre weight such that $\rhobar_v$ is contained in the support of $\cZ_{\sig_v}$ defined in Theorem \ref{thm:geom-BM}.
\end{defn}

Following \cite[\S9.1]{LLHLM-2020-localmodelpreprint}, we formulate a generalization of Conjecture \ref{conj:SWC-tame} for $\rbar$ not necessarily tame at places above $p$.

\begin{conj}\label{conj:SWC-gen}
The set $W_\gen(\rbar)$ is equal to $W^{\BM}_\gen(\rbar_p)$.
\end{conj}

We also define $W_g(\rbar_p)=\{\sig=\otimes_{v\in S_p}\sig_v \mid \sig_v \in W_g(\rbar_p|_{G_{F_v}})\}$, a set of Serre weights of $\GSp_4(\cO_F/p)$, where $W_g(\rbar_p|_{G_{F_v}})$ is the set defined in Proposition \ref{prop:geom-SW-EG-stack}. Note that any $\sig \in W_g(\rbar_p)$ is $3h_\eta$-deep.

Let $\cX_{\Sym,F_p}:= \prod_{v\in S_p,\cO}\cX_{\Sym,F_v}$. By Theorem \ref{thm:irred-comp-pgma-stack}, there is a bijection between irreducible components in the underlying reduced substack $\cX_{\Sym,F_p,\red}$ and isomorphism classes of Serre weights of $\GSp_4(\cO_F/p)$. If $\sig=\otimes_{v\in S_p} \sig_v$ is a Serre weight of   $\GSp_4(\cO_F/p)$, we denote its corresponding irreducible component by $\cC_{\sig} := \prod_{v\in S_p, \cO}\cC_{\sig_v}$.

\begin{lemma}\label{lem:irred-comp-EGstack-to-versal}
Let $\sig\in W_g(\rbar_p)$. Let $R_{\rbar_p}:= \ctimes_{v\in S_p}R_{\rbar|_{G_{F_v}}}^\square$ with a versal morphism
\begin{align*}
    \iota_{\rbar_p}: \spf R_{\rbar_p} \ra \cX_{\Sym,F_p}.
\end{align*}
If $\rbar_p|_{I_{\Qp}}$ is tame and $(6(h_\eta+1)-2)$-generic, then  $\iota^*_{\rbar_p}(\cC_{\sig})$ is an irreducible cycle corresponding to an irreducible subscheme $\cC_\sig(\rbar_p)\subset \spec R_{\rbar_p}/\varpi$. Moreover, if $\cC_{\sig}\subset \prod_{v\in S_p,\cO}\cX_{\Sym,F_v,\F}^{\lam_v+\eta_v,\tau_v}$, then $\cC_\sig(\rbar_p)$ is contained in $\spec R_{\rbar_p}^{\lam+\eta,\tau}/\varpi$.
\end{lemma}
\begin{proof}
This follows from Theorem \ref{thm:mod-p-local-model} and Proposition \ref{prop:unibranch-T-fixed-point}. 
\end{proof}

The following is the main result of this subsection.

\begin{thm}\label{thm:SWC}
There exists a polynomial $P(X_1,X_2,X_3)\in \Z[X_1,X_2,X_3]$ independent of $p$ such that if $\rbar:G_F \ra \GSp_4(\F)$ is automorphic, satisfies Taylor--Wiles conditions, and for each $v|p$, $\rbar|_{G_{F_v}}$ is tame with a $P$-generic lowest alcove presentation $(s_v,\mu_v-\eta)$, then Conjectures \ref{conj:SWC-tame} and \ref{conj:SWC-gen} hold for $\rbar$.
\end{thm}

\begin{proof}
We take $P$ to be the product $P_{6(h_\eta+1)-2} P_{2\eta,e}P_{\eta,e}^{\eta_0}$ and assume that $\mu_v$ is $P$-generic. Note that $W^?(\rbar_p|_{I_{\Qp}}) = W_g(\rbar_p)$ by Proposition \ref{prop:geom-SW-EG-stack}.

\subsubsection*{Step 1}
We show that $W(\rbar)\cap W_{\mathrm{obv}}(\rbar_p|_{I_{\Qp}})$ is nonempty. 
Let $\tau$ be the tame inertial $L$-parameter with a lowest alcove presentation 
such that $\tilw(\tau)=\tilw(\rbar_p) t_{-\eta -w_0\eta}$ (note that this condition and $P_{2\eta,e}P_{6(h_\eta+1)-2}$-genericity of $\mu_v$ imply that $\tau$ is $P_{2\eta,e}P_{6(h_\eta+1)-2}$-generic). By \cite[Lemma 2.6.7]{LLHLM-2020-localmodelpreprint}, $W^?(\rbar_p|_{I_{\Qp}})\subset \JH(\ov{\sig}(\eta,\tau))$.  
For $\psi_p\in \Psi(\rbar_p)$  compatible with $(2\eta,\tau)$ and a weak patching functor $M^\psip_\infty$ for $\rbar_p$ (which exists by Theorem \ref{thm:existence-patching-functor}), the automorphy of $\rbar$ and Theorem \ref{thm:weight-elim} imply that $M_\infty^{\psip} (\sig^\circ(\eta,\tau)) \neq 0$. By Theorem \ref{thm:pcrys-local-model} and $P_{2\eta,e}$-genericity of $\tau$, $M_\infty^{\psip}(\sig^\circ(\eta,\tau))$ has full support on $\spec R_{\rbar_p}^{2\eta,\tau,\psip}$. By Lemma \ref{lem:irred-comp-EGstack-to-versal}, we have $\cC_{\sig}(\rbar_p)\subset \spec R_{\rbar_p}^{2\eta,\tau}/\varpi$ for any $\sig \in W^?(\rbar_p|_{I_{\Qp}})$. Then there exists $\sig'\in W^?(\rbar_p|_{I_{\Qp}})$ such that $\til{Z}(M_\infty^{\psip}(\sig'))$ is supported on $\cC_\sig(\rbar_p)$. 
Since the support of $\til{Z}(M_\infty^{\psip}(\sig'))=\cZ_{\sig'}(\rhobar)$ is contained in $\{\cC_{\kappa}(\rbar_p)|\kappa \uparrow\sigma'\}$, this implies, in particular, that $\sig \uparrow \sig'$.
Take $\sig\simeq F(\lam) \in W_{\mathrm{obv}}(\rbar_p|_{I_{\Qp}})$ such that $\lam$ is in the highest $p$-restricted alcove. Then $\sig \uparrow \sig'$ implies that $\sig=\sig'$, and $\sig$ also belongs to $W(\rbar)$.

\subsubsection*{Step 2} We show that $W_{\mathrm{obv}}(\rbar_p|_{I_{\Qp}})\subset W(\rbar)$. Any $\sig \in W_{\mathrm{obv}}(\rbar_p|_{I_{\Qp}})$ is equal to $F_{\rbar_p}(w^\mo)$ for a unique $w\in \uW$. We write $\tau_{\sig}$ for the $19$-generic tame inertial type such that $\tilw(\rbar_p,\tau)=t_{w(\eta)}$. By \cite[Corollary 2.6.5]{LLHLM-2020-localmodelpreprint}, $\sig \in \JH(\ov{\sig}(\tau_{\sig}))$. There exists $\psip\in \Psi(\rbar_p)$ compatible with $(\eta,\tau_{\sig})$ and a weak patching functor $M_\infty^\psip$ for $\rbar_p$.

By the previous step, we can choose $\sig\in W_{\mathrm{obv}}(\rbar_p|_{I_{\Qp}}) \cap W(\rbar)$. Then we have $M_\infty^{\psip}(\sig^\circ(\tau_{\sig}))\neq 0$. Thus, there exists a potentially crystalline lift $r_0$ of $\rbar$ of type $(\eta_v,\tau_{\sig,v})$ at $v|p$. By Theorem \ref{thm:longest-shape-def-ring}, this shows that $\rbar$ is potentially diagonalizably automorphic. We now consider an arbitrary $\sig \in W_{\mathrm{obv}}(\rbar_p|_{I_{\Qp}})$. By applying Theorem \ref{thm:longest-shape-def-ring}, \cite[Theorem 1.1]{Booher-minimal-deformation}, and \cite[Theorem 3.4]{patrikis-tang-2020potential}, we can find a lift $r$ of $\rbar$ that is potentially diagonalizable of type $(\eta_v,\tau_{\sig,v})$ at $v|p$. By \cite[Lemma 4.4.4]{EL}, $r$ is automorphic and $M_\infty^{\psip}(\sig^\circ(\tau_{\sig}))\neq 0$. By reducing modulo $p$, $M_\infty^{\psip}(\sig)\neq 0$ and thus $\sig\in W(\rbar)$.

\subsubsection*{Step 3} We show that $W(\rbar)=W^?(\rbar_p|_{I_{\Qp}})$. 
Let $\sig\in W^?(\rbar|_{I_{\Qp}})$ corresponding to $(\tilw,\tilw_1)$ under the bijection in Proposition \ref{prop:JH-factor-W?}. Let $\tau$ be a tame inertial type such that $\tilw(\rbar_p,\tau)= (\tilw_h\tilw)^\mo w_0 \tilw_1$. By \cite[Proposition 8.6.3]{LLHLM-2020-localmodelpreprint}, $\sig\in \JH(\ov{\sig}(\tau)) \cap W^?(\rbar|_{I_{\Qp}})$, and if $\sig' \in \JH(\ov{\sig}(\tau)) \cap W^?(\rbar|_{I_{\Qp}})$ satisfies $\sig \uparrow \sig'$, then $\sig=\sig'$. 
Since $\JH(\ov{\sig}(\tau)) \cap W^?(\rbar|_{I_{\Qp}})$ is nonempty if and only if  $\JH(\ov{\sig}(\tau)) \cap W_{\mathrm{obv}}(\rbar|_{I_{\Qp}})$ is nonempty, the previous step implies that $M_\infty^{\psip}(\sig^\circ(\tau))\neq 0$ and it has full support on $\spec R_{\rbar_p}^{\eta,\tau}$, by Theorem \ref{thm:pcrys-local-model} and the $P^{\eta_0}_{\eta,e}$-genericity of $\mu$. Since $\cC_\sig(\rbar_p)\subset \spec R_{\rbar_p}^{\eta,\tau}$ by Lemma \ref{lem:irred-comp-EGstack-to-versal}, there exists  $\sig'\in \JH(\ov{\sig}(\tau))$ such that $M^\psip_\infty(\sig')$ is supported on $\cC_{\sig}(\rbar_p)$. As we argued at the end of Step 1, this implies $\sig\uparrow \sig'$. Then $\sig=\sig'$ by our choice of $\tau$ and thus $\sig \in W(\rbar)$. Thus Conjecture \ref{conj:SWC-tame} holds for $\rbar$.

\subsubsection*{Step 4} It remains to prove that $W_{\gen}(\rbar) = W^{\mathrm{BM}}_{\gen}(\rbar_p)$. Let $\sig$ be a  generic Serre weight. Then $\sig$ is $(\cS_{\Lam},\hat{\cS}_{\Lam})$-generic. Let $M_\infty^{\psip}$ be a weak patching functor for $\rbar_p$ constructed in Theorem \ref{thm:existence-patching-functor}(1). Then $M_\infty^{\psip}(\sig)\neq 0$ if and only if $\sig \in W(\rbar)$. On the other hand, $\sig \in W^{\BM}_{\gen}(\rbar_p)$ is equivalent to $\rbar_p \in \iota_{\rbar_p}^*(\cZ_{\sig})$, which is equivalent to $M_\infty^{\psip}(\sig)\neq 0$ by the construction of $\cZ_{\sig}$. Thus Conjecture \ref{conj:SWC-gen} holds for $\rbar$.
\end{proof}

\subsection{Modularity lifting in generic tamely potentially crystalline case}\label{sub:modularity-lifting}

We prove the following modularity lifting result.

\begin{thm}\label{thm:modularity-lifting}
Let $r:G_F \ra \GSp_4(E)$ be a continuous representation satisfying the following conditions:
\begin{enumerate}
    \item $r$ is unramified at all but finitely many places;
    \item $r|_{G_{F_v}}$ is potentially crystalline of type $(\lam_v+\eta_v,\tau_v)$ with a lowest alcove presentation $(s_v,\mu_v-\eta_v)$ with $P_{\lam+\eta,e}$-generic $\mu_v$;
    \item $\rbar$ satisfies Taylor--Wiles conditions;
    \item $\rbar|_{G_{F_v}}$ is tame at $v|p$; and
    \item $\rbar\simeq \rbar_{\pi,p,\iota}$ for some cuspidal automorphic representation  $\pi$ of $\GSp_4(\A_F)$ of weight $\lam$ and central character $\chi$ such that $\sig(\tau_v)$ is a $K$-type for $\pi$ at $v|p$. 
\end{enumerate}
Then $r\simeq r_{\til{\pi},p,\iota}$ for some cuspidal automorphic representation $\til{\pi}$  of $\GSp_4(\A_F)$ of weight $\lam$ central character $\chi$ such that $\sig(\tau_v)$ is a $K$-type for $\til{\pi}$ at $v|p$. 
\end{thm}

The theorem follows from the following lemma and a base change argument.

\begin{lemma}\label{lem:modularity-lifting-lemma}
Let $r:G_F \ra \GSp_4(E)$ be a continuous representation satisfying the following conditions:
\begin{enumerate}
    \item $r$ is unramified at all but finitely many places;
    \item if $r$ is ramified at a place $v\nmid p$, then $\rbar|_{G_{F_v}}$ is trivial, $r|_{G_{F_v}}$ has only unipotent ramification, and $q_v=1 \mod p$;
    \item $r|_{G_{F_v}}$ is potentially crystalline of type $(\lam_v+\eta_v,\tau_v)$ with a lowest alcove presentation $(s_v,\mu_v-\eta_v)$ with $P_{\lam+\eta,e}$-generic $\mu_v$;
    \item $\rbar$ satisfies Taylor--Wiles conditions;
    \item $\rbar|_{G_{F_v}}$ is tame at $v|p$; and
    \item $\rbar\simeq \rbar_{\pi,p,\iota}$ for some cuspidal automorphic representation  $\pi$ of $\GSp_4(\A_F)$ of weight $\lam$ and central character $\chi$ such that $\sig(\tau_v)$ is a $K$-type for $\pi$ at $v|p$ and for all finite places $v$ of $F$, $(\pi_v)^{\Iw(v)}\neq 0$. 
\end{enumerate}
Then $r\simeq r_{\til{\pi},p,\iota}$ for some cuspidal automorphic representation $\til{\pi}$  of $\GSp_4(\A_F)$ of weight $\lam$ central character $\chi$ such that $\sig(\tau_v)$ is a $K$-type for $\til{\pi}$ at $v|p$. 
\end{lemma}

\begin{proof}
We apply Taylor's Ihara avoidance argument. We apply the construction in \S\ref{subsubsec:construction-of-patching-functors} to $F$, $\chi$, $S=S_{r}\cup\{v_0\}$, $R=S\backslash (S_p\cup\{v_0\})$, and for each $v\in R$, take $\zeta_v$ as in item (1) (resp.~item (2)) of Proposition \ref{prop:Ihara-avoid-def} to get a weak patching functor which we denote by $M_{\infty}^{1,\psip}$ (resp.~$M_\infty^{\zeta,\psip}$). Note that by their construction, the two weak patching functors coincide on $\Rep_\F^{\psip}(\GSp_4(\cO_p))$. 
By the last assumption, Theorem \ref{thm:pcrys-local-model}, $P_{\lam+\eta,e}$-genericity of $\mu$, \cite[Proposition 7.4.2]{BCGP-ab_surf-2018abelian}, and Proposition \ref{prop:Ihara-avoid-def}(2),  $M_\infty^{\zeta,\psip}(\sig^\circ(\lam,\tau))$ has full support over $R^{\zeta,\lam+\eta,\tau,\psip}_\infty$. By the congruence between $M_\infty^{1,\psip}$ and $M_\infty^{\zeta,\psip}$, $M_\infty^{1,\psip}(\osig(\lam,\tau))$ has full support over $R^{1,\lam+\eta,\tau,\psip}_\infty/\varpi$. Then by  Proposition \ref{prop:Ihara-avoid-def}(1), $M_\infty^{1,\psip}(\sig^\circ(\lam,\tau))$ has full support over $R^{1,\lam+\eta,\tau,\psip}_\infty$. This proves the claim.
\end{proof}

\begin{proof}[Proof of Theorem \ref{thm:modularity-lifting}]
There exists a solvable extension of totally real fields $F'/F$ satisfying the following conditions:
\begin{enumerate}
    \item $F'/F$ is linearly disjoint from $F^{\ker \rbar}$;
    \item any place $v|p$ of $F$ splits completely in $F'$;
    \item if $r|_{G_{F'}}$ is ramified at a place $w$ of $F'$ lying over a place $v\nmid p$ of $F$, then $\rbar|_{G_{F'_w}}$ is trivial, $r|_{G_{F'_w}}$ has only unipotent ramification, and $q_w=1 \mod p$;
    \item there is a cuspidal automorphic representation ${\pi'}$ of $\GSp_4(\A_{F'})$ of weight $\lam'=(\lam_w)$, where $\lam_w=\lam_v$ for a place $v$ of $F$ and $w|v$, which is a base change of $\pi$, and for all finite places $w$ of $F'$, $(\pi'_w)^{\Iw(w)}\neq 0$ (here we use \cite[Proposition 4.13]{Mok-Comp-2014-MR3200667}). 
\end{enumerate}
Then by \cite[Lemma 8.3.2]{BCGP-ab_surf-2018abelian} (which easily generalizes to our setup), it suffices to show that there exists a cuspidal automorphic representation $\til{\pi}'$  of $\GSp_4(\A_{F'})$ of weight $\lam'$ such that for a place $w|p$ of $F'$ above a place $v$ of $F$, $\sig(\tau_v)$ is a $K$-type for $\til{\pi}'$ at $w$, and $r|_{G_{F'}}\simeq r_{\til{\pi}',p,\iota}$. This follows from Lemma \ref{lem:modularity-lifting-lemma}.
\end{proof}

\appendix

\numberwithin{equation}{section}
\numberwithin{figure}{section} 
\numberwithin{table}{section} 
\makeatletter
\let\c@equation\c@thmapp
\makeatother

\section{Torus-fixed points of certain affine Springer fibers}\label{sec:appdix}

Let $k$ be a field. The affine flag variety $\Fl$ associated with $\GSp_4$ over $k$ is defined by the quotient $\cI_k \bss L\GSp_4$ where $L\GSp_4$ is the loop group whose $R$-points are given by $\GSp_4(R\DP{v})$ for a $k$-algebra $R$, and $\cI_k$ is the standard Iwahori subgroup scheme in $L\GSp_4$. 

Let $\bfa\in \Lie\GSp_4(k)$ be a regular semisimple element. For simplicity, we assume that $\bfa$ is diagonal, and we write $\bfa = \diag(\bfa_1,\dots, \bfa_4)$. We define a sub-ind-scheme $L\GSp_4^{\nba_{\bfa}} \subset L\GSp_4$  such that for a $k$-algebra $R$, $g \in L\GSp_4(R)$ if and only if
\begin{align*}
    v \frac{dg}{dv} g^\mo + g \bfa g^\mo \in \frac{1}{v} \Lie \cI_k (R).
\end{align*}
Note that $L\GSp_4^{\nba_{\bfa}}$ is stable under left multiplication by $\cI_k$.

\begin{defapp}
    The \textit{deformed affine Springer fiber} $\Fl^{\nba_{\bfa}}$ associated with $\bfa$ is the subscheme of $\Fl$ given by the image of $L\GSp_4^{\nba_{\bfa}}$ under the quotient map $L\GSp_4 \epi \Fl$.
\end{defapp}

\cite{BA} considers certain irreducible components of the usual affine Springer fibers for $\GL_n$ and classifies their torus-fixed points. This was used to classify torus-fixed points in certain irreducible components of deformed affine Springer fibers for $\GL_n$ in \cite{LLHLM-2020-localmodelpreprint}. The goal of this appendix is to show how the main result of \cite{BA} generalizes to certain components of deformed affine Springer fibers for $\GSp_4$.

\begin{rmkapp}\label{rmk:convention}
Note that we use the left quotient by Iwahori to define the affine flag variety $\Fl$, as opposed to the right quotient in \cite{BA}. These conventions can be compared by taking the inverse at the level of $L\GSp_4$. 
\end{rmkapp}

Recall that $\mathrm{Fl}$ admits $T^\vee$-action induced by the right multiplication on $L\GSp_4$. The $T^\vee$-fixed points of $\mathrm{Fl}$ are given by the image of the map $\tilW \ra \mathrm{Fl}(k)$ sending $wt_{\lambda}\in \tilW$ to $v^{\phi(\lam)}\phi(w)^\mo  \cI_k$. Note that this is the composition of the map $(-)^*:\tilW \ra \tilW^\vee$ and the natural embedding of $\tilW^\vee$ into $\mathrm{Fl}(k)$.  For simplicity, we write the image of $\tilw$ under this map by $\tilw^{*}$. 

Let $\tilw\in \tilW$. Recall the definitions
\begin{align*}
    S^\circ(\tilw^*w_0)^{\nba_\bfa} &:= \cI_k \bss \cI_k\tilw^{*}w_0 \cI_k \cap \mathrm{Fl}^{\nba_{\bfa}} \\
    S^{\nba_\bfa}(\tilw^*w_0) &:= \ov{S^\circ(\tilw^*w_0)^{\nba_\bfa} }
\end{align*}
where the closure is taken inside $\Fl^{\nba_\bfa}$. 
If $\tilw \in \tilW_{1}^+$, then Theorem \ref{thm:schubert-cell-with-monodromy} shows that $ S^{\nba_\bfa}(\tilw^*w_0)$ is irreducible of dimension 4.

\begin{rmkapp}\label{rmk:affine_weyl_vs_ext_aff_weyl}
Suppose that $\tilw \in \tilW$ is equal to $\tilw'\delta$ for some $\tilw'\in W_a$ and $\del\in \Omega$. Since $\delta^{
*}$ normalizes $\cI_k$, we have $ S^{\nba_\bfa}(\tilw^*w_0)= \del^*S^{\nba_\bfa}(\tilw^{\prime *}w_0)$. 
\end{rmkapp}

\begin{thmapp}\label{thm:BA-GSp4}
Recall that $\bfa$ is regular semisimple. In addition, we assume that $\bfa$ is 3-generic. 
Then, the set of $T^\vee$-fixed points of $S^{\nba_\bfa}(\tilw^*w_0)^{T^\vee}$ is given by
\begin{align*}
   S^{\nba_\bfa}(\tilw^*w_0)^{T^\vee} = \{\tilz^{*} \mid \til{z} \in \tilW, \til{z} \le w_0 \tilw \}.
\end{align*}
\end{thmapp}

\begin{proof}
We follow the proof of \cite[Theorem 3.1]{BA} with minor modifications. 

By Remark \ref{rmk:affine_weyl_vs_ext_aff_weyl}, it suffices to consider $\tilw \in W_a$. Let $w$ be an element of $W^\vee$. By \cite[Proposition 4.3.1 and 4.3.6]{LLHLM-2020-localmodelpreprint}, we have the equality of subvarieties in $\Fl$
\begin{align*}
    S^{\nba_\bfa}(\tilw^*w_0) = S^{\nba_{w(\bfa)}}(\tilw^*w_0)w. 
\end{align*}
This implies that $\tilz \in S^{\nba_{w(\bfa)}}(\tilw^*w_0)$ if and only if $\tilz w \in S^{\nba_\bfa}(\tilw^*w_0)$. Thus, it suffices to show that 
\begin{align*}
    \{ \til{z}^{*}w_0 \mid \tilz \in \tilW^{+}, \til{z} \le \tilw \} \subset S^{\nba_\bfa}(\tilw^*w_0)^{T^\vee}.
\end{align*}

Recall that we have an isomorphism
\begin{align*}
    \tilw^* w_0 N_{\tilw^*w_0} \simeq \cI_k\bss \cI_k \tilw^* w_0 \cI_k.
\end{align*}
We define the subvariety $N_{\tilw^*w_0}^{\nba_\bfa} \subset N_{\tilw^*w_0}$ by the condition  $g\in N_{\tilw^*w_0}^{\nba_\bfa}$ if and only if $\tilw^* w_0 g \in L\GSp_4^{\nba_\bfa}$, or equivalently, 
\begin{align}\label{eqn:monodromy}
    v^2 \frac{dg}{dv}g^\mo + vg \bfa g^\mo \in \Ad ((\tilw^*w_0)^\mo)\PR{ \Lie \cI_k}.
\end{align}
Then we have $\tilw^* w_0 N_{\tilw^*w_0}^{\nba_\bfa} \simeq S^\circ(\tilw^*w_0)^{\nba_\bfa}$. 

By \cite[Lemma 4.7.1]{LLHLM-2020-localmodelpreprint} (which generalizes to $\GSp_4$ in an obvious manner), $\tilz^*w_0 \in S^{\nba_\bfa}(\tilw^*w_0)$ if and only if
\begin{align*}
    S^\circ(\tilw^*w_0)^{\nba_\bfa} \cap L^\mm\cG \tilz^*w_0 \neq \emptyset.
\end{align*}
We prove a slightly stronger condition. Let $\tilw_0 \in W_a^+$ be the element corresponding to the highest restricted alcove. Since $\tilw \uparrow \tilw_0$, we have
\begin{align*}
    \Ad((\tilw_0^*w_0)^\mo)(\Lie \cI_k)\cap \cI_k \subset \Ad((\tilw^*w_0)^\mo)(\Lie \cI_k) \cap \cI_k.
\end{align*}
As the left hand side of \eqref{eqn:monodromy} is in $\cI_k$, this implies that
\begin{align*}
    \cI_k\bss \cI_k \tilw^*w_0 N_{\tilw_0^*w_0}^{\nba_\bfa} \subset S^\circ(\tilw^* w_0)^{\nba_\bfa}
\end{align*}
and we claim that 
\begin{align}\label{eqn:intersection-nonempty}
    \cI_k\bss \cI_k \tilw^*w_0 N_{\tilw_0^*w_0}^{\nba_\bfa} \cap L^\mm\cG \tilz^*w_0 \neq \emptyset.
\end{align}
As explained in \cite[Appendix B.3]{BA}, this follows from the non-vanishing of the determinants of certain matrices. To explain this, we introduce some notation.

We identify $\mathrm{Fl}(k)$ with the set of full affine flags in a 4-dimensional $k\DP{v}$-vector space. Let $e_1, \dots , e_4$ be the standard basis of the $k\DP{v}$-vector space. For vectors $v_1,\dots,v_4$, we let $\RG{v_1,\dots,v_4}$ be the $k\DB{v}$-span of $v_1,\dots,v_4$ and $\RG{v_1,\dots,v_4}_{v^\mo}$ be the $k[1/v]$-span of $v_1,\dots,v_4$. Then \eqref{eqn:intersection-nonempty} holds if and only if there exists $M^{\tilw_0}\in w_0 N_{\tilw^*_0w_0}^{\nba_\bfa} w_0 (k)$ such that
\begin{align}\label{eqn:flag-intersection}
     \RG{ve_1,\dots,ve_l,e_{l+1},\dots, e_4} \tilw^* M^{\tilw_0} \cap \RG{ e_1,\dots,e_l,v^\mo e_{l+1},\dots, v^\mo e_4}_{v^\mo}\tilz^* = 0
\end{align}
for all $0\le l \le 3$. In our case, we can check these conditions directly.

We view both $\RG{ve_1,\dots,ve_l,e_{l+1},\dots, e_4} \tilw^*$ and $\RG{ e_1,\dots,e_l,v^\mo e_{l+1},\dots, v^\mo e_4}_{v^\mo}\tilz^*$ as $k$-vector spaces, and we want to show that the image of the former under the linear transform given by $M^{\tilw_0}$ has trivial intersection with the latter inside $k\DP{v}^4$.

Let $\omega_l = (1,\dots,1,0,\dots,0)$ be the cocharacter of $\GL_4$ whose first $l$ entries are 1 and the remaining entries are 0. Then $\RG{ve_1,\dots,ve_l,e_{l+1},\dots, e_4} \tilw^*$ contains $v^{\tilw(\omega_l)_1} \RG{e_1,\dots, e_4}$ which is stabilized by $M^{\tilw_0}$. Similarly, $\RG{ e_1,\dots,e_l,v^\mo e_{l+1},\dots, v^\mo e_4}_{v^\mo}\tilz^*$ contains $v^{\tilz(\omega_l)_4-1} \RG{e_1,\dots, e_4}_{v^\mo}$. Thus, to verify \eqref{eqn:flag-intersection}, we need to show that no non-trivial linear combination of $v^k e_j$'s, for $v^ke_j \in \RG{ve_1,\dots,ve_l,e_{l+1},\dots, e_4} \tilw^*$ and $k < \tilw(\omega_l)_1$, is mapped into $\RG{ e_1,\dots,e_l,v^\mo e_{l+1},\dots, v^\mo e_4}_{v^\mo}\tilz^*$ by $M^{\tilw_0}$.

Now we compute a block matrix whose $ij$-th block, for $1 \le i \le \tilw(\omega_l)_1 - \tilw(\omega_l)_4$ and $0 \le j\le \tilw(\omega_l)_1 - \tilz(\omega_l)_4 -1$, is a matrix representation of $M^{\tilw_0}$ with respect to the vectors
\begin{align*}
    v^{ \tilw(\omega_l)_4 +i -1}e_4, \dots, v^{ \tilw(\omega_l)_4 +i -1}e_{r_{li}}  &\ \ \text{in the domain} \\
    v^{\tilz(\omega_l)_4+j}e_{4}, \dots, v^{\tilz(\omega_l)_4+j}e_{s_{lj}} &\ \ \text{in the codomain,}
\end{align*}
where $1\le r_{li}\le 4$ is the smallest integer such that $v^{ \tilw(\omega_l)_4 +i -1}e_{r_{li}} \in \RG{ve_1,\dots,ve_l,e_{l+1},\dots, e_4} \tilw^*$ and $1\le s_{lj} \le 4$ is the smallest integer such that $v^{\tilz(\omega_l)_4+j}e_{s_{lj}} \notin \RG{ e_1,\dots,e_l,v^\mo e_{l+1},\dots, v^\mo e_4}_{v^\mo}\tilz^*$. 
Then \eqref{eqn:flag-intersection} holds if and only if this block matrix has a trivial kernel for some $k$-point of $M^{\tilw_0}$.

Explicitly, $N_{\tilw^*_0w_0}$ is given by 
\begin{align*}
    N_{\tilw^*_0w_0} = \pma{  1 & c_{12} & c_{13} + c_{13}' v & c_{14} + c_{14}'v + c_{14}''v^2 
    \\
    & 1 & c_{23} & c_{13} + c_{13}' v - c_{12}c_{23} \\
    & & 1 & -c_{12} \\
    & & & 1
    }
\end{align*}
and $N_{\tilw^*_0w_0}^{\nba_\bfa}$ is the locus given by the equations 
\begin{align*}
    c_{13}  &= \frac{\bfa_2 - \bfa_3}{\bfa_1 - \bfa_3} c_{12}c_{23} \\
    c_{14} &= - \frac{(\bfa_1-\bfa_2)(\bfa_2-\bfa_3)}{(\bfa_1-\bfa_4)(\bfa_2-\bfa_4)} c_{12}^2 c_{23}
    \\
    c_{14}' &= \frac{\bfa_2-\bfa_3-1}{\bfa_1-\bfa_4-1} c_{12}c_{13}'.
\end{align*}
We use the coordinate functions $c_{ij}$'s to describe $M^{\tilw_0}$ and the block matrices discussed above.

For simplicity, we compute block matrices for $\tilw = \tilw_0$ and $\tilz = e$. The remaining cases follow from a similar computation. When $l=0$, the block matrix is given by
\begin{align*}
\left(
\begin{tabular}{c c c c }
  &          & $c_{13}'$ & $c_{14}'$             \\ \hline
1 & $c_{12}$ & $c_{13}$  & $c_{14}$              \\
  & 1        & $c_{23}$  & $c_{13}-c_{12}c_{23}$ \\
  &          & 1         & $-c_{12}$            
\end{tabular}%
\right)
\end{align*}
where the rows correspond to $v^\mo e_4, e_4, e_3,e_2$ in the domain and  the columns correspond to $e_4,e_3,e_2,e_1$ in the codomain, respectively. Its determinant is given by $c'_{14}$.  

Similarly, the block matrices for $l=1,2,3$ are given by
\begin{align*}
\left(\begin{tabular}{ccc}
  &          & $c'_{13}$ \\ \hline
1 & $c_{12}$ & $c_{13}$  \\
  & 1        & $c_{23}$ 
\end{tabular}\right), \ \ \left(\begin{tabular}{cc|cccc}
  &          &   &          &           & $c''_{14}$            \\ \hline
1 & $c_{12}$ &   &          & $c'_{13}$ & $c'_{14}$             \\
  & 1        &   &          &           & $c'_{13}$             \\ \hline
  &          & 1 & $c_{12}$ & $c_{13}$  & $c_{14}$              \\
  &          &   & 1        & $c_{23}$  & $c_{13}-c_{12}c_{23}$ \\
  &          &   &          & 1         & $-c_{12}$            
\end{tabular}\right), \ \ \left(
\begin{tabular}{c|cccc}
1 &   &          & $c'_{13}$ & $c'_{14}$             \\
  &   &          &           & $c'_{13}$             \\ \hline
  & 1 & $c_{12}$ & $c_{13}$  & $c_{14}$              \\
  &   & 1        & $c_{23}$  & $c_{13}-c_{12}c_{23}$ \\
  &   &          & 1         & $-c_{12}$            
\end{tabular}%
    \right)
\end{align*}
whose determinants are $c'_{13}$, $-c''_{14}$, and $-c'_{13}$, respectively. 
Therefore, \eqref{eqn:flag-intersection} for $\tilw= \tilw_0$ and $\tilz=e$ holds for $M^{\tilw_0}$ with $c'_{13},c'_{14}, c''_{14} \neq 0$. Such     $M^{\tilw_0}$ exists by the description of $N^{\nba_{\bfa}}_{\tilw_0^*w_0}$ and our assumptions on $\bfa$.
\end{proof}

\bibliography{mybib}{}
\bibliographystyle{alpha}

\end{document}